\documentclass[12pt,a4paper, reqno]{amsart}
\usepackage{}
\usepackage{txfonts}
\usepackage{color}
\usepackage{url}
\usepackage[all]{xypic}
\usepackage{graphicx}
\usepackage[utf8]{inputenc}
\usepackage{extarrows}
\usepackage{wasysym}
\usepackage{amsmath,amsthm}
\usepackage{fourier}
\usepackage{bbm}
\usepackage{amsfonts}
\usepackage{amsxtra}
\usepackage{amssymb}
\usepackage{amscd}
\usepackage{amsthm}
\usepackage{mathrsfs}
\usepackage{leqno}
\usepackage{multirow}
 \usepackage{color}
\usepackage{amsthm}  
\usepackage[backref=page]{hyperref}
\usepackage{amsmath}
\allowdisplaybreaks[4]

\numberwithin{equation}{section}
\newtheorem*{theorem*}{Theorem}
\newtheorem{lemma}{Lemma}[section]
\newtheorem{proposition}[lemma]{Proposition}
\newtheorem{remark}[lemma]{Remark}
\newtheorem{example}[lemma]{Example}

\newtheorem{theorem}[lemma]{Theorem}
\newtheorem{definition}[lemma]{Definition}

\newtheorem{corollary}[lemma]{Corollary}

\newtheorem*{question*}{Question}
\newtheorem*{assumption*}{Assumption}
\newtheorem*{axiom*}{Axiom}

\newtheorem*{theorem*1}{Theorem B}
\newtheorem*{theorem*2}{Theorem C}
\baselineskip 20pt \textwidth 18cm  \theoremstyle{plain}
\newtheorem*{proposition*1}{Proposition A}

\newcommand{\Hom}{\operatorname{Hom}}

\renewcommand{\Im}{\operatorname{Im}}

\newcommand{\ind}{\operatorname{ind}}

\newcommand{\Ind}{\operatorname{Ind}}

\newcommand{\Res}{\operatorname{Res}}

\newcommand{\Ha}{\operatorname{H}}
\newcommand{\Oa}{\operatorname{O}}

\newcommand{\sgn}{\operatorname{sgn}}

\newcommand{\C}{\mathbb C}

\newcommand{\N}{{\mathbb N}}
\newcommand{\Q}{\mathbb Q}
\newcommand{\R}{\mathbb R}
\newcommand{\Za}{\mathbb Z}

\newcommand{\T}{\mathbb T}
\newcommand{\Z}{\mathbb Z}

\newcommand{\Gal}{\operatorname{Gal}}
\newcommand{\GL}{\operatorname{GL}}

\newcommand{\GSp}{\operatorname{GSp}}

\newcommand{\SL}{\operatorname{SL}}
\newcommand{\SO}{\operatorname{SO}}
\newcommand{\Sp}{\operatorname{Sp}}

\newcommand{\SU}{\operatorname{SU}}
\newcommand{\U}{\operatorname{U}}

\newcommand{\Span}{\operatorname{Span}}

\newcommand{\Ree}{\operatorname{Re}}

\newcommand{\diag}{\operatorname{diag}}

\newcommand{\PGL}{\operatorname{PGL}}
\newcommand{\supp}{\operatorname{supp}}

\newcommand{\Arg}{\operatorname{Arg}}

\begin{document}
\title{Siegel modular  forms associated to  Weil representations: $\SL_2(\R)\&\GL_2(\R)$ cases }
\author{Chun-Hui Wang}
\address{School of Mathematics and Statistics\\Wuhan University \\Wuhan, 430072,
P.R. CHINA}
\email{cwang2014@whu.edu.cn}
\begin{abstract}
We investigate explicit modular forms of weights $1/2$ and $3/2$—classical, minus, and fermionic theta series—arising from the classical Weil representation associated to $\SL_2(\R)$ via the $2$-cocycles of Rao, Kudla, Perrin, Lion--Vergne and Satake--Takase. We reorganize these forms using (tensor) induction, and  subsequently extend our study to the similitude group $\GL_2(\R)$.
\end{abstract}
\maketitle
\setcounter{secnumdepth}{5}
\tableofcontents{}
\section{Introduction}
 \subsection{Notations and conventions}
 Let $\R$ and $\C$ denote  the usual real numbers and complex numbers. Let $i=\sqrt{-1}\in \C$. For an element $z=u+iv\in \C$ with $u,v\in \R$, denote $u=\Ree z$, $v=\Im z$, $\overline{z}=u-iv$ and $|z|^2=u^2+v^2$. The principal argument function $\Arg: \C \setminus \{0\} \to  [-\pi, \pi)$ assigns to each $z \in \C \setminus \{0\}$ its uniquely defined angle $\Arg z\in [-\pi, \pi)$. 
 Define the domains:  $\mathbb{H}=\{ x+yi\in \C\mid y> 0\}$ and   $\mathbb{H}^{\pm}=\{ x+yi\in \C\mid y\neq 0\}$.
Let $\U(\C)=\{ \cos t+ i\sin t \in \C\mid -\pi\leq t <\pi\}$.

 Let $(W=\R^2, \langle, \rangle)$ be a symplectic vector space  of  dimension $2$   over  $\R$, endowed with the symplectic form:
 \[\langle v, v'\rangle=x_1x_1^{'\ast}-x_1'x_1^{\ast},\] 
 for $v=(x_1, x_1^{\ast}), v'=(x_1', x_1^{'\ast})\in \R^2$.  Let $\{e_1=(1, 0),e_1^{\ast}=(0, 1)\}$ be a symplectic basis of $\R^2$. Let $X=\Span_\R\{ e_1\}$, $X^{\ast}=\Span_\R\{e_1^{\ast}\}$. Let $\Ha(W)=\R^2\oplus \R$ denote the  Heisenberg group, defined  by
\[(v, t)\cdot(v',t')=(v+v', t+t'+\tfrac{\langle v,v'\rangle}{2}),\]
for $ v, v'\in \R^2, t,t'\in \R$.  
In $\SL_2(\R)$, we denote  $P(\R)=\{ g=\begin{pmatrix}a & b\\ 0& a^{-1}\end{pmatrix} \in\SL_2(\R) \}$, $ P_{>0}(\R)=\{ g=\begin{pmatrix}a & b\\ 0& a^{-1}\end{pmatrix} \in P(\R)\mid a>0 \}$, $N(\R)=\{ g=\begin{pmatrix}1 & b\\ 0& 1\end{pmatrix} \in\SL_2(\R) \}$, $N_-(\R)=\{ g=\begin{pmatrix}1 & 0\\ c& 1\end{pmatrix} \in\SL_2(\R) \}$, $\SO_2(\R)=\{ g= \begin{pmatrix} \cos t & -\sin t\\ \sin t &  \cos t\end{pmatrix}\mid -\pi \leq t <\pi\}$. Let $\SL_2^{\pm}(\R)=\{ g\in \GL_2(\R)\mid \det g=1 \textrm{ or} -1\}$.  In $\SL_2^{\pm}(\R)$, let   $P^{\pm}_{>0}(\R)=\{ g=\begin{pmatrix} a & b\\ 0& a^{-1} \det g  \end{pmatrix} \in\SL^{\pm}_2(\R)\mid a>0\}$,  $\Oa_2(\R)=\{ g= \begin{pmatrix} \cos t & \sin t\\ \mp\sin t & \pm \cos t\end{pmatrix}\mid -\pi \leq t < \pi\}$,
 $\SL_2^{\pm}(\Z)=\{ g=\begin{pmatrix}a & b\\ c& d\end{pmatrix} \in \SL_2^{\pm}(\R)\mid a,b,c,d\in \Z\}$, $F_2=\langle s(-1)\rangle$ for $s(-1)=\begin{pmatrix} 1 & 0\\ 0& -1\end{pmatrix}\in \SL^{\pm}_2(\R)$. Let  $\Ha^{\pm}(W)=F_2\ltimes (\R\oplus \R \oplus \R)$, with the action: $(x, y; t)^{s(-1)}=(x,-y; -t)$. 

 Let $V=\C$, and let $\Ha(V)=\C\oplus \R$ denote another   Heisenberg group with the multiplication: $$(z, t)(z',t')=(z+z', t+t'+\tfrac{1}{2}\Im(\overline{z}z')).$$
 Let $\Gal(\C/\R)=\langle \sigma\rangle$ act on $\Ha(\C)$ by  $(z, t)^{\sigma}=(z^{\sigma}, -t)$. Form the semi-direct product  $\Gal(\C/\R)\ltimes \Ha(\C)$ with this action.
 
   Let $(\,\cdot\,,\cdot\,)_{\mathbb R}\colon \mathbb R\times \mathbb R\to\{\pm 1\}$ denote the Hilbert symbol(extending the classical definition for use), given by  
\[
(a,b)_{\mathbb R}= 
\begin{cases}
-1,& a<0\ \text{and}\ b<0,\\
1,& \text{otherwise}.
\end{cases}
\]
  For any real number \(a\), define  
\[
((a)) := 
\begin{cases}
0,& \text{if } a \in \mathbb Z,\\
a - \lfloor a \rfloor - \tfrac12,& \text{otherwise}.
\end{cases}
\]
For coprime integers \(c,d\) with \(c>0\), the Dedekind sum \(s(d,c)\) is defined by  
\[
s(d,c)=\sum_{k=1}^{c}\Bigl(\!\Bigl(\frac{k}{c}\Bigr)\!\Bigr)\Bigl(\!\Bigl(\frac{kd}{c}\Bigr)\!\Bigr).
\]
For elements in $\GL_2(\R)$, we denote by    $\omega=\begin{pmatrix}0 & -1\\ 1& 0\end{pmatrix}$, $u(b)=\begin{pmatrix}
  1&b\\
  0 & 1
\end{pmatrix}$,  $u_-(c)=\begin{pmatrix}
  1&0\\
  c & 1
\end{pmatrix}$, $h(a)=\begin{pmatrix}
  a&0\\
  0 & a^{-1}
\end{pmatrix} $, $h_{\epsilon}=\begin{pmatrix}
  1&0\\
  0& \epsilon
\end{pmatrix}$.

If $\psi$ is a non-trivial character of $\R$ and $a\in \R$,   we will  write   $\psi^a$ for the character: $t\longrightarrow \psi(at)$.  We will let  $\psi_0$ denote the fixed character of $\R$ defined as: $ t \longmapsto e^{2\pi it}$, for $t\in \R$. Let $\mu_n=\langle e^{\tfrac{2\pi i}{n}}\rangle$,  $ e^{\tfrac{2\pi i}{n}} \in \C^{\times}$. Let $\mathbb{T}=\{ e^{i\theta}\mid \theta \in \R\}=\U(\C)$. For a finite set  $G$, let $|G|$ denote the cardinality of $G$. On $\R$, let $du$ denote the Lebsegue measure. On $\C$, let  $dz=du+idv$, and $d(z)=dudv$, unless overstated. 

\subsection{Weil representations}
Let $\psi$ be a non-trivial continuous unitary character of $\R$.  According to the Stone-von Neumann's theorem, there exists  only one unitary irreducible complex representation of $\Ha(W)$ with central character $\psi$, known as the Heisenberg representation. Details are in Section  \ref{Weilmodel}. 

 According to Weil's work, the Heisenberg representation  can be extended to  a projective representation of $\SL_2(\R)$, and then  to an actual representation of a $\C^{\times}$-covering group over $\SL_2(\R)$. This group is called the Metaplectic group, and the representation is called the Weil representation.   In the real field case,  this special  central cover can descend  to an  $8$-degree or  $2$-degree cover over $\SL_2(\R)$. There are three common models for realizing the Weil representation: the Sch\"odinger model, the Lattice model and the Fock model.  Details are in Section \ref{Weilmodel}.
 
 Various cocycles are linked to  Weil representations. Each has unique characteristics and historical origins.  In the present context,  $\widetilde{c}_{X^{\ast}}$ is a $\mu_8$-valued cocycle on $\SL_{2}(\R)$, whereas $\overline{c}_{X^{\ast}}$ represents a $\mu_2$-valued cocycle  on the same group. Their difference stems from the constant $m_X^{\ast}$ 
and the relation established in  Lem. \ref{relacto}(3).   Neither of these cocycles is continuous.  For further details on this aspect, we refer the reader to Rao \cite{Ra}, Kudla \cite{Ku}, Perrin \cite{Pe}, and Lion-Vergne \cite{LiVe}. We mention that these works also treat the case of higher-rank symplectic groups.    In \cite{Ta1}, Takase introduced another real analytic cocycle $\widehat{c}_{z}$ linked to  the Fock model of the Weil representation, which  is useful in studying the irreducible $\SO_2(\R)$-components of the Weil representation. We reorganize their results in Sections \ref{TRCGSLR} and \ref{Weilmodel}. For our purpose, we extend the $2$-cocycles and Weil representations from $\SL_2(\R)$ to  $\GL_2(\R)$ in  Sections \ref{TrcgGLR} and \ref{chap4exWe}, following Barthel's work  \cite{Ba} on the $p$-adic field.  For a deeper analysis of the Gobal field case, we direct the reader to Waldspurger's seminal works \cite{Wa1, Wa2}.

Throughout this paper, lowercase symbols ($\widetilde{c}_{X^{\ast}}$, $\overline{c}_{X^{\ast}}$, $\widehat{c}_{z}$) indicate $2$-cocycles on $\SL_2(\R)$, while uppercase symbols ($\widetilde{C}_{X^{\ast}}$, $\overline{C}_{X^{\ast}}$, $\widehat{C}_{z}$) indicate their counterparts on $\GL_2(\R)$. The associated central extension groups are denoted by $\widetilde{\SL}_2(\R)$, $\overline{\SL}_2(\R)$, $\widehat{\SL}^z_2(\R)$, $\widetilde{\GL}_2(\R)$, $\overline{\GL}_2(\R)$, and $\widehat{\GL}^z_2(\R)$, respectively. When $z=z_0$, we  abbreviate $\widehat{\SL}_2(\R)$ and $\widehat{\GL}_2(\R)$ for $\widehat{\SL}^{z_0}_2(\R)$ and $\widehat{\GL}^{z_0}_2(\R)$, respectively. The notations for the corresponding central covering subgroups follow the same convention.

Lowercase $\pi$  always denotes the Weil representation of $\SL_2$, while uppercase $\Pi$ denotes that of $\GL_2$. 
  \subsection{ Modular forms of half-integral weight}\label{siegelmod} Theta functions, as is well known, have a long and rich history. In what follows, we mention only the papers and results most relevant to our work. For broader historical context and recent developments, we refer the reader to the introduction in \cite{Str}.

  Let $z\in \mathbb{H}$.  According to D. Zagier's book \cite[p.27]{Za}, the theta series, the minus theta series, and the fermionic theta series are defined respectively by:
  $$\theta(\tfrac{z}{2})= \sum_{n\in \Z} e^{i   \pi z n^2 },  \quad \theta^{M}(\tfrac{z}{2})=\sum_{n\in \Z} (-1)^{n}e^{i   \pi z n^2 }, \quad \theta^F(\tfrac{z}{2})= \sum_{n\in \Z} e^{i   \pi z(n+\tfrac{1}{2})^2} .$$
Moreover, these satisfy the identity(cf. \cite[pp. 27--30]{Za})
  \[\theta^{M}(\tfrac{z}{2})\cdot \theta(\tfrac{z}{2}) \cdot \theta^F(\tfrac{z}{2})=2\eta^3(z),\]
  where $\eta(z)$ is the Dedekind eta-function. By Shimura's classical  paper \cite{Sh} on modular forms of half-integral weight,  the theta series $\theta(z)$ is a modular form of weight $1/2$ for the group $\Gamma_0(4)$ satisfying the equality: 
  \begin{theorem}\label{shimura}
$\theta(\tfrac{az+b}{cz+d})=\nu(r)\sqrt{cz+d} \theta(z)$, for $r=\begin{pmatrix} a& b\\ c&d\end{pmatrix}\in \Gamma_0(4), \nu( r)= \left(\frac{c}{d}\right)\epsilon^{-1}_d$.
\end{theorem}
In Lion-Vergne's book \cite{LiVe}, the authors approach theta series via Weil representations. Indeed, in \cite[Sect. 2.4.13, Theorem]{LiVe}, they state the result for $\theta(\tfrac{z}{2})$ rather than $\theta(z)$.  According to the theory of modular forms(cf. \cite{DiSh}), these two functions differ only by a certain slash operator.     Let:
\begin{itemize}
\item[(1)] $\theta^{1/2}(z) =\theta(\tfrac{z}{2})=\sum_{n \in \mathbb{Z}}  e^{i\pi z n^2}$;
\item[(2)] $\Gamma_{\theta}=\{ g=\begin{pmatrix} a & b \\ c& d\end{pmatrix} \in \SL_2(\Z) \mid  a c \equiv 0(\bmod 2),  b d\equiv 0(\bmod 2) \}$;
\item[(3)] $\lambda(r)=\left\{ \begin{array}{cl} \left(\frac{2c}{d}\right)\epsilon^{-1}_d, & r=\begin{pmatrix} a& b\\ c& d\end{pmatrix}\in \Gamma(2),\\ \left(\frac{2d}{c}\right)\epsilon_ce^{-\tfrac{i \pi}{4}}(c,d)_{\R},& r=\begin{pmatrix} a& b\\ c& d\end{pmatrix} \in \Gamma_{\theta} \setminus \Gamma(2).\end{array} \right.$
\end{itemize}
\begin{theorem}[{\cite[Sect. 2.4.13, Theorem]{LiVe}}]
$\theta^{1/2}(\tfrac{az+b}{cz+d})=\lambda(r) \sqrt{cz+d}  \,\theta^{1/2}(z)$, for $ r=\begin{pmatrix} a& b\\ c& d\end{pmatrix}\in \Gamma_{\theta}$, $z\in \mathbb{H}$.
\end{theorem}
Note that for $r\in \Gamma(2)$, the above result is comparable with Shimura's Theorem \ref{shimura}.  The function $\lambda(r)$ above  is adopted from \cite[p.42, Prop.2.3.3.]{CoSt}, which coincides with the expression given by Lion-Vergne in \cite[Sect.2.4.9, Thm.]{LiVe}.

Let $N$ be  a positive integer and  $\chi$ is a  primitive Dirichlet character modulo $N$. We say $\chi$  even if $\chi(-1)=1$ and odd if $\chi(-1)=-1$. When  $N=1$, we have  $\chi\equiv1$.  Let $t$ be a positive integer.
 \begin{definition} 
 \begin{itemize}
\item[(1)] If  $\chi$ is even, define $\theta^{+}_{\chi, t}(z) = \sum_{n \in \mathbb{Z}} \chi(n)\, e^{2\pi i tzn^2  }$.
\item[(2)] If $\chi$ is odd, define $\theta^{-}_{\chi, t}(z) = \sum_{n \in \mathbb{Z}} \chi(n)\, n e^{2\pi i t  zn^2 }$.
\end{itemize}
\end{definition}
\begin{theorem}[{\cite[Sect. 2.4.19, Theorem]{LiVe}}]\label{lionver}
Let  $ r=\begin{pmatrix} a& b\\ c& d\end{pmatrix}\in \Gamma_0(4N^2t)$, $z\in \mathbb{H}$. Then:
\begin{itemize}
\item[(1)]  $\theta^{+}_{\chi, t}( \tfrac{az+b}{cz+d}) = \chi(d)  \left(\frac{t}{d}\right)\epsilon^{-1}_d \left(\frac{c}{d}\right) (cz+d)^{1/2}  \,\theta^{+}_{\chi, t}(z)$.
\item[(2)] $\theta^{-}_{\chi, t}( \tfrac{az+b}{cz+d}) = \chi(d)  \left(\frac{t}{d}\right)\epsilon^{-1}_d \left(\frac{c}{d}\right) (cz+d)^{3/2}   \,\theta^{-}_{\chi, t}(z)$.
\end{itemize}
\end{theorem}
Using the same framework, and aided by the works of Satake–Takase  on Fock models for Weil representations (see \cite{Sa1, Sa2, Ta1, Ta2}), we extend the above results to the minus theta series and the fermionic theta series. The corresponding multiplier systems are computed explicitly.   Let $\kappa=1$ or $3$, $\mathfrak{w}=\begin{pmatrix} t& 0\\ 0&1\end{pmatrix}$. 
 \begin{definition} 
 \begin{itemize}
 \item If  $\chi$ is even, define:
\begin{itemize}
\item[(1)] $\theta^{1/2}_{\chi,t}(z) = \sum_{n \in \mathbb{Z}} \chi(n)\, e^{i\pi tz n^2}$, 
\item[(2)] $\theta^{M, 1/2}_{\chi,t}( z)= \sum_{n \in \mathbb{Z}} (-1)^{n}\chi(n)\, e^{i\pi tz n^2}$, 
\item[(3)] $\theta^{F, 1/2}_{\chi,t}( z)=\sum_{n\in \Z}  \chi(n+\tfrac{N+1}{2}) e^{i   \pi tz (n+\tfrac{1}{2})^2 }$, for odd $N$.
\end{itemize}
 \item If $\chi$ is odd, define:
\begin{itemize}
\item[(1)] $\theta^{3/2}_{\chi,t}(z) = \sum_{n \in \mathbb{Z}} \chi(n)\, n e^{i\pi tz n^2}$,
\item[(2)] $\theta^{M, 3/2}_{\chi,t}( z)= \sum_{n \in \mathbb{Z}} (-1)^{n}\chi(n)\,n e^{i\pi tz n^2} $,
\item[(3)] $\theta^{F, 3/2}_{\chi,t}( z)=\sum_{n\in \Z}  \chi(n+\tfrac{N+1}{2}) (n+\tfrac{1}{2})e^{i   \pi tz (n+\tfrac{1}{2})^2 } $, for odd $N$.
\end{itemize}
\end{itemize}
\end{definition}
\begin{definition}
\begin{itemize}
\item[(1)]  $\nu_{\theta,t}(r)= \left\{ \begin{array}{lcl}  \left(\frac{2ct}{d}\right)\epsilon^{-1}_d, & & r=\begin{pmatrix} a& tb\\ ct^{-1} & d\end{pmatrix}\in \Gamma(2),\\ \left(\frac{2ct}{d}\right)\epsilon^{-1}_de^{-\tfrac{i \pi}{4}}(-c,d)_{\R}, & & r=\begin{pmatrix} a& tb\\ ct^{-1} & d\end{pmatrix}\begin{pmatrix} 0& -1\\ 1& 0\end{pmatrix}\in  \Gamma(2)\omega.\end{array}\right.$
\item[(2)] $\nu_{\theta^M,t}(r)= \left\{ \begin{array}{ll} \left(\frac{2ct}{d}\right)\epsilon^{-1}_d (\tfrac{-1}{d})^{\tfrac{c}{2t}}(-i)^{\tfrac{c}{2t}}, & r=\begin{pmatrix} a& tb\\ ct^{-1} & d\end{pmatrix}\in \Gamma(2),\\ \left(\frac{2ct}{d}\right)\epsilon^{-1}_d (\tfrac{-1}{d})^{\tfrac{c}{2t}}(-i)^{\tfrac{c}{2t}} e^{\tfrac{i \pi}{4}}(-1,dt-c)_{\R}(-c,dt-c)_{\R},  & r=\begin{pmatrix} a& tb\\ ct^{-1} & d\end{pmatrix}\begin{pmatrix} 1& 0\\ -1& 1\end{pmatrix}\in  \Gamma(2)u_-(-1).\end{array}\right.$
\item[(3)] $\nu_{\theta^F,t}(r)= \left\{ \begin{array}{lcl}\left(\frac{2ct}{d}\right)\epsilon^{-1}_d (\tfrac{-1}{d})^{\tfrac{bt}{2}}i^{\tfrac{bt}{2}}, & & r=\begin{pmatrix} a& tb\\ ct^{-1} & d\end{pmatrix}\in \Gamma(2),\\ \left(\frac{2ct}{d}\right)\epsilon^{-1}_d (\tfrac{-1}{d})^{\tfrac{bt}{2}}i^{\tfrac{bt}{2}}e^{\tfrac{i \pi}{4}}, & & r=\begin{pmatrix} a& tb\\ ct^{-1} & d\end{pmatrix}\begin{pmatrix} 1& 1\\0& 1\end{pmatrix}\in  \Gamma(2)u(1).\end{array}\right.$
    \end{itemize}
    \end{definition}
\begin{proposition*1}
\begin{itemize}
\item[(1)]   $\theta^{\kappa/2}_{\chi,t}( \tfrac{az+b}{cz+d})=\nu_{\theta,t}(r)(cz+d)^{\kappa/2}\chi(d)\theta^{\kappa/2}_{\chi,t}(z)$, for  $r=\begin{pmatrix} a& b\\ c & d\end{pmatrix}\in  \mathfrak{w}^{-1}[ \Gamma_{\theta}\cap \Gamma_0(N^2)]\mathfrak{w}$.
\item[(2)]   $\theta^{M, \kappa/2}_{\chi,t}( \tfrac{az+b}{cz+d})=\nu_{\theta^M,t}(r)(cz+d)^{\kappa/2}\chi(d)\theta^{M, \kappa/2}_{\chi,t}(z)$, for  $r=\begin{pmatrix} a& b\\ c & d\end{pmatrix}\in  \mathfrak{w}^{-1}[\Gamma^0(2)\cap \Gamma_0(N^2)]\mathfrak{w}$. 
    \item[(3)]   $\theta^{F, \kappa/2}_{\chi,t}( \tfrac{az+b}{cz+d})=\nu_{\theta^F,t}(r)(cz+d)^{\kappa/2}\chi(d)\theta^{F, \kappa/2}_{\chi,t}(z)$, for  $r=\begin{pmatrix} a& b\\ c & d\end{pmatrix}\in  \mathfrak{w}^{-1}[  \Gamma_0(2)\cap \Gamma_0(N^2)]\mathfrak{w}$.
\end{itemize}
\end{proposition*1}
\begin{proof}
See Section \ref{examples}, Propositions \ref{examthetaN}, \ref{examthetaMNt},\ref{examthetaFNt}. 
\end{proof}  
The first item aligns with Lion-Vergne's theorem (above Thm.~\ref{lionver}) modulo the slash operator corresponding to $\begin{pmatrix} 2& 0\\ 0&1\end{pmatrix}$. To establish this result, we examine and compare three models of the Weil representation—the Schrödinger model, the lattice model, and the Fock model. These modular forms can be interpreted as intertwining operators, a perspective dating back to \cite{Ma}. The underlying methodology can also be applied to the study of Siegel modular forms for higher-rank metaplectic groups, as demonstrated in \cite{Wa3}. To compute the explicit multiplier systems,  we use the corresponding $2$-cocycles  associated  with   $\GL_2(\R)$ to construct the relevant slash operators, following the approach of Lion-Vergne \cite{LiVe}. One can see Section  \ref{examples} for the details.
\subsection{The results}
In the Langlands program, the theory is formulated not only in terms of modular forms, but also  in the language of automorphic representations (cf. \cite{Gan}, \cite{Ku2} for advanced treatments).  Under  the same spirit,  we recast the result above in the language of representation theory and  generate it by  making use of induced representations and tensor-induced representations.    Our treatment of representations is based on  the works of  Bushnell-Henniart \cite{BuHe},  Curtis-Reiner \cite{CuRe},  Kaniuth-Taylor  \cite{KaTa}, and  Serre \cite{Ser}.

Let: 
$$\overline{\lambda}: \overline{\Gamma}_{\theta} \to \mathbb{T}; \quad [r,\epsilon] \longmapsto \lambda(r)\epsilon,$$
$$\overline{\lambda}_{\chi}:\overline{\Gamma}_{\theta} \cap \overline{\Gamma}_0(N^2) \to \mathbb{T}; \quad [r,\epsilon] \longmapsto \lambda(r)\chi(r)\epsilon,$$
$$\overline{\gamma}_{\chi}=\operatorname{Ind}_{\overline{\Gamma}_{\theta} \cap \overline{\Gamma}_0(N^2)}^{\overline{\Gamma}_0(N^2)} [\overline{\lambda}_{\chi}]^{-1}.$$

By Lemmas \ref{lambdaoverlin} and \ref{thetachichar}, both $\overline{\lambda}$ and $\overline{\lambda}_{\chi}$ are characters. Moreover, by Lemmas \ref{gammachiodd} and \ref{exenchigamma}(1), $\overline{\gamma}_{\chi}$ is irreducible if  $N$ is odd or $N\equiv 2\pmod{4}$; whereas by Lemma \ref{exenchigamma}(2), $\overline{\gamma}_{\chi}$ is reducible if $N\equiv 0\pmod{4}$.

\begin{itemize}
\item For odd $N$, define $\overline{M}_{q_1}=[u(N^2),1]$, $\overline{M}_{q_2}=[1,1]$, and $\overline{M}_{q_3}=[u_-(-N^2),1]$; for even $N$, define $\overline{M}_{q_1}=[u(1),1]$ and $\overline{M}_{q_2}=[1,1]$.

\item For odd $N$, define
$$\Theta^{\kappa/2}_{\chi}: \mathbb{H} \longrightarrow M_{1\times 3}(\mathbb{C}); \quad z\longmapsto \Big([\theta^{\kappa/2}_{\chi}]^{\overline{M}_{q_1}}(z),[\theta^{\kappa/2}_{\chi}]^{\overline{M}_{q_2}}(z),[\theta^{\kappa/2}_{\chi}]^{\overline{M}_{q_3}}(z)\Big).$$

\item For even $N$, define
$$\Theta^{\kappa/2}_{\chi}: \mathbb{H} \longrightarrow M_{1\times 2}(\mathbb{C}); \quad z\longmapsto \Big([\theta^{\kappa/2}_{\chi}]^{\overline{M}_{q_1}}(z),[\theta^{\kappa/2}_{\chi}]^{\overline{M}_{q_2}}(z)\Big).$$
\item  $V_{G\to H}$: the  transfer map from $G$ to $H$.
\end{itemize}
\begin{theorem*1}
Let $g\in G=\overline{\Gamma}_0(N^2)$, $h\in H=\overline{\Gamma}_{\theta} \cap  \overline{\Gamma}_0(N^2 )$, $z\in \mathbb{H}$, $\kappa=1$ or $3$.
\begin{itemize}
\item[\rm(1)] $\theta^{\kappa/2}_{\chi}(hz)=J_{\kappa/2}(h,z)\,\overline{\lambda}_{\chi}(h)\,\theta^{\kappa/2}_{\chi}(z)$.
\item[\rm(2)] $\Theta^{\kappa/2}_{\chi}(gz)=J_{\kappa/2}(g,z)\,\Theta^{\kappa/2}_{\chi}(z)\,\overline{\gamma}_{\chi}(g^{-1})$.
\item[\rm(3)] 
\begin{itemize}
\item[\rm(a)] If $N$ is odd, then $[\theta^{\kappa/2}_{\chi}]^{\otimes}(gz)=J_{\kappa/2}(g,z)^3\,[\theta^{\kappa/2}_{\chi}]^{\otimes}(z)\,\overline{\lambda}_{\chi}(V_{G\to H}(g))$.
\item[\rm(b)] If $N$ is even, then $[\theta^{\kappa/2}_{\chi}]^{\otimes}(gz)=J_{\kappa/2}(g,z)^2\,[\theta^{\kappa/2}_{\chi}]^{\otimes}(z)\,\overline{\lambda}_{\chi}(V_{G\to H}(g))$.
\end{itemize}
\end{itemize}
\end{theorem*1}
\begin{proof}
For part (1), see Theorem \ref{mainthm1'}; for part (2), see Theorems \ref{maintheorem1} and \ref{maintheorem1even}; for part (3), see Theorem \ref{Twistedbypro2}.
\end{proof}
In Section \ref{GammathetaN2PM}, we extend the above results to the similitude group $\GL_2(\R)$. The main result is as follows.
 \begin{theorem*2}
 Let $g\in G= \overline{\Gamma}^{\pm}_0(N^2 )$, $h\in H=\overline{\Gamma}^{\pm}_{\theta} \cap  \overline{\Gamma}^{\pm}_0(N^2 )$, $z\in \mathbb{H}^{\pm}$, $\delta=\delta^+$ or $\delta^-$.
 \begin{itemize}
\item[\rm(1)] $\theta^{\pm,\delta}_{\chi}(hz)=J_{\delta}(h,z)\,\theta^{\pm,\delta}_{\chi}(z)\,\overline{\lambda}^{\pm}_{\chi}(h^{-1})$.
\item[\rm(2)] $\Theta^{\pm,\delta}_{\chi}(gz)=J_{\delta}(g,z)\,\Theta^{\pm,\delta}_{\chi}(z)\,\overline{\gamma}^{\pm}_{\chi}(g^{-1})$.
\end{itemize}
\end{theorem*2}
\begin{proof}
See Theorem  \ref{mainthm1234}  for part (1) and Theorem \ref{maintheoremplus}  for part (2).
\end{proof}
The notation and conventions are similar.  The main difference in the similitude case is that we employ the Weil--Deligne group to define the automorphic factor $J_{\delta}(-,z)$ in Section \ref{dewiaf}. Additionally, both $\theta^{\pm,\delta}_{\chi}(z)$ and $\Theta^{\pm,\delta}_{\chi}(z)$ are matrix-valued functions. In Lemma \ref{gammatwodimplus}, we prove that $\overline{\lambda}^{\pm}_{\chi}$ is irreducible. Similarly, by Lemma \ref{irrdgamachi}, $\overline{\gamma}^{\pm}_{\chi}$ is irreducible if $N$ is odd or $N\equiv 2\pmod{4}$, and reducible if $N\equiv 0\pmod{4}$.
\section{The relevant covering groups: $\SL_2(\R)$ case}\label{TRCGSLR}
 Let $\psi=\psi_0^{\mathfrak{e}}$ be a non-trivial character of $\R$, for some $\mathfrak{e}\in \R^{\times}$. In this section, we use the following notations:  $g_1, g_2, g_3=g_1g_2\in \SL_2(\R)$ with $g_i=\begin{pmatrix}
a_i & b_i\\
c_i& d_i\end{pmatrix}$, $g=\begin{pmatrix} a& b\\ c& d\end{pmatrix}\in \SL_2(\R)$, $y\in \R^{\times}$, $z\in \mathbb{H}$.
\subsection{ The cocycles $\overline{c}_{X^{\ast}}$ and $\widetilde{c}_{X^{\ast}}$}  Let us define:
\begin{itemize}
\item[(1)] $x(g)=\left\{\begin{array}{lr} d\R^{\times 2} & \textrm{ if } c=0,\\
c\R^{\times 2} & \textrm{ if } c\neq 0.\end{array}\right.$
\item[(2)]$\gamma(\psi)=\psi_0( \tfrac{\sgn \mathfrak{e}}{8})=e^{\tfrac{(\sgn \mathfrak{e})\pi i}{4}}.$
\item[(3)]$\gamma(a, \psi)= \tfrac{\gamma(\psi^a)}{\gamma(\psi)}.$
\item[(4)]$m_{X^{\ast},\psi}(g)=\left\{\begin{array}{lr} e^{\tfrac{i \pi (\sgn \mathfrak{e})}{4}[1-\sgn(d)]}& \textrm{ if } c=0,\\e^{\tfrac{i \pi (\sgn \mathfrak{e})}{4}[- \sgn(c)]}& \textrm{ if } c\neq 0.\end{array}\right.$
\end{itemize}
If $\psi=\psi_0$, we write $m_{X^{\ast}} $ for $m_{X^{\ast},\psi}$ for simplicity. 
\begin{lemma}\label{relacto}
\begin{itemize}
\item[(1)] $\overline{c}_{X^{\ast}}(g_1, g_2)=(x(g_1), x(g_2))_\R(-x(g_1)x(g_2), x(g_3))_\R$.
\item[(2)] $\widetilde{c}_{X^{\ast}}(g_1, g_2)= \gamma(\psi^{\tfrac{1}{2}c_1c_2c_3})=e^{\tfrac{i \pi (\sgn \mathfrak{e})}{4} \sgn(c_1c_2c_3)}.$
\item[(3)] $\overline{c}_{X^{\ast}}(g_1, g_2)=m_{X^{\ast},\psi}(g_1g_2)^{-1} m_{X^{\ast},\psi}(g_1) m_{X^{\ast},\psi}(g_2) \widetilde{c}_{ {X^{\ast}}}(g_1,g_2)$.
\end{itemize}
\end{lemma}
\begin{proof}
See \cite[p.364, Remark]{Ra} for (1), \cite[p.359, Coro.4.3]{Ra} for (2) and  \cite[p.360,(5.1)]{Ra} for (3).
\end{proof}
\begin{lemma}
Let $p\in P_{>0}(\R) $. Then $\widetilde{c}_{X^{\ast}}(p, g)=1=\widetilde{c}_{X^{\ast}}( g,p)$ and $\overline{c}_{X^{\ast}}(p, g)=\overline{c}_{X^{\ast}}( g,p)=1$. 
\end{lemma}
\begin{proof}
The first equalities follow from Lemma (2) above.  
For the second, observe that
\begin{align*}
\overline{c}_{X^{\ast}}(p,g)
&=(x(p),x(g))_{\mathbb R}\,(-x(p)x(g),x(pg))_{\mathbb R}=(-x(p)x(g),x(p)x(g))_{\mathbb R}=1,
\\[2mm]
\overline{c}_{X^{\ast}}(g,p)
&=(x(g),x(p))_{\mathbb R}\,(-x(g)x(p),x(gp))_{\mathbb R}=(-x(g)x(p),x(g)x(p))_{\mathbb R}=1.
\end{align*}
\end{proof} 
\subsubsection{$K=\SO_2(\R)$ } Note that  $\U(\C) \simeq \SO_2(\R)$, through  the mapping $ (\cos t + i\sin t) \longrightarrow \begin{pmatrix} \cos t & -\sin t\\ \sin t & \cos t\end{pmatrix}$. Following \cite[pp.75-77]{LiVe},  let us define   the  function $u: \R \longrightarrow \Z$ by
$$ u(t)=\left\{ \begin{array}{ll}
2k & \textrm{ if } t=k\pi, \\
2k+1  & \textrm{ if } k\pi <t<(k+1)\pi.
\end{array}\right.$$
\begin{lemma}
$\sgn[\sin(t_1)\sin(t_2)\sin(t_1+t_2)]=u(t_1)+u(t_2)-u(t_1+t_2)$.
\end{lemma}
\begin{proof}
See \cite[pp.75-77]{LiVe}.
\end{proof}
Let us modify the above $u$ by a character of $\R$ in the following way:
$$u'(t)=u(t)+ \tfrac{2}{\pi}t.$$
\begin{lemma}\label{chap6theq}
\begin{itemize}
\item[(1)] $u'(t+k\pi)=u'(t)+4k$, for $k\in \Z$.
\item[(2)] $\sgn[\sin(t_1)\sin(t_2)\sin(t_1+t_2)]=u'(t_1)+u'(t_2)-u'(t_1+t_2)$.
\end{itemize}
\end{lemma}
\begin{proof}
1) Note that $u(t+k\pi)=u(t)+2k$. So $u'(t+k\pi)=u(t+k\pi)+\tfrac{2}{\pi}(t+k\pi)=u(t)+ 2k+ \tfrac{2}{\pi}t+2k=u'(t)+4k$.\\
2) Since $u'$ is different from $u$ by a character of $\R$, the result follows from the above lemma.
\end{proof}
\begin{lemma}\label{chap6thetaf}
 The restriction of $[\widetilde{c}_{X^{\ast}}]$ on $\U(\C)$ is trivial, with an explicit  trivialization:
$$
\widetilde{s}:  \U(\C) \longrightarrow \C^{\ast};
  \cos t+ i\sin t\longmapsto   e^{-\tfrac{i\pi \sgn(\mathfrak{e}) }{4}u'(t)},$$
such that $\widetilde{c}_{X^{\ast}}(g_1, g_2) =\widetilde{s}(g_1)^{-1}\widetilde{s}(g_2)^{-1}\widetilde{s}(g_1g_2)$, for $g_i\in \U(\C)$.
\end{lemma}
\begin{proof}
It is well defined by the above lemma (1).
\begin{align*}
\widetilde{c}_{X^{\ast}}(e^{i t_1}, e^{i t_2}) & =e^{\tfrac{i \pi \sgn(\mathfrak{e})}{4} \sgn[\sin(t_1)\sin(t_2)\sin(t_1+t_2)]}\\
&\stackrel{ \textrm{ Lem. \ref{chap6theq}(2)}}{=} e^{\tfrac{i \pi \sgn(\mathfrak{e})}{4} u'(t_1)}e^{\tfrac{i \pi\sgn(\mathfrak{e})}{4} u'(t_2)} e^{\tfrac{-i \pi\sgn(\mathfrak{e})}{4} u'(t_1+t_2)}.
\end{align*}
\end{proof}
\begin{lemma}\label{chap6trivial}
The restriction of $[\overline{c}_{X^{\ast}}]$ on $\U(\C)$ is  trivial,  with an explicit trivialization:
$\overline{s}(e^{it})=e^{\sgn(\mathfrak{e})\tfrac{it}{2}}$, for $-\pi\leq t< \pi$,
such that $\overline{c}_{X^{\ast}}(g_1, g_2) =\overline{s}(g_1)^{-1}\overline{s}(g_2)^{-1}\overline{s}(g_1g_2)$, for $g_i\in \U(\C)$.
\end{lemma}
\begin{proof}
By Lem. \ref{relacto}(3) and the above lemma,  we have:
\begin{align*}
\overline{c}_{X^{\ast}}(g_1, g_2)&=m_{X^{\ast},\psi}(g_1g_2)^{-1} m_{X^{\ast},\psi}(g_1) m_{X^{\ast},\psi}(g_2) \widetilde{c}_{ {X^{\ast}}}(g_1,g_2)\\
&=[\widetilde{s}(g_1g_2)m_{X^{\ast},\psi}(g_1g_2)^{-1}] [\widetilde{s}(g_1)m_{X^{\ast},\psi}(g_1)^{-1}]^{-1} [\widetilde{s}(g_2)m_{X^{\ast},\psi}(g_2)^{-1}]^{-1}.
\end{align*}
 \[u(t)=\left\{ \begin{array}{rl}
0 & \textrm{ if } t=0, \\
1  & \textrm{ if } 0 <t<\pi,\\
-1& \textrm{ if } -\pi <t<0,\\
-2 & \textrm{ if } t=-\pi.
\end{array}\right.\]\[ \widetilde{s}(e^{it})=e^{-\sgn(\mathfrak{e})\tfrac{ it}{2}}\left\{ \begin{array}{rl}
1 & \textrm{ if } t=0, \\
e^{-\sgn(\mathfrak{e})\tfrac{i\pi  }{4}} & \textrm{ if } 0 <t<\pi,\\
e^{\sgn(\mathfrak{e})\tfrac{i\pi }{4}}& \textrm{ if } -\pi <t<0,\\
e^{\sgn(\mathfrak{e})\tfrac{i\pi  }{2}} & \textrm{ if } t=-\pi.
\end{array}\right.m_{X^{\ast},\psi}(e^{it})=\left\{ \begin{array}{rl}
1 & \textrm{ if } t=0, \\
e^{-\sgn(\mathfrak{e})\tfrac{i\pi  }{4}} & \textrm{ if } 0 <t<\pi,\\
e^{\sgn(\mathfrak{e})\tfrac{i\pi }{4}}& \textrm{ if } -\pi <t<0,\\
e^{\sgn(\mathfrak{e})\tfrac{i\pi  }{2}}  & \textrm{ if } t=-\pi.
\end{array}\right.\]
Hence:
\[\overline{c}_{X^{\ast}}(g_1, g_2)=\overline{s}(g_1)\overline{s}(g_2)\overline{s}(g_1g_2)^{-1},\]
\[\overline{c}_{X^{\ast}}(g_1, g_2)=\overline{c}_{X^{\ast}}(g_1, g_2)^{-1}=\overline{s}(g_1)^{-1}\overline{s}(g_2)^{-1}\overline{s}(g_1g_2).\]
\end{proof}
\begin{remark}\label{chap6triviall}
\begin{itemize}
\item[(1)] The trivialization map for  $\widetilde{c}_{X^{\ast}}$ can not be valued in $\mu_8$.
\item[(2)] The trivialization map for  $\overline{c}_{X^{\ast}}$ can not be valued in $\mu_2$.
\end{itemize}
\end{remark}
\begin{proof}
Assume, for contradiction, that a second trivialization
\(\widetilde{s}'\colon\U(\C)\to\mu_{8}\) exists.
Then \(\widetilde{s}'=\widetilde{s}\cdot\chi\) for some measurable homomorphism
\(\chi\colon\U(\C)\to\mu_{8}\subseteq\mathbb{T}\).
Because \(\U(\C)\) is separable, \(\chi\) is automatically continuous,
hence of the form \(\chi(e^{it})=e^{int}\) for some \(n\in\mathbb{Z}\).
Thus
\[
e^{int}\cdot e^{-\tfrac{1}{2}\sgn(\mathfrak{e})it}\in\mu_{8}
\quad\text{for all }t\in\mathbb{R};
\]
this is impossible.
Hence no such \(\widetilde{s}'\) exists.
The second statement is proved in the same way.
\end{proof}
 \subsection{The cocycle $\widehat{c}_z$}\label{Thecocyclewcz} Let $z=x+iy\in \C$ with $x>0$. Then $z=re^{it}$ with $-\tfrac{\pi}{2} < t< \tfrac{\pi}{2}$ and   $\sqrt{z}^{-1}=\sqrt{r}^{-1/2}e^{-it/2}$. Let $z,z'\in \mathbb{H}$. Define:
\begin{itemize}
\item  $ \gamma(z',z)= \big(\tfrac{z'-\overline{z}}{2i}\big)^{-1/2} \cdot (\Im z')^{1/4}\cdot (\Im z)^{1/4}$.
\item  $\epsilon(g;z',z)= \frac{\gamma(gz',gz)}{\gamma(z',z)}$. 
\item  $\widehat{c}_z(g_1,g_2)=\epsilon(g_1;z, g_2(z))$. 
\item $\alpha_z(g)=\tfrac{cz+d}{|cz+d|}$.
\item $J(g,z)=cz+d$.
\end{itemize}
According to \cite{Ta1}, $\widehat{c}_z(g_1,g_2)$ defines a real-analytic $2$-cocycle of order $2$.
\begin{lemma}\label{epzz'}
$\epsilon(g;z',z)=e^{i\Big[ \Arg \Big((cz'+d)^{1/2}\Big)-\Arg\Big( (cz+d)^{1/2}\Big)\Big]}$.
\end{lemma}
\begin{proof}
\[\Im(gz)=\tfrac{\Im(z)}{|cz+d|^{2}}=\tfrac{\Im(z)}{|J(g,z)|^{2}}.\]
\begin{align}
\epsilon(g;z',z)
&=\Big(\tfrac{g(z')-\overline{g(z)}}{2i}\Big)^{-1/2}  \cdot  \Big(\tfrac{z'-\overline{z}}{2i}\Big)^{1/2} \cdot | J(g, z')|^{-1/2}\cdot  | J(g, z)|^{-1/2}\label{epsilonzz'}\\
&=\Big(\tfrac{z'-\overline{z}}{2i} \tfrac{1}{(cz'+d)(c\overline{z}+d)}\Big)^{-1/2}  \cdot  \Big(\tfrac{z'-\overline{z}}{2i}\Big)^{1/2} \cdot | J(g, z')|^{-1/2}\cdot  | J(g, z)|^{-1/2}\\
&=\Big(\tfrac{z'-\overline{z}}{2i}\cdot \tfrac{(cz'+d)}{|cz'+d|}^{-1}\cdot \tfrac{(cz+d)}{|cz+d|}\Big)^{-1/2} \cdot  \Big(\tfrac{z'-\overline{z}}{2i}\Big)^{1/2}. \label{equ3}
\end{align}
Write $\theta=\Arg(z'-\overline{z})$, $\theta_1=\Arg(cz+d)$, $\theta_2=\Arg(cz'+d)$.  Note:
\[cz'+d-(c\overline{z}+d)=c(z'-\overline{z});\]
\[\Arg(\tfrac{z'-\overline{z}}{2i})=\theta-\tfrac{\pi}{2}.\]
\begin{itemize}
\item If $c=0$, then $\epsilon(g;z',z)=1=e^{ i\Big[ \Arg \Big((cz'+d)^{1/2}\Big)-\Arg\Big( (cz+d)^{1/2}\Big)\Big]}$.
\item If $c>0$, then $\theta$, $\theta_1, \theta_2$ all belong to $(0,\pi)$.
\[\Arg(c\overline{z}+d)=-\Arg(cz+d) \in (-\pi, 0), \quad \Arg[-(c\overline{z}+d)]=\pi+\Arg(c\overline{z}+d)=\pi-\Arg(cz+d).\]
\[\tfrac{\Arg(cz'+d)+\Arg[-(c\overline{z}+d)]}{2}=\Arg[c(z'-\overline{z})]=\Arg(z'-\overline{z}).\]
\[\tfrac{\theta_2+\pi-\theta_1}{2}=\theta, \quad \theta_1-\theta_2=\pi-2\theta.\]
\[\Arg\Big(\tfrac{z'-\overline{z}}{2i}\cdot \tfrac{(cz'+d)}{|cz'+d|}^{-1}\cdot \tfrac{(cz+d)}{|cz+d|}\Big)=\Arg[\exp\big(i(\theta-\tfrac{\pi}{2}+\theta_1-\theta_2)\big)] =\tfrac{\pi}{2}-\theta;\]
\[\Arg\Big(\tfrac{z'-\overline{z}}{2i}\cdot \tfrac{(cz'+d)}{|cz'+d|}^{-1}\cdot \tfrac{(cz+d)}{|cz+d|}\Big)=\Arg\tfrac{z'-\overline{z}}{2i}-\Arg\tfrac{(cz'+d)}{|cz'+d|}+\Arg\tfrac{(cz+d)}{|cz+d|}.\]
\begin{align}
(\ref{equ3})
&=e^{ i\Big[ \Arg \Big((cz'+d)^{1/2}\Big)-\Arg\Big( (cz+d)^{1/2}\Big)\Big]}.
\end{align}
\item If $c<0$, then $\theta\in (0,\pi)$, and $\theta_1, \theta_2$ both  belong to $(-\pi, 0)$.
\[\Arg(c\overline{z}+d)=-\Arg(cz+d) \in(0,\pi), \quad \Arg[-(c\overline{z}+d)]=\Arg(c\overline{z}+d)-\pi=-\Arg(cz+d)-\pi.\]
\[\tfrac{\Arg(cz'+d)+\Arg[-(c\overline{z}+d)]}{2}=\Arg[c(z'-\overline{z})]=\theta-\pi.\]
\[\tfrac{\theta_2-\pi-\theta_1}{2}=\theta-\pi, \quad \theta_1-\theta_2=\pi-2\theta.\]
\[\Arg\Big(\tfrac{z'-\overline{z}}{2i}\cdot \tfrac{(cz'+d)}{|cz'+d|}^{-1}\cdot \tfrac{(cz+d)}{|cz+d|}\Big)=\Arg[\exp\big(i(\theta-\tfrac{\pi}{2}+\theta_1-\theta_2)\big)] =\tfrac{\pi}{2}-\theta;\]
\[\Arg\Big(\tfrac{z'-\overline{z}}{2i}\cdot \tfrac{(cz'+d)}{|cz'+d|}^{-1}\cdot \tfrac{(cz+d)}{|cz+d|}\Big)=\Arg\tfrac{z'-\overline{z}}{2i}-\Arg\tfrac{(cz'+d)}{|cz'+d|}+\Arg\tfrac{(cz+d)}{|cz+d|}.\]
\begin{align}
(\ref{equ3})
&=e^{ i\Big[ \Arg \Big((cz'+d)^{1/2}\Big)-\Arg\Big( (cz+d)^{1/2}\Big)\Big]}.
\end{align}
\end{itemize} 
\end{proof}
\begin{lemma}
$\widehat{c}_z(g_1,g_2)^2= \alpha_z(g_1) \alpha_z(g_2) \alpha_z(g_1g_2)^{-1}$.
\end{lemma}
\begin{proof}
Let $g_2(z)=z'$. 
\begin{align*}
\widehat{c}_z(g_1,g_2)^2&=\epsilon(g_1;z, g_2(z))^2\\
&=\Bigg\{e^{i\Big[ \Arg \Big((c_1z+d_1)^{1/2}\Big)-\Arg\Big( (c_1z'+d_1)^{1/2}\Big)\Big]}\Bigg\}^2\\
&=e^{i \Arg (c_1z+d_1)}-e^{i \Arg (c_1z'+d_1)}\\
&=\tfrac{c_1z+d_1}{|c_1z+d_1|}\cdot \Big[\tfrac{c_1z'+d_1}{|c_1z'+d_1|}\Big]^{-1}\\
&=\tfrac{J(g_1,z)}{|J(g_1,z)|}\cdot \Big[\tfrac{J(g_1, g_2(z))}{|J(g_1, g_2(z))|}\Big]^{-1}\\
&=\tfrac{J(g_1,z)}{|J(g_1,z)|}\cdot  \tfrac{J(g_2,z)}{|J(g_2,z)|} \cdot \Big[ \tfrac{J(g_1g_2,z)}{|J(g_1g_2,z)|}\Big]^{-1}.
\end{align*} 
\end{proof}
\begin{lemma}\label{usopbig}
Let $g_1\in  P_{>0}(\R) $ and  $g_2\in \U(\C)$. Then $\widehat{c}_{z_0}(g_1,g)=1=\widehat{c}_{z_0}(g,g_2)$, for $g\in \SL_{2}(\R)$. 
\end{lemma}
\begin{proof}
1) Write $g_1=\begin{pmatrix} a_1& b_1\\ 0 &a_1^{-1} \end{pmatrix}$, with $a_1>0$. By \cite[p.118]{Ta1} or (\ref{epsilonzz'}), we have:
\begin{align*}
 \widehat{c}_{z_0}(g_1,g)
&=\Big(\tfrac{g_1(i)-\overline{g_1g(i)}}{2i}\Big)^{-1/2}  \cdot  \Big(\tfrac{i-\overline{g(i)}}{2i}\Big)^{1/2} \cdot | J(g_1, i)|^{-1/2}\cdot  | J(g_1, g(i))|^{-1/2}\\
&=\Big(a_1^2\tfrac{i-\overline{g(i)}}{2i}\Big)^{-1/2}  \cdot  \Big(\tfrac{i-\overline{g(i)}}{2i}\Big)^{1/2} \cdot  a_1\\
&=1.
\end{align*}
2) $g_2 (i)=i$ and $\gamma(i,i)=1$. 
\begin{align*}
 \widehat{c}_{z_0}(g,g_2)
&=\gamma(g(i),g(i))\\
&=\big(\tfrac{g(i)-\overline{g(i)}}{2i}\big)^{-1/2} \cdot (\Im g(i))^{1/4}\cdot (\Im g(i))^{1/4}\\
&=(\Im g(i))^{-1/2}\cdot (\Im g(i))^{1/2}\\
&=1.
\end{align*}
\end{proof}
For any $z\in \mathbb{H}$, write $z=p_z(i)$, for some $p_z\in  P_{>0}(\R)$. 
\begin{lemma}\label{ztoz0}
$\widehat{c}_z(g_1,g_2)=\widehat{c}_{z_0}(g_1^{ p_z},  g_2^{p_z})$.
\end{lemma}
\begin{proof}
\begin{align*}
\widehat{c}_z\big(g_1,g_2\big)&=\epsilon\big(g_1;z, g_2(z)\big)=\tfrac{\gamma\big(g_1(z),g_1g_2(z)\big)}{\gamma\big(z,g_2(z)\big)}\\
&=\tfrac{\gamma\big(g_1p_z(i),g_1g_2p_z(i)\big)}{\gamma\big(p_z(i),g_2p_z(i)\big)}\\
&=\tfrac{\gamma\big(g_1p_z(i),g_1p_z[p_z^{-1}g_2p_z](i)\big)}{\gamma\big(i,p_z^{-1}g_2p_z(i)\big)} \cdot \tfrac{\gamma\big(i,p_z^{-1}g_2p_z(i)\big)}{\gamma\big(p_z(i),g_2p_z(i)\big)}\\
&=\widehat{c}_{z_0}\big(g_1p_z,p_z^{-1}g_2p_z\big)\widehat{c}_{z_0}\big(p_z,p_z^{-1}g_2p_z\big)^{-1}\\
&=\widehat{c}_{z_0}\big(g_1p_z,p_z^{-1}g_2p_z\big)\\
&=\widehat{c}_{z_0}\big(p_z^{-1}g_1p_z,p_z^{-1}g_2p_z\big).
\end{align*}
\end{proof}
 Let $\overline{\SL}_2(\mathbb R)$, $\overline{\SL}^{\mu_8}_2(\mathbb R)$, $\widetilde{\SL}_2(\mathbb R)$, and $\widehat{\SL}^z_2(\mathbb R)$ denote the central extensions of $\SL_2(\mathbb R)$ arising from $\overline{c}_{X^{\ast}}$,  $\overline{c}_{X^{\ast}}$, $\widetilde{c}_{X^{\ast}}$, and $\widehat{c}_z$, with centers $\mu_2$,  $\mu_8$, $\mu_8$, and $\mathbb{T}$, respectively. 
\subsection{Link $\widehat{c}_{z_0}$  to $\overline{c}_{X^{\ast}}$ and $\widetilde{c}_{X^{\ast}}$}\label{linkedST}
In Lemma \ref{chap6trivial}, we have defined $\overline{s}$ for $\U(\C)$. Assume $\psi=\psi_0$ there. Let us extend this function   to the group  $\SL_2(\R)$ as follows: 
\begin{equation}\label{depkg}
g=p_gk_g=\begin{pmatrix} \tfrac{1}{ \sqrt{c^2+d^2}}    &  \tfrac{bd+ac}{\sqrt{c^2+d^2}} \\ 0& \sqrt{c^2+d^2}  \end{pmatrix}\cdot \begin{pmatrix} \tfrac{d}{\sqrt{c^2+d^2}}   & -\tfrac{c}{\sqrt{c^2+d^2}}\\ \tfrac{c}{\sqrt{c^2+d^2}}& \tfrac{d}{\sqrt{c^2+d^2}}\end{pmatrix},
\end{equation}
for $p_g\in P_{>0}(\R)$, $k_g\in \U(\C)$. We define:
\begin{align}
 \overline{s}(g)= \overline{s}(k_g)=e^{\tfrac{ it}{2}}=\tfrac{\sqrt{ci+d}}{\sqrt[4]{c^2+d^2}},
\end{align}
 for  $k_g=e^{i t}=\tfrac{ci+d}{\sqrt{c^2+d^2}}$,  with $-\pi \leq t< \pi$. Note that
 \begin{equation}\label{eqalpha}
 \overline{s}(g)^2=\alpha_{z_0}(g).
 \end{equation}
\begin{lemma}\label{twocoz0}
$\widehat{c}_{z_0}(g_1,g_2)=\overline{c}_{X^{\ast}}(g_1,g_2)\overline{s}(g_1)\overline{s}(g_2) \overline{s}(g_1g_2)^{-1}.$
\end{lemma}
\begin{proof}
Let $p=\begin{pmatrix} a & b \\0& a^{-1}\end{pmatrix} \in P_{>0}(\R)$, $u=\begin{pmatrix} \cos t & -\sin t\\ \sin t & \cos t\end{pmatrix}\in \U(\C), u_1,u_2\in \U(\C)$, $g\in \SL_2(\R)$. 
\begin{itemize}
\item[(1)] $\overline{c}_{X^{\ast}}(p,g)\overline{s}(p)\overline{s}(g) \overline{s}(pg)^{-1}=\overline{c}_{X^{\ast}}(p,g)=1=\widehat{c}_{z_0}(p,g)$.
\item[(2)] $\overline{c}_{X^{\ast}}(u_1,u_2)\overline{s}(u_1)\overline{s}(u_2) \overline{s}(u_1u_2)^{-1}=1=\widehat{c}_{z_0}(u_1,u_2)$.
\item[(3)]$\overline{c}_{X^{\ast}}(pu_1,u_2)=\overline{c}_{X^{\ast}}(p,u_1)\overline{c}_{X^{\ast}}(pu_1,u_2)=
\overline{c}_{X^{\ast}}(p,u_1u_2)\overline{c}_{X^{\ast}}(u_1,u_2)=\overline{c}_{X^{\ast}}(u_1,u_2)$;
$$\overline{c}_{X^{\ast}}(pu_1,u_2)\overline{s}(pu_1)\overline{s}(u_2)\overline{s}(pu_1u_2)^{-1}
=\overline{c}_{X^{\ast}}(u_1,u_2)\overline{s}(u_1)\overline{s}(u_2) \overline{s}(u_1u_2)^{-1}
=1=\widehat{c}_{z_0}(pu_1,u_2).$$
\item[(4)] Write $up=p'u'$, for $p'\in P_{>0}(\R)$, $u'\in \U(\C)$. Then:
\begin{itemize}
\item[(a)] $\overline{c}_{X^{\ast}}(u,p)\overline{s}(u)\overline{s}(p) \overline{s}(up)^{-1}=\overline{s}(u)\overline{s}(up)^{-1}=\overline{s}(u)\overline{s}(u')^{-1}$.
\item[(b)] 
\begin{align*}
\Arg \big(  \big[J(u, p(z_0))\big]^{1/2}\big)
&= \Arg \big(  \big[ J(u, p(z_0)) J(p,z_0)\big]^{1/2} \big)\\
&=\Arg \big(  \big[J(up, z_0)\big]^{1/2}\big)\\
&=\Arg \big(  \big[J(p', u'(z_0))J(u', z_0)\big]^{1/2}\big)\\
&=\Arg \big(   \big[J(p', z_0)\big]^{1/2} \big[ J(u', z_0)\big]^{1/2}\big)\\
&=\Arg \big( J(u', z_0)^{1/2}\big).
\end{align*}
\begin{align*}
\widehat{c}_{z_0}(u,p)&=\epsilon(u;z_0, p(z_0))\\
&=e^{ i\Big[\Arg \big(  \big[ J(u, z_0)\big]^{1/2}\big)-\Arg \big(   \big[  J(u, p(z_0))\big]^{1/2}\big)\Big]}\\
&=\overline{s}(u)\overline{s}(u')^{-1}.
\end{align*}
\end{itemize}
\item[(5)] Write $g_i=p_i u_i$ for $p_i\in P_{>0}(\R)$, $u_i\in \U(\C)$. Note that the right hand side also defines  a cocycle and we write it by $c'$.   Then:
\[\widehat{c}_{z_0}(g_1,g_2)=\widehat{c}_{z_0}(p_1u_1,p_2u_2)=\widehat{c}_{z_0}(u_1,p_2);\]
\[c'(p,g)=1=c'(g,u), \textrm{ for } p\in  P_{>0}(\R), u\in \U(\C).\]
\[c'(g_1,g_2)=c'(p_1u_1,p_2u_2)=c'(u_1,p_2)=\widehat{c}_{z_0}(u_1,p_2).\]
\end{itemize}
\end{proof}
For $g\in \SL_2(\R)$, we define:
\begin{align}
 \widetilde{s}(g)= m_{X^{\ast},\psi}(g)\overline{s}(g).
\end{align}
Note that in Lemma~\ref{chap6thetaf} we have already defined a function $\widetilde s(g)$ for $g\in\U(\C)$.
\begin{itemize}
  \item[(1)] If $\sgn(\mathfrak e)<0$, the restriction of $\widetilde s$ to $\U(\C)$ agrees with the function $\widetilde s$ given in Lemma~\ref{chap6thetaf}.
  \item[(2)] If $\sgn(\mathfrak e)>0$, the restriction of $\widetilde s$ to $\U(\C)$ is obtained from the function in Lemma~\ref{chap6thetaf} by twisting by the character $t\mapsto t$ of $\U(\C)$.
\end{itemize}
\begin{lemma}\label{eightcoz0}
$\widehat{c}_{z_0}(g_1,g_2)=\widetilde{c}_{X^{\ast}}(g_1,g_2) \widetilde{s}(g_1) \widetilde{s}(g_2)  \widetilde{s}(g_1g_2)^{-1}.$
\end{lemma}
\begin{proof}
\begin{align*}
\widehat{c}_{z_0}(g_1,g_2)&\stackrel{\textrm{ Lem. }\ref{twocoz0}}{=}\overline{c}_{X^{\ast}}(g_1,g_2)\overline{s}(g_1)\overline{s}(g_2) \overline{s}(g_1g_2)^{-1}\\
&=\widetilde{c}_{ {X^{\ast}}}(g_1,g_2)m_{X^{\ast},\psi}(g_1g_2)^{-1} m_{X^{\ast},\psi}(g_1) m_{X^{\ast},\psi}(g_2) \overline{s}(g_1)\overline{s}(g_2) \overline{s}(g_1g_2)^{-1}.
\end{align*}
\end{proof}
As a consequence, we have  the following  group isomorphisms:
 \begin{equation}\label{overwideSL}
 \overline{\SL}^{\mu_8}_2(\mathbb R)\stackrel{\iota}{\to}\widetilde{\SL}_2(\mathbb R);[g,\epsilon] \longmapsto [g,m_{X^{\ast},\psi}(g)\epsilon],
 \end{equation}
  \begin{equation}\label{overhatSL}
  \overline{\SL}^{\mu_8}_2(\mathbb R)\stackrel{\iota'}{\to}\Im(\iota') (\subseteq\widehat{\SL}_2(\mathbb R));[g,\epsilon] \longmapsto [g,\overline{s}(g)^{-1}\epsilon],
   \end{equation}
    \begin{equation}\label{overwidehatSL}
  \widetilde{\SL}_2(\mathbb R)\stackrel{\iota'}{\to}\Im(\iota') (\subseteq\widehat{\SL}_2(\mathbb R));[g,\epsilon] \longmapsto [g,\widetilde{s}(g)^{-1}\epsilon].
   \end{equation}
\begin{lemma}\label{constants1}
If $\psi=\psi_0$, then $ \widetilde{s}(g)=1$, for $g\in P_{>0}(\R)$, $g=\omega$,  and  $g=h(a)$.
\end{lemma}
\begin{proof}
1) If $g\in P_{>0}(\R)$, then $m_{X^{\ast},\psi}(g)=1=\overline{s}(g)$. \\
2) If $g=\omega$, then $m_{X^{\ast},\psi}(g)=e^{\tfrac{-i \pi }{4}}$ and $\overline{s}(g)=e^{\tfrac{i\pi }{4}}$.\\
3)  If $g=h(-1)$, then $m_{X^{\ast},\psi}(g)=e^{\tfrac{i \pi }{2}}$ and $\overline{s}(g)=e^{-\tfrac{i\pi }{2}}$.\\
4) If $g=h(a)$ with $\sgn(a)<0$,  then $m_{X^{\ast},\psi}(g)=e^{\tfrac{i \pi }{2}}$ and $\overline{s}(g)=e^{-\tfrac{i\pi }{2}}$.
\end{proof}
\begin{lemma}\label{constants2}
If $\psi=\psi_0$, then $ \widetilde{s}(g) \widetilde{s}(g^{h_{-1}}) =1$, for any $g\in \SL_2(\R)$.
\end{lemma}
\begin{proof}
1) If $g\in P_{>0}(\R)$, then $g^{h_{-1}}\in P_{>0}(\R)$. So $\widetilde{s}(g) =\widetilde{s}(g^{h_{-1}}) =1$.\\
2) If $g=\begin{pmatrix} a& -b\\ b & a\end{pmatrix}\in \SO_2(\R)$, then $g^{h_{-1}}=\begin{pmatrix} a& b\\ -b & a\end{pmatrix}\in \SO_2(\R)$. Write $u=a+bi=e^{it}$ with $-\pi \leq t< \pi$. Then  $\overline{u}=\left\{ \begin{array}{cl} e^{-it}, &  -\pi < t< \pi,\\-1, &  t=-\pi.\end{array} \right.$ \\
a)  If $-\pi < t< \pi$, then $\overline{s}(g)\overline{s}(g^{h_{-1}})=e^{it/2}e^{-it/2}=1$.   $m_{X^{\ast},\psi_0}(g)=e^{\tfrac{i \pi [- \sgn(b)] }{4}}$,  $m_{X^{\ast},\psi_0}(g^{h_{-1}})=^{\tfrac{i \pi [\sgn(b)] }{4}}$. Therefore, $\widetilde{s}(g) \widetilde{s}(g^{h_{-1}}) =1$.\\
b) If $t=-\pi$, then $\overline{s}(g)\overline{s}(g^{h_{-1}})= e^{-i\pi/2}e^{-i\pi/2}=-1$. $m_{X^{\ast},\psi_0}(g)=e^{\tfrac{i \pi }{2}}=m_{X^{\ast},\psi_0}(g^{h_{-1}})$. Therefore, $\widetilde{s}(g) \widetilde{s}(g^{h_{-1}}) =-1(i^2)=1$.\\
3) If $g\in \SL_2(\R)$, we write $g=pk$, for $p\in  P_{>0}(\R)$ and $k\in \U(\C)$, then $p^{h_{-1}} \in P_{>0}(\R)$ and $k^{h_{-1}}\in \U(\C)$.  So $\widetilde{s}(g)=\widetilde{s}(p)\widetilde{s}(k)$ and $\widetilde{s}(g^{h_{-1}})=\widetilde{s}(p^{h_{-1}})\widetilde{s}(k^{h_{-1}})$. Hence the result holds.
\end{proof}

 \subsection{The discrete group $\SL_2(\Z)$}
Let
$g=\begin{pmatrix}
a & b\\ c & d\end{pmatrix}\in \SL_2(\Z)$. Recall the definitions  of Dedekind sums from \cite{Po}, \cite{RaWh}:
\begin{itemize}
\item $((x))=\left\{\begin{array}{ll} x-[x]-\tfrac{1}{2} & \textrm{ if } x\in \R\setminus \Z,\\ 0 & \textrm{ if } x\in \Z.\end{array}\right.$
\item $s(d,c)=\sum_{k\bmod(|c|)} ((\tfrac{k}{c}))((\tfrac{kd}{c}))$, for two coprime integers $c,d$.(Dedekind sum)
\item $s(-d, c)=-s(d,c)$, $s(d,-c)=s(d,c)$.
\end{itemize}
In \cite{As}, Asai gave the following function:
$$\mu(g)=\left\{\begin{array}{ll} \tfrac{b}{12d}+\tfrac{1-\sgn(d)}{4} & \textrm{ if } c=0,\\ \tfrac{a+d}{12c}-\sgn(c)\bigg( \tfrac{1}{4}+s(d, |c|)\bigg) & \textrm{ if } c\neq 0.\end{array}\right.$$
Define:
\begin{itemize}
\item $\overline{\beta}_1(g)=e^{-\pi i \mu(g)}$.
\item $\widetilde{\beta}_1(g)=\overline{\beta}_1(g)m_{X^{\ast},\psi}(g).$
\item  $\widehat{\beta}_1(g)=\overline{\beta}_1(g)\overline{s}(g)^{-1}.$
\end{itemize}
 Let $g_i= \begin{pmatrix} a_i& b_i\\c_i& d_i\end{pmatrix} \in \SL_2(\Z)$.
\begin{lemma}\label{chap8csp}
\begin{itemize}
\item[(1)] $\overline{c}_{X^{\ast}}(g_1,g_2)=\overline{\beta}_1(g_1)^{-1}\overline{\beta}_1(g_2)^{-1}\overline{\beta}_1(g_1g_2)$.
\item[(2)] $\widetilde{c}_{X^{\ast}}(g_1,g_2)=\widetilde{\beta}_1(g_1)^{-1}\widetilde{\beta}_1(g_2)^{-1}\widetilde{\beta}_1(g_1g_2)$.
\item[(3)] $\widehat{c}_{z_0}(g_1,g_2)=\widehat{\beta}_1(g_1)^{-1}\widehat{\beta}_1(g_2)^{-1}\widehat{\beta}_1(g_1g_2)$.
\end{itemize}
\end{lemma}
\begin{proof}
1) See \cite[Sections 2-3]{As}.\\
2) \begin{align*}
\widetilde{c}_{ {X^{\ast}}}(g_1,g_2) &=m_{X^{\ast},\psi}(g_1g_2) m_{X^{\ast},\psi}(g_1)^{-1} m_{X^{\ast},\psi}(g_2)^{-1}\overline{c}_{X^{\ast}}(g_1, g_2) \\
&=[\overline{\beta}_1(g_1g_2)m_{X^{\ast},\psi}(g_1g_2)] [\overline{\beta}_1(g_1)m_{X^{\ast},\psi}(g_1)]^{-1} [\overline{\beta}_1(g_2)m_{X^{\ast},\psi}(g_2)]^{-1}.
\end{align*}
3) \begin{align*}
\widehat{c}_{z_0}(g_1,g_2) &=\overline{c}_{X^{\ast}}(g_1,g_2)\overline{s}(g_1)\overline{s}(g_2) \overline{s}(g_1g_2)^{-1} \\
&=[\overline{s}(g_1)^{-1}\overline{\beta}_1(g_1)]^{-1} \cdot [\overline{s}(g_2)^{-1}\overline{\beta}_1(g_2)]^{-1} \cdot [\overline{s}(g_1g_2)^{-1}\overline{\beta}_1(g_1g_2)].
\end{align*}
\end{proof}
\subsection{The discrete group $\Gamma_{\theta}$}
 Let $c, d$ be two coprime integers with $c d\neq  0$.  Following \cite[p.162]{LiVe}, let us define the symplectic Gauss sum:
$$ G(c,d)\stackrel{\Delta}{=}|d|^{-\tfrac{1}{2}}\sum_{n=0}^{|d|-1} e^{-\tfrac{\pi i c n^2}{d}}.$$
If $c=0$, $d=\pm 1$, we also  define 
$$ G(c,d)\stackrel{\Delta}{=}|d|^{-\tfrac{1}{2}}\sum_{n=0}^{|d|-1} e^{-\tfrac{\pi i c n^2}{d}}=1.$$
\begin{lemma}
If  $cd \neq 0$ is even, $G(d,c)G(c,d)=e^{-\tfrac{\pi i}{4} \sgn(cd)}$.
\end{lemma}
\begin{proof}
See \cite[p.170, Prop.]{LiVe}.
\end{proof}
\begin{lemma}\label{chap8betacd}
For two non-zero coprime integers $c, d$, with  $2\mid cd$, we have:
$$G(d,c)=\left\{\begin{array}{lr}
 (\tfrac{c'}{|d|}) e^{-\tfrac{\pi i}{4} \sgn(cd)} & \textrm{ if } c=2c', |d|\equiv 1(\bmod 4),\\i  \sgn(d) (\tfrac{c'}{|d|} )   e^{-\tfrac{\pi i}{4} \sgn(cd)}   & \textrm{ if }  c=2c', |d|\equiv 3(\bmod 4),\\
(\tfrac{d'}{|c|})& \textrm{ if } d=2d',  |c| \equiv 1(\bmod 4),  \\
-i \sgn(c)(\tfrac{d'}{|c|})& \textrm{ if } d=2d', | c| \equiv 3(\bmod 4).
\end{array}\right.$$
\end{lemma}
\begin{proof}
 By the above lemma,  $G(c,d)G(d,c)=e^{-\tfrac{\pi i}{4} \sgn(cd)}$.  Moreover, by \cite[Thm.1.5.2]{BeEvWi}, for $d>0$, $2\mid c$ and $c=2c'$,  we have: $$G(c,d)=\left\{\begin{array}{cc}
 (\tfrac{c'}{d})  & \textrm{ if } d\equiv 1(\bmod 4),\\ -(\tfrac{c'}{d})i  & \textrm{ if } d\equiv 3(\bmod 4).\end{array}\right.$$
 If  $d<0$, and $c=2c'$, then
 $$ G(c,d)=\overline{G(c,-d)}=\left\{\begin{array}{cc}
 (\tfrac{c'}{|d|})  & \textrm{ if } |d|\equiv 1(\bmod 4),\\ (\tfrac{c'}{|d|})i  & \textrm{ if } |d|\equiv 3(\bmod 4).\end{array}\right.$$ Then the result follows.
\end{proof}

  Following \cite[p.155]{LiVe}, for $g=\begin{pmatrix}
                                        a & b \\
                                        c& d 
                                      \end{pmatrix} \in \Gamma_{\theta}$,  let us define:
\begin{align}
&\widetilde{\beta}(g)=\left\{ \begin{array}{cl} G(d,c) & \textrm{ if } c\neq 0,\\ 1&  \textrm{ if } c=0.\end{array}\right.\label{betag}\\
&\overline{\beta}(g)=\widetilde{\beta}(g)m_{X^{\ast}, \psi_0}(g)^{-1}= \left\{\begin{array}{cc}  G(d,c)e^{\tfrac{\pi i \sgn(c)}{4}} & \textrm{ if } c\neq 0,\\e^{-\tfrac{\pi i[1-\sgn(d)]}{4}} & \textrm{ if } c=0.\end{array}\right.\\
&\widehat{\beta}(g)=\widetilde{\beta}(g)m_{X^{\ast}, \psi_0}(g)^{-1}\overline{s}(g)^{-1}=\left\{\begin{array}{cc}  G(d,c)e^{\tfrac{\pi i \sgn(c)}{4}}\Big[\tfrac{\sqrt{ci+d}}{\sqrt[4]{c^2+d^2}}\Big]^{-1} & \textrm{ if } c\neq 0,\\1& \textrm{ if } c=0.\end{array}\right. 
\end{align}
 If $c=0$, $d=\pm 1$ and  $\widehat{\beta}(g)=\widetilde{\beta}(g)m_{X^{\ast}, \psi_0}(g)^{-1}\overline{s}(g)^{-1}=e^{-\tfrac{\pi i[1-\sgn(d)]}{4}} \Big[\tfrac{\sqrt{d}}{\sqrt{|d|}}\Big]^{-1}=1$.
\begin{lemma}\label{spgamma12}
Let   $g_i=\begin{pmatrix} a_i& b_i\\c_i& d_i\end{pmatrix}\in \Gamma_{\theta}$.
\begin{itemize}
\item[(1)] If $\psi=\psi_0$, $\widetilde{c}_{X^{\ast}}(g_1,g_2)=\widetilde{\beta}(g_1)^{-1}\widetilde{\beta}(g_2)^{-1}\widetilde{\beta}(g_1g_2)$.
\item[(2)] $\overline{c}_{X^{\ast}}(g_1, g_2) =\overline{ \beta}(g_1)^{-1}\overline{ \beta}(g_2)^{-1}\overline{ \beta}(g_1g_2)$.
\item[(3)] $\widehat{c}_{z_0}(g_1, g_2) =\widehat{ \beta}(g_1)^{-1}\widehat{ \beta}(g_2)^{-1}\widehat{ \beta}(g_1g_2)$.
\end{itemize}
\end{lemma}
\begin{proof}
1) See \cite[pp.149--150]{LiVe}.\\
2)
\begin{align*}
\overline{c}_{X^{\ast}}(g_1, g_2)&=m_{X^{\ast},\psi_0}(g_1g_2)^{-1} m_{X^{\ast},\psi_0}(g_1) m_{X^{\ast},\psi_0}(g_2) \widetilde{c}_{ {X^{\ast}}}(g_1,g_2)\\
&=[ \widetilde{\beta}(g_1g_2)m_{X^{\ast},\psi_0}(g_1g_2)^{-1}] [ \widetilde{\beta}(g_1)m_{X^{\ast},\psi_0}(g_1)^{-1}]^{-1} [\widetilde{\beta}(g_2)m_{X^{\ast},\psi_0}(g_2)^{-1}]^{-1}.
\end{align*}
3)\begin{align*}
 &\widehat{c}_{z_0}(g_1,g_2) =\overline{c}_{X^{\ast}}(g_1,g_2)\overline{s}(g_1)\overline{s}(g_2) \overline{s}(g_1g_2)^{-1}\\ 
 &=[ \widetilde{\beta}(g_1g_2)m_{X^{\ast},\psi_0}(g_1g_2)^{-1} \overline{s}(g_1g_2)^{-1} ] [ \widetilde{\beta}(g_1)m_{X^{\ast},\psi_0}(g_1)^{-1}\overline{s}(g_1)^{-1}]^{-1} [\widetilde{\beta}(g_2)m_{X^{\ast},\psi_0}(g_2)^{-1}\overline{s}(g_2)^{-1}]^{-1}.
 \end{align*}
\end{proof}
Note that when $\psi=\psi_0$,  $\widetilde{\beta}$   differs from  $\widetilde{\beta}_1$ by a character $\chi_{\Gamma_{\theta}}$.  
Let $\chi_{\Gamma_{\theta}}(g)=\widetilde{\beta}_1(g)/\widetilde{\beta}(g)$, for $g=\begin{pmatrix} a & b\\ c& d\end{pmatrix}\in \Gamma_{\theta}$. According to A. Putman's answer on the question titled `Generators for congruence group $\Gamma_{\theta}$' on MathOverflow, $\Gamma(2)/\{\pm I\}$ is a free group with two generators  $u(2)=\begin{pmatrix} 1 & 2\\ 0 & 1\end{pmatrix},u_-(2)=\begin{pmatrix} 1 & 0\\ 2 & 1\end{pmatrix}$. Then $\chi_{\Gamma_{\theta}}$ is a character of $\Gamma_{\theta}$,  which is determined by $\chi_{\Gamma_{\theta}}(\omega)$,  $\chi_{\Gamma_{\theta}}(u(2))$, $\chi_{\Gamma_{\theta}}(u_-(2))$.
\begin{lemma}
$\chi_{\Gamma_{\theta}}(\omega)=\chi_{\Gamma_{\theta}}(\pm I)=1$, $\chi_{\Gamma_{\theta}}(u(2))=e^{-\tfrac{\pi i}{6}}$, $\chi_{\Gamma_{\theta}}(u_-(2))=e^{\tfrac{\pi i}{6}}$.
\end{lemma}
\begin{proof}
1) $\widetilde{\beta}_1(\omega)=\overline{\beta}_1(g)m_{X^{\ast},\psi_0}(g)=e^{-\pi i (-\tfrac{1}{4})} \cdot e^{-\tfrac{i \pi }{4}}=1$, $\widetilde{\beta}(\omega)=1$.\\
2) $\widetilde{\beta}_1(u(2))=e^{-\pi i \tfrac{1}{6}}$, $\widetilde{\beta}(u(2))=1$.\\
3)$\widetilde{\beta}_1(u_-(2))=e^{-\pi i \tfrac{1}{12}}$, $\widetilde{\beta}(u_-(2))=e^{-\pi i \tfrac{1}{4}}$.
\end{proof}
\subsection{Automorphic factor I}\label{autoI} Recall the notations $J(g, z)$ and $\alpha_z(g)$ from Sect. \ref{Thecocyclewcz}. Additionally, we define:
\begin{itemize}
\item[(1)] $J'_{1/2}(g, z)=   \epsilon(g; z,z_0) \cdot |J(g, z)|^{1/2}$; (Coming from \cite[p.131, \wasyparagraph 5]{Ta1})
\item[(2)] $J_{1/2}(g, z)=J'_{1/2}(g, z)\cdot\widetilde{s}(g)$;
\item[(3)] $J_{3/2}(g, z)=  J_{1/2}(g, z)J(g, z)$.
\end{itemize}
\begin{lemma}\label{J121}
  $J'_{1/2}(g, z)^2= \alpha_{z_0}(g)^{-1}  J(g, z)$ and  $J'_{1/2}(g_1g_2, z)=J'_{1/2}(g_1, g_2(z))J'_{1/2}(g_2, z)\widehat{c}_{z_0}(g_1,g_2)$.
\end{lemma}
\begin{proof}
See \cite[pp.\,131--132]{Ta1}; we provide the details for completeness. \\
1) $J'_{1/2}(g, z)^2= \epsilon(g; z,z_0)^2 \cdot |J(g, z)| \stackrel{\textrm{ Lem.} \ref{epzz'}}{=}e^{i\Big[ \Arg (cz+d)-\Arg (cz_0+d)\Big]}\cdot |J(g, z)|=[\tfrac{cz_0+d}{|cz_0+d|}]^{-1} (cz+d)=\alpha_{z_0}(g)^{-1}  J(g, z)$.\\
2) Put $g_3=g_1g_2$, $z'=g_2z$, $z''=g_2z_0$.  Then
\begin{align*}
\epsilon(g_1; g_2z,g_2z_0)&=e^{i\Big[ \Arg \Big((c_1z'+d_1)^{1/2}\Big)-\Arg\Big( (c_1z''+d_1)^{1/2}\Big)\Big]}\\
&=e^{i\Big[ \Arg \Big((c_1z'+d_1)^{1/2}\Big)-\Arg\Big( (c_1z_0+d_1)^{1/2}\Big)\Big]} e^{i\Big[ \Arg \Big((c_1z_0+d_1)^{1/2}\Big)-\Arg\Big( (c_1z''+d_1)^{1/2}\Big)\Big]} \\
&=\epsilon(g_1; g_2z,z_0) \epsilon(g_1;z_0, g_2z_0).
\end{align*}
Hence
\begin{align*}
J'_{1/2}(g_1g_2, z)=&\epsilon(g_1g_2; z,z_0)\cdot |J(g_1g_2, z)|^{1/2} \\
&=\frac{\gamma(g_1g_2z,g_1g_2z_0)}{\gamma(z,z_0)}\cdot  |J(g_1g_2, z)|^{1/2} \\
&=\epsilon(g_1; g_2z,g_2z_0) \epsilon(g_2; z,z_0)\cdot |J(g_1g_2, z)|^{1/2}\\
&=\epsilon(g_1;z_0, g_2z_0) \epsilon(g_1; g_2z,z_0) \epsilon(g_2; z,z_0) \cdot |J(g_1, g_2z)|^{1/2}|J(g_2, z)|^{1/2}\\
&=\widehat{c}_{z_0}(g_1,g_2) \cdot \epsilon(g_1; g_2z,z_0) |J(g_1, g_2z)|^{1/2}  \cdot  \epsilon(g_2; z,z_0)|J(g_2, z)|^{1/2}\\
&=J'_{1/2}(g_1, g_2(z))J'_{1/2}(g_2, z)\widehat{c}_{z_0}(g_1,g_2).
\end{align*}
\end{proof}
\begin{remark}\label{remakonUC}
$J'_{1/2}(g, z_0)=| J(g, z_0)|^{1/2}=1$ and $J_{1/2}(g, z_0)=\widetilde{s}(g)$, for $g\in \U(\C)$.
\end{remark}
For $z\in \mathbb{H}$, let $p_z$ be the corresponding element in $ P_{>0}(\R)$.
\begin{lemma}
 $J'_{1/2}(g, z)=\widehat{c}_{z_0}(g,p_z)^{-1}\cdot |J(g, z)|^{1/2}$.
\end{lemma}
\begin{proof}
It follows from that $\widehat{c}_{z_0}(g,p_z)= \epsilon(g;z_0,z) \stackrel{ \textrm{ Lem. }\ref{epzz'} }{=}\epsilon(g;z,z_0)^{-1}$.
\end{proof}
\begin{lemma}\label{J121J1}
  $J_{1/2}(g_1g_2, z)=J_{1/2}(g_1, g_2(z))J_{1/2}(g_2, z)\widetilde{c}_{X^{\ast}}(g_1,g_2)$.
\end{lemma}
\begin{proof}
\begin{align*}
\text{Left side}
&=\widetilde{s}(g_1g_2) J'_{1/2}(g_1, g_2(z))J'_{1/2}(g_2, z)\widehat{c}_{z_0}(g_1,g_2)\\
&=\widetilde{s}(g_1g_2) \widetilde{s}(g_1)^{-1} \widetilde{s}(g_2)^{-1}  J_{1/2}(g_1, g_2(z))J_{1/2}(g_2, z)\widehat{c}_{z_0}(g_1,g_2)\\
&\stackrel{\textrm{ Lem. }\ref{eightcoz0}}{=} \text{Right side}.
\end{align*}
\end{proof}
\begin{lemma}\label{J12squr2}
$[J_{1/2}(g, z)m_{X^{\ast}}(g)^{-1}]^2=cz+d$.
\end{lemma}
\begin{proof}
$\text{Left side}=[J'_{1/2}(g, z)\widetilde{s}(g)m_{X^{\ast}}(g)^{-1}]^2=[J'_{1/2}(g, z)\overline{s}(g)]^2\stackrel{ \textrm{ Eq. }  \ref{eqalpha}}{=}\alpha_{z_0}(g)^{-1} J(g, z) \cdot\alpha_{z_0}(g)= \text{Right side}$.
\end{proof}

\begin{definition}\label{defczd}
  $\sqrt{cz+d}\stackrel{Def.}{=} J_{1/2}(g, z)m_{X^{\ast}}(g)^{-1}$, for $g\in \SL_2(\R)$.
  \end{definition}
  \begin{remark}\label{remachoArg}
The above definition is compatible with our choice of $\Arg$.
\end{remark}
  \begin{proof}
  $J_{1/2}(g, z)m_{X^{\ast}}(g)^{-1}=J'_{1/2}(g, z)\overline{s}(g)=e^{i\Big[ \Arg \Big((cz+d)^{1/2}\Big)-\Arg\Big( (cz_0+d)^{1/2}\Big)\Big]} |J(g, z)|^{1/2} \Big(\tfrac{J(g,z_0)}{|J(g,z_0)|}\Big)^{1/2}=(cz+d)^{1/2}$.
    \end{proof}
  Let $g_3=g_1g_2$ and $g_i=\begin{pmatrix} a_i& b_i\\ c_i &d_i \end{pmatrix}$. 
    \begin{lemma}
    
    $\sqrt{c_3z+d_3}=\sqrt{c_1(g_2 z)+d_1}\cdot \sqrt{c_2 z+d_2}\cdot \overline{c}_{X^{\ast}}(g_1,g_2)$.
  \end{lemma}
  \begin{proof}
 \begin{align*}
 \text{Left side}
 &=J_{1/2}(g_1g_2, z)m_{X^{\ast}}(g_1g_2)^{-1}\\
 &\stackrel{\textrm{Lem. } \ref{J121J1}}{=}J_{1/2}(g_1, g_2(z))J_{1/2}(g_2, z)\widetilde{c}_{X^{\ast}}(g_1,g_2)m_{X^{\ast}}(g_1g_2)^{-1}\\
 &=J_{1/2}(g_1, g_2(z))J_{1/2}(g_2, z) m_{X^{\ast}}(g_1)^{-1} m_{X^{\ast}}(g_2)^{-1}\overline{c}_{X^{\ast}}(g_1, g_2)\\
 &= \text{Right side}. 
 \end{align*}
  \end{proof}
  For $\overline{g}=(g,\epsilon)\in\overline{\SL}_{2}(\mathbb{R})$ or $\overline{\SL}_{2}^{\mu_8}(\mathbb{R})$, define:
\begin{align*}
J_{1/2}(\overline{g},z)&\coloneqq\epsilon^{-1}J_{1/2}(g,z)m_{X^{\ast}}(g)^{-1}=\epsilon^{-1}\sqrt{cz+d},\\[2pt]
J_{3/2}(\overline{g},z)&\coloneqq J_{1/2}(\overline{g},z)\,J(g,z)=\epsilon^{-1}\sqrt{cz+d}(cz+d).
\end{align*}
Let $\kappa=1$ or $3$. 
\begin{corollary}\label{Automorphicfactor2}
For every $\overline{g}_{1},\overline{g}_{2}\in\overline{\SL}_{2}(\mathbb{R})$, $J_{\kappa/2}(\overline{g}_{1}\overline{g}_{2},z)
      =J_{\kappa/2}(\overline{g}_{1},g_{2}z)\,
       J_{\kappa/2}(\overline{g}_{2},z)$.
\end{corollary}
 \begin{corollary}
$J_{\kappa/2}(-,z_0): \overline{\SO}_2(\R) \to \U(\C); \overline{g}\longmapsto J_{\kappa/2}(\overline{g},z_0)$ is  a group homomorphism. Moreover, when $\kappa=1$, it is  an isomorphism.
 \end{corollary}
\section{The  relevant covering groups: $\GL_2(\R)$ case}\label{TrcgGLR}
Following Barthel \cite{Ba}, who extended the 2-cocycle $\widetilde{c}_{X^{\ast}}$ from $\Sp$ to $\GSp$ over a $p$-adic field, we carry out the analogous construction over $\R$.  In this section we extend the 2-cocycles $\overline{c}_{X^{\ast}}$, $\widetilde{c}_{X^{\ast}}$, and $\widehat{c}_{z}$ from $\SL_2(\R)$ to the full group $\GL_2(\R)$.
 \subsection{$\mu_8$-covering}
 Let $\alpha$  now be a continuous automorphism of  $\SL_2(\R)$.  By Moore's cohomology theory, there exists a unique automorphism $\alpha^{\ast}$ of   $\widetilde{\SL}_2(\R)$ such that the following diagram
\[
\begin{CD}
1 @>>> \mu_8    @>>>  \widetilde{\SL}_2(\R)@>>>\SL_2(\R) @>>> 1\\
@.     @|    @VV{\alpha^{\ast}}V  @VV{\alpha}V \\
1 @>>> \mu_8    @>>>  \widetilde{\SL}_2(\R)@>>>\SL_2(\R) @>>> 1
\end{CD}
\]
is commutative. Moreover, there exists a unique function $\nu(\alpha,-): \SL_2(\R) \longrightarrow \mu_8$, such that the following equality holds:
\begin{align}\label{chap3equ}
\widetilde{c}_{X^{\ast}}(g_1^{\alpha}, g_2^{\alpha})=\widetilde{c}_{X^{\ast}}(g_1, g_2)\nu(\alpha,g_1)^{-1}\nu(\alpha,g_2)^{-1}\nu(\alpha,g_1g_2).
\end{align}
The map  $\alpha^{\ast}$ is given as:
$$\alpha^{\ast}:  \widetilde{\SL}_2(\R) \longrightarrow \widetilde{\SL}_2(\R); [g, \epsilon] \longmapsto [g^{\alpha}, \nu(\alpha, g) \epsilon].$$
\begin{lemma}\label{chap3twoau}
For two automorphisms $\alpha_1, \alpha_2$,  $\nu(\alpha_1\alpha_2, g)=\nu(\alpha_1,g)\nu(\alpha_2, g^{\alpha_1})$.
\end{lemma}
\begin{proof}
Because  $(g, \epsilon)^{\alpha_1\alpha_2}=(  g^{\alpha_1}, \nu( \alpha_1,g) \epsilon)^{\alpha_2}=( g^{\alpha_1\alpha_2}, \nu(\alpha_1,g) \nu(\alpha_2,g^{\alpha_1})\epsilon)$.
\end{proof}
\begin{example}\label{chap3example1}
If $\alpha=h\in  \SL_2(\R)$, and  $h$ acts on $\SL_2(\R)$ by  conjugation, i.e., $g^h=h^{-1}gh$, for $g\in \SL_2(\R)$,  then $\nu(\alpha,g)=\widetilde{c}_{X^{\ast}}(h^{-1}, g h)\widetilde{c}_{X^{\ast}}(g,h)$.
\end{example}
\begin{proof}
By the properties of  cocycles, we obtain:
\begin{equation}
\tfrac{\widetilde{c}_{X^{\ast}}(g_1^h, g_2^h)}{\widetilde{c}_{X^{\ast}}(g_1,g_2)}
=[\widetilde{c}_{X^{\ast}}(h^{-1}, g_1g_2h)\widetilde{c}_{X^{\ast}}(g_1g_2,h)][\widetilde{c}_{X^{\ast}}(h^{-1}, g_1h)\widetilde{c}_{X^{\ast}}(g_1,h)]^{-1} [\widetilde{c}_{X^{\ast}}(h^{-1}, g_2h)\widetilde{c}_{X^{\ast}}(g_2,h)]^{-1}.
\end{equation}
So $\nu(\alpha,g)=\widetilde{c}_{X^{\ast}}(h^{-1}, g h)\widetilde{c}_{X^{\ast}}(g,h)$ by the uniqueness.
\end{proof}
\begin{example}[Barthel]\label{chap3example2}
If $\alpha=\begin{pmatrix}
1& 0\\
0& y
\end{pmatrix}\in \GL_2(\R)$ and  $g=\begin{pmatrix} a& b\\ c& d\end{pmatrix}\in \SL_2(\R)$,  then
\begin{align}\label{chap3alpp}
\nu(\alpha,g)=\left\{\begin{array}{lr} (y,  a)_{\R} & c=0,\\ (c, y)_\R \gamma(y, \psi^{\tfrac{1}{2}})^{-1} & c\neq 0.\end{array}\right.
\end{align} 
\end{example}
\begin{proof}
The proof is similar to   the $p$-adic case in  \cite[p.212, Prop.1.2.A]{Ba}.
\end{proof}
From now on,  if $\alpha =\begin{pmatrix}
1& 0\\
0& y
\end{pmatrix}$, we will write $\nu(y, g)$ for $\nu(\alpha, g)$,  and $g^y$ for $g^{\alpha}$. 
There exists an exact sequence: 
$$1 \longrightarrow \SL_2(\R) \longrightarrow \GL_2(\R)\stackrel{\lambda}{\longrightarrow} \R^{\times}\longrightarrow 1.$$
Let us choose a section map
\begin{align}\label{chap3ss}
s: \R^{\times} \longrightarrow  \GL_2(\R); y \longmapsto \begin{pmatrix}
1& 0\\
0& y
\end{pmatrix}.
\end{align}
Then $\GL_2(\R)\simeq \R^{\times} \ltimes \SL_2(\R)$. By lifting the action of $\R^{\times}$ from $\SL_2(\R)$ to $\widetilde{\SL}_2(\R)$,  we   obtain a group $\R^{\times} \ltimes \widetilde{\SL}_2(\R)$; let us denote it by $\widetilde{\GL}_2(\R)$. Then there exists  an exact sequence:
$$ 1\longrightarrow \mu_8 \longrightarrow \widetilde{\GL}_2(\R) \stackrel{ }{\longrightarrow} \GL_2(\R) \longrightarrow 1.$$
 Let $\widetilde{C}_{X^{\ast}}$ denote  the $2$-cocycle  associated to this exact sequence from $\widetilde{c}_{X^{\ast}}$. For  $(y_1, g_1, \epsilon_1), (y_2, g_2, \epsilon_2)\in \R^{\times} \ltimes \widetilde{\SL}_2(\R)$,
\begin{equation}
\begin{split}
(y_1, g_1, \epsilon_1)\cdot (y_2, g_2, \epsilon_2)&=(y_1y_2, [g_1, \epsilon_1]^{y_2}[g_2, \epsilon_2])\\
& =(y_1y_2, [g_1^{y_2}, \nu(y_2, g_1)\epsilon_1][g_2, \epsilon_2])\\
&=(y_1y_2, [g_1^{y_2}g_2, \nu( y_2,g_1)\widetilde{c}_{X^{\ast}}(g_1^{y_2}, g_2)\epsilon_1\epsilon_2]).
\end{split}
\end{equation}
Therefore, we obtain:
\begin{lemma}\label{chap3yg1}
\begin{equation}\label{chap3yg}
\widetilde{C}_{X^{\ast}}([y_1,g_1], [y_2, g_2]) =\nu( y_2,g_1)\widetilde{c}_{X^{\ast}}(g_1^{y_2}, g_2),
\end{equation}
for   $[y_i, g_i]\in \R^{\times} \ltimes \SL_2(\R)$.
\end{lemma}
\subsection{$\mu_2$-covering}
Instead of the $8$-degree cover,  let us consider the $2$-degree covering case. For an automorphism $\alpha$ of $\SL_2(\R) $, let us define $\nu_2: \alpha \times \SL_2(\R)  \longrightarrow \mu_2$ such that
$$(g, \epsilon)^{\alpha}=(g^{\alpha}, \nu_2( \alpha,g) \epsilon), \qquad  \qquad (g, \epsilon) \in \overline{\SL}_2(\R).$$
Similarly, $\nu_2$ determines an automorphism of $\overline{\SL}_2(\R)$ iff the following equality holds:
\begin{equation}\label{chap3crao2}
\overline{c}_{X^{\ast}}(g_1, g_2)=\overline{c}_{X^{\ast}}(g_1^{\alpha},g_2^{\alpha})\nu_2(\alpha,g_1) \nu_2(\alpha,g_2) \nu_2(\alpha, g_1g_2)^{-1}.
\end{equation}
\begin{lemma}\label{chap3twocover}
For two automorphisms $\alpha_1, \alpha_2$,  $\nu_2(\alpha_1\alpha_2, g)=\nu_2(\alpha_1,g)\nu_2(\alpha_2, g^{\alpha_1})$.
\end{lemma}
\begin{proof}
The argument is similar to that of Lemma \ref{chap3twoau}.
\end{proof}
\begin{lemma}\label{chap3nu23}
$\nu_2( \alpha,g )=\nu(\alpha,g) \tfrac{m_{X^{\ast},\psi}(g)}{m_{X^{\ast},\psi}(g^{\alpha})}$.
\end{lemma}
\begin{proof}
\begin{equation*}
\begin{split}
&\overline{c}_{X^{\ast}}(g_1, g_2)m_{X^{\ast},\psi}(g_1g_2) m_{X^{\ast},\psi}(g_1)^{-1}m_{X^{\ast},\psi}(g_2)^{-1}\\
&\stackrel{\textrm{ Lem. \ref{relacto}}}{=}  \widetilde{c}_{X^{\ast}}(g_1,g_2)\\
&\stackrel{(\ref{chap3equ})}{=}\widetilde{c}_{X^{\ast}}(g_1^{\alpha},g_2^{\alpha})\nu(\alpha,g_1) \nu(\alpha,g_2) \nu(\alpha,g_1g_2)^{-1}\\
&\stackrel{\textrm{ Lem. \ref{relacto}}}{=}\overline{c}_{X^{\ast}}(g_1^{\alpha},g_2^{\alpha}) m_{X^{\ast},\psi}(g_1^{\alpha}g_2^{\alpha}) m_{X^{\ast},\psi}(g_1^{\alpha})^{-1}m_{X^{\ast},\psi}(g_2^{\alpha})^{-1}\nu(\alpha,g_1) \nu(\alpha,g_2) \nu(\alpha,g_1g_2)^{-1}.
\end{split}
\end{equation*}
Hence:
$$\tfrac{\overline{c}_{X^{\ast}}(g_1, g_2)}{\overline{c}_{X^{\ast}}(g_1^{\alpha},g_2^{\alpha})}= [\tfrac{m_{X^{\ast},\psi}(g_1g_2)}{m_{X^{\ast},\psi}(g_1^{\alpha}g_2^{\alpha})}]^{-1} \tfrac{m_{X^{\ast},\psi}(g_1)}{ m_{X^{\ast},\psi}(g_1^{\alpha})}
\tfrac{m_{X^{\ast},\psi}(g_2)}{ m_{X^{\ast},\psi}(g_2^{\alpha})}\nu(\alpha,g_1) \nu(\alpha,g_2) \nu(\alpha,g_1g_2)^{-1}.$$
According to the above (\ref{chap3crao2}), we have $\nu_2( \alpha,g)=\nu(\alpha,g) \tfrac{m_{X^{\ast},\psi}(g)}{m_{X^{\ast},\psi}(g^{\alpha})}$.
\end{proof}
\begin{example}\label{twocoverex}
If $\alpha=h\in  \SL_2(\R)$, and  $h$ acts on $\SL_2(\R)$ by  conjugation, i.e., $g^h=h^{-1}gh$, for $g\in \SL_2(\R)$,  then $\nu_2(\alpha,g)=\overline{c}_{X^{\ast}}(h^{-1}, g h)\overline{c}_{X^{\ast}}(g,h)=\overline{c}_{X^{\ast}}(h^{-1}, g)\overline{c}_{X^{\ast}}(h^{-1}g,h)$. 
\end{example}
\begin{proof}
The argument is similar to that of Lemma \ref{chap3example1}. 
\end{proof}
\begin{example}\label{nu2abcd}
Let $g=\begin{pmatrix} a& b\\ c& d\end{pmatrix}\in \SL_2(\R)$. Then $\nu_2(\omega,g)=\left\{ \begin{array}{cr}  ( d, -bc)_{\R} & \textrm{ if } d\neq 0, c\neq 0, b\neq 0,\\  ( d, ac)_{\R} &\textrm{ if } d\neq 0, c\neq 0, b=0,\\  ( d, b)_{\R} &  \textrm{ if } d\neq 0, c= 0, b\neq 0, \\ 1 & \textrm{ if } d\neq 0, c= 0, b=0,\\ ( -c, -b)_{\R} & \textrm{ if } d= 0, c\neq 0, b\neq 0.  \end{array}\right.$
\end{example}
\begin{proof}
Take $h=\omega$ in the preceding lemma.  Then 
\[g\omega=\begin{pmatrix} b& -a\\ d& -c\end{pmatrix}, \quad \omega^{-1}g\omega=\begin{pmatrix} d& -c\\ -b& a\end{pmatrix},\] 
\[\nu_2(\omega,g)=\overline{c}_{X^{\ast}}(\omega^{-1}, g \omega)\overline{c}_{X^{\ast}}(g,\omega), \]
\[\overline{c}_{X^{\ast}}(g,\omega)= (x(g), x(  \omega))_{\R}(-x(g) x(  \omega), x(g \omega))_{\R},\]
\[\overline{c}_{X^{\ast}}(\omega^{-1}, g \omega)=(x(\omega^{-1}), x( g \omega))_{\R}(-x(\omega^{-1}) x( g \omega), x(\omega^{-1}g \omega))_{\R}.\]
\begin{itemize}
\item[(1)] If $d\neq 0$, $c\neq 0$ and $b\neq 0$, $\overline{c}_{X^{\ast}}(g,\omega)=(c, 1)_{\R} (-c, d)_{\R}=(-c, d)_{\R}$,    $\overline{c}_{X^{\ast}}(\omega^{-1}, g \omega)=(-1, d)_{\R}( d, -b)_{\R}$, $\nu_2(\omega,g)=(-c, d)_{\R}(-1, d)_{\R}( d, -b)_{\R}=( d, -bc)_{\R}$.
\item[(2)] If $d\neq 0$, $c\neq 0$ and $b= 0$, $\overline{c}_{X^{\ast}}(g,\omega)=(c, 1)_{\R} (-c, d)_{\R}=(-c, d)_{\R}$, $\overline{c}_{X^{\ast}}(\omega^{-1}, g \omega)=(-1, d)_{\R}( d, a)_{\R}$, $\nu_2(\omega,g)=(-c, d)_{\R}(-1, d)_{\R}( d, a)_{\R}=( d, ac)_{\R}$.
\item[(3)] If $d\neq 0$, $c= 0$ and $b\neq 0$, $\overline{c}_{X^{\ast}}(g,\omega)=(d, 1)_{\R} (-d, d)_{\R}=1$, $\overline{c}_{X^{\ast}}(\omega^{-1}, g \omega)=(-1, d)_{\R}( d, -b)_{\R}$,$\nu_2(\omega,g)=(-1, d)_{\R}( d, -b)_{\R}=( d, b)_{\R}$.
\item[(4)] If $d\neq 0$, $c= 0$ and $b= 0$, $\overline{c}_{X^{\ast}}(g,\omega)=(d, 1)_{\R} (-d, d)_{\R}=1$,  $\overline{c}_{X^{\ast}}(\omega^{-1}, g \omega)=(-1, d)_{\R}( d, a)_{\R}$, $\nu_2(\omega,g)=(-1, d)_{\R}( d, a)_{\R}=( d, -a)_{\R}=( d, -d^{-1})_{\R}=1$.
\item[(5)] If $d= 0$, $c\neq 0$ and $b\neq 0$, $\overline{c}_{X^{\ast}}(g,\omega)=(c, 1)_{\R} (-c, -c)_{\R}=(-c,-c)_{\R}$,$\overline{c}_{X^{\ast}}(\omega^{-1}, g \omega)=(-1, -c)_{\R}( -c, -b)_{\R}$, $\nu_2(\omega,g)=(-c,-c)_{\R}(-1, -c)_{\R}( -c, -b)_{\R}=( -c, -b)_{\R}$.
\end{itemize}
 
\end{proof}
\begin{lemma}\label{nu2}
For $g=\begin{pmatrix} a& b\\ c& d\end{pmatrix}\in \SL_2(\R)$, $y\in \R^{\times}$, we have:
\begin{equation}\label{chap3alppu2}
\nu_2(y,g)=\left\{\begin{array}{lr}  (y,  a)_\R & c=0,\\ 1 & c\neq 0. \end{array}\right.
\end{equation}
\end{lemma}
\begin{proof}
1)  If $c=0$, $m_{{X^{\ast},\psi}}(g)=m_{{X^{\ast},\psi}}(g^{y})$, so $\nu_2(y, g)=\nu(y, g)=(y,  a)_{\R}$.\\
2)  If $c\neq 0$, then:
\begin{itemize}
\item[] $m_{{X^{\ast},\psi}}(g)=\gamma(c, \psi^{\tfrac{1}{2}})^{-1} \gamma(\psi^{\tfrac{1}{2}})^{-1}=\gamma({\psi}^{\tfrac{1}{2} c})^{-1}$,
 \item[] $m_{{X^{\ast},\psi}}(g^{y})=\gamma({\psi}^{\tfrac{1}{2} cy^{-1}})^{-1}$,
\item[] $\tfrac{m_{{X^{\ast},\psi}}(g)}{  m_{{X^{\ast},\psi}}(g^{y})}=\gamma(y^{-1}, \psi^{\tfrac{1}{2}}) (c,y)_\R$,
 \item[] $\nu_2(y, g)=\nu(y, g)\tfrac{m_{{X^{\ast},\psi}}(g)}{  m_{{X^{\ast},\psi}}(g^{y})}=( c, y)_\R \gamma(y, \psi^{\tfrac{1}{2}})^{-1}\gamma(y^{-1}, \psi^{\tfrac{1}{2}}) (c,y)_\R=1$.
      \end{itemize}
\end{proof}

Lifting the natural action of $\mathbb R^{\times}$ on $\SL_{2}(\mathbb R)$ to the two–fold cover $\overline{\SL}_{2}(\mathbb R)$ produces the semidirect product  
\[
\overline{\GL}_{2}(\mathbb R)=\mathbb R^{\times}\ltimes\overline{\SL}_{2}(\mathbb R),
\]  
which fits into the exact sequence  
\[
1\longrightarrow \mu_{2}\longrightarrow\overline{\GL}_{2}(\mathbb R)\longrightarrow\mathbb R^{\times}\ltimes\SL_{2}(\mathbb R)\longrightarrow 1.
\]  
We denote the corresponding $2$-cocycle on $\mathbb R^{\times}\ltimes\SL_{2}(\mathbb R)$ by $\overline C_{X^{\ast}}$.  Explicitly,  
\[
\overline C_{X^{\ast}}([y_{1},g_{1}],[y_{2},g_{2}])
=\nu_{2}(y_{2},g_{1})\;\overline c_{X^{\ast}}(g_{1}^{y_{2}},g_{2}),
\qquad [y_{i},g_{i}]\in\mathbb R^{\times}\ltimes\SL_{2}(\mathbb R).
\]  
Inserting the identity from Lemma~\ref{chap3nu23} yields  
\begin{equation}\label{chap3292inter}
\overline C_{X^{\ast}}([y_{1},g_{1}],[y_{2},g_{2}])
=m_{X^{\ast},\psi}(g_{1}^{y_{2}}g_{2})^{-1}\,
m_{X^{\ast},\psi}(g_{1})\,
m_{X^{\ast},\psi}(g_{2})\,
\widetilde C_{X^{\ast}}([y_{1},g_{1}],[y_{2},g_{2}]).
\end{equation}
Abusing notation, we set  
\[
m_{X^{\ast},\psi}[y,g]=m_{X^{\ast},\psi}(g),
\qquad [y,g]\in\mathbb R^{\times}\ltimes\SL_{2}(\mathbb R).
\]  
Then, for all $h_{1},h_{2}\in\GL_{2}(\mathbb R)$,
\begin{equation}\label{chap3293inter}
\overline C_{X^{\ast}}(h_{1},h_{2})
=m_{X^{\ast},\psi}(h_{1}h_{2})^{-1}\,
m_{X^{\ast},\psi}(h_{1})\,
m_{X^{\ast},\psi}(h_{2})\,
\widetilde C_{X^{\ast}}(h_{1},h_{2}).
\end{equation}
\subsection{Link $\widehat{c}_{z}$  to $\overline{c}_{X^{\ast}}$ and $\widetilde{c}_{X^{\ast}}$}\label{linkedST1} Assume $z=p_z(i)$. 
\begin{lemma}\label{beznu2nu}
\begin{itemize}
\item[(1)] $\widehat{c}_{z}(g_1,g_2)=\overline{c}_{X^{\ast}}(g_1,g_2)\nu_2(p_z,g_1)^{-1} \nu_2(p_z,g_2)^{-1} \nu_2(p_z, g_1g_2)\overline{s}(g_1^{p_z})\overline{s}(g_2^{p_z}) \overline{s}(g_1^{p_z}g_2^{p_z})^{-1}$.
\item[(2)]$\widehat{c}_{z}(g_1,g_2)=\widetilde{c}_{X^{\ast}}(g_1,g_2)\nu(p_z,g_1)^{-1} \nu(p_z,g_2)^{-1} \nu(p_z, g_1g_2)\widetilde{s}(g_1^{p_z})\widetilde{s}(g_2^{p_z}) \widetilde{s}(g_1^{p_z}g_2^{p_z})^{-1}$.
\end{itemize}
\end{lemma}
\begin{proof}
1) 
\begin{align*}
\widehat{c}_{z}(g_1,g_2)&\stackrel{\textrm{ Lem. }\ref{ztoz0}}{=}\widehat{c}_{z_0}(g_1^{ p_z},  g_2^{p_z})\\
&\stackrel{\textrm{ Lem. }\ref{twocoz0}}{=}\overline{c}_{X^{\ast}}(g_1^{ p_z},g_2^{ p_z})\overline{s}(g_1^{ p_z})\overline{s}(g_2^{ p_z}) \overline{s}(g_1^{ p_z}g_2^{ p_z})^{-1}\\
&\stackrel{\textrm{ Eq. }\ref{chap3crao2}}{=}\overline{c}_{X^{\ast}}(g_1,g_2)\nu_2(p_z,g_1)^{-1} \nu_2(p_z,g_2)^{-1} \nu_2(p_z, g_1g_2)\overline{s}(g_1^{ p_z})\overline{s}(g_2^{ p_z}) \overline{s}(g_1^{ p_z}g_2^{ p_z})^{-1}.
\end{align*}
2) 
\begin{align*}
\widehat{c}_{z}(g_1,g_2)&\stackrel{\textrm{ Lem. }\ref{ztoz0}}{=}\widehat{c}_{z_0}(g_1^{ p_z},  g_2^{p_z})\\
&\stackrel{\textrm{ Lem. }\ref{eightcoz0}}{=}\widetilde{c}_{X^{\ast}}(g_1^{ p_z},g_2^{ p_z})\widetilde{s}(g_1^{ p_z})\widetilde{s}(g_2^{ p_z}) \widetilde{s}(g_1^{ p_z}g_2^{ p_z})^{-1}\\
&\stackrel{\textrm{ Eq. }\ref{chap3equ}}{=}\widetilde{c}_{X^{\ast}}(g_1,g_2)\nu(p_z,g_1)^{-1} \nu(p_z,g_2)^{-1} \nu(p_z, g_1g_2)\widetilde{s}(g_1^{ p_z})\widetilde{s}(g_2^{ p_z}) \widetilde{s}(g_1^{ p_z}g_2^{ p_z})^{-1}.
\end{align*}
\end{proof}
\subsection{One Lemma}
\begin{lemma}\label{equah-1}
$\widehat{c}_{z_0}(g_1,g_2) \widehat{c}_{z_0}(g_1^{h_{-1}}, g_2^{h_{-1}})=1$, for $g_i\in \SL_2(\R)$.
\end{lemma}
\begin{proof}
\begin{align*}
&\widehat{c}_{z_0}(g_1,g_2) \widehat{c}_{z_0}(g_1^{h_{-1}}, g_2^{h_{-1}})\label{betz0zh1}\\
&=[\overline{c}_{X^{\ast}}(g_1,g_2)\overline{c}_{X^{\ast}}(g^{h_{-1}}_1,g^{h_{-1}}_2)][\overline{s}(g_1)\overline{s}(g^{h_{-1}}_1)][\overline{s}(g_2)\overline{s}(g^{h_{-1}}_2)] [\overline{s}(g_1g_2)\overline{s}(g_1^{h_{-1}}g_2^{h_{-1}})]^{-1}\\
&=[\nu_2(-1, g_1)\nu_2(-1, g_2)\nu_2(-1, g_1g_2)^{-1}][\overline{s}(g_1)\overline{s}(g^{h_{-1}}_1)][\overline{s}(g_2)\overline{s}(g^{h_{-1}}_2)] [\overline{s}(g_1g_2)\overline{s}(g_1^{h_{-1}}g_2^{h_{-1}})]^{-1}.
\end{align*}
Let $g=pk_g\in \SL_2(\R)$ with $p\in P_{>0}(\R)$ and  $k_g\in \U(\C)$.  Then
\[p^{h_{-1}}\in P_{>0}(\R), \qquad k_g^{h_{-1}}\in \U(\C), \qquad g^{h_{-1}}=p^{h_{-1}}k_g^{h_{-1}},\]
and  \[ \nu_2(-1, g)=\nu_2(-1, k_g),\qquad \overline{s}(g)=\overline{s}(k_g).\]
 For any $g\in \U(\C)$, we have \[\nu_2(-1, g)=\overline{s}(g)\overline{s}(g^{h_{-1}})=\left\{ \begin{array}{cl} -1 & \textrm{ if } g= \begin{pmatrix} -1& 0\\0 &-1\end{pmatrix},\\ 1 & \textrm{ otherwise}.\end{array}\right.\]
Consequently, we may (and do) assume  $g_i\in \U(\C)$ and  the claimed identity follows.
\end{proof}
\subsection{$\mathbb{T}$-covering} 
Rather than working with the two–fold cover itself, we now focus on the cocycle $\widehat{c}_z$.  
For every automorphism \(\alpha\) of \(\mathrm{SL}_{2}(\mathbb R)\) we introduce a map  
\[
\nu_{z}\colon\alpha\times\mathrm{SL}_{2}(\mathbb R)\longrightarrow \mathbb{T}
\]  
such that  
\[
(g,\epsilon)^{\alpha}=\bigl(g^{\alpha},\nu_{z}(\alpha,g)\,\epsilon\bigr),\qquad (g,\epsilon)\in\widehat{\SL}_2^z(\mathbb R).
\]  
This \(\nu_{z}\) extends to an automorphism of \(\widehat{\SL}_2^z(\mathbb R)\) precisely when  
\begin{equation}\label{chap3crao2beta2}
\widehat{c}_{z}(g_1,g_2)=\widehat{c}_{z}(g_1^{\alpha},g_2^{\alpha})\,\nu_{z}(\alpha,g_1)\,\nu_{z}(\alpha,g_2)\,\nu_{z}(\alpha,g_1g_2)^{-1}.
\end{equation}
Let us determine the explicit form of \(\nu_{z}(\alpha,g)\).
\begin{lemma}\label{znunu}
$\nu_z(\alpha,g)=\nu(\alpha,g)\tfrac{\nu(p_z,g^{\alpha})\widetilde{s}(g^{p_z})}{\nu(p_z,g)\widetilde{s}(g^{{\alpha}p_z})}$, for $g\in \SL_2(\R)$.
\end{lemma}
\begin{proof}
\begin{align*}
&\tfrac{\widehat{c}_z(g_1^{\alpha},g_2^{\alpha})}{\widehat{c}_z(g_1, g_2)}=\nu_z(\alpha,g_1)^{-1} \nu_z(\alpha,g_2)^{-1} \nu_z(\alpha, g_1g_2)\\
&\stackrel{\textrm{ Lem.}\ref{beznu2nu}}{=}\tfrac{\widetilde{c}_{X^{\ast}}(g_1^{\alpha},g_2^{\alpha})\nu(p_z,g_1^{\alpha})^{-1} \nu(p_z,g_2^{\alpha})^{-1} \nu(p_z, g_1^{\alpha}g_2^{\alpha})\widetilde{s}(g_1^{{\alpha}p_z})\widetilde{s}(g_2^{{\alpha}p_z}) \widetilde{s}(g_1^{{\alpha}p_z}g_2^{{\alpha}p_z})^{-1}}{\widetilde{c}_{X^{\ast}}(g_1,g_2)\nu(p_z,g_1)^{-1} \nu(p_z,g_2)^{-1} \nu(p_z, g_1g_2)\widetilde{s}(g_1^{p_z})\widetilde{s}(g_2^{p_z}) \widetilde{s}(g_1^{p_z}g_2^{p_z})^{-1}}\\
&=\tfrac{\nu(p_z,g_1^{\alpha})^{-1} \nu(p_z,g_2^{\alpha})^{-1} \nu(p_z, g_1^{\alpha}g_2^{\alpha})\widetilde{s}(g_1^{{\alpha}p_z})\widetilde{s}(g_2^{{\alpha}p_z}) \widetilde{s}(g_1^{{\alpha}p_z}g_2^{{\alpha}p_z})^{-1}}{\nu(p_z,g_1)^{-1} \nu(p_z,g_2)^{-1} \nu(p_z, g_1g_2)\widetilde{s}(g_1^{p_z})\widetilde{s}(g_2^{p_z}) \widetilde{s}(g_1^{p_z}g_2^{p_z})^{-1}}\nu(\alpha,g_1)^{-1}\nu(\alpha,g_2)^{-1}\nu(\alpha,g_1g_2)\\
&=\Big[\tfrac{\nu(p_z,g_1^{\alpha})\widetilde{s}(g_1^{p_z})\nu(\alpha,g_1)}{\nu(p_z,g_1)\widetilde{s}(g_1^{{\alpha}p_z})}\Big]^{-1}
\Big[\tfrac{\nu(p_z,g_2^{\alpha})\widetilde{s}(g_2^{p_z})\nu(\alpha,g_2)}{\nu(p_z,g_2)\widetilde{s}(g_2^{{\alpha}p_z})}\Big]^{-1}
\Big[\tfrac{\nu(p_z,g_1^{\alpha}g_2^{\alpha})\widetilde{s}(g_1^{p_z}g_2^{p_z})\nu(\alpha,g_1g_2)}{\nu(p_z,g_1g_2)\widetilde{s}(g_1^{{\alpha}p_z}g_2^{{\alpha}p_z})}\Big].
\end{align*}
\end{proof}

\begin{example}
 $\nu_{z_0}(\alpha,g)=\nu(\alpha,g)\tfrac{\widetilde{s}(g)}{\widetilde{s}(g^{{\alpha}})}$.
\end{example}
\begin{example}\label{exz0z}
If  $\alpha=\begin{pmatrix} 1& 0 \\ 0 & y\end{pmatrix}$, then 
$\nu_{z_0}(y,g)=\nu(y,g)\tfrac{\widetilde{s}(g)}{\widetilde{s}(g^{y})}=\nu_2( y,g ) \tfrac{\overline{s}(g)}{\overline{s}(g^{y})}$.
\end{example}
Lifting the natural action of \(\mathbb R^{\times}\) on \(\mathrm{SL}_{2}(\mathbb R)\) to \(\widehat{\SL}_2^z(\mathbb R)\) yields the semidirect product  
\[
\widehat{\GL}_2^z(\mathbb R)=\mathbb R^{\times}\ltimes\widehat{\SL}_2^z(\mathbb R)
\]  
fitting into the exact sequence  
\[
1\longrightarrow \mathbb{T}\longrightarrow\widehat{\GL}_2^z(\mathbb R)\longrightarrow\mathbb R^{\times}\ltimes\mathrm{SL}_{2}(\mathbb R)\longrightarrow 1.
\]  
We retain the symbol $\widehat{C}_z$ for the corresponding \(2\)-cocycle on \(\mathbb R^{\times}\ltimes\mathrm{SL}_{2}(\mathbb R)\); explicitly,
\begin{equation}\label{betaz}
\widehat{C}_z\bigl([y_{1},g_{1}],[y_{2},g_{2}]\bigr)=\nu_{z}(y_{2},g_{1})\,\widehat{c}_{z}(g_{1}^{y_{2}},g_{2}),
\qquad [y_{i},g_{i}]\in\mathbb R^{\times}\ltimes\mathrm{SL}_{2}(\mathbb R).
\end{equation}
\begin{remark}
For simplicity, when \( z = z_0 \) we write
\[
\widehat{\GL}_2(\mathbb R)=\widehat{\GL}^{z_0}_2(\mathbb R)
\quad\text{and}\quad
\widehat{\SL}_2(\mathbb R)=\widehat{\SL}^{z_0}_2(\mathbb R).
\]
\end{remark}
For $h=[y,g]\in \R^{\times}\ltimes \SL_2(\R)\simeq \GL_2(\R)$, let us also write:
\[ \widetilde{s}(h)=\widetilde{s}(g), \qquad \overline{s}(h)=\overline{s}(g).\]
\begin{lemma}\label{Cz_0wid}
Let $h_i=[y_i, g_i]\in \R^{\times} \ltimes \SL_2(\R)$. 
\begin{itemize}
\item[(1)] $\widehat{C}_{z_0}(h_1,h_2)=\overline{C}_{X^{\ast}}(h_1,h_2)\overline{s}(h_1)\overline{s}(h_2) \overline{s}(h_1h_2)^{-1}$.
\item[(2)]$\widehat{C}_{z_0}(h_1,h_2)=\widetilde{C}_{X^{\ast}}(h_1,h_2)\widetilde{s}(h_1)\widetilde{s}(h_2) \widetilde{s}(h_1h_2)^{-1}$.
\end{itemize}
\end{lemma}
\begin{proof}
1) 
\begin{align*}
\widehat{C}_{z_0}(h_1,h_2)&\stackrel{(\ref{betaz})}{=}\nu_{z_0}(y_2,g_1)\widehat{c}_{z_0}(g_1^{y_2}, g_2)\\
&\stackrel{\textrm{ Ex.} \ref{exz0z}}{=}\nu_2( y_2,g ) \tfrac{\overline{s}(g)}{\overline{s}(g^{y_2})}\widehat{c}_{z_0}(g_1^{y_2}, g_2)\\
&\stackrel{\textrm{ Lem.} \ref{beznu2nu}(1)}{=}\nu_2( y_2,g ) \tfrac{\overline{s}(g)}{\overline{s}(g^{y_2})}\overline{c}_{X^{\ast}}(g_1^{y_2},g_2)\overline{s}(g_1^{y_2})\overline{s}(g_2) \overline{s}(g_1^{y_2}g_2)^{-1}\\
&=\overline{C}_{X^{\ast}}(h_1,h_2)\overline{s}(h_1)\overline{s}(h_2) \overline{s}(h_1h_2)^{-1}.
\end{align*}
2) 
\begin{align*}
\widehat{C}_{z_0}(h_1,h_2)&\stackrel{(\ref{betaz})}{=}\nu_{z_0}(y_2,g_1)\widehat{c}_{z_0}(g_1^{y_2}, g_2)\\
&\stackrel{\textrm{ Ex.} \ref{exz0z}}{=}\nu(y_2,g_1)\tfrac{\widetilde{s}(g_1)}{\widetilde{s}(g_1^{y_2})}\widehat{c}_{z_0}(g_1^{y_2}, g_2)\\
&\stackrel{\textrm{ Lem.} \ref{beznu2nu}(2)}{=}\nu(y_2,g_1)\tfrac{\widetilde{s}(g_1)}{\widetilde{s}(g_1^{y_2})}\widetilde{c}_{X^{\ast}}(g_1^{y_2},g_2)\widetilde{s}(g_1^{y_2})\widetilde{s}(g_2) \widetilde{s}(g_1^{y_2}g_2)^{-1}\\
&=\widetilde{C}_{X^{\ast}}(h_1,h_2)\widetilde{s}(h_1)\widetilde{s}(h_2) \widetilde{s}(h_1h_2)^{-1}.
\end{align*}
\end{proof}
 Let $\overline{\GL}_2(\mathbb R)$, $\overline{\GL}^{\mu_8}_2(\mathbb R)$, $\widetilde{\GL}_2(\mathbb R)$, and $\widehat{\GL}^z_2(\mathbb R)$ denote the central extensions of $\GL_2(\mathbb R)$ arising from $\overline{C}_{X^{\ast}}$,  $\overline{C}_{X^{\ast}}$, $\widetilde{C}_{X^{\ast}}$, and $\widehat{C}_z$, with centers $\mu_2$,  $\mu_8$, $\mu_8$, and $\mathbb{T}$, respectively. As a consequence, we have  the following  group isomorphisms:
 \begin{equation}\label{overwideGL}
 \overline{\GL}^{\mu_8}_2(\mathbb R)\stackrel{\iota}{\to}\widetilde{\GL}_2(\mathbb R);[h,\epsilon] \longmapsto [h,m_{X^{\ast},\psi}(h)\epsilon],
 \end{equation}
  \begin{equation}\label{overhatGL}
  \overline{\GL}^{\mu_8}_2(\mathbb R)\stackrel{\iota'}{\to}\Im(\iota') (\subseteq\widehat{\GL}_2(\mathbb R));[h,\epsilon] \longmapsto [h,\overline{s}(h)^{-1}\epsilon],
   \end{equation}
    \begin{equation}\label{overwidehatGL}
  \widetilde{\GL}_2(\mathbb R)\stackrel{\iota'}{\to}\Im(\iota') (\subseteq\widehat{\GL}_2(\mathbb R));[h,\epsilon] \longmapsto [h,\widetilde{s}(h)^{-1}\epsilon].
   \end{equation}
\subsection{The decomposition}
 For any $h\in \begin{pmatrix}
a & b\\
c& d\end{pmatrix}\in \GL_2(\R)$, we  write: $$h=\begin{pmatrix}
\sqrt{ |\det(h)|} & 0\\
0& \sqrt{ |\det(h)|}\end{pmatrix}\cdot \Big[\begin{pmatrix}
\sqrt{|\det h|}^{-1} & \\
& \sqrt{|\det h|}^{-1}\end{pmatrix}\begin{pmatrix}
a & b\\
c& d\end{pmatrix} \Big] .$$ So $\GL_2(\R)\simeq \R^{\times}_{>0} \cdot \SL_2^{\pm}(\R) $.  Let  $\overline{\SL}^{\pm}_2(\R)$, $\widetilde{\SL}^{\pm}_2(\R)$, $\widehat{\SL}^{\pm}_2(\R)$ denote the corresponding subgroups of $\overline{\GL}_2(\R)$, $\widetilde{\GL}_2(\R)$ and $\widehat{\GL}_2(\R)$ respectively.
\begin{lemma}\label{isoGLR}
\begin{itemize}
\item[(1)] $\overline{\GL}_2(\R) \simeq \R^{\times}_{>0}\times\overline{\SL}^{\pm}_2(\R)  $.
\item[(2)] $\widetilde{\GL}_2(\R) \simeq \R^{\times}_{>0}\times\widetilde{\SL}^{\pm}_2(\R)$.
\item[(3)] $\widehat{\GL}_2(\R)\simeq\R^{\times}_{>0} \times\widehat{\SL}^{\pm}_2(\R)$.
\end{itemize}
\end{lemma}
\begin{proof}
For $h_1=\begin{pmatrix}
a & b\\
c& d\end{pmatrix}\in \GL_2(\R)$ and $h_2=t\in \R^{\times}_{>0}\subseteq  \GL_2(\R)$, write $y_1=\det h_1$,  $g_1=\begin{pmatrix}
a & b\\
y_1^{-1}c& y_1^{-1}d\end{pmatrix}$, $y_2=t^2$, $g_2=\begin{pmatrix}
t& 0\\
0& t^{-1}\end{pmatrix}$, $g_1^{y_2} g_2=\begin{pmatrix}
at & bt\\
y_1^{-1}t^{-1}c& y_1^{-1}t^{-1}d\end{pmatrix}=g_2g_1$.  Then:
 $$\overline{C}_{X^{\ast}}(h_1, h_2)=\nu_2(y_2,g_1)\overline{c}_{X^{\ast}}(g_1^{y_2}, g_2)=1 \cdot \overline{c}_{X^{\ast}}(g_1^{y_2}, g_2)=1, $$
  $$\overline{C}_{X^{\ast}}(h_2, h_1)=\nu_2(y_1,g_2)\overline{c}_{X^{\ast}}(g_2^{y_1}, g_1)=1\cdot \overline{c}_{X^{\ast}}(g_2, g_1)=1, $$
  $$\widetilde{C}_{X^{\ast}}(h_1, h_2) =\nu(y_2,g_1)\widetilde{c}_{X^{\ast}}(g_1^{y_2}, g_2)=\nu(y_2,g_1)=1,$$
   $$\widetilde{C}_{X^{\ast}}(h_2, h_1) =\nu(y_1,g_2)\widetilde{c}_{X^{\ast}}(g_2^{y_1}, g_1)=\nu(y_1,g_2)=1,$$
\begin{align*}
\widehat{C}_{z_0}(h_1,h_2)&=\overline{C}_{X^{\ast}}(h_1,h_2)\overline{s}(h_1)\overline{s}(h_2) \overline{s}(h_1h_2)^{-1}\\
&=\overline{s}(h_1)\overline{s}(h_2) \overline{s}(h_1h_2)^{-1}\\
&=\overline{s}(g_1)\overline{s}(g_2) \overline{s}(g_1^{y_2} g_2)^{-1}\\
&=\overline{s}(g_1)\overline{s}(g_2) \overline{s}(g_2 g_1)^{-1}\\
&=\overline{s}(g_1)\overline{s}( g_1)^{-1}=1,
\end{align*}
\begin{align*}
\widehat{C}_{z_0}(h_2,h_1)&=\overline{C}_{X^{\ast}}(h_2,h_1)\overline{s}(h_2)\overline{s}(h_1) \overline{s}(h_2h_1)^{-1}\\
&=\overline{s}(h_1)\overline{s}(h_2) \overline{s}(h_2h_1)^{-1}\\
&=\overline{s}(h_1)\overline{s}(h_2) \overline{s}(h_1h_2)^{-1}\\
&=\widehat{C}_{z_0}(h_1,h_2)=1.
\end{align*}
\end{proof}
Under the above isomorphisms, $\overline{\GL}_2(\Q) \nsimeq \Q^{\times}_{>0} \times\overline{\SL}^{\pm}_2(\Q) $,  $\widetilde{\GL}_2(\Q) \nsimeq \Q^{\times}_{>0}\times\widetilde{\SL}^{\pm}_2(\Q)  $ and $\widehat{\GL}_2(\Q) \nsimeq  \Q^{\times}_{>0}\times\widehat{\SL}^{\pm}_2(\Q) $. For any $h=\begin{pmatrix}
a & b\\
c& d\end{pmatrix} \in \GL_2(\Q)$, we write:
\[h= \underbrace{\begin{pmatrix}
\sqrt{ |\det(h)|} & 0\\
0& \sqrt{ |\det(h)|}\end{pmatrix}}_{h_1}\cdot  \underbrace{\begin{pmatrix}
\sqrt{ |\det(h)|}^{-1} & 0\\
0& \sqrt{ |\det(h)|}\end{pmatrix}}_{h_2}  \cdot\underbrace{\begin{pmatrix}
a & b\\
c |\det(h)|^{-1}& d|\det(h)|^{-1}\end{pmatrix}}_{h_3}    .\]
\[y_3=\det h_3=\frac{\det h}{|\det h|}, g_3=\begin{pmatrix}
a & b\\
c\det(h)^{-1}& d\det(h)^{-1}\end{pmatrix}, y_2=\det h_2=1, g_2=h_2\]
\[\widetilde{C}_{X^{\ast}}(h_2, h_3) =\nu(y_3,g_2)\widetilde{c}_{X^{\ast}}(g_2^{y_3}, g_3)=\nu(y_3,g_2)\stackrel{(\ref{chap3alpp})}{=}1\]
\[\overline{C}_{X^{\ast}}(h_2, h_3)=\nu_2(y_3,g_2)\overline{c}_{X^{\ast}}(g_2^{y_3}, g_3)\stackrel{(\ref{chap3alppu2})}{=}1 \cdot \overline{c}_{X^{\ast}}(g_2^{y_3}, g_3)=1.\]
\[\widehat{C}_{z_0}(h_2, h_3)=\nu_{z_0}(y_3,g_2)\widehat{c}_{z_0}(g_2^{y_3}, g_3)=\nu_{z_0}(y_3,g_2)=\nu_2(y_3,g_2)=1.\]
Thus, for $\overline{h}=[h,t_1]\in \overline{\GL}_2(\Q)$,  $\widetilde{h}=[h,t_2]\in \widetilde{\GL}_2(\Q) $ and $\widehat{h}=[h,t_3]\in \widehat{\GL}_2(\Q) $,  we can write:
\[\overline{h}=h_1 \cdot h_2 \cdot \underbrace{[h_3,t_1]}_{\overline{h}_3} , \qquad \widetilde{h}= h_1\cdot h_2 \cdot\underbrace{[h_3,t_2]}_{\widetilde{h}_3},  \qquad \widehat{h}= h_1\cdot h_2 \cdot\underbrace{[h_3,t_3]}_{\widehat{h}_3}.\]
Note that if  $\det h\in \Q^{\times2}$, then all $h_1, h_2, h_3$ belong to $\GL_2(\Q)$.

\subsection{The parabolic group}
\begin{lemma}\label{prbc}
For $h=\begin{pmatrix}
a& b\\
c& d\end{pmatrix}\in  \SL^{\pm}_2(\R)$, $p=\begin{pmatrix}
a_1& b_1\\
0&   a_1^{-1} \det p \end{pmatrix}\in P^{\pm}_{>0}(\R)$, we have:
\begin{itemize}
\item $\widetilde{C}_{X^{\ast}}(p, h)=1$;
\item $\widetilde{C}_{X^{\ast}}(h, p)=\left\{\begin{array}{cl}  (\det p, a)_{\R}  & \textrm{ if } c=0,\\(c\det h, \det p)_{\R} \gamma(\det p, \psi^{\tfrac{1}{2}})^{-1}  & \textrm{ if } c\neq 0. \end{array}\right.$
\end{itemize}
\end{lemma}
\begin{proof}
 Let   $g_h=\begin{pmatrix}
a& b\\
c(\det h)^{-1}& d(\det h)^{-1}\end{pmatrix}$, and $g_p=\begin{pmatrix}
a_1& b_1\\
0&  a_1^{-1}\end{pmatrix}$. Then:
$$\widetilde{C}_{X^{\ast}}(p, h)=\nu(\det h,g_p)\widetilde{c}_{X^{\ast}}((g_p)^{\det h}, g_h)=(\det h, a_1)_{\R}=1;$$
$$\widetilde{C}_{X^{\ast}}( h,p)=\nu(\det p,g_h)\widetilde{c}_{X^{\ast}}((g_h)^{\det p}, g_p)=\nu(\det p,g_h)=\left\{\begin{array}{cl} (\det p, a)_{\R} & \textrm{ if } c=0,\\(c\det h, \det p)_{\R} \gamma(\det p, \psi^{\tfrac{1}{2}})^{-1}  & \textrm{ if } c\neq 0. \end{array}\right.$$
\end{proof}
\begin{lemma}\label{chap6Cover}
For $h=\begin{pmatrix}
a& b\\
c& d\end{pmatrix}\in  \SL^{\pm}_2(\R)$, $p=\begin{pmatrix}
a_1 & b_1\\
0& a_1^{-1} \det p  \end{pmatrix}\in P^{\pm}_{>0}(\R)$, we have:
\begin{itemize}
\item $\overline{C}_{X^{\ast}}(p, h)=1$;
\item $\overline{C}_{X^{\ast}}(h, p)=\left\{\begin{array}{cl}  (\det p, a)_{\R}  & \textrm{ if } c=0,\\1  & \textrm{ if } c\neq 0. \end{array}\right.$
\end{itemize}
\end{lemma}
\begin{proof}
 Let $g_h$ and $g_p$ be as in the proof of the above lemma. Then:
\begin{align*}
\overline{C}_{X^{\ast}}(p, h)&=\nu_2(\det h,g_p)\overline{c}_{X^{\ast}}((g_p)^{\det h}, g_h)\\
&=(\det h, a_1)_{\R}(x(g_p^{\det h}), x(g_h))_\R(-x(g_p^{\det h})x(g_h), x(g_p^{\det h}g_h))_\R\\
&=1;
\end{align*}
\begin{align*}
\overline{C}_{X^{\ast}}( h,p)&=\nu_2(\det p,g_h)\overline{c}_{X^{\ast}}((g_h)^{\det p}, g_p)=\nu_2(\det p,g_h)=\left\{\begin{array}{cl} (\det p, a)_{\R} & \textrm{ if } c=0,\\1  & \textrm{ if } c\neq 0. \end{array}\right.
\end{align*}
\end{proof}
\begin{lemma}\label{chap6Coverhat}
For $h=\begin{pmatrix}
a& b\\
c& d\end{pmatrix}\in  \SL^{\pm}_2(\R)$, $p=\begin{pmatrix}
a_1 & b_1\\
0& a_1^{-1} \det p  \end{pmatrix}\in P^{\pm}_{>0}(\R)$, we have:
 $\widehat{C}_{z_0}(p, h)=1$.
\end{lemma}
\begin{proof}
\[\overline{s}(p)=1,  \overline{s}(ph)=\overline{s}(pp_h k_h)=\overline{s}(k_h)=\overline{s}(h);\]
\[\widehat{C}_{z_0}(p,h)=\overline{C}_{X^{\ast}}(p,h)\overline{s}(p)\overline{s}(h) \overline{s}(ph)^{-1}=1.\]
\end{proof}
 
\begin{lemma}\label{prog0}
Let $p=\begin{pmatrix} a& b\\ 0& d\end{pmatrix} \in \GL_{2}(\R)$, with $ a>0, d>0$ and $h\in \GL_2(\R)$. Then  $\overline{C}_{X^{\ast}}( h,p)=\overline{C}_{X^{\ast}}(p, h)=\widetilde{C}_{X^{\ast}}( h,p)=\widetilde{C}_{X^{\ast}}( p,h)=1$.
\end{lemma}
\begin{proof}
Decompose $p=\begin{pmatrix} a& b\\ 0& d\end{pmatrix} = \begin{pmatrix} a& 0\\ 0& a\end{pmatrix}\begin{pmatrix} 1& b/a\\ 0& d/a\end{pmatrix}$.
By Lemma \ref{isoGLR} the results hold for the first factor; by Lemmas \ref{prbc} and \ref{chap6Cover} they hold for the second. Hence the result follows.
\end{proof}
\begin{corollary}
The centers of $\overline{\GL}_2(\R)$, $\widetilde{\GL}_2(\R)$, and $\widehat{\GL}_2(\R)$ are $\R_{>0}^{\times}\times \mu_2$, $\R_{>0}^{\times}\times \mu_8$, and $\R_{>0}^{\times}\times \T$, respectively.
\end{corollary}
\begin{proof}
Since  $\R^{\times}$ is the center of $\GL_2(\R)$, it follows from Lemma \ref{isoGLR} that the respective sets are contained in the corresponding centers. For $-I=\diag(-1,-1), h_{-1}=\diag(1,-1) \in \GL_2(\R)$, we have:
\[
\widehat{C}_{z_0}(h_{-1},-I)=\overline{C}_{X^{\ast}}(h_{-1},-I)=1=\widetilde{C}_{X^{\ast}}(h_{-1},-I), \qquad \overline{C}_{X^{\ast}}(-I,h_{-1})=-1=\widetilde{C}_{X^{\ast}}(-I,h_{-1}),
\]
\[
\overline{s}(h_{-1})=1, \quad \overline{s}(-h_{-1})=\overline{s}(-I)=-i,
\]
\[
\widehat{C}_{z_0}(-I,h_{-1})=\overline{C}_{X^{\ast}}(-I,h_{-1})\,\overline{s}(-I)\,\overline{s}(h_{-1})\,\overline{s}(-h_{-1})^{-1}=-1.
\]
Hence $-I$ does not lie in the center in any of the three cases, and the result follows.
\end{proof}
 \subsection{The maximal compact subgroup $K$}
We consider the maximal compact subgroup $K=\Oa_2(\R)$ in $ \GL_2(\R)$ or $\SL_2^{\pm}(\R)$. 
\begin{lemma}\label{chap6thetaf3}
 The restriction of $[\overline{C}_{X^{\ast}}]$ on $\Oa_2(\R)$ is  not trivial.
\end{lemma}
\begin{proof}
Because $-I\in\Oa_2(\R)$ commutes with $h_{-1}\in \Oa_2(\R)$, any trivialization would satisfy
\[
\overline{C}_{X^{\ast}}(-I,h_{-1})
=f(-I)f(h_{-1})f((-I)h_{-1})^{-1}
=f(-I)f(h_{-1})f(h_{-1}(-I))^{-1}
=\overline{C}_{X^{\ast}}(h_{-1},-I).
\]
However, Lemma~\ref{chap6Cover} gives
\[
\overline{C}_{X^{\ast}}(-I,h_{-1})=-1\neq 1=\overline{C}_{X^{\ast}}(h_{-1},-I),
\]
so no such function $f$ can exist.
\end{proof}
\begin{corollary}\label{chap6nonsp222}
The restrictions of both $[\overline{C}_{X^{\ast}}]$ and $[\widehat{C}_{z_0}]$ to $\Oa_2(\mathbb R)$ are non-trivial.
\end{corollary}
\begin{lemma}
$\widehat{C}_{z_0}(h, k)=1$, for $h\in \GL_2(\R)$, $k\in \SO_2(\R)$.
\end{lemma}
\begin{proof}
By Lem. \ref{isoGLR}, it suffices to show this for $h\in \SL^{\pm}_2(\R)$. If $\det h=1$, the result follows from Lem. \ref{usopbig}. If $\det h=-1$, then $h=h_{\det h}[h_{\det h}h]$, and
$
\widehat{C}_{z_0}(h, k)=\nu_{z_0}(1,h_{\det h}h)\,\widehat{c}_{z}(h_{\det h}h,k)=1.\qedhere$
\end{proof}
 \subsection{The discrete subgroups}\label{chap6nonspI}

\begin{lemma}\label{chap6nonsp}
The restriction of\/ $[\overline{C}_{X^{\ast}}]$ to\/ $\Gamma^{\pm}(2)$ is non-trivial.
\end{lemma}
\begin{proof}
The proof is similar to that of Lemma \ref{chap6thetaf3}.
\end{proof}
\begin{corollary}\label{chap6nonsp2}
The restrictions of both $[\overline{C}_{X^{\ast}}]$ and $[\widehat{C}_{z_0}]$ to $\Gamma_{\theta}^{\pm}$, $\SL^{\pm}_2(\mathbb Z)$ are non-trivial.
\end{corollary}

\subsection{The projective group $\PGL_2^{\pm}(\R)$}
 Let $\GL_2^+(\R)=\R^{\times}_{>0}\SL_2(\R)$, $\PGL_2(\R)=\GL_2(\R)/\R^{\times}$, $\PGL_2^{\pm}(\R)=\GL_2(\R)/\R_{>0}^{\times}$.
Then:
\begin{itemize}
\item $\R_{>0}^{\times}=\R^{\times 2}$ and $\R^{\times}/\R^{\times 2} \simeq \mu_2$.
\item  There exists an exact sequence: $1 \longrightarrow \SL_2(\R)/\{\pm 1\} \longrightarrow \PGL_2(\R) \stackrel{\dot{\lambda}}{\longrightarrow} \R^{\times}/\R^{\times 2} \longrightarrow 1.$
\item  There exists an exact sequence: $1 \longrightarrow \SL_2(\R) \longrightarrow \GL_2(\R) \stackrel{\lambda}{\longrightarrow} \R^{\times}\longrightarrow 1.$
\item  There exists an exact sequence: $1\longrightarrow \SL_2(\R) \longrightarrow \PGL_2^{\pm}(\R)\stackrel{\dot{\lambda}}{\longrightarrow}   \R^{\times}/\R^{\times 2} \longrightarrow 1.$
\end{itemize}
By Lemma~\ref{isoGLR} we set
\[
\overline{\PGL}_2^{\pm}(\mathbb R)=\overline{\GL}_2(\mathbb R)/\mathbb R_{>0}^{\times},\qquad
\widetilde{\PGL}_2^{\pm}(\mathbb R)=\widetilde{\GL}_2(\mathbb R)/\mathbb R_{>0}^{\times},\qquad
\widehat{\PGL}_2^{\pm}(\mathbb R)=\widehat{\GL}_2(\mathbb R)/\mathbb R_{>0}^{\times},
\]
and obtain the exact sequences
\[
1\to \overline{\SL}_2(\mathbb R)\to \overline{\PGL}_2^{\pm}(\mathbb R)
\stackrel{\dot{\overline{\lambda}}}{\longrightarrow}\mu_2\to 1,
\]
\[
1\to \widetilde{\SL}_2(\mathbb R)\to \widetilde{\PGL}_2^{\pm}(\mathbb R)
\stackrel{\dot{\widetilde{\lambda}}}{\longrightarrow}\mu_2\to 1,
\]
\[
1\to \widehat{\SL}_2(\mathbb R)\to \widehat{\PGL}_2^{\pm}(\mathbb R)
\stackrel{\dot{\widehat{\lambda}}}{\longrightarrow}\mu_2\to 1.
\]
Fix $r\in\mathbb R_{>0}^{\times}$.  Then $\mu_2\cong\langle -r\rangle/\langle r^{2}\rangle$, and the map
\[
\dot{\iota}_{r}\colon \langle -r\rangle/\langle r^{2}\rangle\to\PGL_2^{\pm}(\mathbb R),\qquad
-r\longmapsto\Bigl[\begin{pmatrix}0&1\\ r&0\end{pmatrix}\Bigr],
\]
is a group homomorphism.  Consequently
\[
\PGL_2^{\pm}(\mathbb R)\cong
\bigl(\langle -r\rangle/\langle r^{2}\rangle\bigr)\ltimes\SL_2(\mathbb R).
\]
We extend the construction to the covering groups via the explicit lifts

\[
\overline{\dot{\iota}}_{r}\colon
\langle -r\rangle/\langle r^{2}\rangle\to\overline{\PGL}_2^{\pm}(\mathbb R),\qquad
-r\longmapsto\Bigl[\Bigl(\begin{pmatrix}0&1\\ r&0\end{pmatrix},1\Bigr)\Bigr],
\]

\[
\widetilde{\dot{\iota}}_{r}\colon
\langle -r\rangle/\langle r^{2}\rangle\to\widetilde{\PGL}_2^{\pm}(\mathbb R),\qquad
-r\longmapsto\Bigl[\Bigl(\begin{pmatrix}0&1\\ r&0\end{pmatrix},
e^{\tfrac{\pi i\sgn(\mathfrak{e})}{4}}\Bigr)\Bigr],
\]

\[
\widehat{\dot{\iota}}_{r}\colon
\langle -r\rangle/\langle r^{2}\rangle\to\widehat{\PGL}_2^{\pm}(\mathbb R),\qquad
-r\longmapsto\Bigl[\Bigl(\begin{pmatrix}0&1\\ r&0\end{pmatrix},
e^{\tfrac{\pi i}{4}}\Bigr)\Bigr].
\]
\begin{lemma}
$\overline{\dot{\iota}}_{r}$, $\widetilde{\dot{\iota}}_{r}$, $\widehat{\dot{\iota}}_{r}$ all are group homomorphisms.
\end{lemma}
\begin{proof}
 \begin{align*}
\overline{C}_{X^{\ast}}(\begin{pmatrix}0&1\\ r&0\end{pmatrix}, \begin{pmatrix}0&1\\ r&0\end{pmatrix})&=\overline{C}_{X^{\ast}}(h_{-r}\omega^{-1}, h_{-r}\omega^{-1})\\
&=\nu_2(-r,\omega^{-1})\overline{c}_{X^{\ast}}((\omega^{-1})^{h_{-r}},\omega^{-1})\\
&\stackrel{\textrm{ Lem.}\ref{nu2}}{=}\overline{c}_{X^{\ast}}(\begin{pmatrix}0&-r\\ r^{-1}&0\end{pmatrix},\omega^{-1})=1.
\end{align*}
 \begin{align*}
\widetilde{C}_{X^{\ast}}(\begin{pmatrix}0&1\\ r&0\end{pmatrix}, \begin{pmatrix}0&1\\ r&0\end{pmatrix})&=\widetilde{C}_{X^{\ast}}(h_{-r}\omega^{-1}, h_{-r}\omega^{-1})\\
&=\nu(-r,\omega^{-1})\widetilde{c}_{X^{\ast}}((\omega^{-1})^{h_{-r}},\omega^{-1})\\
&=\nu(-r,\omega^{-1})\\
&\stackrel{\textrm{ Ex.}\ref{chap3example2}}{=}(-1, -r)_\R \gamma(-r, \psi^{\tfrac{1}{2}})^{-1}=-e^{\tfrac{\pi i \sgn(\mathfrak{e})}{2}}=(e^{-\tfrac{\pi i \sgn(\mathfrak{e})}{4}})^2.
\end{align*}
 \begin{align*}
\widehat{C}_{z_0}(\begin{pmatrix}0&1\\ r&0\end{pmatrix}, \begin{pmatrix}0&1\\ r&0\end{pmatrix})&=\nu_{z_0}(-r,\omega^{-1})\widehat{c}_{z_0}((\omega^{-1})^{h_{-r}},\omega^{-1})\\
&\stackrel{\textrm{ Ex.}\ref{exz0z}}{=}\nu_2( -r,\omega^{-1} ) \tfrac{\overline{s}(\omega^{-1})}{\overline{s}((\omega^{-1})^{h_{-r}})}\widehat{c}_{z_0}(\begin{pmatrix}0&-r\\ r^{-1}&0\end{pmatrix},\omega^{-1})\\
&=1 \cdot (-i)\cdot 1=(e^{-\tfrac{\pi i}{4}})^2.
\end{align*}
\end{proof}
Consequently we obtain semi-direct-product decompositions:
\[
\overline{\PGL}_2^{\pm}(\mathbb R)\cong
\bigl(\langle -r\rangle/\langle r^{2}\rangle\bigr)\ltimes\overline{\SL}_2(\mathbb R), \quad \widetilde{\PGL}_2^{\pm}(\mathbb R)\cong
\bigl(\langle -r\rangle/\langle r^{2}\rangle\bigr)\ltimes\widetilde{\SL}_2(\mathbb R), \quad \widehat{\PGL}_2^{\pm}(\mathbb R)\cong
\bigl(\langle -r\rangle/\langle r^{2}\rangle\bigr)\ltimes\widehat{\SL}_2(\mathbb R).
\]
If $r\in\mathbb Q_{>0}^{\times}$ we may work over $\mathbb Q$ and set
$$\overline{\PGL}^{\pm r}_2(\Q)=\langle -r\rangle/\langle r^2\rangle \ltimes  \overline{\SL}_2(\Q),\quad  \widetilde{\PGL}^{\pm r}_2(\Q) = \langle -r\rangle/\langle r^2\rangle \ltimes \widetilde{\SL}_2(\Q), \quad \widehat{\PGL}^{\pm r}_2(\Q) = \langle -r\rangle/\langle r^2\rangle \ltimes \widehat{\SL}_2(\Q) .$$

Let \(h\in\GL_2(\R)\).

\begin{itemize}
  \item If \(\det h>0\), write
  \[
    h=\sqrt{\det h}\;\bigl[\tfrac{1}{\sqrt{\det h}}h\bigr]
    \quad\text{with}\quad \tfrac{1}{\sqrt{\det h}}h\in\SL_2(\R).
  \]

  \item If \(\det h<0\), write 
  \[h= \underbrace{\sqrt{-(\det h) r^{-1}}}_{h_1} \cdot \underbrace{\begin{pmatrix}0&1\\ r&0\end{pmatrix}}_{h_2}\cdot  \underbrace{\Big[\tfrac{1}{ \sqrt{-(\det h) r^{-1}}}\begin{pmatrix}0&r^{-1}\\ 1&0\end{pmatrix}  h\Big]}_{h_3} \quad\text{with}\quad h_3\in\SL_2(\R).\]
      \end{itemize}
 Now consider $h \in \mathrm{GL}_2(\mathbb{Q})$. If $r \in \mathbb{Q}^{\times}_{>0}$ and $\det h \in -r \mathbb{Q}^{\times 2}$, then in the above decomposition, all factors $h_1$, $h_2$, and $h_3$ belong to $\mathrm{GL}_2(\mathbb{Q})$. \\
    
 Let $ \overline{h}=(h,t_1)\in\overline{\GL}_2(\R)$, $ \widetilde{h}=(h,t_2)\in\widetilde{\GL}_2(\R)$, $ \widehat{h}=(h,t_3)\in\widehat{\GL}_2(\R)$.
  \begin{itemize}
  \item If \(\det h>0\), write
  \[
    \overline{h}=[\sqrt{\det h},1] \cdot\bigl[\tfrac{1}{\sqrt{\det h}}h,t_1\bigr], \quad  \widetilde{h}=[\sqrt{\det h},1] \cdot\bigl[\tfrac{1}{\sqrt{\det h}}h,t_2\bigr],\quad 
    \widehat{h}=[\sqrt{\det h},1] \cdot\bigl[\tfrac{1}{\sqrt{\det h}}h,t_3\bigr].
  \]

  \item If \(\det h<0\), write: 
  \[\overline{h}=[h_1,1] \cdot [h_2,1] \cdot [h_3, \overline{C}_{X^{\ast}}(h_2,h_3)^{-1}t_1],\]
    \[\widetilde{h}=[h_1,1] \cdot [h_2,e^{\tfrac{\pi i\sgn(\mathfrak{e})}{4}}] \cdot [h_3, e^{-\tfrac{\pi i\sgn(\mathfrak{e})}{4}}\widetilde{C}_{X^{\ast}}(h_2,h_3)^{-1}t_2],\]
    \[\widehat{h}=[h_1,1] \cdot [h_2,e^{\tfrac{\pi i}{4}}] \cdot [h_3, e^{-\tfrac{\pi i}{4}}\widehat{C}_{X^{\ast}}(h_2,h_3)^{-1}t_3].\]
      \end{itemize}
     \subsection{Automorphic factor II}\label{autoII}
     Consider the left  action of $\GL_2(\R)$ on $\mathbb{H}^{\pm}$ by $h(z)=\tfrac{az+b}{cz+d}$, for $h=\begin{pmatrix} a   & b \\ c& d\end{pmatrix}\in \GL_2(\R)$, $z\in \mathbb{H}^{\pm}$.  This action is well-defined and  transitive. 
     Restricting to the subgroup  $\SL^{\pm}_2(\R)$,  the stabilizer of $i$ is just the group $\SO_2(\R)$, and  there exists a bijective map:
$$i:\SL_2^{\pm}(\R)/\SO_2(\R) \longrightarrow\mathbb{H}^{\pm}; [g] \longmapsto gi.$$
 Note that  $\SL_2^{\pm}(\R)=P_{>0}^{\pm}(\R)\SO_2(\R)$,  $ P_{>0}^{\pm}(\R) \cap \SO_2(\R)=1$.
So there exists a bijective map:
\begin{align}\label{chap6ph}
i^{\pm}: P^{\pm}_{>0}(\R) \longrightarrow \mathbb{H}^{\pm} ;  g=\begin{pmatrix} a& b \\ 0 &(\det g)   a^{-1}\end{pmatrix} \longmapsto ab\det g + ia^2 \det g.
\end{align}
For $g=\begin{pmatrix} a   & b \\ c& d\end{pmatrix}\in \SL^{\pm}_2(\R)$, let us write $g=p_g k_{g}$, for $p_g\in P^{\pm}_{>0}(\R) $ and $k_g\in \SO_2(\R)$. Then:
\begin{lemma}\label{decomgl2}
$p_g=\begin{pmatrix} \tfrac{1}{ \sqrt{c^2+d^2}}    & \det g \tfrac{bd+ac}{\sqrt{c^2+d^2}} \\ 0& \det g\sqrt{c^2+d^2}  \end{pmatrix}$, $k_{g}=\det g \cdot \begin{pmatrix} \tfrac{d}{\sqrt{c^2+d^2}}   & -\tfrac{c}{\sqrt{c^2+d^2}}\\ \tfrac{c}{\sqrt{c^2+d^2}}& \tfrac{d}{\sqrt{c^2+d^2}}\end{pmatrix}$.
\end{lemma}
\begin{proof}
\begin{align*}
\tfrac{1}{ \sqrt{c^2+d^2}} \tfrac{d\det g}{\sqrt{c^2+d^2}} +(\tfrac{bd+ac}{\sqrt{c^2+d^2}} )\tfrac{c}{\sqrt{c^2+d^2}}&= \tfrac{d\det g}{ c^2+d^2}+ \tfrac{cbd+ac^2}{c^2+d^2}\\
&= \tfrac{dad-dbc+cbd+ac^2}{c^2+d^2}=a;
\end{align*}
\begin{align*}
(\tfrac{1}{ \sqrt{c^2+d^2}} )\tfrac{-c\det g }{\sqrt{c^2+d^2}} +(\tfrac{bd+ac}{\sqrt{c^2+d^2}} )\tfrac{d}{\sqrt{c^2+d^2}}&= \tfrac{-c\det g}{ c^2+d^2}- \tfrac{bd^2+acd}{c^2+d^2}\\
&= \tfrac{-cad+bc^2+bd^2+acd}{c^2+d^2}=b.
\end{align*}
\end{proof}
For $z\in \mathbb{H}^{\pm}$, let us write $p_z$ for the corresponding element in $P_{>0}^{\pm}(\R)$ via the map $i^{\pm}$. Until the end of this subsection, we let  $h=\begin{pmatrix} a   & b \\ c& d\end{pmatrix}\in \GL_2(\R)$, $h_i=\begin{pmatrix} a_i   & b_i \\ c_i& d_i\end{pmatrix}\in\GL_2(\R)$, $z\in \mathbb{H}^{\pm}$. 
Define:
\begin{itemize}
\item[(1)] $J(h, z)=cz+d$;
\item[(2)] $\alpha_z(h)=\tfrac{cz+d}{|cz+d|}$;
\item[(3)] $J'_{1/2}(h, z)=  \widehat{C}_{z_0}(h, p_z)^{-1}\cdot |J(h, z)|^{1/2}$; 
\item[(4)] $J_{1/2}(h, z)=J'_{1/2}(h, z)\cdot\widetilde{s}(h)$;
\item[(5)] $J_{3/2}(h, z)=  J_{1/2}(h, z)J(h, z)$.
\item[(6)] $\sgn z= \left\{\begin{array}{cr}  1 & \textrm{ if } z\in \mathbb{H},\\-1  & \textrm{ if } z\in \mathbb{H}^-. \end{array}\right.$
\item[(7)] $\sgn(h,z)=\overline{C}_{X^{\ast}}(h, p_z)$.
\end{itemize}
The restriction of the above notions to $\SL_2(\R)$  is comparable to that given in Section \ref{autoI} above. It can be checked that $J(h_1h_2, z)=J(h_1, h_2(z))J(h_2, z)$.
\begin{lemma}\label{sgnalpha}
\begin{itemize}
\item[(1)] $\overline{s}(h)^2= \sgn(\det h)\alpha_{z_0}(h)$.
\item[(2)] $[\widehat{C}_{z_0}(h_1,h_2)]^2=\alpha_{z_0}(h_1)\alpha_{z_0}(h_2)\alpha_{z_0}(h_1h_2)^{-1}$.
\item[(3)] $\sgn(h,z)=\left\{\begin{array}{cl}  (\sgn(z), a)_{\R}  & \textrm{ if } c=0,\\1  & \textrm{ if } c\neq 0. \end{array}\right.$
\end{itemize}
\end{lemma}
\begin{proof}
1) Note that  $\overline{s}(h)=\overline{s}(h_{\det h}h)$. So $[\overline{s}(h)]^2=[\overline{s}(h_{\det h}h)]^2\stackrel{(\ref{eqalpha})}{=}\alpha_{z_0}(h_{\det h}h)=\tfrac{(\det h)ci+(\det h)d}{|(\det h)ci+(\det h)d|}=\sgn(\det h)\alpha_{z_0}(h)$. \\
2) 
\begin{align*}
\text{Left side}
&\stackrel{\textrm{ Lem.}\ref{Cz_0wid}}{=}[\overline{s}(h_1)]^2[\overline{s}(h_2)]^2 [\overline{s}(h_1h_2)]^{-2}\\
&=\alpha_{z_0}(h_1)\alpha_{z_0}(h_2)\alpha_{z_0}(h_1h_2)^{-1}= \text{Right side}.
\end{align*}
3) Write $h=h_{|\det h|^{-1}} g_h$, for $h_{|\det h|}=\diag(1, |\det h|)$ and  $g_h=\begin{pmatrix} a   & b \\ |\det h|^{-1}c& d|\det h|^{-1}\end{pmatrix}\in \SL_2^{\pm}(\R)$. Then
\[\sgn(h,z)=\overline{C}_{X^{\ast}}(h, p_z)=\overline{C}_{X^{\ast}}(g_h, p_z)\stackrel{\textrm{ Lem. } \ref{chap6Cover}}{=}\left\{\begin{array}{cl}  (\det p_z, a)_{\R}  & \textrm{ if } c=0,\\1  & \textrm{ if } c\neq 0. \end{array}\right.\]\end{proof}
\begin{lemma}\label{tildehatover}
\begin{itemize}
\item[(1)] $J'_{1/2}(h_1h_2, z)=J'_{1/2}(h_1, h_2(z))J'_{1/2}(h_2, z)\widehat{C}_{z_0}(h_1,h_2)$.
\item[(2)] $J_{1/2}(h_1h_2, z)=J_{1/2}(h_1, h_2(z))J_{1/2}(h_2, z)\widetilde{C}_{X^{\ast}}(h_1,h_2)$.
\item[(3)]  $J_{1/2}(h_1h_2, z)m_{X^{\ast}}(h_1h_2)^{-1}=J_{1/2}(h_1, h_2(z))J_{1/2}(h_2, z)m_{X^{\ast}}(h_1)^{-1}m_{X^{\ast}}(h_2)^{-1}\overline{C}_{X^{\ast}}(h_1,h_2)$.
\end{itemize}
\end{lemma}
\begin{proof}
1) \[|J(h_1h_2, z)|^{1/2}=|J(h_1, h_2(z))|^{1/2} |J(h_2, z)|^{1/2};\]
\[\widehat{C}_{z_0}(h_1, h_2) \widehat{C}_{z_0}(h_1h_2,  p_z)=  \widehat{C}_{z_0}(h_1, h_2 p_z)  \widehat{C}_{z_0}(h_2,  p_z).\] 
 If $h_i\in \SL_2^{\pm}(\R)$, we have:
\begin{align*}
 \widehat{C}_{z_0}(h_1, h_2)&=  \widehat{C}_{z_0}(h_1, h_2 p_z)  \widehat{C}_{z_0}(h_2,  p_z) \widehat{C}_{z_0}(h_1h_2,  p_z)^{-1}\\
&= \widehat{C}_{z_0}(h_1, p_{h_2 p_z}k_{h_2 p_z})  \widehat{C}_{z_0}(h_2,  p_z) \widehat{C}_{z_0}(h_1h_2,  p_z)^{-1}\\
&= \widehat{C}_{z_0}(h_1, p_{h_2 p_z}k_{h_2 p_z})  \widehat{C}_{z_0}( p_{h_2 p_z},k_{h_2 p_z}) \widehat{C}_{z_0}(h_2,  p_z) \widehat{C}_{z_0}(h_1h_2,  p_z)^{-1}\\
&= \widehat{C}_{z_0}(h_1, p_{h_2 p_z})  \widehat{C}_{z_0}( h_1p_{h_2 p_z},k_{h_2 p_z}) \widehat{C}_{z_0}(h_2,  p_z) \widehat{C}_{z_0}(h_1h_2,  p_z)^{-1}\\
&= \widehat{C}_{z_0}(h_1, p_{h_2 (z)})  \widehat{C}_{z_0}(h_2,  p_z) \widehat{C}_{z_0}(h_1h_2,  p_z)^{-1};
\end{align*}

\begin{align*}
\text{Right side}
&=\widehat{C}_{z_0}(h_1, p_{ h_2(z)})^{-1}\cdot |J(h_1,h_2(z))|^{1/2}\widehat{C}_{z_0}(h_2, p_z)^{-1}\cdot |J(h_2, z)|^{1/2}\widehat{C}_{z_0}(h_1,h_2)\\
&=\widehat{C}_{z_0}(h_1h_2,  p_z)^{-1}|J(h_1h_2, z)|^{1/2}\\
&= \text{Left side}.
\end{align*}
In general, write $h_i=|\det(h_i)|^{1/2}g_i$, with $g_i=|\det(h_i)|^{-1/2}h_i\in \SL_2^{\pm}(\R)$. Then $h_2(z)=g_2(z)$. 
\[J'_{1/2}(h_i, z) =\widehat{C}_{z_0}(g_i, p_z)^{-1}\cdot |J(g_i, z)|^{1/2} \cdot |\det(h_i)|^{1/4}.\]
\begin{align*}
\text{Right side}
&=\widehat{C}_{z_0}(g_1, p_{h_2(z)})^{-1}\cdot |J(g_1, h_2(z))|^{1/2} \cdot |\det(h_1)|^{1/4}\widehat{C}_{z_0}(g_2, p_z)^{-1}\cdot |J(g_2, z)|^{1/2} \cdot |\det(h_2)|^{1/4}\widehat{C}_{z_0}(h_1,h_2)\\
&=\widehat{C}_{z_0}(g_1, p_{g_2(z)})^{-1}\cdot |J(g_1, g_2(z))|^{1/2} \widehat{C}_{z_0}(g_2, p_z)^{-1}\cdot |J(g_2, z)|^{1/2} \cdot |\det(h_1h_2)|^{1/4}\widehat{C}_{z_0}(g_1,g_2)\\
&=J'_{1/2}(g_1g_2, z) \cdot |\det(h_1h_2)|^{1/4}=\text{Left side}.
\end{align*}
2) It follows from Lemma \ref{Cz_0wid}. \\
3) It follows from the equality (\ref{chap3293inter}).
\end{proof}
As a consequence, we have:
\begin{lemma}
\begin{itemize}
\item[(1)] $J'_{3/2}(h_1h_2, z)=J'_{3/2}(h_1, h_2(z))J'_{3/2}(h_2, z)\widehat{C}_{z_0}(h_1,h_2)$.
\item[(2)] $J_{3/2}(h_1h_2, z)=J_{3/2}(h_1, h_2(z))J_{3/2}(h_2, z)\widetilde{C}_{X^{\ast}}(h_1,h_2)$.
\item[(3)]  $J_{3/2}(h_1h_2, z)m_{X^{\ast}}(h_1h_2)^{-1}=J_{3/2}(h_1, h_2(z))J_{3/2}(h_2, z)m_{X^{\ast}}(h_1)^{-1}m_{X^{\ast}}(h_2)^{-1}\overline{C}_{X^{\ast}}(h_1,h_2)$.
\end{itemize}
\end{lemma}
\begin{lemma}\label{J121H}
  $J'_{1/2}(h, z)^2= \alpha_{z_0}(h)^{-1}  J(h, z)$.
\end{lemma}
\begin{proof}
\begin{align*}
\text{Left side}
&=\widehat{C}_{z_0}(h, p_z)^{-2}\cdot |J(h, z)|\\
&=\alpha_{z_0}(h)^{-1}\alpha_{z_0}(p_z)^{-1}\alpha_{z_0}(hp_z)\cdot |J(h, z)|\\
&=\alpha_{z_0}(h)^{-1}|J(h, z)|\Big[\tfrac{J(p_z, i)}{|J(p_z, i)|}\Big]^{-1}\tfrac{J(hp_z, i)}{|J(hp_z, i)|}\\
&=\alpha_{z_0}(h)^{-1}|J(h, z) |\tfrac{J(h, z)}{|J(h, z)|}\\
&= \text{Right side}.
\end{align*}
\end{proof}
\begin{lemma}\label{J12squr}
$[J_{1/2}(h, z)m_{X^{\ast}}(h)^{-1}]^2=\sgn(\det h)\cdot (cz+d)$.
\end{lemma}
\begin{proof}
$\text{Left side}=[J'_{1/2}(h, z)\widetilde{s}(h)m_{X^{\ast}}(h)^{-1}]^2=[J'_{1/2}(h, z)\overline{s}(h)]^2\stackrel{ \textrm{ Lem.} \ref{sgnalpha}}{=}\alpha_{z_0}(h)^{-1} J(h, z) \cdot\sgn(\det h)\alpha_{z_0}(h)= \text{Right side}$.
\end{proof}
\begin{lemma}\label{jpz}
\begin{itemize}
\item[(1)] For $p=\begin{pmatrix} a& b\\ 0& \det(p) a^{-1}\end{pmatrix}\in P^{\pm}_{>0}(\R)$ and $p'=\begin{pmatrix} a& 0\\ 0& a\end{pmatrix}$ with $a>0$, $J_{1/2}(p, z)=a^{-1/2}$ and $J_{1/2}(p', z)=a^{1/2}$ .
\item[(2)] For  $k=\begin{pmatrix} a& -b\\ b& a\end{pmatrix}\in \SO_2(\R)$, $J_{1/2}(k, z_0)=\widetilde{s}(k)$ and $J_{1/2}(k, \overline{z_0})=\widetilde{s}(k) (-bi+a)$.  
\end{itemize}
\end{lemma}
\begin{proof}
1) \begin{align*}
J_{1/2}(p, z)&=J'_{1/2}(p, z)\cdot\widetilde{s}(p)\\
&=\widehat{C}_{z_0}(p, p_z)^{-1}\cdot |J(p, z)|^{1/2}\cdot\widetilde{s}(p)\\
&=\overline{C}_{X^{\ast}}(p, p_{z})^{-1}\overline{s}( p)^{-1}\overline{s}(p_{z})^{-1} \overline{s}(pp_{z})|J(p, z)|^{1/2}\cdot\widetilde{s}(p)\\
&=|J(p, z)|^{1/2}\cdot\overline{s}(p) m_{X^{\ast}}(p)\\
&=|\det(p) a^{-1}|^{1/2}.
\end{align*}
 \begin{align*}
J_{1/2}(p', z)&=J'_{1/2}(p', z)\cdot\widetilde{s}(p')\\
&=\widehat{C}_{z_0}(p', p_z)^{-1}\cdot |J(p', z)|^{1/2}\cdot\widetilde{s}(p')\\
&=\overline{C}_{X^{\ast}}(p', p_{z})^{-1}\overline{s}( p')^{-1}\overline{s}(p'_{z})^{-1} \overline{s}(p'p_{z})|J(p', z)|^{1/2}\cdot\widetilde{s}(p')\\
&=|J(p', z)|^{1/2}\cdot\overline{s}(p') m_{X^{\ast}}(p')\\
&= a^{1/2}.
\end{align*}
2) By Remark \ref{remakonUC}, $J_{1/2}(k, z_0)=\widetilde{s}(k)$.  
 \begin{align*}
J_{1/2}(h_{-1}kh_{-1}, z_0)&=J_{1/2}(h_{-1}, kh_{-1} (z_0))J_{1/2}(kh_{-1}, z_0)\widetilde{C}_{X^{\ast}}(h_{-1},kh_{-1} )\\
&=J_{1/2}(h_{-1}, kh_{-1} (z_0))J_{1/2}(k, h_{-1}(z_0))J_{1/2}( h_{-1}, z_0)\widetilde{C}_{X^{\ast}}(h_{-1},kh_{-1} )\widetilde{C}_{X^{\ast}}(k, h_{-1} )\\
&=J_{1/2}(k, -z_0)\widetilde{C}_{X^{\ast}}(k, h_{-1} )\\
&=J_{1/2}(k, -z_0)\overline{C}_{X^{\ast}}(k, h_{-1} )m_{X^{\ast}}(kh_{-1})
m_{X^{\ast}}(k)^{-1}
m_{X^{\ast}}(h_{-1})^{-1}\\
&=J_{1/2}(k, -z_0)\overline{C}_{X^{\ast}}(k, h_{-1} )m_{X^{\ast}}(k^{-1})
m_{X^{\ast}}(k)^{-1}.
\end{align*}
\[J_{1/2}(k, -z_0)=J_{1/2}(k^{-1}, z_0)\overline{C}_{X^{\ast}}(k, h_{-1} )^{-1}m_{X^{\ast}}(k^{-1})^{-1}
m_{X^{\ast}}(k).\]
If $b\neq 0$, then $\overline{C}_{X^{\ast}}(k, h_{-1} )=1$, $J_{1/2}(k^{-1}, z_0)= \widetilde{s}(k^{-1})= m_{X^{\ast}}(k^{-1})\overline{s}(k^{-1})=m_{X^{\ast}}(k^{-1})\overline{s}(k)(-bi+a)$.
Therefore,
\[J_{1/2}(k, -z_0)=m_{X^{\ast}}(k^{-1})\overline{s}(k)(-bi+a)m_{X^{\ast}}(k^{-1})^{-1}m_{X^{\ast}}(k)=\overline{s}(k)m_{X^{\ast}}(k)(-bi+a)=\widetilde{s}(k)(-bi+a).\]
If $b=0$, $a=-1$, then $k=-I$ and
\[J_{1/2}(k, -z_0)=J_{1/2}(-I, z_0)\overline{C}_{X^{\ast}}(-I, h_{-1} )^{-1}=\widetilde{s}(-I)\cdot (-1)=-\widetilde{s}(-I).\] 
If $b=0$, $a=1$, then $k=I$ and
\[J_{1/2}(k, -z_0)=J_{1/2}(I, z_0)\overline{C}_{X^{\ast}}(I, h_{-1} )^{-1}=\widetilde{s}(I).\] 
\end{proof}
\begin{lemma}
  $\sgn(h,z)\sqrt{\sgn(\det h) \cdot (cz+d)}= J_{1/2}(h, z)m_{X^{\ast}}(h)^{-1}$.
  \end{lemma}
  \begin{proof}
 (1) Let $h\in \SL_2(\R)$. If  $z\in \mathbb H$, then the result holds by Definition \ref{defczd} and  Remark \ref{remachoArg}. \\
 If $z\in \mathbb H^-$, and $c\neq 0$, then $\sgn(h,z)=1$. 
 \[\overline{C}_{X^{\ast}}(h,h_{-1})=1=\overline{C}_{X^{\ast}}(h_{-1}, hh_{-1})=\overline{C}_{X^{\ast}}(h_{-1},h_{-1}p_{z})= \overline{C}_{X^{\ast}}(h_{-1}, hp_{z})=1;\]
 \begin{align*}
 \overline{C}_{X^{\ast}}(h^{h_{-1}}, h_{-1}p_{z})&=\overline{C}_{X^{\ast}}(h,h_{-1}) \overline{C}_{X^{\ast}}(h_{-1}, hh_{-1}) \overline{C}_{X^{\ast}}(h^{h_{-1}}, h_{-1}p_{z})\\
 &=\overline{C}_{X^{\ast}}(h,h_{-1})\overline{C}_{X^{\ast}}(hh_{-1},h_{-1}p_{z}) \overline{C}_{X^{\ast}}(h_{-1}, hp_{z})\\
 &=\overline{C}_{X^{\ast}}(h,p_{z})\overline{C}_{X^{\ast}}(h_{-1},h_{-1}p_{z}) \overline{C}_{X^{\ast}}(h_{-1}, hp_{z})=\overline{C}_{X^{\ast}}(h,p_{z}).
\end{align*}
\begin{align*}
\sqrt{cz+d}&=\sqrt{(-c)(-z)+d}\\
&= J_{1/2}(h^{h_{-1}}, -z)m_{X^{\ast}}(h^{h_{-1}})^{-1}\\
&=\widetilde{s}(h^{h_{-1}}) \widehat{C}_{z_0}(h^{h_{-1}}, p_{-z})^{-1}\cdot |J(h^{h_{-1}}, -z)|^{1/2}m_{X^{\ast}}(h^{h_{-1}})^{-1}\\
&=\widetilde{s}(h^{h_{-1}}) \widehat{C}_{z_0}(h^{h_{-1}}, h_{-1}p_{z})^{-1}\cdot |J(h^{h_{-1}}, -z)|^{1/2}m_{X^{\ast}}(h^{h_{-1}})^{-1}\\
&=\overline{s}(h^{h_{-1}})\widehat{C}_{z_0}(h^{h_{-1}}, h_{-1}p_{z})^{-1} \cdot   |J(h^{h_{-1}}, -z)|^{1/2}\\
&=\overline{s}(h^{h_{-1}})\overline{C}_{X^{\ast}}(h^{h_{-1}}, h_{-1}p_{z})^{-1}\overline{s}( h^{h_{-1}})^{-1}\overline{s}(h_{-1}p_{z})^{-1} \overline{s}(h_{-1}hp_{z})\cdot   |J(h^{h_{-1}}, -z)|^{1/2}\\
&=\overline{C}_{X^{\ast}}(h, p_{z})^{-1}\overline{s}(p_{z})^{-1} \overline{s}(hp_{z}) |J(h, z)|^{1/2}\\
&=\overline{s}(h)\widehat{C}_{z_0}(h, p_{z})^{-1} \cdot   |J(h, z)|^{1/2}\\
&=J_{1/2}(h, z)m_{X^{\ast}}(h)^{-1}.
\end{align*}
If $z\in \mathbb H^-$,  and $c= 0$,  $a>0, d>0$,  then $h\in P_{>0}(\R)$, $\sgn(h,z)=1$ and  $\widehat{C}_{z_0}(h, p_z)^{-1}\stackrel{ \textrm{Lem. }\ref{chap6Coverhat} }{=}1$.
\[J_{1/2}(h, z)m_{X^{\ast}}(h)^{-1}=\overline{s}(h)\widehat{C}_{z_0}(h, p_{z})^{-1} \cdot   |J(h, z)|^{1/2}=d^{1/2}=\sgn(h,z)\sqrt{\sgn(\det h)(cz+d)}.\]
If $z\in \mathbb H^-$,  and $c= 0$,  $a<0, d<0$, then $-h\in P_{>0}(\R)$ and $\sgn(h,z)=-1$. 
\[\overline{C}_{X^{\ast}}(h, p_{z})^{-1}=-1, \quad \overline{s}( h)^{-1}=e^{\tfrac{\pi i}{2}}, \quad \overline{s}(p_{z})^{-1} =1, \quad  \overline{s}(hp_{z})=e^{-\tfrac{\pi i}{2}};\]
\[\widehat{C}_{z_0}(h, p_z)^{-1}=\overline{C}_{X^{\ast}}(h, p_{z})^{-1}\overline{s}( h)^{-1}\overline{s}(p_{z})^{-1} \overline{s}(hp_{z})=-1;\]
\[J_{1/2}(h, z)m_{X^{\ast}}(h)^{-1}=\overline{s}(h)\widehat{C}_{z_0}(h, p_{z})^{-1} \cdot   |J(h, z)|^{1/2}=-e^{-\tfrac{\pi i}{2}}|d|^{1/2}=-\sqrt{d}=\sgn(h,z)\sqrt{\sgn(\det h)(cz+d)}.\]
(2) If $h\in \GL_2(\R)$ with  $\det h>0$, then $h=(\det h)^{1/2} ((\det h)^{-1/2} h)$ and $\sgn(h,z)=\sgn((\det h)^{-1/2} h,z)$.
\begin{align*}
\text{Right side}&=J_{1/2}((\det h)^{1/2} , (\det h)^{-1/2} hz) J_{1/2}( (\det h)^{-1/2} h,z) m_{X^{\ast}}(h)^{-1}\widetilde{C}_{X^{\ast}}((\det h)^{1/2} , (\det h)^{-1/2} h)\\
&=J_{1/2}((\det h)^{1/2} , (\det h)^{-1/2} hz) J_{1/2}( (\det h)^{-1/2} h,z) m_{X^{\ast}}((\det h)^{-1/2} h)^{-1}\\
&=(\det h)^{1/4}\sgn((\det h)^{-1/2} h,z)\sqrt{(\det h)^{-1/2} \cdot (cz+d)}\\
&=\text{Left side}.
\end{align*}
(3) If $h\in \GL_2(\R)$ with $\det h<0$, then $h=h_{-1} [h_{-1}h]$. 
\[\sgn(h,z)=\sgn(h_{-1}h,z).\] 
\begin{align*}
\text{Right side}&=J_{1/2}(h_{-1}, h_{-1} h(z)) J_{1/2}( h_{-1}h,z) m_{X^{\ast}}(h)^{-1}\widetilde{C}_{X^{\ast}}(h_{-1} ,h_{-1}h)\\
&=J_{1/2}(h_{-1}, h_{-1} hz) J_{1/2}( h_{-1}h,z) m_{X^{\ast}}(h_{-1}h)^{-1}\\
&=1^{-1/2}\cdot\sgn(h_{-1}h,z) \sqrt{-cz-d}\\
&=\text{Left side}.
\end{align*}
    \end{proof}
  Let $h_3=h_1h_2$. 
    \begin{lemma}
    $\sgn(h_3,z)\sqrt{\sgn(\det h_3)(c_3z+d_3)}=\sgn(h_1,h_2(z))\sqrt{\sgn(\det h_1)(c_1(h_2 z)+d_1)}\cdot \sgn(h_2,z) \sqrt{\sgn(\det h_2)(c_2 z+d_2)}\cdot \overline{C}_{X^{\ast}}(h_1,h_2)$.
  \end{lemma}
  \begin{proof}
 \begin{align*}
 \text{Left side}
 &=J_{1/2}(h_1h_2, z)m_{X^{\ast}}(h_1h_2)^{-1}\\
 &\stackrel{\textrm{Lem. } \ref{tildehatover}(2)}{=}J_{1/2}(h_1, h_2(z))J_{1/2}(h_2, z)\widetilde{C}_{X^{\ast}}(h_1,h_2)m_{X^{\ast}}(h_1h_2)^{-1}\\
 &=J_{1/2}(h_1, h_2(z))J_{1/2}(h_2, z) m_{X^{\ast}}(h_1)^{-1} m_{X^{\ast}}(h_2)^{-1}\overline{C}_{X^{\ast}}(h_1, h_2)\\
 &= \text{Right side}. 
 \end{align*}
  \end{proof}
  \subsection{Automorphic factor plus}\label{autoIII}
 
  For $\overline{h}=(h,\epsilon)\in\overline{\GL}_{2}(\mathbb{R})$ or $\overline{\GL}_{2}^{\mu_8}(\mathbb{R})$, define:
\begin{align*}
J_{1/2}(\overline{h},z)&\coloneqq\epsilon^{-1}J_{1/2}(h,z)m_{X^{\ast}}(h)^{-1}=\epsilon^{-1}\sgn(h,z)\sqrt{\sgn(\det h)(cz+d)},\\[2pt]
J_{3/2}(\overline{h},z)&\coloneqq J_{1/2}(\overline{h},z)\,J(h,z)=\epsilon^{-1}\sgn(h,z)\sqrt{\sgn(\det h)(cz+d)}(cz+d).
\end{align*}
\begin{lemma}\label{Automorphicfactor}
For  $\overline{h}_{1},\overline{h}_{2}\in\overline{\GL}_{2}(\mathbb{R})$, we have:
\begin{enumerate}
\item $J_{1/2}(\overline{h}_{1}\overline{h}_{2},z)
      =J_{1/2}(\overline{h}_{1},h_{2}z)\,
       J_{1/2}(\overline{h}_{2},z)$;
\item $J_{3/2}(\overline{h}_{1}\overline{h}_{2},z)
      =J_{3/2}(\overline{h}_{1},h_{2}z)\,
       J_{3/2}(\overline{h}_{2},z)$.
\end{enumerate}
\end{lemma}
\begin{proof}
Note that $\overline{h}_{1}\overline{h}_2=[h_1h_2,\overline{C}_{X^{\ast}}(h_1,h_2)\epsilon_1\epsilon_2]$.
 \begin{align*}
1) \quad \text{Left side}
&=\overline{C}_{X^{\ast}}(h_1,h_2)^{-1}\epsilon_1^{-1}\epsilon_2^{-1}J_{1/2}(h_1h_2,z)m_{X^{\ast}}(h_1h_2)^{-1}\\
&\stackrel{\textrm{ Lem. }\ref{tildehatover}(2)}{=}\overline{C}_{X^{\ast}}(h_1,h_2)^{-1}\epsilon_1^{-1}\epsilon_2^{-1}J_{1/2}(h_1, h_2(z))J_{1/2}(h_2, z)\widetilde{C}_{X^{\ast}}(h_1,h_2)m_{X^{\ast}}(h_1h_2)^{-1}\\
&\stackrel{(\ref{chap3293inter})}{=}\overline{C}_{X^{\ast}}(h_1,h_2)^{-1}\epsilon_1^{-1}\epsilon_2^{-1}J_{1/2}(h_1, h_2(z))J_{1/2}(h_2, z) m_{X^{\ast}}(h_1)^{-1} m_{X^{\ast}}(h_2)^{-1}\overline{C}_{X^{\ast}}(h_1, h_2)\\
&= \text{Right side}.
\end{align*}
2) It is a consequence of Part (1).
\end{proof}
\subsection{The Weil-Deligne group}\label{theWdg1}
Recall the Weil group or the Weil-Deligne group  $\mathbb{W}_{\R}$  for  $\R$ from \cite{La} or \cite{Tate}. It can be seen as a subgroup of $\SU(2)$ generated by
\[ \{ \begin{pmatrix} z & 0\\ 0 &\overline{z}\end{pmatrix}   \mid z\in \C^{\times}\},  \qquad  \qquad \omega_{\sigma}=\begin{pmatrix}0& -1\\ 1 &0\end{pmatrix}.\] 
We identify $\C^{\times}$ as a subgroup of $\mathbb{W}_{\R}$.  We decompose $\Oa_2(\R)$ as  $\langle h_{-1}\rangle \ltimes \SO_2(\R)$.
\begin{lemma}\label{weilDelignegroup}
There exists a group morphism $\iota$ from $\overline{\Oa}^{\mu_8}_2(\R)$ to $\mathbb{W}_{\R}$ by sending $[h_{-1},i]$ to  $\omega_{\sigma}$, $[1,\epsilon]$ to  $\begin{pmatrix}\epsilon^{-1}& 0\\0 &\overline{\epsilon^{-1}}\end{pmatrix}$ and $g=\begin{pmatrix}a & -b\\ b& a\end{pmatrix}\in \SO_2(\R)$ to  $\begin{pmatrix} \sqrt{a+bi} & 0\\ 0 &\overline{\sqrt{ a+bi}}\end{pmatrix} $.
  \end{lemma}
\begin{proof}
1) $[h_{-1},i] [h_{-1},i]=[1, -\overline{C}_{X^{\ast}}(h_{-1}, h_{-1})]=[1, -\nu_2( -1,I)\overline{c}_{X^{\ast}}(I,I)]=[1, -1]$.  So the restriction of  $\iota$ to $\langle [h_{-1},i]\rangle$ is a group morphism. \\
2) Note that the maps $f_1: \overline{\SO}_2(\R) \to \U(\C);\overline{g} \longmapsto J_{1/2}(\overline{g},z_0) $ and   $f_2: \overline{\SO}_2(\R) \to \U(\C); [g,\epsilon] \longmapsto \overline{J_{1/2}(\overline{g},z_0)} $ both are group isomorphisms. \\
3)  \begin{align*}
[h_{-1},i] [g,1][h_{-1},-i]&=[g^{h_{-1}}, \overline{C}_{X^{\ast}}(h_{-1}, g) \overline{C}_{X^{\ast}}( h_{-1}g, h_{-1})]\\
&=[g^{h_{-1}}, \nu_2(1, 1)\overline{c}_{X^{\ast}}(1, g)\nu_2(-1, g)\overline{c}_{X^{\ast}}(g^{h_{-1}}, 1)]\\
&=[g^{h_{-1}},\nu_2(-1, g)].
\end{align*} We write $u=a+bi=e^{i \tau_u}$ with $-\pi \leq \tau_u < \pi$.  If $b\neq0 $, we know that $\nu_2(-1, g)=1$. If $b=0$ and $a=1$, then $\nu_2(-1, g)=1$.  If $b=0$ and $a=-1$, then $\nu_2(-1, g)=-1$.  
\[J_{1/2}([g^{h_{-1}},\nu_2(-1, g)],z_0)=\sqrt{a-bi} \nu_2(-1, g)=\left\{ \begin{array}{lr} e^{-i \tfrac{\tau_u}{2}} &  -\pi < \tau_u < \pi,\\ -e^{i \tfrac{\tau_u}{2}} &  \tau_u = -\pi,\end{array}\right.= e^{-i \tfrac{\tau_u}{2}}=\overline{\sqrt{a+bi}}.\]
\end{proof}
By restriction, we can embed $\overline{\Oa}_2(\R)$ into $\mathbb{W}_{\R}$. In particular, on $\overline{\SO}_2(\R)$,  the embedding sends $[g,\epsilon]$ to $\epsilon^{-1}\overline{s}(g)=\epsilon\overline{s}(g)$.
\section{Congruence subgroups} 
In this section, we write $p$ for an odd prime number and assume $\psi=\psi_0$. Let $N$ be a positive integer. 
\subsection{   Congruence subgroups of $\SL_2(\Z)$}\label{Congruencesubgroupsof}    
In $\SL_2(\Z)$, we consider the following subgroups:
$$\Gamma(N)=\{ g=\begin{pmatrix} a& b\\ c& d\end{pmatrix} \in \SL_2(\Z) \mid  \begin{pmatrix} a& b\\ c& d\end{pmatrix}\equiv \begin{pmatrix} 1& 0\\0& 1\end{pmatrix}(\bmod N)\},$$
$$\Gamma(2N,N)=\{ g\in \Gamma(N) \mid a\equiv d\equiv 1(\bmod 2N), b \equiv c \equiv 0(\bmod N)\}\quad (\,N \textrm{ even}), $$
$$\Gamma(p,p)=\{ g\in \SL_2(\Z)\mid a\equiv d(\bmod p), b \equiv c \equiv 0(\bmod p)\}, $$
$$\Gamma_0(N)=\{ g=\begin{pmatrix} a& b\\ c& d\end{pmatrix} \in \SL_2(\Z) \mid  \begin{pmatrix} a& b\\ c& d\end{pmatrix}\equiv \begin{pmatrix} \ast & \ast\\ 0 & \ast \end{pmatrix}(\bmod N)\},$$
$$\Gamma^0(N)=\{ g=\begin{pmatrix} a& b\\ c& d\end{pmatrix} \in \SL_2(\Z) \mid  \begin{pmatrix} a& b\\ c& d\end{pmatrix}\equiv \begin{pmatrix} \ast & 0 \\ \ast & \ast \end{pmatrix}(\bmod N)\}.$$
Let us choose $c, d\in \Z$ such that $-pc+4d=1$. Then:
      \[ \begin{pmatrix} pc&d\\-4p&-p\end{pmatrix}=\begin{pmatrix} 0& 1\\-p&0\end{pmatrix}  \begin{pmatrix}4& 1\\ pc& d\end{pmatrix} = \begin{pmatrix} c&d\\ -4& -p\end{pmatrix} \begin{pmatrix} p&0\\0&1\end{pmatrix}; \]
      \[\Gamma_0(4)\begin{pmatrix} p&0\\0&1\end{pmatrix}=\Gamma_0(4)\begin{pmatrix} c&d\\ -4& -p\end{pmatrix} \begin{pmatrix} p&0\\0&1\end{pmatrix}=\Gamma_0(4)\begin{pmatrix} 0& 1\\-p&0\end{pmatrix}  \begin{pmatrix}4& 1\\ pc& d\end{pmatrix}.\]

 For elements in $\GL_2(\Z)$, we use the following notations:
\[ g_{-p}=\begin{pmatrix} 0 & 1\\ p& 0 \end{pmatrix}, g_p=\begin{pmatrix} 0 & 1\\ -p& 0 \end{pmatrix}, r_x= \begin{pmatrix} 0 & -1 \\ 1 &x 
\end{pmatrix} \textrm{ for } 1\leq x\leq p,  r_0=\begin{pmatrix}4& 1\\ pc& d\end{pmatrix}.\]
Then the following results can be checked:
\begin{itemize}
\item[(1)] 
$ g_p \SL_2(\Z) g_p^{-1}=g_p^{-1}\SL_2(\Z)g_p$, $ [g_p^{-1}\SL_2(\Z)g_p] \cap \SL_2(\Z)=\Gamma_0(p)$, $g_p^{-1} \Gamma_{0}(p) g_p=\Gamma_{0}(p)$.
\item[(2)] 
$ g_{-p} \SL_2(\Z) (g_{-p})^{-1}=(g_{-p})^{-1}\SL_2(\Z)g_{-p}$, $ [(g_{-p})^{-1}\SL_2(\Z)g_{-p}] \cap \SL_2(\Z)=\Gamma_0(p)$, $(g_{-p})^{-1} \Gamma_{0}(p) g_{-p}=\Gamma_{0}(p)$.
\item[(3)] $\SL_2(\Z)=\mathop{\sqcup}_{0\leq x\leq p} \Gamma_0(p) r_x$, $  g_{p}^{-1} \SL_2(\Z)g_{p} =\mathop{\sqcup}_{0\leq x\leq p} \Gamma_0(p) g_{p}^{-1}r_xg_{ p}$ and $  (g_{- p})^{-1} \SL_2(\Z)g_{-p} =\mathop{\sqcup}_{0\leq x\leq p} \Gamma_0(p) (g_{- p})^{-1}r_xg_{- p}$.
\item[(4)] $\SL_2(\Z)\setminus [\SL_2(\Z)g_p \SL_2(\Z)] =\mathop{\sqcup}_{0\leq x\leq p} \SL_2(\Z) g_p r_x$.
\item[(5)] If $N$ is odd, $  \Gamma_0(N) \cap \Gamma_{\theta}=\{ \begin{pmatrix} a & b \\ c &d 
\end{pmatrix} \in \SL_2(\Z) \mid   c \equiv 0(\bmod N), ab \equiv 0(\bmod 2),  cd \equiv 0(\bmod 2)\}$.
\item[(6)]  If $N$ is even, $  \Gamma_0(N) \cap \Gamma_{\theta}=\{ \begin{pmatrix} a & b \\ c &d 
\end{pmatrix} \in \SL_2(\Z) \mid   c \equiv 0(\bmod N), a \equiv d  \equiv 1(\bmod 2),  b\equiv 0(\bmod 2)\}$.
\end{itemize}
\begin{lemma}
$\bigcap_{1 \le x \le p} r_x^{-1} g_p^{-1} \SL_2(\mathbb{Z}) \, g_p r_x 
= r_x^{-1} \Gamma(p,p)r_x.$
\end{lemma}

\begin{proof}
Recall that
\[
g_p r_x = \begin{pmatrix} 1 & x \\ 0 & p \end{pmatrix}, \quad 
(g_p r_x)^{-1} = \frac{1}{p} \begin{pmatrix} p & -x \\ 0 & 1 \end{pmatrix}.
\]
A direct computation shows:
\[
\begin{pmatrix} 1 & x \\ 0 & p \end{pmatrix}^{-1}
\begin{pmatrix} a & b \\ c & d \end{pmatrix}
\begin{pmatrix} 1 & x \\ 0 & p \end{pmatrix}
=
\begin{pmatrix}
a - \dfrac{cx}{p} & ax + bp - \dfrac{cx^2}{p} - dx \\[6pt]
\dfrac{c}{p} & \dfrac{cx}{p} + d
\end{pmatrix}.
\]

Let \( g = \begin{pmatrix} a & b \\ c & d \end{pmatrix} \) and \( g' = \begin{pmatrix} a' & b' \\ c' & d' \end{pmatrix} \) be elements of \( \SL_2(\mathbb{Z}) \), and let \( 1 \le x \ne x' \le p \). Then:
\[
(g_p r_x)^{-1} g \, g_p r_x = (g_p r_{x'})^{-1} g' \, g_p r_{x'}
\]
if and only if the following two cases hold:

\noindent \textbf{Case (1): \( c \ne 0 \).}  
The equality is equivalent to the system:
\[
\begin{cases}
a + d = a' + d', & \text{(1)} \\[2pt]
a - \dfrac{cx}{p} = a' - \dfrac{c'x'}{p}, & \text{(2)} \\[2pt]
c = c', & \text{(3)} \\[2pt]
\dfrac{ad - 1}{c} \in \mathbb{Z}, & \text{(4)} \\[2pt]
\dfrac{a'd' - 1}{c} \in \mathbb{Z}, & \text{(5)} \\[2pt]
a, c, d, a', d' \in \mathbb{Z}. & \text{(6)}
\end{cases}
\]

\noindent \textbf{Case (2): \( c = 0 \).}  
Then:
\[
\begin{cases}
a = d = a' = d' \in \{ \pm 1 \}, \\
b = b' \in \mathbb{Z}.
\end{cases}
\]
We now analyze \textbf{Case (1)} in detail. From (2) and (3), write \( c = p c_1 \). Then (2) becomes:
\[
a - c_1 x = a' - c_1 x' \quad \Rightarrow \quad a' = a - c_1(x - x').
\]
From (1), \( d' = d + c_1(x - x') \).  
Now consider the expression:
\[
\frac{a'd' - 1}{c} = \frac{ad - 1}{c} + \frac{\big[(a - d) - c_1(x - x')\big](x - x')}{p}.
\]

Suppose \( (g_p r_x)^{-1} g \, g_p r_x \in \bigcap_{1 \le y \le p} r_y^{-1} g_p^{-1} \SL_2(\mathbb{Z}) \, g_p r_y \).  
Then for all \( 1 \le x' \le p \), the above expression must be an integer.  
This implies:
\[
\frac{\big[(a - d) - c_1(x - x')\big](x - x')}{p} \in \mathbb{Z} \quad \text{for all } x'.
\]
In particular, for varying \( x' \), this forces:
\[
\begin{cases}
a \equiv d \pmod{p}, \\
p \mid c_1.
\end{cases}
\]

Conversely, suppose \( a \equiv d \pmod{p} \) and \( p^2 \mid c \).  
Define:
\[
\begin{cases}
a' = a - \dfrac{c}{p}(x - x'), \\
c' = c, \\
d' = d + \dfrac{c}{p}(x - x'), \\
b' = b + \left[ \dfrac{a - d}{p} - \dfrac{c(x - x')}{p^2} \right](x - x').
\end{cases}
\]
Then one verifies that \( g' \in \SL_2(\mathbb{Z}) \) and
\[
(g_p r_x)^{-1} g \, g_p r_x = (g_p r_{x'})^{-1} g' \, g_p r_{x'}.
\]

Finally, observe that
\[
g_p^{-1} \begin{pmatrix} a & b \\ c & d \end{pmatrix} g_p
= \begin{pmatrix} d &-\dfrac{c}{p}  \\[2pt] -bp & a \end{pmatrix}.
\]
Hence,
\[
\bigcap_{1 \le x \le p} r_x^{-1} g_p^{-1} \SL_2(\mathbb{Z}) \, g_p r_x
= r_x^{-1} \big[\Gamma(p,p) \big] r_x,
\]
since the intersection is independent of the choice of \( x \in \{1, \dots, p\} \).
\end{proof}

\begin{lemma}
\begin{itemize}
\item[(1)] $\SL_2(\Z)=\Gamma_0(p) \sqcup \Gamma_0(p)s\Gamma_0(p)$, for any $s\in \SL_2(\Z)\setminus \Gamma_0(p)$.
\item[(2)] $g_p^{-1}\SL_2(\Z)g_p=\Gamma_0(p) \sqcup \Gamma_0(p)s\Gamma_0(p)$, for any $s\in g_p^{-1}\SL_2(\Z)g_p\setminus \Gamma_0(p)$.
\end{itemize}
\end{lemma}
\begin{proof}
1) $\Gamma_0(p) r_0\Gamma_0(p)=\Gamma_0(p)$. For $1\leq x\leq p$, we have: 
\[\Gamma_0(p) r_x \Gamma_0(p)=\Gamma_0(p)   \begin{pmatrix} 0 & -1\\ 1& 0 \end{pmatrix} \begin{pmatrix} 1 &x\\ 0& 1 \end{pmatrix}\Gamma_0(p)=\Gamma_0(p)   \begin{pmatrix} 0 & -1\\ 1& 0 \end{pmatrix} \Gamma_0(p).\]
2) It is a consequence of Part (1).
\end{proof}

\begin{lemma}\label{coniso}
Let $N$ be an odd positive integer  number. 
\begin{itemize}
\item[(1)] The canonical  map $\iota:\bigl(\Gamma_0(N)\cap\Gamma_{\theta}\bigr)\backslash\Gamma_0(N)\longrightarrow\Gamma_{\theta}\backslash\mathrm{SL}_2(\mathbb Z)$
is bijective.
\item[(2)] The canonical  map $\iota:\bigl(\Gamma(N)\cap\Gamma_{\theta}\bigr)\backslash\Gamma(N)\longrightarrow\Gamma_{\theta}\backslash\mathrm{SL}_2(\mathbb Z)$
is bijective.
\end{itemize}
\end{lemma}
\begin{proof}
1) Write $G=\Gamma_0(N)\cap\Gamma_{\theta}$.
For $x,y\in\Gamma_0(N)$, the equality $Gx=Gy$ gives $xy^{-1}\in G\subseteq\Gamma_{\theta}$, hence $\Gamma_{\theta}x=\Gamma_{\theta}y$; thus $\iota$ is well-defined.
Conversely, if $\iota(Gx)=\iota(Gy)$, then $\Gamma_{\theta}x=\Gamma_{\theta}y$, so $xy^{-1}\in\Gamma_{\theta}$.
Because $x,y\in\Gamma_0(p)$, we also have $xy^{-1}\in\Gamma_0(N)$, hence $xy^{-1}\in G$ and $Gx=Gy$; therefore $\iota$ is injective.

To see surjectivity, recall that $|\Gamma_{\theta}\backslash\mathrm{SL}_2(\mathbb Z)|=3$.
The three cosets are represented by
\[
\begin{pmatrix}1&0\\0&1\end{pmatrix},\quad
\begin{pmatrix}1&1\\0&1\end{pmatrix},\quad
\begin{pmatrix}1&0\\-N&1\end{pmatrix}.
\]
Each of these matrices lies in $\Gamma_0(N)$, so their $G$-cosets map under $\iota$ to the three distinct cosets of $\Gamma_{\theta}$.
Consequently $\iota$ is surjective and hence bijective.\\
2) The proof is analogous, using the coset representatives:
\[
\begin{pmatrix}1&0\\0&1\end{pmatrix},\quad
\begin{pmatrix}1&N\\0&1\end{pmatrix},\quad
\begin{pmatrix}1&0\\-N&1\end{pmatrix}.
\]
\end{proof}
\begin{lemma}\label{conisoeve}
Let $N$ be an even positive integer  number. Then $\Gamma_0(N)=\bigl(\Gamma_0(N)\cap\Gamma_{\theta}\bigr) \sqcup \bigl(\Gamma_0(N)\cap\Gamma_{\theta}\bigr)u(1)$.
\end{lemma}
\begin{proof}
Let $g=\begin{pmatrix} a& b\\ c& d\end{pmatrix} \in \Gamma_0(N)$. Note that $2\nmid a$ and $2\nmid d$.  If $2\mid b$, then $g\in \bigl(\Gamma_0(N)\cap\Gamma_{\theta}\bigr)$. If $2\nmid b$, then $gu(-1)=\begin{pmatrix} a& b-a\\ c& d-c\end{pmatrix}$; $2\mid (b-a)$ and $gu(-1)\in \bigl(\Gamma_0(N)\cap\Gamma_{\theta}\bigr)$.
\end{proof}
\subsection{One-dimensional representations of $\Gamma_{\theta}, \widetilde{\Gamma}_{\theta}, \overline{\Gamma}_{\theta}$}\label{onedim}
 For  $r\in \Gamma_{\theta}$, we  define
   \begin{align}\label{lambda}
   \lambda(r)= m_{X^{\ast}}(r)\widetilde{\beta}^{-1}(r).
    \end{align}
  \begin{lemma}\label{mutirr12}
  $ \lambda(r_1)\lambda(r_2)=\lambda(r_1r_2)\overline{c}_{X^{\ast}}(r_1,r_2)$, for $r_i\in \Gamma_{\theta}$.
\end{lemma}
\begin{proof}
By Lemma \ref{spgamma12}(1), $\widetilde{\beta}(r_1)^{-1}\widetilde{\beta}(r_2)^{-1}\widetilde{\beta}(r_1r_2)=\widetilde{c}_{X^{\ast}}(r_1,r_2)$. Consequently,
\begin{align*}
& m_{X^{\ast}}(r_1)\widetilde{\beta}(r_1)^{-1} m_{X^{\ast}}(r_2)\widetilde{\beta}(r_2)^{-1}\\
&= m_{X^{\ast}}(r_1)m_{X^{\ast}}(r_2)\widetilde{\beta}(r_1r_2)^{-1}\widetilde{c}_{X^{\ast}}(r_1,r_2)\\
&= \widetilde{\beta}(r_1r_2)^{-1} m_{X^{\ast}}(r_1r_2)\overline{c}_{X^{\ast}}(r_1, r_2).
\end{align*}
\end{proof}
 For $\widetilde{r}=(r,t_r)\in \widetilde{\Gamma}_{\theta}$, we  define:
\[\widetilde{\lambda}: \widetilde{\Gamma}_{\theta} \longrightarrow \mathbb{T}; \widetilde{r}\longmapsto t_r \widetilde{\beta}^{-1}(r).\]
\begin{lemma}
$\widetilde{\lambda}$ is a character of $\widetilde{\Gamma}_{\theta}$.
\end{lemma}
\begin{proof}
Let $\widetilde{s}=(s,t_s)\in \widetilde{\Gamma}_{\theta}$. Then $\widetilde{r}\widetilde{s}=[rs, t_rt_s \widetilde{c}_{X^{\ast}}(t_r,t_s)]$. 
$$\widetilde{\lambda}(\widetilde{r}\widetilde{s})=t_rt_s \widetilde{c}_{X^{\ast}}(r,s)\widetilde{\beta}^{-1}(rs)\stackrel{\textrm{Lem. } \ref{spgamma12}}{=}t_rt_s \widetilde{\beta}^{-1}(r)\widetilde{\beta}^{-1}(s)=\widetilde{\lambda}(\widetilde{r})\widetilde{\lambda}(\widetilde{s}).$$
\end{proof}
Recall the isomorphism: $$\overline{\Gamma}_{\theta}\stackrel{\iota}{\to} \Im (\iota) (\subseteq\widetilde{\Gamma}_{\theta});[r,\epsilon] \longmapsto [r,m_{X^{\ast}}(r)\epsilon].$$   Let $\widetilde{r}=(r,t_r)\in \overline{\Gamma}_{\theta}$.  Let us define:
\begin{align}\label{overlinewidegamma}
\overline{\lambda}=\widetilde{\lambda} \circ \iota: \overline{\Gamma}_{\theta} \longrightarrow \mathbb{T}; [r,\epsilon]\longmapsto \epsilon m_{X^{\ast}}(r)\widetilde{\beta}^{-1}(r).
\end{align}
\begin{lemma}\label{lambdaoverlin}
$\overline{\lambda}(r,\epsilon)=\lambda(r) \epsilon$, for $(r,\epsilon)\in \overline{\Gamma}_{\theta}$.
\end{lemma}
\begin{proof}
$\overline{\lambda}(r,\epsilon)=m_{X^{\ast}}(r)\epsilon \widetilde{\beta}^{-1}(r) =\epsilon \lambda(r)$.
\end{proof}
Note:
\[ \Gamma_{\theta}=\Gamma(2) \sqcup \Gamma(2) \omega.\]
If $a,b \in \Z\setminus \{0\}$, we let $\left(\frac{a}{b}\right)$ denote the Kronecker  symbol,  and  define $\left(\frac{0}{1}\right)=1$, $\left(\frac{0}{-1}\right)=-1$.\footnote{Our definition differs from Shimura's \cite{Sh} in that we take $\Arg: \C^{\times} \to [ -\pi, \pi)$ instead of $ ( -\pi, \pi]$ in \cite{Sh}.} Following \cite{Sh}, let $\epsilon_d=\left\{ \begin{array}{cl}1 & \textrm{ if } d\equiv 1(\bmod 4),\\ i & \textrm{ if } d\equiv 3(\bmod 4).\end{array}\right.$
\begin{lemma}\label{gammar}
 $\overline{\lambda}([r,1])=\left\{ \begin{array}{ll} \left(\frac{2c}{d}\right)\epsilon^{-1}_d, & \textrm{ if } r=\begin{pmatrix} a& b\\ c& d\end{pmatrix} \in \Gamma(2),\\ \left(\frac{2c}{d}\right)\epsilon^{-1}_de^{-\tfrac{i \pi}{4}}(-c,d)_{\R},& \textrm{ if } r=\begin{pmatrix} a& b\\ c& d\end{pmatrix}\begin{pmatrix} 0& -1\\ 1& 0\end{pmatrix} \in \Gamma(2)\omega.\end{array} \right.$
\end{lemma}
\begin{proof}a) Let $r\in \Gamma(2)$. 
\begin{itemize}
\item[(1)] If $c=0$, $m_{X^{\ast}}(r)=e^{\tfrac{\pi i}{4}(1-\sgn(d))}$ and $\widetilde{\beta}^{-1}(r)=1$, then $\overline{\lambda}([r,1])= m_{X^{\ast}}(r)\widetilde{\beta}(r)^{-1}=\left(\frac{2c}{d}\right)\epsilon_d^{-1}$.
\item[(2 i)] If $c\neq 0$, $\sgn d>0$ and $d\equiv 1(\bmod 4)$, then $m_{X^{\ast}}(r)=e^{-\tfrac{\pi i \sgn(c)}{4}}$, $\widetilde{\beta}(r)=\left(\frac{2c}{d}\right) e^{-\tfrac{\pi i \sgn(c)}{4}}$ and $\overline{\lambda}([r,1])= m_{X^{\ast}}(r)\widetilde{\beta}(r)^{-1}=\left(\frac{2c}{d}\right)=\left(\frac{2c}{d}\right)\epsilon_d^{-1}$.
\item[(2 ii)] If $c\neq 0$, $\sgn d<0$ and $d\equiv 1(\bmod 4)$, then $m_{X^{\ast}}(r)=e^{-\tfrac{\pi i \sgn(c)}{4}}$, $\widetilde{\beta}(r)=(-i)\left(\frac{2c}{-d}\right) e^{\tfrac{\pi i \sgn(c)}{4}}$ and  \[\overline{\lambda}([r,1])=e^{-\tfrac{\pi i \sgn(c)}{4}}i\left(\frac{2c}{d}\right) \sgn(c) e^{-\tfrac{\pi i \sgn(c)}{4}}=\left(\frac{2c}{d}\right)= \left(\frac{2c}{d}\right)\epsilon_d^{-1}.\]
\item[(2 iii)]  If $c\neq 0$, $\sgn d>0$ and $d\equiv 3(\bmod 4)$, then $m_{X^{\ast}}(r)=e^{-\tfrac{\pi i \sgn(c)}{4}}$, $\widetilde{\beta}(r)=i\left(\frac{2c}{d}\right) e^{-\tfrac{\pi i \sgn(c)}{4}}$ and \[\overline{\lambda}([r,1])=e^{-\tfrac{\pi i \sgn(c)}{4}}(-i)\left(\frac{2c}{d}\right)  e^{\tfrac{\pi i \sgn(c)}{4}}=(-i)\left(\frac{2c}{d}\right) = \left(\frac{2c}{d}\right)\epsilon_d^{-1}.\]
    \item[(2 iv)]  If $c\neq 0$, $\sgn d<0$ and $d\equiv 3(\bmod 4)$, then $m_{X^{\ast}}(r)=e^{-\tfrac{\pi i \sgn(c)}{4}}$, $\widetilde{\beta}(r)=\left(\frac{2c}{-d}\right) e^{\tfrac{\pi i \sgn(c)}{4}}$ and \[\overline{\lambda}([r,1])=e^{-\tfrac{\pi i \sgn(c)}{4}}\left(\frac{2c}{d}\right) \sgn(c)  e^{-\tfrac{\pi i \sgn(c)}{4}}=(-i)\left(\frac{2c}{d}\right) = \left(\frac{2c}{d}\right)\epsilon_d^{-1}.\]
        \end{itemize}
b)  $m_{X^{\ast}}(\omega)=e^{-\tfrac{i \pi}{4}}$,  $\widetilde{\beta}^{-1}(\omega)=1$,  $\overline{\lambda}([\omega,1])= e^{-\tfrac{i \pi}{4}}$. Let $r=r_1\omega$, with $r_1=\begin{pmatrix} a& b\\ c& d\end{pmatrix}\in \Gamma(2)$.
\begin{align*}
\overline{c}_{X^{\ast}}(r_1,\omega)&=(x(r_1),x(\omega))_{\mathbb R}\,(-x(r_1)x(\omega),x(r_1\omega))_{\mathbb R}=(-x(r_1)x(\omega),x(r_1\omega))_{\mathbb R}\\
&=\left\{\begin{array}{ll} (-c,d)_{\R}&   \textrm{ if } c\neq  0, \\ (-d,d)_{\R}=1 & \textrm{ if } c= 0,\end{array}\right.\\
&=(-c,d)_{\R}.
\end{align*}
\begin{align*}
\overline{\lambda}([r_1\omega,1])&=\overline{\lambda}([r_1,1])\overline{\lambda}([\omega,1])\overline{c}_{X^{\ast}}(r_1,\omega)^{-1}\\
&=\overline{\lambda}([r_1,1])\overline{\lambda}([\omega,1])\overline{c}_{X^{\ast}}(r_1,\omega)\\
&=\left(\frac{2c}{d}\right)\epsilon_d^{-1}e^{-\tfrac{i \pi}{4}}\overline{c}_{X^{\ast}}(r_1,\omega)\\
&=\left(\frac{2c}{d}\right)\epsilon_d^{-1}e^{-\tfrac{i \pi}{4}}(-c,d)_{\R}.
\end{align*}
\end{proof}
\begin{remark}
The function defined above coincides with the multiplier system  $\nu$ appearing in  \cite[Prop.2.3.3(b)]{CoSt}.
\end{remark}
\begin{proof}
By obversation, the two expressions agree when $r\in \Gamma(2)$.  It therefore suffices to verify the remaining case where $r=\begin{pmatrix} a& b\\ c& d\end{pmatrix} \in \Gamma(2)\omega$. Write  $r=r_1\omega$, with  $$r_1=r\omega^{-1}=\begin{pmatrix} a& b\\ c& d\end{pmatrix}\begin{pmatrix} 0& 1\\ -1& 0\end{pmatrix}=\begin{pmatrix} -b& a\\ -d& c\end{pmatrix}.$$
Consequently, we obtain:
\[\overline{\lambda}([r,1])=\left(\frac{-2d}{c}\right)\epsilon^{-1}_ce^{-\tfrac{i \pi}{4}}(d,c)_{\R}=\left(\frac{2d}{c}\right)[\left(\frac{-1}{c}\right)\epsilon^{-1}_c]e^{-\tfrac{i \pi}{4}}(d,c)_{\R}=\left(\frac{2d}{c}\right)\epsilon_ce^{-\tfrac{i \pi}{4}}(c,d)_{\R},\]
which recovers precisely the formula given in \cite[Prop.~2.3.3(b)]{CoSt}.
\end{proof}
\subsection{Two-dimensional representations of $ \widetilde{\Gamma}^{\pm}_{\theta}, \overline{\Gamma}^{\pm}_{\theta}$}\label{twodim}
Let $\widetilde{\Gamma}^{\pm}_{\theta}, \overline{\Gamma}^{\pm}_{\theta}$ denote the corresponding covering subgroups of $\widetilde{\SL}^{\pm}_2(\Z), \overline{\SL}_2^{\pm}(\Z)$ respectively. Let us define:
\[\widetilde{\lambda}^{\pm}\stackrel{Def.}{=}\Ind_{\widetilde{\Gamma}_{\theta}}^{\widetilde{\Gamma}^{\pm}_{\theta}} \widetilde{\lambda}, \qquad\quad
\overline{\lambda}^{\pm}\stackrel{Def.}{=}\Ind_{\overline{\Gamma}_{\theta}}^{\overline{\Gamma}^{\pm}_{\theta}} \overline{\lambda}.\]
\begin{lemma}\label{gammatwodim}
Both $\widetilde{\lambda}^{\pm}$, $\overline{\lambda}^{\pm}$ are irreducible representations of two dimension.
\end{lemma}
\begin{proof}
1)  Take the element \(r=\begin{pmatrix}3 & 4\\ 2& 3\end{pmatrix}\in\Gamma(2)\subseteq\Gamma_{\theta}\) and set \(\widetilde{r}=[r,1]\).    Then:
\[[h_{-1},1]^{-1}=[h_{-1}, \widetilde{C}_{X^{\ast}}([-1,1], [-1, 1])^{-1}]=[h_{-1},1],\]
\[\widetilde{C}_{X^{\ast}}(h_{-1}, r)\widetilde{C}_{X^{\ast}}(h_{-1} r, h_{-1} )=\nu( 1,1)\widetilde{c}_{X^{\ast}}(1, r)\nu(- 1,r)\widetilde{c}_{X^{\ast}}(r^{h_{-1}}, 1)= \gamma(-1, \psi^{\tfrac{1}{2}})^{-1}=i,\]
\[[h_{-1},1]\widetilde{r}[h_{-1},1]=[h_{-1}rh_{-1},\widetilde{C}_{X^{\ast}}(h_{-1}, r)\widetilde{C}_{X^{\ast}}(h_{-1} r, h_{-1} )]=[h_{-1}rh_{-1},i].\]
Consequently,
  \[ \widetilde{\beta}(r)\stackrel{\textrm{Lem. \ref{chap8betacd}}}{=}i \Bigl(\frac{4}{3}\Bigr) e^{-\tfrac{ \pi i}{4}}=i e^{-\tfrac{ \pi i}{4}}, \qquad
    \widetilde{\lambda}(\widetilde{r})
    = \widetilde{\beta}(r)^{-1}=(-i) e^{\tfrac{ \pi i}{4}},
  \]
  whereas
  \[ \widetilde{\beta}(r^{h_{-1}})\stackrel{\textrm{Lem. \ref{chap8betacd}}}{=}i \Bigl(\frac{-4}{3}\Bigr) e^{\tfrac{ \pi i}{4}}=(-i) e^{\tfrac{ \pi i}{4}},\qquad
    \widetilde{\lambda}^{h_{-1}}(\widetilde{r})
    =\widetilde{\lambda}\!\left([r^{h_{-1}}, i]\right)
    =i\widetilde{\beta}(r^{h_{-1}})^{-1}=- e^{-\tfrac{ \pi i}{4}}.
  \]
  Hence \(\widetilde{\lambda}\neq\widetilde{\lambda}^{h_{-1}}\), proving that \(\widetilde{\lambda}^{\pm}\) is irreducible.\\
   2)  Note that $\overline{\Gamma}^{\pm}_{\theta}$  embeds into $\widetilde{\Gamma}^{\pm}_{\theta}$ with image differing on the center, hence the result follows. 
\end{proof}
\subsection{Finite-dimensional representations of $\widetilde{\SL}_{2}(\Z)$ and $\overline{\SL}_{2}(\Z)$}\label{FindimSLwide} Let $N^2$ be an odd positive integer number. 
Recall:  
$$  u(N^2)= \begin{pmatrix} 1& N^2\\ 0& 1\end{pmatrix}, \quad  u_-(-N^2)=\begin{pmatrix} 1& 0\\-N^2& 1\end{pmatrix}, \quad  1= \begin{pmatrix} 1& 0\\0&1\end{pmatrix}.$$ 

\begin{definition}\label{Mq123}
\begin{itemize}
\item[(1)] In $\SL_2(\Z)$, define  $ M_{q_1}=u(N^2)$, $M_{q_2}=1$,  $M_{q_3}=u_-(-N^2)$, and set $\mathcal{M}_2=\{M_{q_1},M_{q_2},M_{q_3}\}$.
\item[(2)] In $\widetilde{\SL}_2(\Z)$, define $ \widetilde{M}_{q_1}=[u(N^2),1], \widetilde{M}_{q_2}=[1,1], \widetilde{M}_{q_3}=[u_-(-N^2),1],$ and set  $\widetilde{\mathcal{M}}_2=\{\widetilde{M}_{q_1},\widetilde{M}_{q_2},\widetilde{M}_{q_3}\}$.
\item[(3)]  In $\overline{\SL}_2(\Z)$, define  $ \overline{M}_{q_1}=[u(N^2),1], \overline{M}_{q_2}=[1,1], \overline{M}_{q_3}=[u_-(-N^2),1],$ and set  $\overline{\mathcal{M}}_2=\{\overline{M}_{q_1},\overline{M}_{q_2},\overline{M}_{q_3}\}$.
\end{itemize}
\end{definition}
Then:
\[\SL_2(\Z)=\sqcup^3_{i=1} \Gamma_{\theta}  M_{q_i}=\sqcup^3_{i=1}  M_{q_i}  \Gamma_{\theta}.\]  
\[\widetilde{\SL}_{2}(\Z)=\sqcup^3_{i=1} \widetilde{\Gamma}_{\theta}  \widetilde{M}_{q_i}=\sqcup^3_{i=1} \widetilde{M}_{q_i} \widetilde{\Gamma}_{\theta} .\] 
\[\overline{\SL}_{2}(\Z)=\sqcup^3_{i=1} \overline{\Gamma}_{\theta}  \overline{M}_{q_i}=\sqcup^3_{i=1} \overline{M}_{q_i} \overline{\Gamma}_{\theta} .\] 
\begin{definition}
\begin{itemize}
\item[(1)]  $\widetilde{\gamma}=\Ind_{\widetilde{\Gamma}_{\theta}}^{\widetilde{\SL}_{2}(\Z)} \widetilde{\lambda}^{-1}, \widetilde{M}=\Ind_{\widetilde{\Gamma}_{\theta}}^{\widetilde{\SL}_{2}(\Z)} \C.$
\item[(2)]$\overline{\gamma}=\Ind_{\overline{\Gamma}_{\theta}}^{\overline{\SL}_{2}(\Z)} \overline{\lambda}^{-1}, \overline{M}=\Ind_{\overline{\Gamma}_{\theta}}^{\overline{\SL}_{2}(\Z)} \C.$
\end{itemize}
\end{definition}
\begin{proposition}\label{irrdgamma}
\begin{itemize}
\item[(1)]  $\widetilde{\gamma}$ is an irreducible representation of $\widetilde{\SL}_{2}(\Z)$. 
\item[(2)] $\overline{\gamma}$ is an irreducible representation of $\overline{\SL}_{2}(\Z)$. 
\end{itemize}
\end{proposition}
\begin{proof}
1) Let us write $G=\widetilde{\SL}_{2}(\Z)$, $H=\widetilde{\Gamma}_{\theta}$. 
Then $G= H \sqcup [H s H] $, for any $s\notin H $.  For such $s$, write $H_s =s^{-1}Hs$, and set $(\widetilde{\lambda}^{-1})^s(\widetilde{r})=\widetilde{\lambda}^{-1}(s\widetilde{r}s^{-1})$, for $\widetilde{r} \in H_s \cap H$. Then:
$$\Hom_{G}(\widetilde{\gamma},\widetilde{\gamma})\simeq \Hom_{H}(\widetilde{\lambda}^{-1},\Ind_{H}^{G} \widetilde{\lambda}^{-1})\simeq \Hom_{H}(\widetilde{\lambda}^{-1},\widetilde{\lambda}^{-1} \oplus \Ind_{H_s \cap H}^H (\widetilde{\lambda}^{-1})^s)$$
$$\simeq  \C \oplus \Hom_{H_s\cap H}(\widetilde{\lambda}^{-1}, (\widetilde{\lambda}^{-1})^s)\simeq \C \oplus \Hom_{H_s\cap H}(\widetilde{\lambda}, \widetilde{\lambda}^s).$$
In particular, we consider $s=[u(1),1]$ and $s_0=u(1)$.  
\[\widetilde{\lambda}:\widetilde{r}=(r,t_r)\longmapsto t_r \widetilde{\beta}^{-1}(r),\]
\[\widetilde{\lambda}^s:\widetilde{r}^{s^{-1}}=(r^{s_0^{-1}},t_r)\longmapsto t_r \widetilde{\beta}^{-1}(r^{s_0^{-1}}).\]
Let us check the element $u_-(-2)=\begin{pmatrix} 1& 0\\ -2& 1\end{pmatrix}\in \Gamma(2)\subseteq \Gamma_{\theta} \subseteq \SL_2(\Z)$. Then
\[r= s_0^{-1}u_-(-2) s_0=\begin{pmatrix} 1&-1\\ 0& 1\end{pmatrix}\begin{pmatrix} 1& 0\\ -2& 1\end{pmatrix} \begin{pmatrix} 1& 1\\ 0& 1\end{pmatrix}=\begin{pmatrix} 3 & 2 \\ -2 & -1 \end{pmatrix}, \quad [r,1]\in H_s\cap H.\]
\[\widetilde{\beta}^{-1}(r^{s_0^{-1}})\stackrel{\textrm{Lem. \ref{chap8betacd}}}{=}e^{-\tfrac{\pi i}{4}}, \quad \widetilde{\beta}^{-1}(r)\stackrel{\textrm{Lem. \ref{chap8betacd}}}{=}e^{\tfrac{\pi i}{4}}.\]
Hence $\Hom_{G}(\widetilde{\gamma},\widetilde{\gamma})\simeq \C$.\\
2) 
 Recall the isomorphism $\overline{\SL}^{\mu_8}_2(\mathbb R)\stackrel{\iota}{\to}\widetilde{\SL}_2(\mathbb R)$  from (\ref{overwideSL}). It follows from that $\overline{\gamma}=\widetilde{\gamma}\circ \iota$, so the result holds.
\end{proof}
\begin{corollary}\label{irrdgammaodd}
For an odd positive $N$, the restriction of  $\widetilde{\gamma}$ (resp. $\overline{\gamma}$) to  $\widetilde{\Gamma}_0(N^2)$  (resp. $\overline{\Gamma}_0(N^2)$ )  is also irreducible. 
\end{corollary}
\begin{proof}
For $\widetilde{\gamma}$, the proof is similar as above by using  $G=\widetilde{\Gamma}_0(N^2)$, $H=\widetilde{\Gamma}_0(N^2)\cap \widetilde{\Gamma}_{\theta}$, $s=[u(1),1]$, $s_0=u(1)$. However, we consider the element  $u_-(-2N^2)=\begin{pmatrix} 1& 0\\ -2N^2& 1\end{pmatrix}\in \Gamma(2)\cap \Gamma_0(N^2)\subseteq \Gamma_{\theta} \cap\Gamma_0(N^2) \subseteq \SL_2(\Z)$. 
\[r= s_0^{-1}u_-(-2N^2) s_0=\begin{pmatrix} 1&-1\\ 0& 1\end{pmatrix}\begin{pmatrix} 1& 0\\ -2N^2& 1\end{pmatrix} \begin{pmatrix} 1& 1\\ 0& 1\end{pmatrix}=\begin{pmatrix} 1+2N^2 & 2N^2 \\ -2N^2 & 1-2N^2 \end{pmatrix}, \quad [r,1]\in H_s\cap H.\]
\[\widetilde{\beta}^{-1}(r^{s_0^{-1}})\stackrel{\textrm{Lem. \ref{chap8betacd}}}{=}e^{-\tfrac{\pi i}{4}}, \quad \widetilde{\beta}^{-1}(r)\stackrel{\textrm{Lem. \ref{chap8betacd}}}{=}e^{\tfrac{\pi i}{4}}.\]
Hence $\Hom_{G}(\widetilde{\gamma},\widetilde{\gamma})\simeq \C$. For $\overline{\gamma}$, the reason is similar.
\end{proof}
\subsection{Finite-dimensional representations of $\widetilde{\SL}_{2}(\Z)$ II}\label{FindimSLwideII}
For  $g\in \SL_2(\Z)$, write $g=r M_{q}$, for  unique  $r\in \Gamma_{\theta}$ and  $M_{q}\in \mathcal{M}_2$. Define the function:
$$f: \SL_2(\Z) \longmapsto \mu_8; \quad g \longmapsto \widetilde{\beta}(r) \widetilde{c}_{X^{\ast}}(r, M_q).$$
We  modify the cocycle $\widetilde{c}_{X^{\ast}}$ by $f$  to obtain a new cocycle: 
$$\widetilde{c}'_{X^{\ast}}(g_1,g_2)=\widetilde{c}_{X^{\ast}}(g_1,g_2) f(g_1)f(g_2) f(g_1g_2)^{-1}.$$
\begin{lemma}\label{crg}
Let $r, r_1,r_2\in \Gamma_{\theta}$ and $g\in \SL_2(\Z)$. Then $\widetilde{c}'_{X^{\ast}}(r,g)=1=\widetilde{c}'_{X^{\ast}}(r_1,r_2)$.
\end{lemma}
\begin{proof}
1) By Lem. \ref{spgamma12}, $\widetilde{c}'_{X^{\ast}}(r_1,r_2)=1$.\\
2) If $g=M_q$, then $$\widetilde{c}'_{X^{\ast}}(r,g)=\widetilde{c}_{X^{\ast}}(r,M_q) f(r)f(M_q) f(rM_q)^{-1}=\widetilde{c}_{X^{\ast}}(r,M_q)\widetilde{\beta}(r)1 \widetilde{\beta}(r)^{-1} \widetilde{c}_{X^{\ast}}(r, M_q)^{-1}=1.$$
3) For  general $g$,  write $g=r_gM_q$. Then: 
\begin{equation*}
\widetilde{c}'_{X^{\ast}}(r,g)=\widetilde{c}'_{X^{\ast}}(r,r_gM_q)=\widetilde{c}'_{X^{\ast}}(r,r_gM_q)\widetilde{c}'_{X^{\ast}}(r_g,M_q)
=\widetilde{c}'_{X^{\ast}}(r,r_g)\widetilde{c}'_{X^{\ast}}(rr_g,M_q)=1.
\end{equation*}
\end{proof}
\begin{lemma}\label{ep}
$\widetilde{c}'_{X^{\ast}}(g,-)$ defines a character of $\Gamma(2)$. 
\end{lemma}
\begin{proof}
Let $g_1,g_2,g \in \SL_2(\Z)$ and $r,r_1,r_2\in \Gamma(2)$.
\begin{align*}
\widetilde{c}'_{X^{\ast}}(rg_1,g_2)
&=\widetilde{c}'_{X^{\ast}}(r,g_1)\widetilde{c}'_{X^{\ast}}(rg_1,g_2)\\
&=\widetilde{c}'_{X^{\ast}}(r,g_1g_2)\widetilde{c}'_{X^{\ast}}(g_1,g_2)\\
&=\widetilde{c}'_{X^{\ast}}(g_1,g_2).
\end{align*}
\begin{align*}
\widetilde{c}'_{X^{\ast}}(g,r_1r_2)
&=\widetilde{c}'_{X^{\ast}}(g,r_1r_2)\widetilde{c}'_{X^{\ast}}(r_1,r_2)\\
&=\widetilde{c}'_{X^{\ast}}(g,r_1)\widetilde{c}'_{X^{\ast}}(gr_1, r_2)\\
&=\widetilde{c}'_{X^{\ast}}(g,r_1)\widetilde{c}'_{X^{\ast}}(gr_1g^{-1} g, r_2)\\
&=\widetilde{c}'_{X^{\ast}}(g,r_1)\widetilde{c}'_{X^{\ast}}( g, r_2).
\end{align*}
\end{proof}
\begin{lemma}
$\widetilde{c}'_{X^{\ast}}(g,r)=-1$  for some $g\in  \SL_2(\Z)$ and some  $r\in \Gamma(4) $.
\end{lemma}
\begin{proof}
Consider $r=u_-(-4)$ and $g=M_q=u(N)$. Then:
\[M_qr=\begin{pmatrix} 1& N\\ 0& 1\end{pmatrix}\begin{pmatrix} 1& 0\\-4& 1\end{pmatrix}=\begin{pmatrix} 1-4N& N\\-4& 1\end{pmatrix}=\begin{pmatrix} 1-4N & 4N^2\\ -4& 1+4N\end{pmatrix}\cdot \begin{pmatrix}1 & N\\ 0& 1\end{pmatrix};\]
\[f(M_q)=1, \quad\quad f(r)= \widetilde{\beta}(r) =  \widetilde{\beta}(u_-(-4))\stackrel{\textrm{Lem. \ref{chap8betacd}}}{=}e^{\tfrac{\pi i}{4}},\]
 \begin{align*}
 f(M_qr)&=f(\begin{pmatrix} 1-4N& N\\-4& 1\end{pmatrix})=f(\begin{pmatrix} 1-4N & 4N^2\\ -4& 1+4N\end{pmatrix}\cdot \begin{pmatrix}1 & N\\ 0& 1\end{pmatrix})\\
 &=\widetilde{\beta}(\begin{pmatrix} 1-4N & 4N^2\\ -4& 1+4N\end{pmatrix}) \widetilde{c}_{X^{\ast}}(\begin{pmatrix} 1-4N & 4N^2\\ -4& 1+4N\end{pmatrix}, \begin{pmatrix}1 & N\\ 0& 1\end{pmatrix})\\
 &=\widetilde{\beta}(\begin{pmatrix} 1-4N & 4N^2\\ -4& 1+4N\end{pmatrix})\stackrel{\textrm{Lem. \ref{chap8betacd}}}{=}\big(\tfrac{-2}{1+4N}\big)e^{\tfrac{\pi i}{4}}=\big(\tfrac{-1}{1+4N}\big)\big(\tfrac{2}{1+4N}\big)e^{\tfrac{\pi i}{4}}=-e^{\tfrac{\pi i}{4}}.
 \end{align*}
 \begin{align*}
\widetilde{c}'_{X^{\ast}}(g,r)
&=\widetilde{c}_{X^{\ast}}(M_q,r) f(M_q)f(r)f(M_qr)^{-1}\\
&=f(M_q)f(r)f(M_qr)^{-1}=-1.
\end{align*}
\end{proof}
\subsection{Finite-dimensional representations of $\widetilde{\SL}_{2}^{\pm}(\Z)$ and $\overline{\SL}_{2}^{\pm}(\Z)$}\label{FindimSLwidesign}
\begin{definition}
\begin{itemize}
\item[(1)]  $\widetilde{\gamma}^{\pm}=\Ind_{\widetilde{\Gamma}_{\theta}}^{\widetilde{\SL}^{\pm}_{2}(\Z)} \widetilde{\lambda}^{-1}, \widetilde{M}^{\pm}=\Ind_{\widetilde{\Gamma}_{\theta}}^{\widetilde{\SL}^{\pm}_{2}(\Z)} \C.$
\item[(2)]$\overline{\gamma}^{\pm}=\Ind_{\overline{\Gamma}_{\theta}}^{\overline{\SL}^{\pm}_{2}(\Z)} \overline{\lambda}^{-1}, \overline{M}^{\pm}=\Ind_{\overline{\Gamma}_{\theta}}^{\overline{\SL}^{\pm}_{2}(\Z)} \C.$
\end{itemize}
\end{definition}
\begin{proposition}\label{irrdgammapm}
\begin{itemize}
\item[(1)]  $\widetilde{\gamma}^{\pm}$ is an irreducible representation of $\widetilde{\SL}^{\pm}_{2}(\Z)$. 
\item[(2)] $\overline{\gamma}^{\pm}$ is an irreducible representation of $\overline{\SL}^{\pm}_{2}(\Z)$. 
\end{itemize}
\end{proposition}
\begin{proof}
(1) Note that \(\widetilde{\gamma}^{\pm}|_{\widetilde{\SL}_{2}(\Z)} \simeq \widetilde{\gamma} \oplus \widetilde{\gamma}^{h_{-1}}\). Using the notations from the proof of Prop.~\ref{irrdgamma}, we have
\[
\Hom_{G}(\widetilde{\gamma}, \widetilde{\gamma}^{h_{-1}}) 
\simeq \Hom_{H}\big(\widetilde{\lambda}^{-1}, [\widetilde{\lambda}^{-1}]^{h_{-1}}\big) 
\oplus \Hom_{H_s\cap H}\Big(\widetilde{\lambda}^{-1}, \big([\widetilde{\lambda}^{-1}]^{h_{-1}}\big)^s\Big) 
\simeq \Hom_{H_s\cap H}\Big(\widetilde{\lambda}, [\widetilde{\lambda}^{h_{-1}}]^s\Big).
\]
We now take  $s=u_-(1)$. Then
\[[s,1] [r,t_r] [s^{-1},1]=[r^{s^{-1}}, \widetilde{c}_{X^{\ast}}(s,r)  \widetilde{c}_{X^{\ast}}(sr,s^{-1}) t_r]=[r^{s^{-1}},t_r],\]
\[
\widetilde{\lambda} : \widetilde{r} = (r, t_r) \mapsto t_r \widetilde{\beta}^{-1}(r),
\]
\[
[\widetilde{\lambda}^{h_{-1}}]^s : \widetilde{r}^{s^{-1}} = (r^{s^{-1}},   t_r) \mapsto t_r \widetilde{\beta}^{-1}(h_{-1} r^{s^{-1}} h_{-1}).
\]
Consider the element
\[
r = s^{-1}\begin{pmatrix}-1 & 0\\ 2 & -1\end{pmatrix}s 
= \begin{pmatrix}1 & 0\\ -1 & 1\end{pmatrix}\begin{pmatrix}-1 & 0\\ 2 & -1\end{pmatrix}\begin{pmatrix}1 & 0\\ 1 & 1\end{pmatrix} 
= \begin{pmatrix}-1 & 0\\ 2 & -1\end{pmatrix} \in \Gamma(2).
\]
Then
\[
\widetilde{\beta}^{-1}(h_{-1} r^{s^{-1}} h_{-1}) 
= \widetilde{\beta}^{-1}(\begin{pmatrix}-1 & 0\\ -2 & -1\end{pmatrix})
\stackrel{\text{Lem.~\ref{chap8betacd}}}{=} \left[\left(\frac{-1}{1}\right)e^{-\frac{\pi i}{4}\sgn(2)}\right]^{-1} 
= e^{\frac{\pi i}{4}},
\]
and
\[
\widetilde{\beta}^{-1}(\begin{pmatrix}-1 & 0\\ 2 & -1\end{pmatrix} )
\stackrel{\text{Lem.~\ref{chap8betacd}}}{=} \left[\left(\frac{N^2}{1}\right)e^{-\frac{\pi i}{4}\sgn(-2)}\right]^{-1} 
= e^{-\frac{\pi i}{4}}.
\]
Hence \(\Hom_{G}(\widetilde{\gamma}, \widetilde{\gamma}^{h_{-1}}) = 0\), and so \(\widetilde{\gamma}^{\pm}\) is irreducible.

(2) Recall the isomorphism \(\overline{\SL}^{\mu_8}_2(\mathbb{R}) \xrightarrow{\iota} \widetilde{\SL}_2(\mathbb{R})\) from \eqref{overwideSL}. Since \(\overline{\gamma}^{\pm} = \widetilde{\gamma}^{\pm} \circ \iota\), the irreducibility of \(\widetilde{\gamma}^{\pm}\) implies that of \(\overline{\gamma}^{\pm}\).
\end{proof}
\begin{corollary}\label{irrdgammapm2nn}
For an odd positive $N$, the restriction of  $\widetilde{\gamma}^{\pm}$ (resp. $\overline{\gamma}^{\pm}$) to  $\widetilde{\Gamma}^{\pm}_0(N^2)$  (resp. $\overline{\Gamma}^{\pm}_0(N^2)$ )  is also irreducible. 
\end{corollary}
\begin{proof}
For $\widetilde{\gamma}^{\pm}$, the proof is similar as above.  Note that \(\widetilde{\gamma}^{\pm}|_{\widetilde{\Gamma}^{\pm}_0(N^2)} \simeq \widetilde{\gamma} \oplus \widetilde{\gamma}^{h_{-1}}\). Using the notations from the proof of Coro.~\ref{irrdgammaodd}, we have
\[
\Hom_{G}(\widetilde{\gamma}, \widetilde{\gamma}^{h_{-1}}) 
\simeq \Hom_{H_s\cap H}\Big(\widetilde{\lambda}, [\widetilde{\lambda}^{h_{-1}}]^s\Big).
\]
We also  take  $s=u_-(1)$. Consider the element 
$$r= s^{-1}[-u_-(-2N^2)]s=\begin{pmatrix} 1&-1\\ 0& 1\end{pmatrix}\begin{pmatrix} -1& 0\\ 2N^2& -1\end{pmatrix} \begin{pmatrix} 1& 1\\ 0& 1\end{pmatrix}=\begin{pmatrix} -1-2N^2 & -2N^2 \\ 2N^2 & -1+2N^2 \end{pmatrix}\in \Gamma(2N^2).$$
Then
\[
\widetilde{\beta}^{-1}(h_{-1} r^{s^{-1}} h_{-1}) 
= \widetilde{\beta}^{-1}(\begin{pmatrix} -1& 0\\ -2N^2& -1\end{pmatrix} )
\stackrel{\text{Lem.~\ref{chap8betacd}}}{=} \left[\left(\frac{-N^2}{1}\right)e^{-\frac{\pi i}{4}\sgn(2N^2)}\right]^{-1} 
= e^{\frac{\pi i}{4}},
\]
and
\[
\widetilde{\beta}^{-1}(\begin{pmatrix}-1 & 0\\ 2N^2 & -1\end{pmatrix} )
\stackrel{\text{Lem.~\ref{chap8betacd}}}{=} \left[\left(\frac{1}{1}\right)e^{-\frac{\pi i}{4}\sgn(-2N^2)}\right]^{-1} 
= e^{-\frac{\pi i}{4}}.
\]
Hence \(\Hom_{G}(\widetilde{\gamma}, \widetilde{\gamma}^{h_{-1}}) = 0\), and so $\widetilde{\gamma}^{\pm}|_{\widetilde{\Gamma}^{\pm}_0(N^2)}$ is irreducible. Moreover, the irreducibility of \(\widetilde{\gamma}^{\pm}|_{\widetilde{\Gamma}^{\pm}_0(N^2)}\) implies that of \(\overline{\gamma}^{\pm}|_{\overline{\Gamma}^{\pm}_0(N^2)}\).
\end{proof}
\subsection{Twisted by  $\overline{\mathfrak{w}}$}\label{Twistedby}
Let $\overline{\mathfrak{w}}\in \GL_2(\Q)$, and $\overline{\mathfrak{w}}=(\mathfrak{w}, t_{\mathfrak{w}})\in \overline{\GL}_2(\Q)$.  Let $\overline{r}=(r,t_r)=\overline{\mathfrak{w}}^{-1}\overline{s}\overline{\mathfrak{w}}\in \overline{\mathfrak{w}}^{-1}\overline{\Gamma}_{\theta}\overline{\mathfrak{w}}$.  Let us define:
\[\overline{\lambda}^{\overline{\mathfrak{w}}}: \overline{\mathfrak{w}}^{-1}\overline{\Gamma}_{\theta}\overline{\mathfrak{w}}\longrightarrow \mathbb{T}; \qquad \overline{r}\longmapsto \overline{\lambda}(\overline{s}).\]
Then $\overline{\lambda}^{\overline{\mathfrak{w}}}$ is a character of $\overline{\mathfrak{w}}^{-1}\overline{\Gamma}_{\theta}\overline{\mathfrak{w}}$. 
\begin{definition}
$\overline{\gamma}^{\overline{\mathfrak{w}}}=\Ind_{\overline{\mathfrak{w}}^{-1}\overline{\Gamma}_{\theta}\overline{\mathfrak{w}}}^{\overline{\mathfrak{w}}^{-1}\overline{\SL}_{2}(\Z)\overline{\mathfrak{w}}} [\overline{\lambda}^{\overline{\mathfrak{w}}}]^{-1}, \overline{M}^{\overline{\mathfrak{w}}}=\Ind_{\overline{\mathfrak{w}}^{-1}\overline{\Gamma}_{\theta}\overline{\mathfrak{w}}}^{\overline{\mathfrak{w}}^{-1}\overline{\SL}_{2}(\Z)\overline{\mathfrak{w}}} \C.$
\end{definition}
By Prop. \ref{irrdgamma}, $\overline{\gamma}^{\overline{\mathfrak{w}}}$ is an irreducible representation. Similarly, the vector space  $\overline{M}^{\overline{\mathfrak{w}}}$ consists of the functions $f: \overline{\mathfrak{w}}^{-1}\overline{\SL}_{2}(\Z)\overline{\mathfrak{w}} \longrightarrow \C$ such that $f(\overline{r}\overline{g})=\overline{\lambda}^{\overline{\mathfrak{w}}}(\overline{r})^{-1}f(\overline{g})$, for $\overline{r}\in\overline{\mathfrak{w}}^{-1}\overline{\Gamma}_{\theta}\overline{\mathfrak{w}}$ and $\overline{g}\in \overline{\mathfrak{w}}^{-1}\overline{\SL}_{2}(\Z)\overline{\mathfrak{w}}$.  Let $\overline{e}_i^{\overline{\mathfrak{w}}}$ denote the function of $\overline{M}^{\overline{\mathfrak{w}}}$, supported on $\overline{\mathfrak{w}}^{-1}\overline{\Gamma}_{\theta}\overline{M}_{q_i}\overline{\mathfrak{w}}$ and $\overline{e}^{\overline{\mathfrak{w}}}_i(\overline{\mathfrak{w}}^{-1}\overline{M}_{q_i}\overline{\mathfrak{w}})=1$. Then $\overline{M}^{\overline{\mathfrak{w}}}=\oplus_{i=1}^3 \C \overline{e}^{\overline{\mathfrak{w}}}_i$.   Write 
    $$\overline{\gamma}^{\overline{\mathfrak{w}}}(\overline{g})(\overline{e}^{\overline{\mathfrak{w}}}_1, \overline{e}^{\overline{\mathfrak{w}}}_2, \overline{e}^{\overline{\mathfrak{w}}}_3)=(\overline{e}^{\overline{\mathfrak{w}}}_1, \overline{e}^{\overline{\mathfrak{w}}}_2, \overline{e}^{\overline{\mathfrak{w}}}_3)M^{\overline{\mathfrak{w}}}(\overline{g}),$$ for some $3\times 3$-matrix $M^{\overline{\mathfrak{w}}}(\overline{g})$. Then 
    \[\begin{array}{rcll}
    \overline{\gamma}^{\overline{\mathfrak{w}}}: &  \overline{\mathfrak{w}}^{-1}\overline{\SL}_{2}(\Z)\overline{\mathfrak{w}} & \longrightarrow &  M_3(\C);\\
                        &\overline{g}                    & \longmapsto     & M^{\overline{\mathfrak{w}}}(\overline{g}),
    \end{array}\] 
    gives a matrix representation. More precisely, we have:
\begin{itemize}
 \item \begin{itemize}
\item[(a1)]   $\overline{M}_{q_i} \overline{s}=\overline{s}'\overline{M}_{q_j}$, for some $\overline{s}, \overline{s}' \in \overline{\Gamma}_{\theta}$ iff   $[\overline{\mathfrak{w}}^{-1}\overline{M}_{q_i}\overline{\mathfrak{w}}] \overline{r}=\overline{r}'  [\overline{\mathfrak{w}}^{-1}\overline{M}_{q_j}\overline{\mathfrak{w}}]$, for $\overline{r}=[\overline{\mathfrak{w}}^{-1} \overline{s}\overline{\mathfrak{w}}]$, $\overline{r}'=[\overline{\mathfrak{w}}^{-1} \overline{s}'\overline{\mathfrak{w}}]$ .
\item[(b1)] If the above condition holds, then $ [\overline{\gamma}^{\overline{\mathfrak{w}}}(\overline{r})(\overline{e}^{\overline{\mathfrak{w}}}_j)]([\overline{\mathfrak{w}}^{-1}\overline{M}_{q_i}\overline{\mathfrak{w}}])
    =\overline{\lambda}^{\overline{\mathfrak{w}}}(\overline{r}')^{-1}$, and $ \supp [\overline{\gamma}^{\overline{\mathfrak{w}}}(\overline{r})](\overline{e}^{\overline{\mathfrak{w}}}_j)=\overline{\mathfrak{w}}^{-1}\overline{\Gamma}_{\theta} \overline{M}_{q_i}\overline{\mathfrak{w}}$. So $[\overline{\gamma}^{\overline{\mathfrak{w}}}(\overline{r})](\overline{e}^{\overline{\mathfrak{w}}}_j)=
    \overline{\lambda}^{\overline{\mathfrak{w}}}(\overline{r}')^{-1}\overline{e}^{\overline{\mathfrak{w}}}_i$, and $[\overline{\gamma}^{\overline{\mathfrak{w}}}(\overline{r}^{-1})](\overline{e}^{\overline{\mathfrak{w}}}_i)=
    \overline{\lambda}^{\overline{\mathfrak{w}}}(\overline{r}^{\overline{\mathfrak{w}}})\overline{e}^{\overline{\mathfrak{w}}}_j$.
\end{itemize}
 \item   
 \begin{itemize}
\item[(c1)] $\overline{M}_{q'_i}\overline{M}_{q'_j}=\overline{s}\overline{M}_{q'_r}$, for some $\overline{s}\in \overline{\Gamma}_{\theta}$ iff $\overline{\mathfrak{w}}^{-1}\overline{M}_{q'_i}\overline{\mathfrak{w}}\overline{\mathfrak{w}}^{-1}\overline{M}_{q'_j}\overline{\mathfrak{w}}=\overline{r}'\overline{\mathfrak{w}}^{-1}\overline{M}_{q'_r}\overline{\mathfrak{w}}$, for the element $\overline{r}'=\overline{\mathfrak{w}}^{-1}\overline{s}\overline{\mathfrak{w}} \in \overline{\mathfrak{w}}^{-1}\overline{\Gamma}_{\theta}\overline{\mathfrak{w}}$.
    \item [(d1)] If the above condition holds, then $[\overline{\gamma}^{\overline{\mathfrak{w}}}(\overline{\mathfrak{w}}^{-1}\overline{M}_{q'_j}\overline{\mathfrak{w}})](\overline{e}^{\overline{\mathfrak{w}}}_r)
        =\overline{\lambda}^{\overline{\mathfrak{w}}}(\overline{r}')^{-1}\overline{e}^{\overline{\mathfrak{w}}}_i$, and  $\overline{\gamma}^{\overline{\mathfrak{w}}}([\overline{\mathfrak{w}}^{-1}\overline{M}_{q'_j}\overline{\mathfrak{w}}^{-1}]^{-1})(\overline{e}^{\overline{\mathfrak{w}}}_i)
        =\overline{\lambda}^{\overline{\mathfrak{w}}}(\overline{r}')\overline{e}^{\overline{\mathfrak{w}}}_r$.
    \end{itemize} 
\end{itemize}

\subsection{Two induced  representations}
In the next two subsections, we let  $G$ be a discrete  group, $H\le G$ is a subgroup of a finite index $m$, and we fix once and for all a coset decomposition $G=\bigsqcup_{i=1}^{m}x_{i}H$.
Let $(\sigma,W)$ be a representation of the subgroup $H$ of finite dimension. Following \cite{KaTa}, there exist two induction from $H$ to $G$.
Write $B=\mathbb{C}[H]$ and $A=\mathbb{C}[G]$.

\begin{itemize}
\item \emph{Hom-induction}
\[
\Ind_{H}^{G}W
=\{f:G\to W\mid f(hg)=\sigma(h)f(g),\;\forall h\in H,\,g\in G\}
\]
with $G$-action $(g\cdot f)(x)=f(xg)$.

\item \emph{Tensor-induction}
\[
\ind_{H}^{G}W=A\otimes_{B}W,
\]
where $A$ is viewed as a right $B$-module and $G$ acts by left multiplication on the first factor.
\end{itemize}

\begin{lemma}\label{EQInd}
There is an $A$-module isomorphism
\[
\mathcal{F}:\ind_{H}^{G}W
\longrightarrow\Ind_{H}^{G}W,
\qquad
f\otimes w\longmapsto F_{f\otimes w}
\]
with
\[
F_{f\otimes w}(g)
=\sum_{h\in H}f(g^{-1}h)\,hw.
\]
\end{lemma}
\begin{proof}
$\mathcal{F}$ is well-defined: for $b\in H$, $[fb](g)=f(gb^{-1})$.
\begin{align*}
F_{fb\otimes w}(g)
&=\sum_{h\in H}[fb](g^{-1}h)\,hw
=\sum_{h\in H}f(g^{-1}hb^{-1})\,hw
=\sum_{h_1\in H}f(g^{-1}h_1)\,h_1bw
=F_{f\otimes bw}(g).\\
F_{f\otimes w}(bg)
&=\sum_{h\in H}f(g^{-1}b^{-1}h)\,hw
=\sum_{h_1\in H}f(g^{-1}h_1)\,bh_1w
=\sigma(b) F_{f\otimes w}(g).
\end{align*}
$G$-equivariance:
\[
F_{g_{0}(f\otimes w)}(g)
=F_{g_{0}f\otimes w}(g)
=\sum_{h\in H}f(g_{0}^{-1}g^{-1}h)\,hw
=F_{f\otimes w}(g g_{0})
=(g_{0}\cdot F_{f\otimes w})(g).
\]
Finally, $\mathcal{F}$ is bijective: an inverse is given by
\[
\mathcal{G}: \Ind_{H}^{G}W\to A\otimes_{B}W,\qquad
\varphi\longmapsto\sum_{x\in H\backslash G}x\otimes\varphi(x^{-1}).
\]
$G$-equivariance:  $x_i^{-1}g=h_{g,x_i} x_j$ and $g x_j^{-1}=x_ih_{g,x_i}$
\[
\sum_{i=1}^mx_i\otimes [g\varphi](x_i^{-1})
=\sum_{i=1}^mx_i\otimes \varphi(x_i^{-1}g)
=\sum_{i=1}^mx_i\otimes \varphi(h_{g,x_i} x_j)
=\sum_{i=1}^mx_ih_{g,x_i}\otimes \varphi( x_j)
=g\sum_{j=1}^mx_j^{-1}\otimes \varphi( x_j).
\qedhere\]
\end{proof}
Write:
\[G=\sqcup_{i=1}^m  x_i H=\sqcup_{i=1}^m   H x_i^{-1}.\]
Let $\{ f_j \mid j\in J\}$ be a basis of $W$.  We choose  the basis of $\ind_{H}^{G}W$ as $\{ e'_{ij}= x_i \otimes f_j \mid 1\leq i\leq m, j\in J\}$ and the basis of $\Ind_{H}^{G}W$ as $\{ e_{ij }\mid 1\leq i\leq m, j\in J,  \supp  e_{ij} \subseteq   H x_i^{-1} \textrm{ and }  e_{ ij} (x_i^{-1})=f_j\}$. It can be checked that $\mathcal{F}(e_{ij}')=e_{ij}$ and $\mathcal{G}(e_{ij})=e_{ij}'$.
\subsection{ Tensor induction}\label{tensorind}
Let us first recall some results on the tensor induction  and the transfer from  \cite[\wasyparagraph 13A]{CuRe} and \cite[Chap. 7]{Ser}. For  $g\in G$, write 
\begin{equation}\label{eqgx}
gx_{i}=x_{\tau_{g}(i)}h_{g,x_{i}},\qquad \tau_{g}\in S_{m}, h_{g,x_{i}}\in H.
\end{equation}  The \emph{transfer} is the map
\[\begin{array}{lccl}
V_{G\to H}:& G &\longrightarrow             & H^{ab}=\tfrac{H}{D(H)};\\ 
           & g &\longmapsto                 & (\prod_{i=1}^m h_{g,x_i}) \bmod D(H).
\end{array}\]
\begin{theorem}
The transfer map $V_{G\to H}$ is a group homomorphism and is independent of the choice of coset representatives  $\{ x_i\}$.
\end{theorem}
\begin{proof}
See \cite[Theorem 7.1]{Ser}.
\end{proof}
Let  $H \wr S_m= H^m \rtimes S_m$. Formula \eqref{eqgx} yields a map:
\begin{equation}\label{eq:psi}
\psi\colon G\longrightarrow H\wr S_{m},\qquad
g\longmapsto \bigl([h_{g,x_{1}},\dots,h_{g,x_{m}}],\tau_{g}\bigr).
\end{equation}
\begin{lemma}
$\psi$ is a group monomorphism. If another set of representatives $x_i'= x_i h_i$ is chosen. The corresponding map $\psi'$ satisfies
$\psi'(g)= s^{-1} \psi(g) s$, with $s=(h_1, \cdots, h_m)$.
\end{lemma}
\begin{proof}
See \cite[p.332, Lemma 13.3]{CuRe}.
\end{proof}
   Extend $\sigma^{\otimes m}$ to $H\wr S_{m}$ by letting $\tau\in S_{m}$ act as
\[
\tau(w_{1}\otimes\dots\otimes w_{m})
=w_{\tau^{-1}(1)}\otimes\dots\otimes w_{\tau^{-1}(m)}.
\]
The \emph{tensor induction} $\otimes\!\operatorname{ind}_{H}^{G}\sigma$ is the representation of $G$ obtained by restricting this extension via $\psi$.  The lemma above shows that the equivalence class of $\otimes\!\operatorname{ind}_{H}^{G}\sigma$ is well defined. 
\begin{proposition}\label{otimescorr}
If $\sigma$ is a character(i.e.\ $\dim W=1$), then $\otimes\ind_{H}^G \sigma$ is again a character and satisfies  $\otimes\ind_{H}^G \sigma=\sigma \circ V_{G\to H}$. 
\end{proposition}
\begin{proof}
See \cite[Pro. 13.12]{CuRe}.
\end{proof} 

Following \cite[\wasyparagraph 13A]{CuRe}, we now give an explicit realization of the tensor induced representation  $\otimes\!\operatorname{ind}_{H}^{G}\sigma$. Denote $\pi=\otimes\ind_{H}^G \sigma$. Let  $V= (x_1\otimes W) \otimes  \cdots   \otimes  (x_m \otimes W)$. For any $g\in G$, the generator $v=(x_1\otimes w_1) \otimes  \cdots   \otimes  (x_m \otimes w_m)\in V$, 
 \begin{align*}
  \pi(g)v&=  (gx_1\otimes w_1)\otimes \cdots \otimes (g x_m \otimes w_m)\\
  &=(x_{\tau_g{(1)}}\otimes \sigma( h_{g,x_1})w_1)  \otimes\cdots \otimes (x_{\tau_g{(m)}} \otimes\sigma( h_{g,x_m})w_m)\\
 &\stackrel{ \textrm{ Def.}}{=}x_{1} \otimes \sigma( h_{g,x_{\tau_g^{-1}(1)}} )w_{\tau_g^{-1}(1)} \otimes  \cdots \otimes x_{m} \otimes \sigma( h_{g,x_{\tau_g^{-1}(m)}})w_{\tau_g^{-1}(m)}.
 \end{align*}

  Following  Lemma \ref{EQInd},  we give another realization of this  representation. For each coset representative $x_i$ ($i=1,\dots,m$) and $w\in W$ set
\[
F_{x_i\otimes w}\coloneqq\mathcal{F}(x_i\otimes w)\in\Ind_{H}^{G}W,
\]
the unique function supported on the coset $Hx_i^{-1}$ and satisfying
\[
F_{x_i\otimes w}(x_i^{-1})=w.
\]
 Let $V'=\{ \sum c_{w_1\otimes \cdots\otimes w_m}F_{x_1\otimes w_1} \otimes \cdots \otimes F_{x_m\otimes w_m}  \}.$ For $g\in G$,  write $gx_i=x_{\tau_g(i)}h_{g,x_i}$ ($h_{g,x_i}\in H$).
Then
\[
x_i^{-1}g^{-1}=h_{g,x_i}^{-1}x_{\tau_g(i)}^{-1}.
\]
Hence the left translate ${}^{g}F_{x_i\otimes w}\coloneqq g\cdot F_{x_i\otimes w}$ is supported on
\[
\supp{}^{g}F_{x_i\otimes w}\subseteq Hx_i^{-1}g^{-1}=Hx_{\tau_g(i)}^{-1},
\]
and its value at the representative $x_{\tau_g(i)}^{-1}$ is
\[
{}^{g}F_{x_i\otimes w}(x_{\tau_g(i)}^{-1})
=F_{x_i\otimes w}(x_{\tau_g(i)}^{-1}g)
=F_{x_i\otimes w}(h_{g,x_i}x_i^{-1})
=\sigma(h_{g,x_i})F_{x_i\otimes w}(x_i^{-1})
=\sigma(h_{g,x_i})w.
\]
We define the action of $G$ on $V'$ as follows:
\begin{align*}
   &\pi(g)F_{x_1\otimes w_1} \otimes \cdots \otimes F_{x_m\otimes w_m} \\
   &= F_{x_{\tau_g{(1)}}\otimes \sigma( h_{g,x_1})w_1}  \otimes\cdots \otimes F_{x_{\tau_g{(m)}} \otimes\sigma( h_{g,x_m})w_m}\\
   &\stackrel{ \textrm{ Def.}}{=}F_{x_{1} \otimes \sigma( h_{g,x_{\tau_g^{-1}(1)}} )w_{\tau_g^{-1}(1)}} \otimes  \cdots \otimes F_{x_{m} \otimes \sigma( h_{g,x_{\tau_g^{-1}(m)}})w_{\tau_g^{-1}(m)}}.
   \end{align*}

\subsection{Example }\label{exampI}
Let the above $G=\overline{\mathfrak{w}}^{-1}\overline{\SL}_2(\Z)\overline{\mathfrak{w}}$,  $H=\overline{\mathfrak{w}}^{-1}\overline{\Gamma}_{\theta}\overline{\mathfrak{w}}$. 
 Let us define:
$$ \sigma= [\overline{\lambda}^{\overline{\mathfrak{w}}}]^{-1}, \qquad \delta=\otimes\ind_{H}^G \sigma.$$
By Prop. \ref{otimescorr},  $\delta$ is a character of $G$. Put
\[ x_1=\overline{\mathfrak{w}}^{-1}\overline{M}_{q_1}^{-1}\overline{\mathfrak{w}}, \quad x_2=\overline{\mathfrak{w}}^{-1}\overline{M}_{q_2}^{-1}\overline{\mathfrak{w}}, \quad x_3=\overline{\mathfrak{w}}^{-1}\overline{M}_{q_3}^{-1}\overline{\mathfrak{w}}.\]
Recall the notations $\{\overline{e}_1^{\overline{\mathfrak{w}}},\overline{e}^{\overline{\mathfrak{w}}}_2, \overline{e}^{\overline{\mathfrak{w}}}_3 \} $ from Section \ref{Twistedby}. Then  $\overline{e}^{\overline{\mathfrak{w}}}_i=F_{x_i \otimes 1}$. 
\begin{align*}
   &\delta(g)F_{x_1\otimes 1} \otimes F_{x_2\otimes 1} \otimes F_{x_3\otimes 1} \\
   &=F_{x_{1} \otimes \sigma( h_{g,x_{\tau_g^{-1}(1)}} )1} \otimes  F_{x_{2} \otimes \sigma( h_{g,x_{\tau_g^{-1}(2)}} )1}  \otimes F_{x_{m} \otimes \sigma( h_{g,x_{\tau_g^{-1}(m)}})1}\\
   &=F_{x_1\otimes 1} \otimes F_{x_2\otimes 1} \otimes F_{x_3\otimes 1} \sigma(V_{G\to H}(g)).\\
& \\
&\delta(g)F_{x_1\otimes 1} \otimes F_{x_2\otimes 1} \otimes F_{x_3\otimes 1} \\
   &=\overline{\gamma}^{\overline{\mathfrak{w}}} (g)F_{x_1\otimes 1} \otimes \overline{\gamma}^{\overline{\mathfrak{w}}}  (g) F_{x_2\otimes 1} \otimes \overline{\gamma}^{\overline{\mathfrak{w}}}  (g) F_{x_3\otimes 1}\\
   &=F_{x_1\otimes 1} \otimes F_{x_2\otimes 1} \otimes F_{x_3\otimes 1} \sigma(V_{G\to H}(g)).
   \end{align*}
\section{Classical Weil representations}\label{Weilmodel}

The center of $\Ha(W)$ is precisely  $\R$.  Let $\psi= \psi_{0}^{\mathfrak{e}}$.   According to the Stone-von Neumann's theorem, there exists a unique unitary irreducible complex representation with central character  $\psi$, up to unitary equivalence.  Let us call it the Heisenberg representation, and denote it  by $\pi_{\psi}$. The  group $\SL_2(\R)$ can then act on $\Ha(W) $ and leave the center $\R$ pointwise unchanged. Then the Heisenberg representation yields a projective representation of  $\SL_2(\R)$, leading to an actual   representation of a $\C^{\times}$-covering group over $\SL_2(\R)$.  This $\C^{\times}$-covering  group is commonly called  the Metaplectic group, and the true representation is known as the Weil representation.  This central covering  can be reduced to an $8$-fold or $2$-fold covering over $\SL_2(\R)$, as shown in \cite{Pe,Ra,We} and others.  Our main references are Rao~\cite{Ra}, Kudla~\cite{Ku}, and Lion--Vergne~\cite{LiVe}.  
\subsection{Schr\"odinger model}\label{chap2Sch}
 Note that $X^{\ast}\times \R$ is a closed subgroup of $\Ha(W)$. Let $\psi_{X^{\ast}}$ denote the one--dimensional unitary representation of $X^{\ast}\times\R$ obtained by extending $\psi$ trivially along $X^{\ast}$.  Define
$$ \pi_{\psi}=\Ind_{X^{\ast}\times \R}^{\Ha(W)}\psi_{X^{\ast}}, V_{\psi}=\Ind_{X^{\ast} \times \R}^{\Ha(W)} \C.$$
 Then $\pi_{\psi}$ is the Heisenberg representation of   $\Ha(W)$ with central character $\psi$. By Weil's theorem it extends uniquely to a unitary representation of $\widetilde{\SL}_{2}(\R)\ltimes\Ha(W)$, provided the central subgroup $\mu_{8}$ of $\widetilde{\SL}_{2}(\R)$ acts trivially; uniqueness follows because $\SL_{2}(\R)$ is perfect.

With respect to the self-dual Haar measure \(d\mu(y)\) on \(\R\) determined by \(\psi\), the Heisenberg representation \(\pi_{\psi}\) acts on \(L^{2}(\R)\) by the following formulas, valid for all \(f\in\mathcal S(\R)\):
\begin{equation}\label{chap2representationsp11}
\pi_{X^{\ast}, \psi}([x,0]\cdot [x^{\ast},0]\cdot [0,k])f(y)=\psi(k+\langle x+ y,x^{\ast}\rangle) f(x+y),
\end{equation}
\begin{equation}\label{chap2representationsp2}
\pi_{X^{\ast}, \psi}( u(b))f(y)=\psi(\tfrac{1}{2}\langle y,yb\rangle) f(y),
\end{equation}
\begin{equation}\label{chap2representationsp3}
\pi_{X^{\ast}, \psi}( h(a))f(y)=|\det(a)|^{1/2} f(ya),
\end{equation}
\begin{equation}\label{chap2representationsp4}
\pi_{X^{\ast}, \psi}([\omega,t])f(y)=\int_{X^{\ast}} t\psi(\langle y, y^{\ast}\rangle) f(y^{\ast}\omega^{-1}) d\mu(y^{\ast}),
\end{equation}
where
\[
u(b)=\begin{pmatrix}1&b\\0&1\end{pmatrix},\quad
h(a)=\begin{pmatrix}a&0\\0&a^{-1}\end{pmatrix},\quad
\omega=\begin{pmatrix}0&-1\\1&0\end{pmatrix}\in\SL_{2}(\R).
\]
Moreover, for all $g_{1},g_{2}\in\SL_{2}(\R)$,
\[
\pi_{\psi}(g_{1})\pi_{\psi}(g_{2})
=\widetilde{c}_{X^{\ast}}(g_{1},g_{2})\,\pi_{\psi}(g_{1}g_{2}),
\]
where $\widetilde{c}_{X^{\ast}}(\,\cdot\,,\,\cdot\,)$ is Perrin--Rao's $2$-cocycle attached to $(\psi,X^{\ast})$; see \cite{Pe, Ra}.
\subsection{The Lattice model}\label{chap2latticemodel}
Let $L\subset W$ be a lattice (cf.\ \cite[p.\,138]{LiVe}).  Define its orthogonal complement with respect to $\psi$ by
\[
L^{\perp}= \bigl\{ w\in W\bigm| \psi\!\bigl(\langle w,l\rangle\bigr)=1
\text{ for all }l\in L\bigr\}.
\]
A lattice $L$ is called \emph{self-dual} (with respect to $\psi$) if $L=L^{\perp}$.  
Write $\Ha(L)=L\times\R$ for the corresponding subgroup of $\Ha(W)$ and set
\[
\pi'_{\psi}=\Ind_{\Ha(L)}^{\Ha(W)}\psi_L,
\]
where $\psi_L$ denotes the extension of $\psi$ to $\Ha(L)$ .  
Then $\pi'_{\psi}$ is a Heisenberg representation of $\Ha(W)$ with central character $\psi$.

\begin{example}\label{chap2laex}
Take $\psi=\psi_0$ and $L=\Z e\oplus\Z e^{\ast}$.  
Let $\mathcal{H}_{\psi}(L)$ be the space of measurable functions $f\colon W\to\C$ satisfying
\begin{enumerate}
\item[(i)] $f(l+w)=\psi\!\bigl(-\tfrac12\langle x_l,x_l^{\ast}\rangle
            -\tfrac12\langle l,w\rangle\bigr)f(w)$
            for all $l=x_l+x_l^{\ast}\in L=(L\cap X)\oplus(L\cap X^{\ast})$
            and almost all $w\in W$;
\item[(ii)] $\displaystyle\int_{L\backslash W}|f(w)|^{2}\,dw<+\infty$,
\end{enumerate}
with the  natural $W$-invariant measure on $L\backslash W$ (from the self-dual  measure on  $W$ with respect to $\psi$).  
Then $\pi'_{\psi}$ is realized on $\mathcal{H}_{\psi}(L)$ by
\[
\pi'_{\psi}\bigl[(w',t)\bigr]f(w)
=\psi\!\bigl(t+\tfrac12\langle w,w'\rangle\bigr)f(w+w'),
\qquad w,w'\in W,\; t\in\R.
\]
\end{example}
\subsection{The Fock model }\label{FOCKMODEL1}
In this section we fix $\psi=\psi_0$. Let $V=\C$ and equip it with the Hermitian form as follows:
$$(z, z')_V= \overline{z}z', \quad\quad z,z'\in \C.$$
Let $\Ha(V)=\C\oplus \R$ be the group of elements $(v, t)$,  with the
multiplication law given by
$$(z, t)(z',t')=(v+v', t+t'+ \tfrac{\Im(z,z')_V}{2}),   \quad\quad z, z'\in \C, t,t'\in \R.$$
Then there exists a group isomorphism:
\begin{align}
\phi\colon\Ha(V)\longrightarrow\Ha(W),\qquad
(a+ib,t)\mapsto(ae+be^\ast,t).
 \end{align}
 We now recall the Fock model of the Heisenberg representation.
\begin{example}[Fock model]\label{ex3}
Let $\mathcal H_F$ be the Fock space of entire functions $f\colon\C\to\C$ satisfying
\[
\|f\|^2=\int_\C|f(z)|^2e^{-\pi|z|^2}\,d(z)<+\infty.
\]
The Heisenberg representation of $\Ha(V)$ is realized on $\mathcal H_F$ by
\[
\pi_F[(z',t)]f(z)=\psi_0(t)\,e^{-\frac\pi2|z'|^2-\pi z\overline{z'}}f(z+z'),
\qquad z,z'\in\C,\ t\in\R.
\]
\end{example}
\begin{proof}
See \cite[p.\,43, (1.71)]{Fo}.
\end{proof}
Following Satake–Takase \cite{Sa1,Sa2,Ta1,Ta2} we attach the following objects to each \(z=P_{z}+iQ_{z}\in\mathbb H\).

\begin{itemize}
\item  Complex conjugate on \(\mathbb H\):
      \(\displaystyle \widehat z:=-\overline z\).

\item  Measures on \(V\):
      \(\displaystyle d_{z}(v):=Q_{z}^{-1}\,d(v)=d_{\widehat z}(v)\).

\item  Gaussian factor:
      \(\displaystyle q_{z}(v):=e^{-\pi i z^{-1}v^{2}}\).

\item  Kernel functions:
      \[
        \kappa(z',v';z,v)
        :=e^{\pi i(z'-\overline z)^{-1}(v'-\overline v)^{2}},
        \qquad
        \kappa_{z}(v',v)
        :=\kappa(z,v';z,v)
        =e^{\frac{\pi}{2}Q_{z}^{-1}(v'-\overline v)^{2}}.
      \]

\item  2-cocycle:
      \(\displaystyle\alpha_{z}(g_{1},g_{2})
        :=\epsilon\!\bigl(g_{1}^{h_{-1}};z,g_{2}^{h_{-1}}(z)\bigr)\),
      \((g_{i}\in\SL_{2}(\mathbb R))\).

\item  Action of \(\SL_{2}(\mathbb R)\ltimes\Ha(W)\) on \(\mathbb H\times V\):
      \[
        [g,h]\cdot[z,v]
        :=\Bigl[g(z),\,(v+xz+x^{\ast})J(g,z)^{-1}\Bigr],
        \qquad h=(x,x^{\ast};t)\in\Ha(W).
      \]

\item  Automorphy factor
      \(\eta\colon[\SL_{2}(\mathbb R)\ltimes\Ha(W)]\times[\mathbb H\times V]\to\mathbb C^{\times}\):
      \begin{itemize}
        \item[(a)]  For \(h=[x,x^{\ast};t]\in\Ha(W)\):
              \[
                \eta(h;[z,v])
                :=e^{2\pi i\bigl[t+\tfrac12(xz+x^{\ast})x+vx\bigr]}.
              \]

        \item[(b)]  For \(g=\bigl(\begin{smallmatrix}a&b\\c&d\end{smallmatrix}\bigr)\in\SL_{2}(\mathbb R)\):
              \[
                \eta(g;[z,v])
                :=e^{2\pi i\bigl[-\tfrac12(cz+d)^{-1}cv^{2}\bigr]}.
              \]

        \item[(c)]  Mixed case:
              \[
                \eta\bigl([g,h];[z,v]\bigr)
                :=\eta\!\bigl(h^{g^{-1}};[g(z),vJ(g,z)^{-1}]\bigr)\;
                   \eta(g;[z,v]),\qquad
                (g,h)=(1,h^{g^{-1}})(g,0).
              \]
      \end{itemize}
\end{itemize}
\begin{lemma}
For \(\widetilde g_1=[g_1,h_1]\), \(\widetilde g_2=[g_2,h_2]\in\SL_2(\mathbb R)\ltimes\Ha(W)\) and \(\widetilde z=[z,v]\in\mathbb H\times V\) we have
\[
\eta\!\bigl(\widetilde g_1\widetilde g_2,\widetilde z\bigr)
=\eta\!\bigl(\widetilde g_1,\widetilde g_2(\widetilde z)\bigr)\,
  \eta\!\bigl(\widetilde g_2,\widetilde z\bigr).
\]
\end{lemma}
\begin{proof}
See \cite[p.\,395]{Sa2} or \cite[p.\,121]{Ta1}.
\end{proof}

\begin{lemma}
For \(\widetilde g=[g,h]\in\SL_2(\mathbb R)\ltimes\Ha(W)\) and \(\widetilde z=[z,v]\), \(\widetilde z'=[z',v']\in\mathbb H\times V\) we have
\[
\kappa\!\bigl(\widetilde g(\widetilde z);\widetilde g(\widetilde z')\bigr)
=\eta\!\bigl(\widetilde g,\widetilde z\bigr)\,
  \kappa(\widetilde z;\widetilde z')\,
  \overline{\eta\!\bigl(\widetilde g,\widetilde z'\bigr)}.
\]
\end{lemma}
\begin{proof}
See \cite[p.\,396]{Sa2} or \cite[p.\,121]{Ta1}.
\end{proof}

\begin{example}[Generalized Fock model]\label{ex4}
Let \(\mathcal H_{F_{\widehat z}}\) be the space of entire functions \(f\colon\mathbb C\to\mathbb C\) satisfying
\[
\|f\|^2
=\int_{\mathbb C}|f(v)|^{2}\,\kappa_z(v,v)\,d_z(v)
<\infty.
\]
The Heisenberg representation of \(\Ha(W)\) is realized on \(\mathcal H_{F_{\widehat z}}\) by
\[
\pi_{F_{\widehat z}}(h)f(v)
=\psi_0(t)\,
  e^{-2\pi i\!\left[\frac12(x\overline z+x^{*})\,x+vx\right]}
  f\!\bigl(v+x\overline z+x^{*}\bigr),
\]
where \(h=(w,t)\in\Ha(W)\) with \(w=x+x^{*}\in W\) and \(t\in\mathbb R\).
\end{example}
\begin{proof}
See \cite[pp.\,125--126]{Ta1}.
\end{proof}
\begin{remark}\label{uneqi}
Take $z=\widehat{z}=z_0=i$. Following  \cite[p.397]{Sa1}, there exists a unitary equivalence $\mathcal{A}$ from $\mathcal{H}_{F_{z_0}}$ to $\mathcal{H}_F$, defined by
$$\mathcal{A}: \mathcal{H}_{F_{z_0}} \longrightarrow \mathcal{H}_F; f \longmapsto \phi(v)=  e^{-\tfrac{\pi}{2}v^2} f(-vi).$$ Furthermore, $\mathcal{A}$ is an $\Ha(W)$-intertwining  operator. 
\end{remark}
\begin{proof}
1) $v=v_x+i v_y$, $w=-vi=v_y-v_x i$, $w_x=v_y,w_y=-v_x$.
\begin{align*}
 &\int_{\C} \mid \phi(v)\mid^2 e^{-\pi | v|^2} d(v)\\
 &=\int_{\C} \mid e^{-\tfrac{\pi}{2} v^2} f(-vi)\mid^2 e^{-\pi | v|^2} d(v) \\
 &=\int_{\C}  e^{-2\pi  v_x^2} \mid f(-vi)\mid^2 d(v)\\
  &\stackrel{w=-vi }{=}\int_{\C}  e^{-2\pi w_y^2} \mid f(w)\mid^2 d(w)<+\infty.
 \end{align*}
2) 
\begin{align*}
&\mathcal{A}(\pi_{ F_{z_0}, \psi}(h)f)(v)= e^{-\tfrac{\pi}{2} v^2} [\pi_{F_{z_0}, \psi}(h)f](-vi)\\
&=e^{-\tfrac{\pi}{2}v^2}\psi_0(t)e^{-2\pi i[\frac{1}{2}(-xi+x^{*})x-vix]}f(-vi-xi+x^{\ast});
\end{align*}
\begin{align*}
&\pi_{F, \psi}(h)\mathcal{A}(f)(v)=\psi_0(t)e^{-\tfrac{\pi}{2}(x^2+x^{\ast 2})-\pi v(x-x^{\ast} i)}\mathcal{A}(f)(v+x+x^{\ast}i)\\
&=e^{-\tfrac{\pi}{2}(v+x+x^{\ast}i)^2}\psi_0(t)e^{-\tfrac{\pi}{2}(x^2+x^{\ast 2})-\pi v(x-x^{\ast} i)}f(-vi-xi+x^{\ast})\\
&=e^{-\tfrac{\pi}{2} v^2}e^{-2\pi vx}\psi_0(t)e^{-\pi(x^2+ix^{\ast}x)}f(-vi-xi+x^{\ast})\\
&=\mathcal{A}(\pi_{F_{z_0}, \psi_0}(h)f)(v).
\end{align*}
\end{proof}
\begin{lemma}
For $\widehat{z},\widehat{z'}\in \mathbb{H}$, there exists a unitary equivalence operator $U_{\widehat{z},\widehat{z'}}$ mapping from $\mathcal{H}_{F_{\widehat{z}}}$ to $\mathcal{H}_{F_{\widehat{z'}}}$, defined by
$$U_{\widehat{z'},\widehat{z}}: \mathcal{H}_{F_{\widehat{z}}} \longrightarrow \mathcal{H}_{F_{\widehat{z'}}}; \phi \longmapsto f(v')= \gamma(\widehat{z'}, \widehat{z})\int_{\C} \kappa(\widehat{z'},v'; \widehat{z},v)^{-1}\phi(v) \kappa_{\widehat{z}}(v,v) d_{\widehat{z}}(v).$$ Furthermore, $U_{\widehat{z'},\widehat{z}}$ is an $\Ha(W)$-intertwining  operator. 
\end{lemma}
\begin{proof}
See \cite[p.125]{Ta1} or \cite[Lemma 3]{Sa1}.
\end{proof}
\begin{lemma}
For $z,z',z''\in \mathbb{H}$,  $U_{z,z'}\circ U_{z',z''}=  U_{z,z''}$.
\end{lemma}
\begin{proof}
See \cite[pp.399-400]{Sa1}.
\end{proof}
\begin{lemma}
For every \(g=\bigl(\begin{smallmatrix}a&b\\c&d\end{smallmatrix}\bigr)\in\SL_{2}(\mathbb R)\) the map
\begin{align*}
T^{\widehat z}_{g^{h_{-1}}}\colon\mathcal H_{F_{\widehat z}}
&\longrightarrow\mathcal H_{F_{g^{h_{-1}}(\widehat z)}},\\
\phi&\longmapsto\eta\!\bigl((g^{h_{-1}})^{-1},g^{h_{-1}}(\widehat z),v\bigr)\,
\phi\!\bigl(vJ(g^{h_{-1}},\widehat z)\bigr)
=e^{-\pi i(c\overline z+d)cv^{2}}\phi\!\bigl(v(c\overline z+d)\bigr)
\end{align*}
is unitary.  Moreover, for every \(h=(x,x^{\ast};t)\in\Ha(W)\) the diagram
\[
\begin{CD}
\mathcal H_{F_{\widehat z}}
@>T^{\widehat z}_{g^{h_{-1}}}>>
\mathcal H_{F_{g^{h_{-1}}(\widehat z)}}\\
@V\pi_{F_{\widehat z}}(h)VV
@VV\pi_{F_{g^{h_{-1}}(\widehat z)}}(h^{g^{-1}})V\\
\mathcal H_{F_{\widehat z}}
@>>T^{\widehat z}_{g^{h_{-1}}}>
\mathcal H_{F_{g^{h_{-1}}(\widehat z)}}
\end{CD}
\]
commutes.
\end{lemma}
\begin{proof}
1) \[g^{-1}= \begin{pmatrix}
d & -b \\
-c & a
\end{pmatrix}, \quad g^{h_{-1}}= \begin{pmatrix} a& -b\\-c& d\end{pmatrix}, \quad (g^{h_{-1}})^{-1}= \begin{pmatrix} d& b\\c& a\end{pmatrix},\]
\[ g^{h_{-1}}(\widehat{z})=(-a\overline{z} -b)(c\overline{z}+d)^{-1},  J(g^{h_{-1}}, \widehat{z})=c\overline{z}+d;\]
 \begin{align*}
\eta((g^{h_{-1}})^{-1}; g^{h_{-1}}(\widehat{z}), v)
 &=e^{2\pi i\Big[\tfrac{1}{2}\langle vg^{h_{-1}}, vJ((g^{h_{-1}})^{-1},g^{h_{-1}}(\widehat{z}))^{-1}\rangle\Big]}\\
&= e^{2\pi i\Big[\tfrac{1}{2} \langle v(-c), vJ(g^{h_{-1}}, \widehat{z})\rangle\Big]}\\
&=e^{2\pi i\Big[-\tfrac{1}{2} v(\overline{z}+d) cv^2\Big]}.
\end{align*}
2) Write  $z_1=\widehat{z}$,  $z_2=g^{h_{-1}}(\widehat{z})$, $h=(x,x^{\ast};t)$,  $g_1=g^{h_{-1}}$, $h_1=h^{h_{-1}}$, $h^{g^{-1}}=(x_1,x_1^{\ast};t) $, $w_1=(x_1,x_1^{\ast})$. Then:
 \[h^{g^{-1}}=h^{h_{-1}g_1^{-1}h_{-1}}=h_1^{g_1^{-1}h_{-1}};\]
 \[ (h_1^{g_1^{-1}})^{-1}=[h^{g^{-1}h_{-1}}]^{-1}=[x_1,-x_1^{\ast}; -t]^{-1}=[-x_1,x_1^{\ast};t];\]
 \begin{align*}
 &\pi_{F_{z_2}}(h^{g^{-1}}) \circ T^{z_1}_{g_1} f(v)\\
 &=\eta((h_1^{g_1^{-1}})^{-1}; g_1(z_1), v)  [T^{z_1}_{g_1} f](v+w_1 \begin{pmatrix} -g_1(z_1)\\ 1\end{pmatrix}) \\
  &=\eta\big((h_1^{g_1^{-1}})^{-1}; g_1(z_1), v\big) \eta\big(g_1^{-1}, g_1(z_1), v+w_1 \begin{pmatrix} -g_1(z_1)\\ 1\end{pmatrix}\big) f\bigg([v+w_1 \begin{pmatrix} -g_1(z_1)\\ 1\end{pmatrix}]J(g_1,z_1)\bigg) ;
 & \\
 &T^{z_1}_{g_1}\circ \pi_{F_{z_1}}(h) f(v)\\
 &=\eta(g_1^{-1}; g_1(z_1), v) [\pi_{F_{z_1}}(h) f](vJ(g_1,z_1))\\
 &=\eta(g_1^{-1}; g_1(z_1), v) \eta((h^{h_{-1}})^{-1}; z_1, vJ(g_1,z_1)) f(vJ(g_1,z_1)+w \begin{pmatrix} -z_1\\ 1\end{pmatrix});
  \end{align*}
 i)  \begin{align*}
 &(w_1 \begin{pmatrix} -g_1(z_1)\\ 1\end{pmatrix})J(g_1,z_1)\\
 &=w g^{-1} \begin{pmatrix} -g_1(z_1)\\ 1\end{pmatrix}J(g_1,z_1)\\
 &=wh_{-1} g_1^{-1} \begin{pmatrix} -g_1(z_1)\\ -1\end{pmatrix}J(g_1,z_1)\\
 &=w\begin{pmatrix} -z_1\\ 1\end{pmatrix};
  \end{align*}
 ii) 
  \begin{align*}
  &\eta([g_1, h^{h_{-1}}]^{-1}; g_1(z_1), v)\\
  &=\eta\big([1,  h^{h_{-1}}]^{-1} [g_1,  0]^{-1} ; g_1(z_1), v\big)\\
  &=\eta\Big((h^{h_{-1}})^{-1}; z_1, vJ\big(g_1^{-1}, g_1(z_1)\big)^{-1}\Big)\eta\big(g_1^{-1}; g_1(z_1), v\big)\\
  &=\eta\big((h^{h_{-1}})^{-1}; z_1, vJ(g_1,z_1)\big)\eta\big(g_1^{-1}; g_1(z_1), v\big);
   \end{align*}
  \begin{align*}
  &\eta([g_1, h^{h_{-1}}]^{-1}; g_1(z_1), v)\\
  &=\eta\big( [g_1,  0]^{-1} [1,  h_1^{g_1^{-1}}]^{-1} ; g_1(z_1), v\big)\\
  &=\eta\big(g_1^{-1}; g_1(z_1), v+(-x_1,x_1^{\ast}) \cdot \begin{pmatrix} g_1(z_1)\\ 1\end{pmatrix}\big) \eta\big(( h_1^{g_1^{-1}})^{-1}; g_1(z_1), v\big).
   \end{align*}
\end{proof}
\begin{lemma}
\begin{itemize}
\item[(1)] $T^{g_2^{h_{-1}}(\widehat{z})}_{g_1^{h_{-1}}} \circ T^{\widehat{z}}_{g_2^{h_{-1}}}=T^{\widehat{z}}_{g_1^{h_{-1}}g_2^{h_{-1}}}$.
\item[(2)] For $g\in \SL_{2}(\R)$, $U_{ g^{h_{-1}}(\widehat{z'}), g^{h_{-1}}(\widehat{z})} \circ T^{\widehat{z}}_{g^{h_{-1}}}=\epsilon(g^{h_{-1}};\widehat{z'},\widehat{z}) T^{\widehat{z'}}_{g^{h_{-1}}} \circ U_{\widehat{z'},\widehat{z}}$.
\end{itemize}
\end{lemma}
\begin{proof}
See \cite[p.125]{Ta1}.
\end{proof}
The Heisenberg representation \(\pi_{F_{\widehat z}}\) of \(\Ha(W)\) from Example~\ref{ex4} extends to a projective representation of \(\SL_2(\R)\ltimes\Ha(W)\) on \(\mathcal H_{F_{\widehat z}}\) by
\begin{equation}\label{hwMpz}
\left\{ \begin{array}{l}
 \pi_{F_{\widehat{z}}}(h)f(v)=\eta((h^{h_{-1}})^{-1}, \widehat{z}, v) f(v+x\overline{z}+x^{\ast}),\\
 \pi_{F_{\widehat{z}}}(g) f(v)=U_{\widehat{z}, g^{h_{-1}}(\widehat{z})} \circ T^{\widehat{z}}_{g^{h_{-1}}} f(v),
\end{array}\right.
\end{equation}
for \(g\in\SL_2(\R)\) and \(h=(x,x^{\ast};t)\in\Ha(W)\).
\begin{lemma}\label{paiFz}
\begin{itemize}
\item[(1)] The formulas \eqref{hwMpz} are well-defined.
\item[(2)] For \(g_1,g_2\in\SL_2(\R)\),
$
   \pi_{F_{\widehat z}}(g_1)\pi_{F_{\widehat z}}(g_2)
   =\alpha_{\widehat z}(g_1,g_2)^{-1}\,\pi_{F_{\widehat z}}(g_1g_2).
$
\end{itemize}
\end{lemma}
\begin{proof}
\[[g,h]=[g,0]\cdot [1,h]=[1,h^{g^{-1}}]\cdot[g,0].\]
\begin{align*}
&\pi_{F_{\widehat{z}}}(g)  \pi_{F_{\widehat{z}},  \psi}(h)\\
&=U_{\widehat{z}, g^{h_{-1}}(\widehat{z})} \circ T^{\widehat{z}}_{g^{h_{-1}}}\circ \pi_{ F_{\widehat{z}}, \psi}(h)\\
&\xlongequal{\widehat{z}_1=g^{h_{-1}}(\widehat{z})}U_{\widehat{z}, \widehat{z}_1} \circ \pi_{F_{\widehat{z}_1}}(h^{g^{-1}})\circ T^{\widehat{z}}_{g^{h_{-1}}}\\
&= \pi_{F_{\widehat{z}}}(h^{g^{-1}})\circ  U_{\widehat{z}, \widehat{z}_1} \circ T^{\widehat{z}}_{g^{h_{-1}}}\\  
&=\pi_{F_{\widehat{z}}}(h^{g^{-1}}) \pi_{F_{\widehat{z}}}(g).
  \end{align*}
  By the above lemma, 
  \[U_{ g_1^{h_{-1}}(\widehat{z}), g_1^{h_{-1}}g_2^{h_{-1}}(\widehat{z})} \circ T^{g_2^{h_{-1}}(\widehat{z})}_{g_1^{h_{-1}}}=\epsilon(g_1^{h_{-1}};\widehat{z},g_2^{h_{-1}}(\widehat{z})) T^{\widehat{z}}_{g_1^{h_{-1}}} \circ U_{\widehat{z},g_2^{h_{-1}}(\widehat{z})};\]
  \begin{align*}
&\pi_{F_{\widehat{z}}}(g_1)  \pi_{F_{\widehat{z}}, \psi}(g_2)\\
&=U_{\widehat{z}, g_1^{h_{-1}}(\widehat{z})} \circ T^{\widehat{z}}_{g_1^{h_{-1}}}\circ U_{\widehat{z}, g_2^{h_{-1}}(\widehat{z})} \circ T^{\widehat{z}}_{g_2^{h_{-1}}}\\
&=\epsilon(g_1^{h_{-1}};\widehat{z},g_2^{h_{-1}}(\widehat{z}))^{-1}U_{\widehat{z}, g_1^{h_{-1}}(\widehat{z})} \circ U_{ g_1^{h_{-1}}(\widehat{z}), g_1^{h_{-1}}g_2^{h_{-1}}(\widehat{z})} \circ T^{g_2^{h_{-1}}(\widehat{z})}_{g_1^{h_{-1}}}  \circ T^{\widehat{z}}_{g_2^{h_{-1}}}\\
&=\alpha_{\widehat{z}}(g_1,g_2)^{-1} \pi_{F_{\widehat{z}}}(g_1g_2). 
  \end{align*}
\end{proof}
\begin{remark}\label{equalityonab}
  When $z=\widehat z=z_0$, the identities  
  \[
    \alpha_{z_0}(g_1,g_2)^{-1}
    =\widehat{c}_{z_0}(g_1^{h_{-1}},g_2^{h_{-1}})^{-1}
    =\widehat{c}_{z_0}(g_1,g_2)
  \]
  hold, and the extended formulas \eqref{hwMpz}  give  the representation
  $\pi_{F_{z_0}}$ of $\widehat{\SL}_{2}(\mathbb R)\ltimes\Ha(W)$.
\end{remark}
\begin{example}\label{exm}
Take $z=z_0=i$. Let   $g=\begin{pmatrix}a & -b\\ b& a\end{pmatrix}\in \U(\C)$. Then $g^{h_{-1}}=\begin{pmatrix}a & b\\ -b& a\end{pmatrix}$. 
Additionally, the representation $\pi_{F_{z_0}}(g)$ acting on a function $f(v)$ is defined as:
\begin{equation}\label{unitact}
 \pi_{F_{z_0}}(g) f(v)=e^{-\pi i (-bi+a) bv^2}f(v(-bi+a)).
 \end{equation}
\begin{itemize}
\item[(1)] If we take $f(v)=e^{-\tfrac{\pi}{2}v^2}$, then $f(v)\in \mathcal{H}_{F_{z_0}}$. Moreover, $\pi_{F_{z_0}}(g) f(v)=f(v)$.
\begin{align*}
&\int_{\C}  e^{-2\pi  v_y^2} \mid f(v)\mid^2 d_{z_0}(v)\\
&=\int_{\C}  e^{-2\pi  v_y^2} e^{-\pi ( v_x^2-v_y^2)} d_{z_0}(v)\\
&=\int_{\R^2 }  e^{-\pi (v_x^2+v_y^2)} dv_xdv_y<+\infty.
\end{align*}
\begin{align*}
\pi_{F_{z_0}}(g) f(v)&=e^{-\pi i (-bi+a) bv^2}e^{-\tfrac{\pi}{2}(-bi+a)^2v^2}\\
&=e^{-\pi ia bv^2-\pi b^2v^2}e^{\tfrac{\pi}{2}(-a^2+b^2+2abi)v^2}\\
&=e^{-\tfrac{\pi}{2}v^2}=f(v).
\end{align*}
\item[(2)] If we take $g(v)=ive^{-\tfrac{\pi}{2}v^2}$ for $v\in V$, then $g(v)\in \mathcal{H}_{F_{z_0}}$. 
\begin{align*}
\pi_{F_{z_0}}(g) g(v)
&=e^{-\pi i (-bi+a) bv^2}e^{-\tfrac{\pi}{2}(-bi+a)^2v^2}(iv(-bi+a))\\
&=g(v)(-bi+a).
\end{align*}
\end{itemize}
\end{example}
\begin{example}\label{example232} 
Take $z=\widehat{z}=z_0=i$.  Let us give the explicit actions of $ u(b), h(a)$, for  $u(b)=\begin{pmatrix}
  1&b\\
  0 & 1
\end{pmatrix}$,  $h(a)=\begin{pmatrix}
  a&0\\
  0 & a^{-1}
\end{pmatrix} \in \SL_{2}(\R)$. 
\end{example}

1) If $g=h(a)$, then $g^{h_{-1}}=h(a)$,  $\widehat{z'}=g^{h_{-1}}(\widehat{z})=\widehat{z} a^{2}=ia^2$. Write $v=v_x+iv_y$ and $v'=v'_x+iv'_y$.
\begin{align*}
 \gamma(\widehat{z}, \widehat{z'})&=\big(\tfrac{\widehat{z}-\overline{\widehat{z'}}}{2i}\big)^{-1/2} \cdot (\Im \widehat{z})^{1/4}\cdot (\Im \widehat{z'})^{1/4}\\
 &=\big(\tfrac{1+a^2}{2}\big)^{-1/2} \cdot |a|^{1/2};\\
& \\
 \kappa_{\widehat{z'}}(v'a,v'a) & =e^{-2\pi v'_ya[a^2]^{-1} av'_y}\\
 &=e^{-2\pi (v'_y)^2}\\
 &=\kappa_{z_0}(v',v');\\
&\\
 d_{\widehat{z'}}(v')=&a^{-2} d_{z_0}(v');
 \end{align*}
  \begin{align*}
 \kappa(\widehat{z},v; \widehat{z'},v'a)^{-1} 
 &=e^{-\pi i\cdot (\widehat{z}-\overline{\widehat{z'}})^{-1}\cdot (v-\overline{v'}a)^2}\\
  &=e^{-\pi (1+a^2)^{-1}(v-\overline{v'}a)^2}.
 \end{align*}
\begin{align*}
& \pi_{F_z}(g) f(v)\\
&=U_{\widehat{z}, g^{h_{-1}}(\widehat{z})} \circ T^{\widehat{z}}_{g^{h_{-1}}} f(v)\\
&= \gamma(\widehat{z}, \widehat{z'})\int_{\C} \kappa(\widehat{z},v; \widehat{z'},v'')^{-1}\big[T^{\widehat{z}}_{g^{h_{-1}}} f\big](v'') \kappa_{\widehat{z'}}(v'',v'') d_{\widehat{z'}}(v'')\\
&= \gamma(\widehat{z}, \widehat{z'})\int_{\C} \kappa(\widehat{z},v; \widehat{z'},v'')^{-1}f(v''a^{-1})\kappa_{\widehat{z'}}(v'',v'') d_{\widehat{z'}}(v'')\\
& \stackrel{v'=v''\cdot a^{-1}}{=}|a|^2\gamma(\widehat{z}, \widehat{z'})\int_{\C} \kappa(\widehat{z},v; \widehat{z'},v'a)^{-1}f(v')\kappa_{\widehat{z'}}(v'a,v'a) d_{\widehat{z'}}(v')\\
&=|a|^{1/2}\big(\tfrac{2}{1+a^2}\big)^{1/2}\int_{\R^2} e^{-\pi (1+a^2)^{-1}(v-\overline{v'}a)^2}f(v')e^{-2\pi (v'_y)^2}d_{z_0}(v').
\end{align*}
2) If $g=u(b)$, then $g^{h_{-1}}=u(-b)$,  $\widehat{z'}=g^{h_{-1}}(\widehat{z})=i-b$.
\begin{align*}
\gamma(\widehat{z}, \widehat{z'})
 &=\big(\tfrac{\widehat{z}-\overline{\widehat{z'}}}{2i}\big)^{-1/2} \cdot (\Im \widehat{z})^{1/4}\cdot (\Im \widehat{z'})^{1/4}\\
 &=\big(\tfrac{2i +b}{2i}\big)^{-1/2}=\tfrac{1}{\sqrt{1-\tfrac{b}{2}i}}.
 \end{align*}
 \begin{align*}
\kappa_{\widehat{z'}}(v',v')= \kappa_z(v',v')&\stackrel{v'=v'_x+iv'_y}{=}e^{-2\pi (v'_y)^2};\\
 & \\
  \kappa(\widehat{z},v; \widehat{z'},v')^{-1}
 &=e^{-\pi i (\widehat{z}-\overline{\widehat{z'}})^{-1}  (v-\overline{v'})^2}\\
 &=e^{-\pi i(2i+b)^{-1} (v-\overline{v'})^2};
 \end{align*}
\begin{align*}
& \pi_{F_z}(g) f(v)=U_{\widehat{z}, g^{h_{-1}}(\widehat{z})} \circ T^{\widehat{z}}_{g^{h_{-1}}} f(v)\\
&= \gamma(\widehat{z}, \widehat{z'})\int_{\C} \kappa(\widehat{z},v; \widehat{z'},v')^{-1}\big[T^{\widehat{z}}_{g^{h_{-1}}} f\big](v') \kappa_z(v',v') d_z(v')\\
&= \gamma(\widehat{z}, \widehat{z'})\int_{\C} \kappa(\widehat{z},v; \widehat{z'},v')^{-1}f(v')\kappa_z(v',v') d_z(v')\\
&=\tfrac{1}{\sqrt{1-\tfrac{b}{2}i}}\int_{\R^2} e^{-\pi i(2i+b)^{-1} (v-\overline{v'})^2} e^{-2\pi (v'_y)^2} f(v') dv'_x dv'_y.
\end{align*}
\section{Extended Weil representations: $\SL_2^{\pm}(\R)$ case}\label{chap4exWe}
In this section,   we will generate the Weil representation to $\widetilde{\GL}_2(\R)$.   
\subsection{}\label{chap4EWR}
Recall that $\pi_{\psi}$ is  Weil representation of $\widetilde{\SL}_2(\R)\ltimes \Ha(W)$ associated to $\psi$, and $\omega_{\psi}=\pi_{\psi}|_{\widetilde{\SL}_2(\R)}$. Let $\widetilde{\GL}_2(\R)$ act on $\Ha(W)$ via $\GL_2(\R)$.
\begin{lemma}
$\omega_{\psi} \simeq \omega_{\psi^{h^2}}$, for any $h\in \R^{\times}$, and $\omega_{\psi} \ncong \omega_{\psi^{-1}}$.
\end{lemma}
\begin{proof}
We closely follow the argument of \cite[p.~36]{MoViWa} in the $p$\nobreakdash-adic setting.  
Let $\widetilde h=[h,1]\in\widetilde{\GL}_2(\mathbb R)$, $\widetilde g=[g,\epsilon]\in\widetilde{\SL}_2(\mathbb R)$, and write $w_t=(w,t)\in\Ha(W)$.  
Then
\[
\pi_\psi\!\bigl(w_t^{\widetilde h}\bigr)
=\pi_\psi\!\bigl(wh,\,t\lambda_h\bigr)
\simeq\pi_{\psi^{\lambda_h}}\!(w_t),
\]
whence
\[
\pi_\psi\!\bigl([\widetilde g,w_t]^{\widetilde h}\bigr)
=\pi_\psi\!\bigl([\widetilde g^{\widetilde h},\,(wh,t\lambda_h)]\bigr)
\simeq\pi_{\psi^{\lambda_h}}\!\bigl([\widetilde g,w_t]\bigr).
\]
Taking $h\in\mathbb R_{+}^{\times}$ gives $\widetilde g^{\widetilde h}=\widetilde g$ and consequently
\[
\omega_\psi(\widetilde g)=\pi_\psi(\widetilde g)\simeq\pi_{\psi^{h^{2}}}(\widetilde g)=\omega_{\psi^{h^{2}}}(\widetilde g).
\]
The last relation can also be derived from the character computations for $\omega_\psi$ and $\omega_{\psi^{-1}}$ given in \cite[Thm.~1C]{Th} and \cite{Ad}.
\end{proof}
Let  
\[
\Pi_{\psi}
=\Ind_{\,\widetilde{\SL}_2(\mathbb R)\ltimes \Ha(W)}^{\,\widetilde{\SL}_2^{\pm}(\mathbb R)\ltimes \Ha(W)}\pi_{\psi}, \quad V^{\pm}_{\psi}
=\Ind_{\,\widetilde{\SL}_2(\mathbb R)\ltimes \Ha(W)}^{\,\widetilde{\SL}_2^{\pm}(\mathbb R)\ltimes \Ha(W)}V_{\psi},
\qquad
\Omega_{\psi}
=\Res_{\,\widetilde{\SL}_2^{\pm}(\mathbb R)}^{\,\widetilde{\SL}_2^{\pm}(\mathbb R)\ltimes \Ha(W)}\Pi_{\psi}.
\]
By Clifford--Mackey theory,  
\[
\Omega_{\psi}\simeq\Ind_{\,\widetilde{\SL}_2(\mathbb R)}^{\,\widetilde{\SL}_2^{\pm}(\mathbb R)}\omega_{\psi}.
\]

\begin{lemma}\label{chap4mainlm}
The restriction of $\Omega_{\psi}$ to $\widetilde{\SL}_2(\mathbb R)$ decomposes as the direct sum of the two Weil representations  
$\omega_{\psi}\oplus \omega_{\psi^{-1}}$.
\end{lemma}

Consequently, $\Omega_{\psi}$ is independent of the choice of $\psi$; we denote it simply by $\Omega$ and call it the Weil representation of $\widetilde{\SL}_2^{\pm}(\mathbb R)$. As a representation of $\Ha(W)$, we have:
\[
\Pi_{\psi}\simeq\pi_{\psi}\oplus\pi_{\psi^{-1}},
\]
which is reducible.  
Set $\Ha^{\pm}(W)=F_{2}\ltimes \Ha(W)$.  Then $\Pi_{\psi}$ is an irreducible representation of $\Ha^{\pm}(W)$.  Observe that $\mathbb R$ is a normal subgroup of $\Ha^{\pm}(W)$ and sits in the exact sequence  
\[
1\longrightarrow\mathbb R\longrightarrow\Ha^{\pm}(W)\longrightarrow F_{2}\ltimes W\longrightarrow 1.
\]

\subsection{The Schr\"odinger model}\label{chap4sch}
Retain the notations of Section~\ref{chap2Sch}.  Let us consider $F\in V^{\pm}_{\psi}$. Then $$F([\epsilon, h])=F([1, h^{s(\epsilon)}] [\epsilon,0])=\pi_{\psi}(h^{s(\epsilon)})F( [\epsilon,0]).$$  We identify $F( [\epsilon,0])(x)=f([\epsilon,x])$, for $x\in X$. The representation  $\Pi_{\psi}$ of $\widetilde{\GL}_2(\mathbb R)\ltimes \Ha(W)$  can be realized on  
\(L^{2}(\mu_{2}\times X)\cong L^{2}(\mu_{2}\times\R)\). To distinguish the different models, we write it by $ \Pi_{X^{\ast},\psi}$.  It is given   by the following formulas:
\begin{equation}\label{chap4representationsp21}
\Pi_{X^{\ast},\psi}[(x,0)\cdot (x^{\ast},0)\cdot (0,k)]f([\epsilon,y])=\psi(\epsilon k+\epsilon ( x+ y)x^{\ast}) f([\epsilon, x+y]),
\end{equation}
\begin{equation}\label{chap4representationsp22}
\Pi_{X^{\ast},\psi}[ h_{-1}]f([\epsilon,y])=f([-\epsilon, y]),
\end{equation}
\begin{equation}\label{chap4representationsp23}
\Pi_{X^{\ast},\psi}[u(b)]f([\epsilon, y])=\psi(\tfrac{1}{2}\epsilon b y^2) f([\epsilon,y]),
\end{equation}
\begin{equation}\label{chap4representationsp24}
\Pi_{X^{\ast},\psi}[h(a)]f([\epsilon, y])=|a|^{1/2} (a,\epsilon)_\R f([\epsilon,ya]),
\end{equation}
\begin{equation}\label{chap4representationsp25}
\Pi_{X^{\ast},\psi}[\omega]f([\epsilon, y])=\nu(\epsilon,\omega)\int_{X^{\ast}} \psi(\epsilon y y^{\ast}) f([\epsilon,-y^{\ast}]) d\mu(y^{\ast}).
\end{equation}
\begin{equation}\label{chap4representationsp26}
\Pi_{X^{\ast},\psi}[\diag(r,r)]f([\epsilon, y])=f([\epsilon, y]).
\end{equation}
where  $x,y ,  x^{\ast}, y^{\ast},k\in    \R$, $r\in \R^{\times}_{>0}$, $t\in \mu_8$, $\epsilon\in \mu_2$,  and 
\begin{equation*}
h_{-1}= \begin{pmatrix}
  1& 0\\
  0 &-1
\end{pmatrix}, h(a)=\begin{pmatrix}
  a& 0\\
  0 &a^{-1}\end{pmatrix}, u(b)=\begin{pmatrix}
  1&b\\
  0 & 1
\end{pmatrix}, \omega=\begin{pmatrix}
  0& -1\\
 1 &0
\end{pmatrix}, \diag(r,r)=\begin{pmatrix}
 r& 0\\
0 &r
\end{pmatrix}\in \GL_2(\R).
\end{equation*}
Moreover, under such actions, for $h_1, h_2\in \GL_2(\R)$,
$$\Pi_{X^{\ast},\psi}(h_1)\Pi_{X^{\ast},\psi}(h_2) =\widetilde{C}_{X^{\ast}}( h_1,h_2)\Pi_{X^{\ast},\psi}(h_1h_2).$$ 
Let us verify  these formulas:\\

1) $$(x,0)\cdot (x^{\ast},0)\cdot (0,k)=(x, x^{\ast}; k+\tfrac{\langle x, x^{\ast}\rangle}{2})\in \Ha(W);$$
\begin{align*}
&\Pi_{X^{\ast},\psi}[(x,0)\cdot (x^{\ast},0)\cdot (0,k)]f(\epsilon, y)\\
&=\Pi_{X^{\ast},\psi}[(x,0)\cdot (x^{\ast},0)\cdot (0,k)]F( [\epsilon,0])(y)\\
&=F( [\epsilon,0] \cdot [1, (x, x^{\ast}; k+\tfrac{\langle x, x^{\ast}\rangle}{2})])(y)\\
&=F( [\epsilon, (x, x^{\ast}; k+\tfrac{\langle x, x^{\ast}\rangle}{2})])(y)\\
&=F(  [1, (x, x^{\ast}; k+\tfrac{\langle x, x^{\ast}\rangle}{2})^{s(\epsilon)}] [\epsilon,0])(y)\\
&=\pi_{\psi}((x, x^{\ast}; k+\tfrac{\langle x, x^{\ast}\rangle}{2})^{s(\epsilon)})F( [\epsilon,0])(y)\\
&=F( [\epsilon,0])((y,0) \cdot  (x, \epsilon x^{\ast}; \epsilon k+\epsilon \tfrac{\langle x, x^{\ast}\rangle}{2})\\
&=F( [\epsilon,0])\big(  (x+y, \epsilon x^{\ast}; \tfrac{\langle y,\epsilon x^{\ast}\rangle}{2} +\epsilon k+\epsilon \tfrac{\langle x, x^{\ast}\rangle}{2})\big)\\
&=F( [\epsilon,0])\big((\epsilon x^{\ast};  \epsilon k+\epsilon \langle x+y, x^{\ast}\rangle)\big)  (x+y, 0)\\
&=\psi(\epsilon k+\epsilon \langle x+y, x^{\ast}\rangle) F( [\epsilon,0])  (x+y, 0)\\
     &=\psi(\epsilon k+\epsilon (x+ y)x^{\ast}) f([\epsilon, x+y]).
     \end{align*}
Let: 
\[([\epsilon, 1], 1),  ([1, g],1) \in \widetilde{\SL}_2^{\pm}(\mathbb R).\] 
\begin{align*}
([\epsilon, 1], 1) ([1, g],1) &=([\epsilon, g], \widetilde{C}_{X^{\ast}}([\epsilon, 1],[1, g]))\\
&=([\epsilon, g], 1)\\
&=\Big([1, g^{s(\epsilon)}], \widetilde{C}_{X^{\ast}}([1, g^{s(\epsilon)}],[\epsilon, 1])^{-1}\Big)\Big([\epsilon, 1], 1\Big) \\
&=\Big([1, g^{s(\epsilon)}], \nu( \epsilon, g^{s(\epsilon)})^{-1}\widetilde{c}_{X^{\ast}}(g, 1)^{-1}\Big)\Big([\epsilon, 1], 1\Big)\\
&=\Big([1, g^{s(\epsilon)}], \nu( \epsilon, g^{s(\epsilon)})^{-1}\Big)\Big([\epsilon, 1], 1\Big).
\end{align*}
2) If $g= u(b)$, then $\nu( 1,g)=1=\nu( -1,g)=\nu( -1,g^{s(-1)}) $.
\begin{align}\label{22}
\Pi_{X^{\ast},\psi}(g)f([\epsilon, y])&=\Pi_{X^{\ast},\psi}(g)F( [\epsilon,0])(y)\\
&=F( [\epsilon,0] \cdot [g,0])(y)\\
&=F(\Big[ \big( [\epsilon, g], \widetilde{C}_{X^{\ast}}([\epsilon, 1], [1,g])\big),0\Big])(y)\\
&=F(\Big(\big([1, g^{s(\epsilon)}], \nu( \epsilon, g^{s(\epsilon)})^{-1}\big), 0\Big)\Big(\big([\epsilon, 1], 1\big), 0\Big))(y)\\
&=\pi_{\psi}( g^{s(\epsilon)} )\nu( \epsilon, g^{s(\epsilon)})^{-1} F(\Big(\big([\epsilon, 1], 1\big), 0\Big))(y)\\
&=\psi(\epsilon \tfrac{1}{2}\langle y,yb\rangle) F([\epsilon, 0])(y)\\
&=\psi(\epsilon \tfrac{ b y^2}{2}) f([\epsilon,y]).
\end{align}
3) If $g=h(a)$, then $\nu( 1,g)=1$, $\nu( -1,g^{s(\epsilon)})=\nu( -1,g)=(a, -1)_\R$.
\begin{align}\label{23}
\Pi_{X^{\ast},\psi}(g)f([\epsilon, y])&=\pi_{\psi}( g^{s(\epsilon)} )\nu( \epsilon, g^{s(\epsilon)})^{-1} F([\epsilon, 0])(y)\\
&=(a, \epsilon)_\R |a|^{1/2} F([\epsilon, 0])(ya)\\
&=(a, \epsilon)_\R |a|^{1/2} f([\epsilon,ya]).
\end{align}
4) If $g= \omega$,  then $\nu( 1,g)=1$, $\nu( -1,g)=\gamma(-1, \psi^{\tfrac{1}{2}})^{-1}=e^{\tfrac{(\sgn \mathfrak{e})\pi i}{2}}$, $\nu( -1,g^{s(-1)})=e^{-\tfrac{(\sgn \mathfrak{e})\pi i}{2}} $. 
\begin{align}\label{235}
\Pi_{X^{\ast},\psi}(g)f([1, y])&=\pi_{\psi}( g )\nu( 1, g)^{-1}  F([1, 0])(y)\\
&=\int_{X^{\ast}} \psi(\langle y, y^{\ast}\rangle) F([1, 0])(y^{\ast}\omega^{-1}) d\mu(y^{\ast})\\
&=\int_{X^{\ast}} \psi( yy^{\ast}) f([1,-y^{\ast}]) d\mu(y^{\ast}).
\end{align}
\begin{align}\label{24}
\Pi_{X^{\ast},\psi}(g)f([-1, y])&=\pi_{\psi}( g^{s(-1)} )\nu( -1, g^{s(-1)} )^{-1} F([-1, 0])(y)\\
&=\pi_{\psi}(-g)\nu( -1, g^{s(-1)} )^{-1} F([-1, 0])(y)\\
&=\int_{X^{\ast}} \nu( -1, g^{s(-1)} )^{-1} \psi(\langle y, y^{\ast}\rangle)F([-1, 0]) (-y^{\ast}\omega^{-1}) d\mu(y^{\ast})\\
&=\nu( -1,g)\int_{X^{\ast}} \psi( y y^{\ast}) f([-1,y^{\ast}]) d\mu(y^{\ast})\\
&=\nu( -1,g)\int_{X^{\ast}} \psi( -y y^{\ast}) f([-1,-y^{\ast}]) d\mu(y^{\ast}).
\end{align}
\subsection{The Lattice model}\label{chap4la}
Retain notations from Section \ref{chap2latticemodel}.  Then $\Pi_{\psi}\simeq \Pi_{L, \psi} =\Ind_{\Ha(L)}^{\Ha^{\pm}(W)} \psi_L$. Let $\mathcal{H}_{\psi}^{\pm}(L)$ be the set of  measurable functions $f: \mu_2  \times W \longrightarrow \C$ such that
\begin{itemize}
\item[(i)] $f(\epsilon, l+w)=\psi(-\epsilon\tfrac{\langle x_{l}, x^{\ast}_l\rangle}{2}-\epsilon\tfrac{\langle l, w\rangle}{2}) f(\epsilon, w)$, for  all $l=x_l+x_{l}^{\ast}\in L=(X\cap L) \oplus (X^{\ast}\cap L)$, almost all $w\in W$,
\item[(ii)] $\int_{L\setminus W} ||f(\epsilon, w)||^2 dw<+\infty$,
\end{itemize}
where $w\in W$, $\epsilon \in \mu_2$, and $dw$ is a $W$-right invariant measure on $L\setminus W$.  Then $\Pi_{L,\psi}$ can be realized on $\mathcal{H}_{\psi}^{\pm}(L)$ by the following formulas:
\begin{equation}\label{chap4www1}
\Pi_{L,\psi}([1,w'_t])f( \epsilon, w)=\psi(\epsilon t+\epsilon\tfrac{\langle w, w'\rangle}{2})f( \epsilon, w+w'),
 \end{equation}
\begin{equation}\label{chap4www2}
\Pi_{L,\psi}([-1,0])f( \epsilon, w)=f(- \epsilon, x-x^{\ast}),
 \end{equation}
for $w=x+x^{\ast},w'=x'+x^{'\ast}\in W$, $w'_t=(w',t) \in \Ha(W)$, $\epsilon \in \mu_2$, $t\in \R$.
\subsection{The Fock model}\label{FM}
Retain notations from Section \ref{FOCKMODEL1}, and consider  $z=z_0=i$. Let:
\[ \Pi_{F_{z_0}}=\Ind_{\widehat{\SL}_2(\R) \ltimes \Ha(W)}^{\widehat{\SL}^{\pm}_2(\R) \ltimes \Ha(W)}\pi_{F_{z_0}}, \quad \mathcal{H}^{\pm}_{F_{z_0}}=\Ind_{\widehat{\SL}_2(\R) \ltimes \Ha(W)}^{\widehat{\SL}^{\pm}_2(\R) \ltimes \Ha(W)} \mathcal{H}_{F_{z_0}}.\]
 Let $F\in \mathcal{H}^{\pm}_{F_{z_0}}$. Then $$F([\epsilon, h])=F([1, h^{s(\epsilon)}] [\epsilon,0])=\pi_{F_{z_0}}(h^{s(\epsilon)})F( [\epsilon,0]).$$ We identify $F( [\epsilon,0])(v)=f([\epsilon,v])$. So $ \mathcal{H}^{\pm}_{F_{z_0}}$ can be taken as  the   space   of holomorphic  functions $f$ on $\mu_2 \times \C$ such that 
 \[\Vert f\Vert^2= \Big[\int_{\C} \mid f([1,v])\mid^2\kappa_{z_0}(v,v)d_{z_0}(v)+ \int_{\C} \mid f([-1,v])\mid^2\kappa_{z_0}(v,v)d_{z_0}(v)\Big]^{1/2}< +\infty.\]
 The representation $ \Pi_{F_{z_0}}$   of $\widehat{\GL}_2(\R) \ltimes \Ha^{\pm}(W)$  can be realized on $\mathcal{H}^{\pm}_{F_{z_0}}$  by the following formulas:
\begin{equation}\label{z_00}
\Pi_{F_{z_0}}(h)f([\epsilon,v])=\psi_0(\epsilon t)e^{-2\pi i[\frac{1}{2}(x\overline{z}_0+\epsilon x^{*})x+vx]}f([\epsilon,v+x\overline{z}_0+\epsilon x^{\ast}]),
\end{equation}
\begin{equation}\label{z_01}
\Pi_{F_{z_0}}[ h_{-1}]f([\epsilon,v])=f([-\epsilon, v]),
\end{equation}
\begin{equation}\label{z_02}
\Pi_{F_{z_0}}[u]f([1, v])= e^{-\pi i \overline{u} \Im(u)v^2} f([1,v\overline{u}]),
\end{equation}
\begin{equation}\label{z_03}
\Pi_{F_{z_0}}[u]f([-1, v])=ue^{-\pi i u \Im(\overline{u})v^2} f([-1,vu]),
\end{equation}
\begin{equation}\label{z_04}
\Pi_{F_{z_0}}[u(b)]f([\epsilon, v])=\tfrac{1}{\sqrt{1-\epsilon\tfrac{b}{2}i}}\int_{\R^2} e^{-\pi i(2i+\epsilon b)^{-1} (v-\overline{v'})^2} e^{-2\pi (v'_y)^2} f([\epsilon,v']) dv'_x dv'_y,
\end{equation}
\begin{equation}\label{z_05}
\Pi_{F_{z_0}}[h(a)]f([\epsilon, v])=(\epsilon,a)_{\R}\big(\tfrac{2|a|}{1+a^2}\big)^{1/2}\int_{\R^2} e^{-\pi (1+a^2)^{-1}(v-\overline{v'}a)^2}f([\epsilon,v'])e^{-2\pi (v'_y)^2}dv'_xdv'_y,
\end{equation}
\begin{equation}\label{z_06}
\Pi_{F_{z_0}}[\diag(r,r)]f([\epsilon, v])=f([\epsilon, v]),
\end{equation}
where $w=x+x^{\ast}\in W$,  $t\in \R$,  $h=(w,t)\in \Ha(W)$, $r>0$, $u(b)=\begin{pmatrix}1&b\\0&1\end{pmatrix},\quad
h(a)=\begin{pmatrix}a&0\\0&a^{-1}\end{pmatrix},\quad
\omega=\begin{pmatrix}0&-1\\1&0\end{pmatrix}\in\SL_{2}(\R)$, and  $u=\begin{pmatrix}a&-b\\b&a\end{pmatrix} \in \SO_2(\R)$ is identified  to  $u=bi+a$.  Let us verify  these formulas:\\
1) \begin{align*}
\Pi_{F_{z_0}}(h)f([\epsilon,v])&= \Pi_{F_{z_0}}([1,h]) F( [\epsilon,0])(v)\\
&= F( [\epsilon,0][1,h])(v)\\
&=F([1, h^{s(\epsilon)}] [\epsilon,0])(v)\\
&=\pi_{F_{z_0}}(h^{s(\epsilon)})F( [\epsilon,0])(v)\\
&=\psi_0(\epsilon t)\,
  e^{-2\pi i\!\left[\frac12(x\overline{z}_0+\epsilon x^{*})\,x+vx\right]}
  F( [\epsilon,0])\!\bigl(v+x\overline{z}_0+\epsilon x^{*}\bigr)\\
  &=\psi_0(\epsilon t)\,
  e^{-2\pi i\!\left[\frac12(x\overline{z}_0+\epsilon x^{*})\,x+vx\right]}
  f( \epsilon,v+x\overline{z}_0+\epsilon x^{*}]).
  \end{align*}
  2) \begin{align*}
\Pi_{F_{z_0}}(h_{-1})f([\epsilon,v]) &= \Pi_{F_{z_0}}([-1,0]) F( [\epsilon,0])(v)\\
&= F( [\epsilon,0][-1,0])(v)\\
&= F( [-\epsilon,0])(v)\\
&=f([-\epsilon,v]).
\end{align*}
Let: 
\[((\epsilon, 1), 1),  ((1, g),1) \in \widehat{\SL}_2^{\pm}(\mathbb R).\] 
\begin{align*}
((\epsilon, 1), 1) ((1, g),1) &=((\epsilon, g), \widehat{C}_{z_0}((\epsilon, 1),(1, g)))\\
&=((\epsilon, g), 1)\\
&=\Big((1, g^{s(\epsilon)}),  \widehat{C}_{z_0}((1, g^{s(\epsilon)}),(\epsilon, 1))^{-1}\Big)\Big((\epsilon, 1), 1\Big) \\
&=\Big((1, g^{s(\epsilon)}), \nu_{z_0}( \epsilon, g^{s(\epsilon)})^{-1} \widehat{c}_{z_0}(g, 1)^{-1}\Big)\Big((\epsilon, 1), 1\Big)\\
&=\Big((1, g^{s(\epsilon)}), \nu_{z_0}( \epsilon,g^{s(\epsilon)})^{-1}\Big)\Big((\epsilon, 1), 1\Big).
\end{align*}
3) If   $g=\begin{pmatrix}a & -b\\ b& a\end{pmatrix}\in \SO_2(\R)$ corresponding to  $u=a+bi \in  \U(\C)$ , we have: $\nu_{z_0}(1,g)=1$. 
 \[\nu_2(-1, g)=\left\{ \begin{array}{cl} -1 & \textrm{ if } g= \begin{pmatrix} -1& 0\\0 &-1\end{pmatrix},\\ 1 & \textrm{ otherwise},\end{array}\right. \qquad\qquad \tfrac{\overline{s}(g^{h_{-1}})}{\overline{s}(g)}=\left\{ \begin{array}{cl} 1 & \textrm{ if } g= \begin{pmatrix} -1& 0\\0 &-1\end{pmatrix},\\ a-bi & \textrm{ otherwise},\end{array}\right.\] \[\nu_{z_0}(-1,g^{h_{-1}})=\nu_2( -1,g^{h_{-1}} ) \tfrac{\overline{s}(g^{h_{-1}})}{\overline{s}(g)}=a-bi.\] 
 
 \begin{align*}
\Pi_{F_{z_0}}(g)f([\epsilon,v]) &= \Pi_{F_{z_0}}([g,0]) F( [\epsilon,0])(v)\\
&= F( [\epsilon,0][g,0])(v)\\
&= F([ ((\epsilon,g),1),0])(v)\\
&= F([\Big((1, g^{s(\epsilon)}), \nu_{z_0}( \epsilon, g)^{-1}\Big)\Big((\epsilon, 1), 1\Big),0])(v)\\
&=\nu_{z_0}( \epsilon, g^{s(\epsilon)})^{-1}\pi_{F_{z_0}}(1, g^{s(\epsilon)})F([\Big((\epsilon, 1), 1\Big),0])(v)\\
&\stackrel{(\ref{unitact})}{=}\left\{ \begin{array}{cl}e^{-\pi i (-bi+a) bv^2}f([1,v(-bi+a)]), &  \textrm{ if } \epsilon=1, \\
(bi+a)e^{\pi i (bi+a) bv^2} f([-1,v(bi+a)]), &  \textrm{ if } \epsilon=-1. \end{array}\right.
\end{align*}
4) If   $g=u(b)\in \SL_2(\R)$, $\nu_{z_0}(1,g)=1$ and 
\[\nu_{z_0}(-1,g)=\nu_2( -1,g ) \tfrac{\overline{s}(g)}{\overline{s}(g^{h_{-1}})}=\nu_2( -1,g )=1=\nu_2( -1, g^{s(-1)} ).\] 
 \begin{align*}
\Pi_{F_{z_0}}(g)f([\epsilon,v]) &=\nu_{z_0}( \epsilon, g^{s(-1)})^{-1}\pi_{F_{z_0}}(1, g^{s(\epsilon)})F([\Big((\epsilon, 1), 1\Big),0])(v).
\end{align*}
By Example \ref{example232}, we have:
\begin{equation}
 \Pi_{F_{z_0}}(g) f([1,v])=\tfrac{1}{\sqrt{1-\tfrac{b}{2}i}}\int_{\R^2} e^{-\pi i(2i+b)^{-1} (v-\overline{v'})^2} e^{-2\pi (v'_y)^2} f([1,v']) dv'_x dv'_y,
 \end{equation}
 \begin{align}
 \Pi_{F_{z_0}}(g) f([-1,v])
 &=\tfrac{1}{\sqrt{1+\tfrac{b}{2}i}}\int_{\R^2} e^{-\pi i(2i-b)^{-1} (v-\overline{v'})^2} e^{-2\pi (v'_y)^2} f([-1,v']) dv'_x dv'_y.
 \end{align}
 5) If   $g=h(a)\in \SL_2(\R)$, $\nu_{z_0}(1,g)=1$ and $\nu_{z_0}(-1,g^{s(-1)})=\nu_{z_0}(-1,g)=\nu_2( -1,g ) \tfrac{\overline{s}(g)}{\overline{s}(g^{h_{-1}})}=\nu_2( -1,g )=(-1,a)_{\R}$. 
 \begin{align*}
\Pi_{F_{z_0}}(g)f([\epsilon,v]) &=\nu_{z_0}( \epsilon, g^{s(\epsilon)})^{-1}\pi_{F_{z_0}}(1, g^{s(\epsilon)})F([\Big((\epsilon, 1), 1\Big),0])(v).
\end{align*}
 By Example \ref{example232}, we have:
 \begin{equation}
 \Pi_{F_{z_0}}(g) f([1,v])=|a|^{1/2}\big(\tfrac{2}{1+a^2}\big)^{1/2}\int_{\R^2} e^{-\pi (1+a^2)^{-1}(v-\overline{v'}a)^2}f([1,v'])e^{-2\pi (v'_y)^2}dv'_xdv'_y,
 \end{equation}
 \begin{align}
 \Pi_{F_{z_0}}(g) f([-1,v])
 &=(-1,a)_{\R}|a|^{1/2}\big(\tfrac{2}{1+a^2}\big)^{1/2}\int_{\R^2} e^{-\pi (1+a^2)^{-1}(v-\overline{v'}a)^2}f([-1,v'])e^{-2\pi (v'_y)^2}dv'_xdv'_y.
 \end{align}
 \begin{example}\label{exm2}
 Let   $g=\begin{pmatrix}a & -b\\ b& a\end{pmatrix}\in \SO_2(\R)$, corresponding to $u=bi+a\in \U(\C)$. Define, for $\epsilon\in\{\pm1\}$ and $v\in\C$:
\[
f^{1/2}_+([\epsilon,v])=\delta_{\epsilon,1}\,e^{-\frac{\pi}{2}v^2},\quad
f^{1/2}_-([\epsilon,v])=\delta_{\epsilon,-1}\,e^{-\frac{\pi}{2}v^2},
\]
\[
f^{3/2}_+([\epsilon,v])=\delta_{\epsilon,1}\,\tfrac{i}{2}v\,e^{-\frac{\pi}{2}v^2},\quad
f^{3/2}_-([\epsilon,v])=\delta_{\epsilon,-1}\,\tfrac{i}{2}v\,e^{-\frac{\pi}{2}v^2}.
\]
  Then:
\begin{itemize}
\item[(1)] $\Pi_{F_{z_0}}(g)f^{1/2}_+=f^{1/2}_+$, $\Pi_{F_{z_0}}(g)f^{3/2}_+=\overline{u}f^{3/2}_+$,  $\Pi_{F_{z_0}}(g)f^{1/2}_-=uf^{1/2}_-$,  $\Pi_{F_{z_0}}(g)f^{3/2}_-=u^2f^{3/2}_-$.
\item[(2)] $\Pi_{F_{z_0}}(h_{-1}) f^{1/2}_+=f^{1/2}_-$ and $\Pi_{F_{z_0}}(h_{-1})f^{3/2}_+=f^{3/2}_-$.
\end{itemize}
\end{example}
\begin{proof}
Let $f_1(v)=e^{-\tfrac{\pi}{2}v^2}$ and $f_2(v)=\tfrac{1}{2}ive^{-\tfrac{\pi}{2}v^2}$.\\
\[ \nu_{z_0}(-1, g^{h_{-1}})^{-1}=u.\]
1) \begin{align*}
\Pi_{F_{z_0}}(g) f^{1/2}_+([1,v])&=\pi_{F_{z_0}}(g)f_1(v)=f_1(v)=f^{1/2}_+([1,v]).
\end{align*}
\begin{align*}
\Pi_{F_{z_0}}(g)f^{1/2}_-([-1,v])&=\Pi_{F_{z_0}}([-1,1]) \pi_{F_{z_0}}(g^{h_{-1}}) \nu_{z_0}(-1, g^{h_{-1}})^{-1}f_1([1,v])\\
&=\Pi_{F_{z_0}}([-1,1]) u f_1([1,v])=uf^{1/2}_-([-1,v]).
\end{align*}
 \begin{align*}
\Pi_{F_{z_0}}(g) f^{3/2}_+([1,v])&=\pi_{F_{z_0}}(g)f_2(v)=\overline{u}f_2(v)=\overline{u}f^{3/2}_+([1,v]).
\end{align*}
\begin{align*}
\Pi_{F_{z_0}}(g)f^{3/2}_-([-1,v])&=\Pi_{F_{z_0}}([-1,1]) \pi_{F_{z_0}}(g^{h_{-1}}) uf_2(v)= \Pi_{F_{z_0}}([-1,1]) u^2 f_2(v)=u^2f^{3/2}_-([-1,v]).
\end{align*}
2) It follows from Eq. \ref{z_01}.
\end{proof}
Let $\mathcal{D}^+=\Span_{\C}\{f^{1/2}_+, f^{1/2}_-\}$ and  $\mathcal{D}^-=\Span_{\C}\{f^{3/2}_+, f^{3/2}_-\}$.  Let $\delta^+=\Ind_{\widehat{\SO}_2(\R)}^{\widehat{\Oa}_2(\R)} (1\cdot 1_{\mathbb{T}})$ and $\delta^-=\Ind_{\widehat{\SO}_2(\R)}^{\widehat{\Oa}_2(\R)} (\overline{u}\cdot 1_{\mathbb{T}})$.
\begin{lemma}\label{deltapm2}
The actions of $\Pi_{F_{z_0}}(\widehat{\Oa}_2(\R))$ on $\mathcal{D}^+$ and $\mathcal{D}^-$ are isomorphic to $\delta^+$ and $\delta^-$,  respectively.
\end{lemma}
\begin{proof}
$\widehat{\Oa}_2(\R)=\widehat{\SO}_2(\R) \sqcup \widehat{\SO}_2(\R) \widehat{h}_{-1}$, for $\widehat{h}_{-1}= [h_{-1},1]$. Let $g=\begin{pmatrix}a & -b\\ b& a\end{pmatrix}\in \SO_2(\R)$,  $\widehat{g}=[g, \epsilon]\in \widehat{\SO}_2(\R)$ and $u=bi+a$.\\
a) Let $e_1$ and $e_2$ denote  the complex  functions on $\widehat{\Oa}_2(\R)$ such that (1) $\supp e_1=\widehat{\SO}_2(\R)$ and $\supp e_2=\widehat{\SO}_2(\R)\widehat{h}_{-1}$, (2) $e_1(1)=1=e_2(\widehat{h}_{-1})$, (3) $ e_1(\widehat{g})=\epsilon=e_2(\widehat{g}\widehat{h}_{-1})$. Let $V^+$ be the vector space spanned by $e_1,e_2$. The representation $\delta^+$ can be realized  on $V^+$. Let $ \mathcal{A}$ be the $\C$-linear map from $ \mathcal{D}^+$ to $V^+$ by sending $ f^{1/2}_+ $ to $e_1$ and $ f^{1/2}_-$ to $e_2$. Let us check that it  indeed defines an  intertwining operator.    
\begin{align*}
\delta^+(g) e_2(\widehat{h}_{-1})&=e_2(\widehat{h}_{-1}[g,1])\\
&=e_2([g^{h_{-1}},\widehat{C}_{z_0}(h_{-1},g) \widehat{C}_{z_0}(g^{h_{-1}}, h_{-1})^{-1} ]\widehat{h}_{-1})\\
&=e_2([g^{h_{-1}}, \nu_{z_0}(-1,g^{s(-1)})^{-1}]\widehat{h}_{-1})=\nu_{z_0}( -1,g^{s(-1)})^{-1}=u;
\end{align*}
\[\widehat{h}_{-1}\widehat{h}_{-1}=[h_{-1},1][h_{-1},1]=[1, \widehat{C}_{z_0}(h_{-1},h_{-1})]=[1,1];\]
\[\delta^+(g) e_1=e_1, \delta^+(g) e_2=u e_2,  \delta^+(h_{-1}) e_1=e_2,  \delta^+(h_{-1})e_2=e_1;\]
\[\mathcal{A}(\Pi_{F_{z_0}}(g) f^{1/2}_+)=\mathcal{A}(f^{1/2}_+)=e_1=\delta^+(g) e_1;\]
\[\mathcal{A}(\Pi_{F_{z_0}}(g) f^{1/2}_-)=\mathcal{A}(uf^{1/2}_-)=ue_2=\delta^+(g) e_2;\] 
\[\mathcal{A}(\Pi_{F_{z_0}}(h_{-1}) f^{1/2}_+)=\mathcal{A}(f^{1/2}_-)=e_2=\delta^+(h_{-1}) e_1;\]   
\[\mathcal{A}(\Pi_{F_{z_0}}(h_{-1}) f^{1/2}_-)=\mathcal{A}(f^{1/2}_+)=e_1=\delta^+(h_{-1}) e_2.\]  
b) Let $e_3$ and $e_4$ denote  the complex  functions on $\widehat{\Oa}_2(\R)$ such that (1) $\supp e_3=\widehat{\SO}_2(\R)$ and $\supp e_4=\widehat{\SO}_2(\R)\widehat{h}_{-1}$, (2) $e_3(1)=1=e_4(\widehat{h}_{-1})$, (3) $ e_3(\widehat{g})=\epsilon \overline{u}=e_4(\widehat{g}\widehat{h}_{-1})$. Let $V^-$ be the vector space spanned by $e_3,e_4$. The representation $\delta^-$ can be realized  on $V^-$. Let $ \mathcal{B}$ be the $\C$-linear map from $ \mathcal{D}^-$ to $V^-$ by sending $ f^{3/2}_+ $ to $e_3$ and $ f^{3/2}_- $ to $e_4$. Let us check that it  indeed defines an  intertwining operator.    
\begin{align*}
\delta^-(g) e_4(\widehat{h}_{-1})&=e_4(\widehat{h}_{-1}[g,1])\\
&=e_4([g^{h_{-1}},\widehat{C}_{z_0}(h_{-1},g) \widehat{C}_{z_0}(g^{h_{-1}}, h_{-1})^{-1} ]\widehat{h}_{-1})\\
&=e_4([g^{h_{-1}}, \nu_{z_0}(-1,g^{s(-1)})^{-1}]\widehat{h}_{-1})=u\nu_{z_0}( -1,g^{s(-1)})^{-1}=u^2;
\end{align*}
\[\delta^-(g) e_3=\overline{u}e_3, \delta^-(g) e_4=u^2 e_4,  \delta^-(h_{-1}) e_3=e_4,  \delta^-(h_{-1})e_4=e_3;\]
\[\mathcal{B}(\Pi_{F_{z_0}}(g) f^{3/2}_+)=\mathcal{B}(\overline{u}f^{3/2}_+)=\overline{u}e_3=\delta^-(g) e_3;\]
\[\mathcal{B}(\Pi_{F_{z_0}}(g) f^{3/2}_-)=\mathcal{B}(u^2f^{3/2}_-)=u^2e_4=\delta^-(g) e_4;\] 
\[\mathcal{B}(\Pi_{F_{z_0}}(h_{-1}) f^{3/2}_+)=\mathcal{B}(f^{3/2}_-)=e_4=\delta^-(h_{-1}) e_3;\]   
\[\mathcal{B}(\Pi_{F_{z_0}}(h_{-1}) f^{3/2}_-)=\mathcal{B}(f^{3/2}_+)=e_3=\delta^-(h_{-1}) e_4.\]  
\end{proof}
\begin{lemma}\label{deltapm}
Both $\delta^+$ and $\delta^-$ are irreducible representations.
\end{lemma}
\begin{proof}
$( \widehat{h}_{-1})^{-1}\widehat{g}\widehat{h}_{-1}=[g^{h_{-1}}, u]$. Hence:
\[(1\cdot 1_{\mathbb{T}})^{\widehat{h}_{-1}}(\widehat{g})=(1\cdot 1_{\mathbb{T}})([g^{h_{-1}}, u])=u\neq (1\cdot 1_{\mathbb{T}})(\widehat{g}),\]
\[(\overline{u}\cdot 1_{\mathbb{T}})^{\widehat{h}_{-1}}(\widehat{g})=(\overline{u}\cdot 1_{\mathbb{T}})([g^{h_{-1}}, u])=u^2\neq (\overline{u}\cdot 1_{\mathbb{T}})(\widehat{g}).\]
\end{proof}
\section{Intertwining operators for $\SL_2^{\pm}(\R)$ case: Schrödinger model to Lattice model}
In this section, $\psi=\psi_0$ and $z=z_0$.
\subsection{Lattice}(cf. \cite{De, DeSe})
Let $w_1=ae_1+be_1^{\ast}$, $w_2=ce_1+de_1^{\ast}$  be two elements   of $W$. Then $\Za w_1+\Za w_2$ forms a  lattice in $W$ iff $w_1$ and $w_2$ are $\R$-linear independence iff $\begin{pmatrix}a & b\\ c & d \end{pmatrix}\in \GL_2(\R)$. Moreover, it is a $\psi$-self dual lattice iff $\langle w_1, w_2\rangle=1$ or $-1$ iff $\begin{pmatrix}a & b\\ c & d \end{pmatrix} \in \SL_2^{\pm}(\R)$. Let $\mathcal{L}_{\psi}$ denote the set of all $\psi$-self dual lattices in $W$. We can define a right $\SL_2^{\pm}(\R)$-action on $\mathcal{L}_{\psi}$ as follows:
$$\begin{array}{rcl}
  \mathcal{L}_{\psi} \times \SL_2^{\pm}(\R)  &\longrightarrow &\mathcal{L}_{\psi} ; \\
  ( L= \Za w_1  \oplus \Za w_2, g)& \longmapsto & Lg=\Za  (w_1g)  \oplus \Za (w_2g).
\end{array}$$
Let  $L_0=\Za e_1  \oplus \Za e_1^{\ast} \in \mathcal{L}_{\psi}$. Then:
\begin{itemize}
\item  $\mathcal{L}_{\psi}=\{L_0 g \mid g\in \SL_2^{\pm}(\R)\}$.
\item $L_0g=L_0$ iff $g\in \SL_2^{\pm }(\Z)$.
\item There exists a bijection: $  \SL_2^{\pm }(\Z) \setminus \SL_2^{\pm}(\R) \longrightarrow  \mathcal{L}_{\psi}; g \longmapsto L_0 g$.
\end{itemize}
\subsection{Lagrangian Grassmanian}(cf.\cite{De})
By Section \ref{FOCKMODEL1}, we can identity $\C$ with $W$ by $a+ib \longrightarrow ae_1+be_1^{\ast}$, and define  $\langle z,z'\rangle_{\C}=\Im(\overline{z}z')$. A Lagrangian plane   of $(\C, \langle, \rangle_{\C})$ is a set $\R z$, for some $z \in \C^{\times}$. Let $\varLambda $ denote the set of all  Lagrangians in $\C$. Then there exists a
left $\C^{\times}$-action on $\Lambda$ given as follows:
 $$\begin{array}{rcl}
  \C^{\times}\times \Lambda &\longrightarrow &\Lambda ; \\
  ( t, \R z)& \longmapsto & t(\R z)=\R t(z).
\end{array}$$
Then $\C^{\times}$ acts transitively on $\Lambda$, and there exists a bijective map:
$$\C^{\times}/\R^{\times} \longrightarrow \Lambda; z \longmapsto \R z.$$
Note that $\C^{\times}/\R^{\times} \simeq \U(\C)/\mu_2$. 
\subsection{Two models} Let $\Pi_{\psi}$ denote the Heisenberg representation of $\Ha^{\pm}(W)$ defined in Sections ~\ref{chap4sch}, and ~\ref{chap4la}.  
It can be realized on  
\[
\mathcal{H}^{\pm}(Y)=\Ind_{Y\times\R}^{\Ha^{\pm}(W)}\psi_Y
\quad\text{and}\quad
\mathcal{H}^{\pm}(L)=\Ind_{\Ha(L)}^{\Ha^{\pm}(W)}\psi_L,
\]  
for any $Y\in\Lambda$ and $L\in\mathcal{L}_{\psi}$, where $\psi_Y$ and $\psi_L$ extend the character $\psi$ on $\R$.  In this section, we only consider the following case:

\begin{itemize}
 \item  $Y^{\ast}=X^{\ast}$, $Y=X$, and $L=\Z e_1\oplus \Z e_1^{\ast}$.
\end{itemize}

  Note that the actions of $\Ha^{\pm}(W)$ on $\mathcal{H}^{\pm}(X^{\ast})$ and $\mathcal{H}^{\pm}(L)$ are given by the formulas (\ref{chap4representationsp21}) (\ref{chap4representationsp22}) and (\ref{chap4www1}) (\ref{chap4www2}) respectively. By \cite[pp.164-165]{We}, or \cite[pp.142-145]{LiVe}, there exists  a pair of explicit isomorphisms between   $\mathcal{H}^{\pm}(X^{\ast})\simeq L^2(\mu_2\times X)$ and  $\mathcal{H}^{\pm}(L)$, given as follows:
 \begin{equation}\label{chap6eq7}
 \begin{split}
\theta_{L, X^{\ast}}(f')([\epsilon, w])&=\sum_{l\in L/L\cap X^{\ast}} f'([\epsilon,  w+l])\psi(\epsilon\tfrac{\langle l, w\rangle}{2}+\epsilon\tfrac{\langle x_{l}, x^{\ast}_l\rangle}{2})\\
&= \sum_{l\in L\cap X}f'([\epsilon,  w+l])\psi(\epsilon\tfrac{\langle l, w\rangle}{2})\\
&=\sum_{l\in L\cap X}f'([\epsilon, x+l])\psi(\epsilon\langle l, w\rangle) \psi(\epsilon\tfrac{\langle x, x^{\ast}\rangle}{2}), 
 \end{split}
 \end{equation}
 \begin{equation}\label{chap6eq8}
 \begin{split}
 \theta_{X^{\ast}, L}(f)([\epsilon,y])&=  \int_{X^{\ast}/X^{\ast}\cap L} f([1, (\dot{y}^{\ast},0)]\cdot [\epsilon,(y, 0)]) d\dot{y}^{\ast}\\
 &= \int_{X^{\ast}/X^{\ast}\cap L} \psi(\epsilon\tfrac{\langle\dot{y}^{\ast}, y\rangle}{2})f([\epsilon,y+\epsilon\dot{y}^{\ast}]) d\dot{y}^{\ast},
 \end{split}
      \end{equation}
for $ w=x+ x^{\ast} \in W$, $y\in X$,  $f'\in \mathcal{H}^{\pm}(X^{\ast})$ with $f'|_{\mu_2 \times X}\in S(\mu_2\times X)\subseteq L^2(\mu_2 \times X)$,  $f\in L^1(W)\cap \mathcal{H}^{\pm}(L)$. 
\begin{lemma}\label{chap6thein}
$\theta_{L, X^{\ast}}$ and $\theta_{X^{\ast}, L}$ define a pair of inverse intertwining operators between $ \Ind_{\Ha(L)}^{\Ha^{\pm}(W)} \psi_L$ and  $\Ind_{X^{\ast} \times \R}^{\Ha^{\pm}(W)} \psi_{X^{\ast}}$.
\end{lemma}
\begin{proof}
Let $g=(\epsilon', w',t')\in \Ha^{\pm}(W)$, with $w'=x'+x^{\ast'} \in X\oplus X^{\ast}$. 
\begin{align*}
&\theta_{L, X^{\ast}}(\Pi_{X^{\ast},\psi}(g)f')([\epsilon,w])\\
&=\sum_{l\in L\cap X}\Pi_{X^{\ast},\psi}(g)(f')([\epsilon, x+l])\psi(\epsilon\langle l, w\rangle) \psi(\epsilon\tfrac{\langle x, x^{\ast}\rangle}{2}) \\
&=\sum_{l\in L\cap X}f'([\epsilon,  x+l][\epsilon', w',t'])\psi(\epsilon\langle l, w\rangle) \psi(\epsilon\tfrac{\langle x, x^{\ast}\rangle}{2})\\
&=\sum_{l\in L\cap X}f'([\epsilon\epsilon', x+l+w', t'+\tfrac{\langle x+l, w'\rangle}{2}] )\psi(\epsilon\langle l, w\rangle) \psi(\epsilon\tfrac{\langle x, x^{\ast}\rangle}{2})\\
&=\sum_{l\in L\cap X}\psi(\epsilon  \epsilon' t'+\epsilon \epsilon' \langle x+l, w'\rangle)\psi(\epsilon \epsilon'  \tfrac{\langle x', w'\rangle}{2})f'([\epsilon \epsilon',x+l+x'])\psi(\epsilon\langle l, w\rangle) \psi(\epsilon\tfrac{\langle x, x^{\ast}\rangle}{2})\\
&=\sum_{l\in L\cap X}\psi(\epsilon  \epsilon' t'+\epsilon \epsilon' \langle x, w'\rangle)\psi(\epsilon \epsilon'  \tfrac{\langle x', w'\rangle}{2})f'([\epsilon \epsilon', x+l+x'])\psi(\epsilon\epsilon'\langle l, \epsilon'x^{\ast}+x^{\ast'} \rangle) \psi(\epsilon\epsilon'\tfrac{\langle x, \epsilon'x^{\ast}\rangle}{2})\\
&=\sum_{l\in L\cap X}\psi(\epsilon  \epsilon' t'+\epsilon \epsilon' [\tfrac{\langle x, x^{\ast'}\rangle}{2}-\tfrac{\langle x', \epsilon' x^{\ast}\rangle}{2}])f'([\epsilon \epsilon',x+l+x'])\psi(\epsilon\epsilon'\langle l, \epsilon'x^{\ast}+x^{\ast'} \rangle) \psi(\epsilon\epsilon'\tfrac{\langle x+x', x^{\ast'}+\epsilon'x^{\ast}\rangle}{2})\\
&=\psi(\epsilon  \epsilon' t'+\epsilon \epsilon' \tfrac{\langle w^{s(\epsilon')},w'\rangle}{2})\theta_{L, X^{\ast}}(f')([\epsilon \epsilon', w^{s(\epsilon')}+w'])\\
&=\theta_{L, X^{\ast}}(f')([\epsilon,w][\epsilon', w',t'])\\
&=\Pi_{L, \psi}(g)[\theta_{L, X^{\ast}}(f')]([\epsilon,w]).
\end{align*}
\begin{align*}
&\theta_{X^{\ast}, L}(\Pi_{L, \psi}(g)f)([\epsilon,y])\\
&=  \int_{X^{\ast}/X^{\ast}\cap L} f([1, (\dot{y}^{\ast},0)]\cdot [\epsilon,(y, 0)] \cdot [\epsilon',( w',t')]) d\dot{y}^{\ast}\\
&= \int_{X^{\ast}/X^{\ast}\cap L} f([1, (\dot{y}^{\ast},0)]\cdot [\epsilon\epsilon',(y+w', t'+\tfrac{\langle y, w'\rangle}{2})]) d\dot{y}^{\ast}\\
&= \int_{X^{\ast}/X^{\ast}\cap L} f\Big([1, (\dot{y}^{\ast},0)] \cdot [1,(\epsilon\epsilon' x^{\ast'}, \epsilon\epsilon' (t'+  \langle y, x^{\ast'}\rangle + \tfrac{\langle x', x^{\ast'}\rangle}{2}))]\cdot [\epsilon\epsilon',y+x']\Big) d\dot{y}^{\ast}\\
&=\psi( \epsilon\epsilon' (t'+  \langle y, x^{\ast'}\rangle + \tfrac{\langle x', x^{\ast'}\rangle}{2})) \theta_{X^{\ast}, L}(f)([\epsilon\epsilon',y+x'])\\
&=\Pi_{X^{\ast}, \psi}(g)\theta_{X^{\ast}, L}(f)([\epsilon,y]).
\end{align*}
\end{proof}

For $g\in  \SL_2^{\pm}(\R)$, we can define
     \begin{equation}
     \Pi_{L,\psi}( g)f=\theta_{L, X^{\ast}}[\Pi_{X^{\ast},\psi}(g)(f')]=\theta_{L, X^{\ast}}[\Pi_{X^{\ast},\psi}(g)\theta_{X^{\ast},L}(f)].
     \end{equation}
Under such action, for $g_1, g_2\in \SL_2^{\pm}(\R)$,
     \begin{equation}
     \begin{split}
     \Pi_{L, \psi}(g_1)[\Pi_{L, \psi}(g_2)f]&=\theta_{L, X^{\ast}}[\Pi_{X^{\ast}, \psi}(g_1)\theta_{X^{\ast},L}]([\Pi_{L, \psi}(g_2)f])\\
     &=\theta_{L, X^{\ast}}([\Pi_{X^{\ast}, \psi}(g_1)\theta_{X^{\ast},L}\theta_{L, X^{\ast}}
     [\Pi_{X^{\ast}, \psi}(g_2)\theta_{X^{\ast},L}(f)]])\\
     &=\theta_{L, X^{\ast}}[\Pi_{X^{\ast}, \psi}(g_1)\Pi_{X^{\ast}, \psi}(g_2)\theta_{X^{\ast},L}(f)]\\
     &=\widetilde{C}_{X^{\ast}}(g_1, g_2)\Pi_{L, \psi}(g_1g_2)f.
     \end{split}
     \end{equation}
\subsection{Explicit elements I}\label{Gamma2I1}
Let us consider  $g\in \Sp^{\pm}(L)$. Under the basis $\{ e_1,  e_1^{\ast}\}$ of $\mathcal{L}$, $\Sp^{\pm}(L)\simeq \SL_2^{\pm}(\Z)$.  Let $w=x+x^{\ast} \in W$. \\
Case $1$: $g=u(b)\in \Gamma^{\pm}(2)$, and $2\mid b$.
\begin{equation*}
\begin{split}
\Pi_{L, \psi}(g)f([\epsilon, w])&=\sum_{l\in L\cap X}[\Pi_{X^{\ast},\psi}(g)f']([\epsilon, x+l])\psi(\epsilon\langle l, w\rangle) \psi(\epsilon\tfrac{\langle x, x^{\ast}\rangle}{2}) \\
&=\sum_{l\in L\cap X}\psi(\tfrac{1}{2}\langle (x+l),\epsilon (x+l)b\rangle) f'([\epsilon,x+l])\psi(\epsilon\langle l, w\rangle)\psi(\epsilon\tfrac{\langle x, x^{\ast}\rangle}{2})\\
&= \sum_{l\in L\cap X} \int_{X^{\ast}/X^{\ast}\cap L}  \psi(\tfrac{1}{2}\langle (x+l),\epsilon (x+l)b\rangle) \psi(\epsilon \tfrac{\langle\dot{y}^{\ast}, x+l\rangle}{2})f([\epsilon,x+l+\epsilon\dot{y}^{\ast}])  \psi(\epsilon\langle l, w\rangle)\psi(\epsilon\tfrac{\langle x, x^{\ast}\rangle}{2}) d\dot{y}^{\ast}\\
&=  \sum_{l\in L\cap X} \int_{X^{\ast}/X^{\ast}\cap L}   \psi(\epsilon \tfrac{\langle\dot{y}^{\ast}, x+l\rangle}{2})f([\epsilon,x+l+\epsilon\dot{y}^{\ast}])  \psi(\epsilon\langle l, x^{\ast}+xb\rangle)\psi(\epsilon\tfrac{\langle x, x^{\ast}+ xb\rangle}{2}) d\dot{y}^{\ast}\\
&= f([\epsilon, wg])=f([\epsilon, w]g).
\end{split}
\end{equation*}
Case $2$: $g=h(a)\in \Gamma^{\pm}(2)$, $a=\pm 1$.
\begin{equation}
\begin{split}
&\Pi_{L, \psi}(g)f([\epsilon, w])\\
&=\sum_{l\in L\cap X}[\Pi_{X^{\ast},\psi}(g)f']([\epsilon, x+l])\psi(\epsilon\langle l, w\rangle) \psi(\epsilon\tfrac{\langle x, x^{\ast}\rangle}{2}) \\
&=\sum_{l\in L\cap X}|a|^{1/2}(a,\epsilon)_{\R}f'([\epsilon, (x+l)a])\psi(\epsilon\langle l, w\rangle) \psi(\epsilon\tfrac{\langle x, x^{\ast}\rangle}{2})\\
&= \sum_{l\in L\cap X}\int_{X^{\ast}/X^{\ast}\cap L} \psi(\epsilon\tfrac{\langle\dot{y}^{\ast}, (x+l)a\rangle}{2})f([\epsilon,xa+la+\epsilon\dot{y}^{\ast}]) (a,\epsilon)_\R \psi(\epsilon\langle l, w\rangle) \psi(\epsilon\tfrac{\langle x, x^{\ast}\rangle}{2})d\dot{y}^{\ast}\\
&= \sum_{l\in L\cap X}\int_{X^{\ast}/X^{\ast}\cap L} \psi(\epsilon \tfrac{\langle\dot{y}^{\ast}, (x+l)a\rangle}{2})f([\epsilon,xa+la+\epsilon\dot{y}^{\ast}]) (a,\epsilon)_\R \psi(\epsilon\langle la, wa\rangle) \psi(\epsilon\tfrac{\langle xa, x^{\ast}a\rangle}{2}) d\dot{y}^{\ast}\\
&= \sum_{l\in L\cap X}\int_{X^{\ast}/X^{\ast}\cap L} \psi(\epsilon \tfrac{\langle\dot{y}^{\ast}, xa+l\rangle}{2})f([\epsilon,xa+l+\epsilon\dot{y}^{\ast}]) (a,\epsilon)_\R \psi(\epsilon\langle l, wa\rangle) \psi(\epsilon\tfrac{\langle xa, x^{\ast}a\rangle}{2}) d\dot{y}^{\ast}\\
&=(a,\epsilon)_\R f(\epsilon, wa)=(a,\epsilon)_\R f([\epsilon,w]g).
\end{split}
\end{equation}
Case $3$: $g=h_{-1} \in  \Gamma^{\pm}(2)$. 
\begin{equation}
\begin{split}
\Pi_{L, \psi}(g)f([\epsilon, w])&=\sum_{l\in L\cap X}[\Pi_{X^{\ast},\psi}(g)f']([\epsilon, x+l])\psi(\epsilon\langle l, w\rangle) \psi(\epsilon\tfrac{\langle x, x^{\ast}\rangle}{2}) \\
&=\sum_{l\in L\cap X}f'([-\epsilon, x+l])\psi(\epsilon\langle l, w\rangle) \psi(\epsilon\tfrac{\langle x, x^{\ast}\rangle}{2}) \\
&= \sum_{l\in L\cap X}\int_{X^{\ast}/X^{\ast}\cap L} \psi(-\epsilon\tfrac{\langle\dot{y}^{\ast}, x+l\rangle}{2})f([-\epsilon,x+l-\epsilon\dot{y}^{\ast}]) \psi(\epsilon\langle l, w\rangle) \psi(\epsilon\tfrac{\langle x, x^{\ast}\rangle}{2})d\dot{y}^{\ast}\\
&= \sum_{l\in L\cap X}\int_{X^{\ast}/X^{\ast}\cap L} \psi(-\epsilon\tfrac{\langle\dot{y}^{\ast}, x+l\rangle}{2})f([-\epsilon,x+l-\epsilon\dot{y}^{\ast}]) \psi(-\epsilon\langle l, -x^{\ast}\rangle) \psi(-\epsilon\tfrac{\langle x, -x^{\ast}\rangle}{2})d\dot{y}^{\ast}\\
&=f([-\epsilon, x-x^{\ast}])\\
&= f([\epsilon, w]g).
\end{split}
\end{equation}
Case $4$: $g=\omega\in  \SL_2(\Z)$. 
\[\nu(-1,g)=(1, -1)_\R \gamma(-1, \psi^{\tfrac{1}{2}})^{-1}=\gamma(-1, \psi^{\tfrac{1}{2}})^{-1}=i, \quad\quad \nu(-1,g)^{-1}=-i.\]
\begin{equation}\label{chap6eq3}
\begin{split}
&\Pi_{L, \psi}(g)f([\epsilon, w])\\
&=\sum_{l\in L\cap X}[\Pi_{X^{\ast},\psi}(g)f']([\epsilon, x+l])\psi(\epsilon\langle l, w\rangle)\psi(\epsilon\tfrac{\langle x, x^{\ast}\rangle}{2}) \\
&=\sum_{l\in L\cap X} \int_{X^{\ast}} \nu(\epsilon,g) f'([\epsilon, y^{\ast}\omega^{-1}]) \psi(\langle \epsilon(x+l), y^{\ast}\rangle)\psi(\epsilon\langle l, w\rangle)\psi(\epsilon\tfrac{\langle x, x^{\ast}\rangle}{2})dy^{\ast}\\
&=\sum_{l\in L\cap X} \int_{X^{\ast}}  \nu(\epsilon,g) f'([\epsilon, y^{\ast}\omega^{-1}]) \psi(\epsilon\langle x, y^{\ast}\rangle) \psi(\epsilon\langle l, x^{\ast}+y^{\ast}\rangle)\psi(\epsilon\tfrac{\langle x, x^{\ast}\rangle}{2})dy^{\ast}\\
&=\sum_{l\in L\cap X} \int_{X^{\ast}}  \int_{X^{\ast}/X^{\ast}\cap L} \nu(\epsilon,g) f([\epsilon, y^{\ast}\omega^{-1}+\epsilon\dot{z}^{\ast}])\psi(\epsilon\tfrac{\langle  \dot{z}^{\ast}, y^{\ast}\omega^{-1}\rangle}{2}) \psi(\epsilon\langle x, y^{\ast}\rangle) \psi(\epsilon\langle l, x^{\ast}+y^{\ast}\rangle)\psi(\epsilon\tfrac{\langle x, x^{\ast}\rangle}{2})dy^{\ast}d\dot{z}^{\ast}\\
&= \sum_{l^{\ast}\in L\cap X^{\ast}}\int_{X^{\ast}/X^{\ast}\cap L} \nu(\epsilon,g)f([\epsilon, (-x^{\ast}+l^{\ast})\omega^{-1}+\epsilon\dot{z}^{\ast}])\psi(\epsilon\tfrac{\langle \dot{z}^{\ast}, (-x^{\ast}+l^{\ast})\omega^{-1} \rangle}{2}) \psi(\epsilon\langle x, -x^{\ast}+l^{\ast}\rangle) \psi(\epsilon\tfrac{\langle x, x^{\ast}\rangle}{2})d\dot{z}^{\ast}\\
&= \sum_{l\in L\cap X}\int_{X^{\ast}/X^{\ast}\cap L} \nu(\epsilon,g)f([\epsilon,-x^{\ast}\omega^{-1}-l+\epsilon\dot{z}^{\ast}])\psi(\epsilon\tfrac{\langle \dot{z}^{\ast}, -x^{\ast}\omega^{-1}-l \rangle}{2}) \psi(\epsilon\langle x, -x^{\ast}+l\omega^{-1}\rangle) \psi(\epsilon\tfrac{\langle x, x^{\ast}\rangle}{2})d\dot{z}^{\ast}\\
&= \sum_{l\in L\cap X}\int_{X^{\ast}/X^{\ast}\cap L} \nu(\epsilon,g)f([\epsilon,x^{\ast}\omega+l+\epsilon\dot{z}^{\ast}])\psi(\epsilon\tfrac{\langle\dot{z}^{\ast},  x^{\ast}\omega+l\rangle}{2}) \psi(\epsilon\langle l, x\omega\rangle) \psi(-\epsilon\tfrac{\langle x, x^{\ast}\rangle}{2})d\dot{z}^{\ast}\\
&= \sum_{l\in L\cap X}\int_{X^{\ast}/X^{\ast}\cap L}\nu(\epsilon,g) f([\epsilon,x^{\ast}\omega+l+\epsilon\dot{z}^{\ast}])\psi(\epsilon\tfrac{\langle \dot{z}^{\ast}, x^{\ast}\omega+l\rangle}{2}) \psi(\epsilon\langle l, x\omega\rangle) \psi(\epsilon\tfrac{\langle  x^{\ast}\omega, x\omega\rangle}{2})d\dot{z}^{\ast}\\
&=\nu(\epsilon,g)f([\epsilon,wg])=\nu(\epsilon,g)f([\epsilon,w]g).
\end{split}
\end{equation}
Case $5$: $g=\omega^{-1}\in  \SL_2(\Z)$.
\[ \nu(-1,\omega)=i, \qquad \nu(-1,\omega^{-1})=(-1, -1)_\R \gamma(-1, \psi^{\tfrac{1}{2}})^{-1}=-i.\]
\begin{equation}\label{chap6eq9}
\begin{split}
\Pi_{L,\psi}(g)f([\epsilon, w])
&=\Pi_{L,\psi}(-\omega)f([\epsilon, w])\\
&=\widetilde{C}_{X^{\ast}}( \omega,-1)^{-1}\Pi_{L,\psi}(\omega)\Big(\Pi_{L,\psi}(-1)f\Big)([\epsilon, w])\\
&=\nu(\epsilon,\omega)\Big(\Pi_{L,\psi}(-1)f\Big) ([\epsilon,w]\omega)\\
&=\nu(\epsilon,\omega)(-1, \epsilon)_{\R} f ([\epsilon,w]g)\\
&=\nu(\epsilon,g)f ([\epsilon,w]g).
\end{split}
\end{equation}
Case $6$: $g=u_-(c) =\omega^{-1} u(-c) \omega\in  \Gamma^{\pm}(2)$ with $2\mid c$. 
\begin{equation}\label{chap6eq4}
\begin{split}
\Pi_{L,\psi}(g)f([\epsilon,w])&=\Pi_{L,\psi}(\omega^{-1} u(-c) \omega) f([\epsilon,w])\\
&=\widetilde{C}_{X^{\ast}}(\omega^{-1}, u(-c))^{-1}\widetilde{C}_{X^{\ast}}(\omega^{-1}u(-c), \omega)^{-1}\Pi_{L,\psi}(\omega^{-1})\Pi_{L,\psi}(u(-c)) \Pi_{L,\psi}(\omega) f([\epsilon,w])\\
&=\widetilde{C}_{X^{\ast}}(\omega^{-1}u(-c), \omega)^{-1}\Pi_{L,\psi}(\omega^{-1})\Pi_{L,\psi}(u(-c)) \Pi_{L,\psi}(\omega) f([\epsilon,w])\\
&=\widetilde{C}_{X^{\ast}}(\omega^{-1}u(-c), \omega)^{-1} \nu(\epsilon, \omega^{-1})\nu(\epsilon, \omega)f([\epsilon,wg])\\
&=\widetilde{c}_{X^{\ast}}(\omega^{-1}u(-c), \omega)^{-1} f([\epsilon,wg])\\
&=e^{\tfrac{\pi i \sgn(c)}{4}}f([\epsilon,wg]).
\end{split}
\end{equation}
\subsection{Explicit elements II }\label{Gamma2I2}
 We focus on the other elements 
$g\in \GL_{2}(\Q) \setminus \Gamma(2)$. Let $w=x e+x^{\ast} e^{\ast}\in W$.

Case 7: $g=u(b)=u(1) \in \SL_{2}(\Z)$. 
\begin{equation*}\label{ui}
\begin{split}
& \Pi_{L,\psi}(g)f([\epsilon, w])\\
&=\sum_{l\in L\cap X} \int_{X^{\ast}/X^{\ast}\cap L} \psi(\tfrac{1}{2}\langle (x+l),\epsilon (x+l)b\rangle) \psi(\epsilon \tfrac{\langle\dot{y}^{\ast}, x+l\rangle}{2})f([\epsilon,x+l+\epsilon\dot{y}^{\ast}])  \psi(\epsilon\langle l, w\rangle)\psi(\epsilon\tfrac{\langle x, x^{\ast}\rangle}{2}) d\dot{y}^{\ast}\\
&=\sum_{l\in L\cap X} \int_{X^{\ast}/X^{\ast}\cap L} \psi(\epsilon\tfrac{1}{2}\langle l,lb\rangle) f([\epsilon,x+l+\epsilon\dot{y}^{\ast}]) 
\psi(\epsilon\tfrac{\langle\dot{y}^{\ast}, x+l\rangle}{2}) \psi(\epsilon\langle l, x^{\ast}+xb)\psi(\epsilon\tfrac{\langle x, x^{\ast}+xb\rangle}{2})d \dot{y}^{\ast}\\
&=\psi(-\tfrac{1}{4}\epsilon x)\sum_{l\in L\cap X} \int_{X^{\ast}/X^{\ast}\cap L} f([\epsilon,x+l+\epsilon\dot{y}^{\ast}]) \psi(\epsilon\tfrac{\langle\dot{y}^{\ast}, x+l\rangle}{2})  \psi(\epsilon\langle l, x^{\ast}+xb+\tfrac{1}{2} e^{\ast}\rangle)\psi(\epsilon\tfrac{\langle x, x^{\ast}+xb+\tfrac{1}{2} e^{\ast}\rangle}{2})d \dot{y}^{\ast}\\
&=\psi(-\tfrac{1}{4}\epsilon x)f([\epsilon, x+x^{\ast}+xb+\tfrac{1}{2} e^{\ast}] )\\
&=\psi(-\tfrac{1}{4}\epsilon x)f([\epsilon, wg+\tfrac{1}{2} e^{\ast}])\\
&=\psi(\tfrac{1}{4}\epsilon x)f([\epsilon, wg-\tfrac{1}{2} e^{\ast}]).
\end{split}
\end{equation*} 
Case 8: $g=u_-(-1)=\omega^{-1} u(1) \omega\in \SL_{2}(\Z)$, $w=(x,x^{\ast})$, $w\omega^{-1}=(x,x^{\ast})\begin{pmatrix}0 & 1 \\ -1 & 0\end{pmatrix}=(-x^{\ast}, x)$.
\begin{equation*}\label{ui-}
\begin{split}
 \Pi_{L,\psi}(g)f( [\epsilon, w])
 &=\Pi_{L,\psi}(\omega^{-1} u(1) \omega)f( [\epsilon, w])\\
 &=\widetilde{c}_{X^{\ast}}(\omega^{-1}, u(1))^{-1}\widetilde{c}_{X^{\ast}}(\omega^{-1}u(1), \omega)^{-1}\Pi_{L,\psi}(\omega^{-1})\Pi_{L,\psi}(u(1)) \Pi_{L,\psi}(\omega) f([\epsilon, w])\\
&=\widetilde{c}_{X^{\ast}}(\omega^{-1}u(1), \omega)^{-1}\Pi_{L,\psi}(\omega^{-1})\Pi_{L,\psi}(u(1)) \Pi_{L,\psi}(\omega) f([\epsilon, w])\\
&=\widetilde{c}_{X^{\ast}}(\omega^{-1}u(1), \omega)^{-1}[\Pi_{L,\psi}(u(1)) \Pi_{L,\psi}(\omega)] f([\epsilon, w\omega^{-1}])\nu(\epsilon,\omega^{-1})\\
&=\widetilde{c}_{X^{\ast}}(\omega^{-1}u(1), \omega)^{-1}\psi(\tfrac{1}{4}\epsilon x^{\ast})[ \Pi_{L,\psi}(\omega)] f([\epsilon, w\omega^{-1}u(1)+\tfrac{1}{2} e^{\ast}])\nu(\epsilon,\omega^{-1})\\
&=\widetilde{c}_{X^{\ast}}(\omega^{-1}u(1), \omega)^{-1}\psi(\tfrac{1}{4}\epsilon x^{\ast}) f([\epsilon,wg+\tfrac{1}{2} e])\nu(\epsilon,\omega)\nu(\epsilon,\omega^{-1})\\
&=e^{-\tfrac{\pi i}{4}}\psi(\tfrac{1}{4}\epsilon x^{\ast}) f([\epsilon, wg+\tfrac{1}{2} e])\\
&=m_{X^{\ast}}(g)^{-1} \psi(\tfrac{1}{4}\epsilon x^{\ast}) f([\epsilon, wg+\tfrac{1}{2} e]).
 \end{split}
\end{equation*}
Case $8'$: $g=u_-(-N^2)$, for $(N,2)=1$. Then $wu_-(1-N^2)=(x,x^{\ast})\begin{pmatrix}1 & 0\\ 1-N^2 &1\end{pmatrix}=(x+x^{\ast}-N^2 x^{\ast}, x^{\ast})$.
\begin{equation*}\label{ui'-}
\begin{split}
\Pi_{L,\psi}(g)f( [\epsilon, w])&=\widetilde{c}_{X^{\ast}}(u_-(1-N^2), u_-(-1))^{-1}\cdot \Pi_{L,\psi}\big(u_-(1-N^2)\big)\Pi_{L,\psi}\big(u_-(-1)\big)f( [\epsilon, w])\\
&=e^{\tfrac{\pi i}{4}} \cdot e^{-\tfrac{\pi i}{4}} \Pi_{L,\psi}(u_-(-1))f( [\epsilon, wu_-(1-N^2)])\\
&=e^{-\tfrac{\pi i}{4}}\psi(\tfrac{1}{4}\epsilon x^{\ast}) f([\epsilon, wg+\tfrac{1}{2} e]).
 \end{split}
\end{equation*}
Case 9: $g=h(a)=\begin{pmatrix} \tfrac{1}{n} &  0\\ 0& n\end{pmatrix} \in \SL_{2}(\Q)$, for $n\in \N$ and $a=\tfrac{1}{n}$.
\begin{equation*}
\begin{split}
&\Pi_{L,\psi}(g)f([\epsilon, w])\\
&=\sum_{l\in L\cap X}\Pi_{X^{\ast},\psi}(g)(f')([\epsilon, x+l])\psi(\epsilon\langle l, w\rangle) \psi(\epsilon\tfrac{\langle x, x^{\ast}\rangle}{2}) \\
&=\sum_{l\in L\cap X}|a|^{1/2}(a,\epsilon)_{\R}f'([\epsilon, (x+l)a])\psi(\epsilon\langle l, w\rangle) \psi(\epsilon\tfrac{\langle x, x^{\ast}\rangle}{2})\\
&= \sum_{l\in L\cap X}\int_{X^{\ast}/X^{\ast}\cap L} |a|^{1/2}\psi(\epsilon \tfrac{\langle\dot{y}^{\ast}, (x+l)a\rangle}{2})f([\epsilon,xa+la+\epsilon\dot{y}^{\ast}]) (a,\epsilon)_\R \psi(\epsilon\langle l, w\rangle) \psi(\epsilon\tfrac{\langle x, x^{\ast}\rangle}{2})d\dot{y}^{\ast}\\
&= \sum_{l\in L\cap X}\int_{X^{\ast}/X^{\ast}\cap L}|a|^{1/2} \psi(\epsilon \tfrac{\langle\dot{y}^{\ast}, (x+l)a\rangle}{2})f([\epsilon,xa+la+\epsilon\dot{y}^{\ast}]) (a,\epsilon)_\R \psi(\epsilon\langle la, wa^{-1}\rangle) \psi(\epsilon\tfrac{\langle xa^{-1}, x^{\ast}a\rangle}{2}) d\dot{y}^{\ast}\\
&= \sum_{k=1}^{n}\sum_{l\in L\cap X}\int_{X^{\ast}/X^{\ast}\cap L} |a|^{1/2} \psi(\epsilon \tfrac{\langle\dot{y}^{\ast}, \tfrac{1}{n}x+l+\tfrac{k}{n}\rangle}{2})f([\epsilon,\tfrac{1}{n}x+l+\tfrac{k}{n}+\epsilon\dot{y}^{\ast}]) (\tfrac{1}{n},\epsilon)_\R \psi(\epsilon\langle l+\tfrac{k}{n}, nw\rangle) \psi(\epsilon\tfrac{\langle  \tfrac{1}{n}  x, nx^{\ast}\rangle}{2}) d\dot{y}^{\ast}\\
&= \sum_{k=1}^{n}\sum_{l\in L\cap X}\int_{X^{\ast}/X^{\ast}\cap L} |a|^{1/2} \psi(\epsilon \tfrac{\langle\dot{y}^{\ast}, \tfrac{1}{n}x+l+\tfrac{k}{n}\rangle}{2})f([\epsilon,\tfrac{1}{n}x+l+\tfrac{k}{n}+\epsilon\dot{y}^{\ast}]) \psi(\epsilon \tfrac{\langle \tfrac{k}{n}, nw\rangle}{2})  \psi(\epsilon\langle l, nw\rangle) \psi(\epsilon\tfrac{\langle  \tfrac{1}{n}  x +\tfrac{k}{n}, nx^{\ast}\rangle}{2}) d\dot{y}^{\ast}\\
&=|n|^{-1/2} \sum_{k=1}^{n}f([\epsilon, w]g+\tfrac{k}{n}e)\psi(\tfrac{1}{2} \epsilon kx^{\ast}).
\end{split}
\end{equation*} 
Case 10: $g=h(m)=\begin{pmatrix} m&  0\\ 0& \tfrac{1}{m}\end{pmatrix} \in \SL_{2}(\Q)$, for  $m\in \N$. Then $g=\omega^{-1} h(m^{-1}) \omega$.
\begin{equation}
\begin{split}
&\Pi_{L,\psi}(g)f(\epsilon, w)\\
&=\Pi_{L,\psi}(\omega^{-1} h(m^{-1}) \omega)f([\epsilon, w])\\
&=\widetilde{C}_{X^{\ast}}(\omega^{-1}, h(m^{-1})\omega)^{-1}\widetilde{C}_{X^{\ast}}(h(m^{-1}), \omega)^{-1}\Pi_{L,\psi}(\omega^{-1})\Pi_{L,\psi}( h(m^{-1})\Pi_{L,\psi}( \omega)f([\epsilon, w])\\
&=\nu(\epsilon,\omega^{-1})\Pi_{L,\psi}( h(m^{-1}))\Pi_{L,\psi}( \omega)f([\epsilon,w\omega^{-1}])\\
&=\nu(\epsilon,\omega^{-1})|m|^{-1/2} \sum_{k=1}^{m}\Pi_{L,\psi}( \omega)f([\epsilon,w]\omega^{-1}h(m^{-1}) +\tfrac{k}{m}e)\psi(\tfrac{1}{2} \epsilon kx)
\\
&=\nu(\epsilon,\omega^{-1})\nu(\epsilon,\omega)|m|^{-1/2} \sum_{k=1}^{m} f([\epsilon,w]g -\tfrac{k}{m}e^{\ast})\psi(\tfrac{1}{2} \epsilon kx)\\
&=|m|^{-1/2} \sum_{k=1}^{m} f([\epsilon,w]g -\tfrac{k}{m}e^{\ast})\psi(\tfrac{1}{2} \epsilon kx).
\end{split}
\end{equation} 
Case 11: $g=\begin{pmatrix}\tfrac{m}{n}&  0\\ 0& \tfrac{n}{m}\end{pmatrix} \in \SL_{2}(\Q)$, for  $m,n\in \N$, and $(n,m)=1$.
\begin{equation}
\begin{split}
&\Pi_{L,\psi}(g)f([\epsilon, w])\\
&=\widetilde{C}_{X^{\ast}}(h(m), h(n^{-1}))^{-1}\Pi_{L,\psi}(h(m) )\Pi_{L,\psi}(h(n^{-1}) )f([\epsilon, w])\\
&=|m|^{-1/2} \sum_{k=1}^{m}\Big[ \Pi_{L,\psi}(h(n^{-1}) )f\Big]([\epsilon,w]h(m) -\tfrac{k}{m}e^{\ast})\psi(\tfrac{1}{2} \epsilon kx)\\
&=|mn|^{-1/2} \sum_{k=1}^{m}\sum_{l=1}^{n}f([\epsilon, w]g-\tfrac{kn}{m}e^{\ast}+\tfrac{l}{n}e)\psi\Big(\tfrac{1}{2} \epsilon l(\tfrac{x^{\ast}}{m}-\tfrac{k}{m})\Big)\psi(\tfrac{1}{2} \epsilon kx)\\
&=|mn|^{-1/2} \sum_{k=1}^{m}\sum_{l=1}^{n}f([\epsilon, w]g-\tfrac{k}{m}e^{\ast}+\tfrac{l}{n}e)\psi\Big(\tfrac{1}{2} \epsilon \tfrac{l}{m}x^{\ast} +\tfrac{1}{2} \epsilon \tfrac{k}{n}x-\tfrac{1}{2} \epsilon\tfrac{kl}{nm}\Big).
\end{split}
\end{equation} 
Case 12: $g=u(\tfrac{2q}{N})=\begin{pmatrix}1&  \tfrac{2q}{N}\\ 0& 1\end{pmatrix} \in \SL_{2}(\Q)$, for $(q,N)=1$ and a positive integer number $N$.\\
\[u(\tfrac{2q}{N})=h(N^{-1}) u(2qN) h(N).\]
\begin{equation}
\begin{split}
&\Pi_{L,\psi}(g)f([\epsilon, w])\\
&=\widetilde{C}_{X^{\ast}}(h(N^{-1}), u(2qN) h(N))^{-1}\widetilde{C}_{X^{\ast}}(u(2qN) , h(N))^{-1}\Pi_{L,\psi}(h(N^{-1}))\Pi_{L,\psi}(u(2qN))\Pi_{L,\psi}(h(N))f([\epsilon, w])\\
&=\Pi_{L,\psi}(h(N^{-1}))\Pi_{L,\psi}(u(2qN))\Pi_{L,\psi}(h(N))f([\epsilon, w])\\
&=|N|^{-1/2} \sum_{k=1}^{N}\Pi_{L,\psi}(u(2qN))\Pi_{L,\psi}(h(N))f([\epsilon, w]h(N^{-1})+\tfrac{k}{N}e)\psi(\tfrac{1}{2} \epsilon kx^{\ast})\\
&=|N|^{-1/2} \sum_{k=1}^{N}\Pi_{L,\psi}(h(N))f\Big([\epsilon, w]h(N^{-1})u(2qN)+\tfrac{k}{N}e+2qk e^{\ast}\Big)\psi(\tfrac{1}{2} \epsilon kx^{\ast})\\
&=|N|^{-1}\sum_{k=1}^{N} \sum_{l=1}^{N} f([\epsilon, w]h(N^{-1})u(2qN)h(N)+ke+2qk\tfrac{1}{N} e^{\ast} -\tfrac{l}{N}e^{\ast})\psi(\tfrac{1}{2} \epsilon l\tfrac{x}{N})\psi(\tfrac{1}{2} \epsilon l\tfrac{k}{N})\psi(\tfrac{1}{2} \epsilon kx^{\ast})\\
&=|N|^{-1}\sum_{k=1}^{N} \sum_{l=1}^{N} f([\epsilon, w]g+\tfrac{2qk}{N} e^{\ast} -\tfrac{l}{N}e^{\ast})\psi(-\epsilon  \tfrac{qkx}{N})\psi(-\epsilon  \tfrac{qk^2}{N})\psi(\epsilon \tfrac{lx}{2N})\psi( \epsilon \tfrac{lk}{N}).
\end{split}
\end{equation} 
Case 13: $g=\begin{pmatrix}a&  0\\ 0& a^{-1}\end{pmatrix} \in P_{>0}(\R)$. In this case, we choose positive  rational numbers $\tfrac{m_i}{n_i}$ approximating $a$. Note that the restriction of $\Pi_{L,\psi}$ on $P_{>0}(\R)$ is  continuous.  (cf. Casselman' paper \cite{Ca})
\begin{equation}
\Pi_{L,\psi}(g)f([\epsilon, w])=\lim_{i} \Pi_{L,\psi}(\tfrac{m_i}{n_i}) f([\epsilon, w]).
\end{equation} 
Case 14: $g_p=\begin{pmatrix}0&  1\\-p& 0\end{pmatrix} \in \GL_{2}(\Q)$, for  a positive prime $p$. Let $\sqrt{p}=\lim_{i} \tfrac{n_i}{m_i}$, for coprime natural number pair $(n_i,m_i)$. Then:
$$g= \begin{pmatrix}\sqrt{p}&  0\\0& \sqrt{p}\end{pmatrix} \cdot  \begin{pmatrix}0& 1\\-1& 0\end{pmatrix}  \cdot \begin{pmatrix} \sqrt{p}  & 0\\0&\sqrt{p}^{-1} \end{pmatrix}.$$
\begin{equation}
\begin{split}
&\Pi_{L,\psi}(g)f([\epsilon, w])\\
&=\widetilde{C}_{X^{\ast}}( \begin{pmatrix}0& 1\\-1& 0\end{pmatrix} , \begin{pmatrix} \sqrt{p}  & 0\\0&\sqrt{p}^{-1} \end{pmatrix})^{-1}\Pi_{L,\psi}( \begin{pmatrix}0& 1\\-1& 0\end{pmatrix}  )\Pi_{L,\psi}( \begin{pmatrix} \sqrt{p}  & 0\\0&\sqrt{p}^{-1} \end{pmatrix})f([\epsilon, w])\\
&=\nu(\epsilon,\omega^{-1})\Pi_{L,\psi}( \begin{pmatrix} \sqrt{p}  & 0\\0&\sqrt{p}^{-1} \end{pmatrix})f([\epsilon,w]\omega^{-1}).
\end{split}
\end{equation} 
\subsection{Application}
For $f\in \mathcal{H}^{\pm}(L)$, we write $f(1, w)=f_1(w)$ and $f(-1, w)=f_2(w)$. Let: $$\Gamma(4,2)=\{ \begin{pmatrix} a & b \\c& d \end{pmatrix} \mid  a\equiv 1\equiv d (\bmod  4), b\equiv 0\equiv c(\bmod 2)\}.$$  
Then $\Gamma_{\theta}=\langle \omega\rangle\Gamma(4,2)\simeq \langle \omega\rangle  \ltimes \Gamma(4,2)$.
\subsubsection{} Note that every element  of $\Gamma_{\theta}$ can be expressed as 
$r $ or $r \omega$, for some $r \in \Gamma(2)$.  By the above cases (1)---(6), we have:
\begin{equation}\label{chapeq8}
 \Pi_{L, \psi}(r)f_1(w)= \epsilon_r f_1(wr).
 \end{equation}
for some $\epsilon_r\in \mu_8$.  Hence:
\begin{lemma}\label{ramma1SP}
The restriction of the cocycle $[\widetilde{c}_{X^{\ast}}]$ to the subgroup $\Gamma_{\theta}$ is trivial, with an explicit trivialization
 $$\epsilon: \Gamma_{\theta} \to \mu_8; r\longmapsto \epsilon_r,$$ 
 such that $\widetilde{c}_{X^{\ast}}(r_1, r_2) = \epsilon(r_1) \epsilon(r_2) \epsilon(r_1 r_2)^{-1}$, for  all $r_i \in \Gamma_{\theta}$.
\end{lemma}
\begin{lemma}
 $\widetilde{\beta}(r)\epsilon_r=1$, for all $r\in \Gamma_{\theta}$.
 \end{lemma}
 \begin{proof}
 For any $r_1,r_2\in \Gamma_{\theta}$, we have the following relationship:
  \begin{equation}\label{trivramma}
 \widetilde{c}_{X^{\ast}}(r_1,r_2)=\widetilde{\beta}(r_1)^{-1}\widetilde{\beta}(r_2)^{-1}\widetilde{\beta}(r_1r_2)=\epsilon_{r_1}\epsilon_{r_1}\epsilon_{r_1r_2}^{-1}.
 \end{equation} Thus, 
the map  $r \to \widetilde{\beta}(r)\epsilon_r$ defines  a character of $\Gamma_{\theta}$. Since $\Gamma_{\theta}$ is  generated by $ u(2)$, $u(-2)$ and $\omega$,  the result follows.
 \end{proof}
 Let $\widetilde{r}=[r, \epsilon] \in \widetilde{\Gamma}_{\theta}$. According to Eq. (\ref{chapeq8}), we have:
 \begin{equation}\label{chapeq89}
 \Pi_{L,\psi}(\widetilde{r}) f_1( w)= \pi_{\psi}(\widetilde{r})f_1(w)=  \epsilon\widetilde{\beta}(r)^{-1} f_1(wr)=\widetilde{\lambda}(\widetilde{r})f_1(wr).
 \end{equation}
\subsubsection{}  Let us consider  $\Gamma_{\theta}^{\pm}=\langle \omega\rangle \Gamma(2)^{\pm}$. 
Let $r\in \Gamma_{\theta}$.  By the above Cases $1-6$,  $\Pi_{L,\psi}(r) f_2(w)=c_rf_2(wr)$.  So $c_r$ differs from $ \widetilde{\beta}(r)^{-1} $ by a character $\chi_{\omega}$.
By Cases $1$ and $6$, we known that $\chi_{\omega}(u(2))=1=\chi_{\omega}(u_-(2))$. By Case $4$, $\chi_{\omega}(\omega)= \nu(-1, \omega)=\gamma(-1, \psi^{\tfrac{1}{2}})^{-1}=i$. Hence $c_r=\widetilde{\beta}(r)^{-1}\chi_{\omega}(r)$, with $(\chi_{\omega})|_{\Gamma(4,2)}=1$ and $\chi_{\omega}(\omega)=i$.  Let:  $$\lambda^{\epsilon}( r)=\left\{\begin{array}{lc}    \widetilde{\beta}(r)^{-1}  &  \textrm{ if } \epsilon=1,\\ \widetilde{\beta}(r)^{-1} \chi_{\omega}(r) & \textrm{ if } \epsilon=-1. \end{array}\right.$$  
Then: 
$$\Pi_{L,\psi}(r) f([\epsilon, w])= \lambda^{\epsilon}(r) f([\epsilon, wr]).$$
It is known that $\Pi_{L,\psi}(h_{-1}) f([\epsilon, w])=f([-\epsilon, wh_{-1}])$. For any $r \in \Gamma_{\theta}^{\pm}$, let us write $r=h_{\det(r)} (h_{\det(r)}r)$. Then: 
\begin{align*}
 \Pi_{L,\psi}(r) f([\epsilon, w])&= \Pi_{L,\psi}(h_{\det(r)} (h_{\det(r)}r) ) f([\epsilon, w])\\
 &= \widetilde{C}_{X^{\ast}}(h_{\det(r)}, h_{\det(r)}r)^{-1}\Pi_{L,\psi} (h_{\det(r)}r)  f([\epsilon \det(r) , w h_{\det(r)}])\\
 &=\lambda^{\epsilon \det(r)}(  h_{\det(r)}r) f([\epsilon \det(r) , w h_{\det(r)}h_{\det(r)}r])\\
 &=\lambda^{\epsilon\det(r)}( h_{\det(r)}r) f([\epsilon, w] r).
 \end{align*}
\begin{lemma}\label{hminustoh}
 For $\overline{r}=[r, t]\in \overline{\Gamma}_{\theta}$,  $\overline{\lambda}(\overline{r}^{h_{-1}})=\overline{\lambda}(\overline{r})\chi_{\omega}(r) $.
\end{lemma}
\begin{proof}
\[\overline{C}_{X^{\ast}}(h_{-1}, r)\overline{C}_{X^{\ast}}(h_{-1} r, h_{-1} )=\nu_2( 1,1)\overline{c}_{X^{\ast}}(1, r)\nu_2(- 1,r)\overline{c}_{X^{\ast}}(r^{h_{-1}}, 1)= \nu_2(- 1,r)\stackrel{(\ref{chap3alppu2})}{=} \sgn(r,-1),\]
\[[h_{-1},1]\overline{r}[h_{-1},1]=[h_{-1}rh_{-1},\nu_2(- 1,r)t].\]
 For $r=\begin{pmatrix} a& b\\ c& d\end{pmatrix}\in \Gamma(4,2)$, by Lemma \ref{gammar}, we have:
\[\sgn(r,-1)=1, \quad \overline{\lambda}(\overline{r}^{h_{-1}})=\sgn(r,-1)\left(\frac{-2c}{d}\right)\epsilon^{-1}_dt=\left(\frac{2c}{d}\right)\epsilon^{-1}_dt=\overline{\lambda}(\overline{r}).\]
For $r=\omega$,  
\[ \sgn(r,-1)=1,  \quad \overline{\lambda}(\overline{r}^{h_{-1}})=-\epsilon^{-1}_{-1}e^{-\tfrac{i \pi}{4}}t=ie^{-\tfrac{i \pi}{4}}t, \quad\overline{\lambda}(\overline{r})=\epsilon^{-1}_{1}e^{-\tfrac{i \pi}{4}}t=e^{-\tfrac{i \pi}{4}}t.\]
\end{proof}
\section{Intertwining operators for $\SL_2^{\pm}(\R)$ case: Schrödinger model to Fock model}
In this section we fix $\psi=\psi_0$ and $z=z_0$.
By \cite[\wasyparagraph3]{Ta1} or \cite[p.~401]{Sa1} there is a unitary equivalence
\[
\mathcal H(X^*)\simeq L^2(X)\xrightarrow{\;\sim\;} \mathcal H_{F_{z_0}}
\]
of $\Ha(W)$-representations.
We extend it to a unitary equivalence
\[
L^2(\mu_2\times X)\xrightarrow{\;\sim\;} \mathcal H^{\pm}_{F_{z_0}}
\]
of $\Ha^{\pm}(W)$-representations, give two explicit maps, and include full proofs for completeness.
Identify $W=X\oplus X^*\cong\mathbb R\oplus\mathbb R$ and set $ q_{z_0}(v)=e^{-\pi v^2}$.
\begin{itemize}
\item $\theta_{ F_{z_0}, X^{\ast}}: \mathcal{H}^{\pm}(X^{\ast})\simeq L^2(\mu_2 \times X) \longrightarrow \mathcal{H}^{\pm}_{F_{z_0}}$.
  \begin{equation}\label{inter3}
 \begin{split}
\theta_{F_{z_0},X^{\ast}}(f)([\epsilon,  v])&=2^{\tfrac{1}{4}} \int_{\R} e^{2\pi i\big( \tfrac{1}{2} z_0 x^2 +xv\big)}f([\epsilon,x])dx.
    \end{split}
 \end{equation}
 \item  $\theta_{  X^{\ast},F_{z_0}}: \mathcal{H}^{\pm}_{F_{z_0}} \longrightarrow  L^2(\mu_2 \times X) $.
 \begin{equation}\label{inter4}
 \begin{split}
\theta_{  X^{\ast},F_{z_0}}(\phi)([\epsilon, x])&=  2^{\tfrac{1}{4}} \int_{\C}  e^{2\pi i \big[\tfrac{1}{2}z_0x^2-x\overline{v} \big]} e^{-2 \pi (v_y)^2}\phi([\epsilon, v])dv_x dv_y.
  \end{split}
      \end{equation}
      \end{itemize}

 Let us verify that $ \theta_{ F_{z_0},  X^{\ast}}$, $\theta_{  X^{\ast},F_{z_0}}$ both are $\Ha^{\pm}(W)$-intertwining  operators. Let $h=(y,y^{\ast};t)\in \Ha(W)$ and $h_{-1}=\begin{pmatrix} 1 & 0\\0& -1\end{pmatrix}\in \GL_2(\R)$.\\
 1) 
 \begin{align*}
 &\theta_{ F_{z_0}, X^{\ast}} [\Pi_{X^{\ast},\psi }(h) f]([\epsilon,v])\\
 &=2^{\tfrac{1}{4}} \int_{\R} e^{2\pi i \big( \tfrac{1}{2} z_0 x^2 +xv\big)}[\Pi_{X^{\ast},\psi }(h) f]([\epsilon,x])dx\\
 &=2^{\tfrac{1}{4}} \int_{\R} e^{2\pi i  \big( \tfrac{1}{2} z_0 x^2 +xv\big)}\psi(\epsilon  t+\epsilon \langle x,y^{\ast}\rangle+\epsilon \tfrac{1}{2}\langle y, y^{\ast}\rangle) f([\epsilon ,x+y])dx\\
  &=2^{\tfrac{1}{4}} \int_{\R} e^{2\pi i \big( \tfrac{1}{2} z_0 (x-y)^2 +(x-y)v\big)}\psi(\epsilon t+\epsilon \langle x-y,y^{\ast}\rangle+\epsilon \tfrac{1}{2}\langle y, y^{\ast}\rangle) f([\epsilon, x])dx\\
  &=2^{\tfrac{1}{4}}\psi(\epsilon t-\epsilon\tfrac{1}{2}\langle y, y^{\ast}\rangle) \int_{\R} \psi\big(  \tfrac{1}{2} z_0 (x-y)^2 + (x-y)v\big)\psi(\epsilon \langle x,y^{\ast}\rangle) f([\epsilon, x])dx\\
  &=2^{\tfrac{1}{4}}\psi(\epsilon t-\epsilon \tfrac{1}{2} y y^{\ast}+\tfrac{1}{2} z_0 y^2 - yv) \int_{\R}\psi\big( \tfrac{1}{2} z_0x^2 - xz_0y + xv+\epsilon xy^{\ast}\big) f([\epsilon, x])dx;
 \end{align*}
 \begin{align*}
 &\Pi_{F_{z_0}}(h) [\theta_{ F_{z_0}, X^{\ast}} f]([\epsilon,v])\\
 &=\psi(\epsilon t)e^{-2\pi i[\frac{1}{2}(y\overline{z}_0+\epsilon y^{*})y+vy]}[\theta_{ F_{z_0}, X^{\ast}} f]([\epsilon, v+y\overline{z}_0+\epsilon y^{\ast}])\\
&=\psi(\epsilon t)e^{-2\pi i[\frac{1}{2}(y\overline{z}_0+\epsilon y^{*})y+vy]}2^{\tfrac{1}{4}} \int_{\R} e^{2\pi i \big( \tfrac{1}{2} z_0 x^2 +x(v+y\overline{z}_0+\epsilon y^{\ast})\big)} f([\epsilon,x]) dx\\
&=2^{\tfrac{1}{4}} \psi(\epsilon t+\tfrac{1}{2}z_0y^2-\tfrac{1}{2}\epsilon y^{*}y-vy) \int_{\R} e^{2\pi i \big( \tfrac{1}{2} z_0 x^2 +x(v+y\overline{z}_0+\epsilon y^{\ast})\big)} f([\epsilon,x]) dx\\
&= \theta_{ F_{z_0}, X^{\ast}} [\Pi_{ X^{\ast},\psi}(h) f]([\epsilon,v]).
  \end{align*}
2)
\begin{align*}
& \kappa_{z_0}(v-y\overline{z_0}-\epsilon y^{\ast},v-y\overline{z_0}-\epsilon y^{\ast})\\
& \stackrel{v=v_x+iv_y}{=}e^{-2\pi (v_y+y)(v_y+y)}\\
&=\kappa_{z_0}(v,v)e^{2\pi i(-\overline{v} y +v y+iy^2)};
\end{align*}
 \begin{align*}
&\theta_{  X^{\ast},F_{z_0}}([\Pi_{F_{z_0}}(h)\phi])([\epsilon,x])\\
&= 2^{\tfrac{1}{4}}\int_{\C}  e^{2\pi i\big[\tfrac{1}{2}z_0x^2-x\overline{v} \big]} \kappa_{z_0}(v,v)[\Pi_{F_{z_0}}(h)\phi]([\epsilon, v])d_{z_0}(v)\\
&=  2^{\tfrac{1}{4}}\int_{\C} e^{2\pi i\big[\tfrac{1}{2}z_0x^2-x\overline{v} \big]} \kappa_{z_0}(v,v)\psi(\epsilon t)e^{-2\pi i[\frac{1}{2}(y\overline{z_0}+\epsilon y^{*})y+vy]}\phi([\epsilon, v+y\overline{z_0}+\epsilon y^{\ast}])d_{z_0}(v)\\
&=  2^{\tfrac{1}{4}} \psi(\epsilon t) \int_{\C}e^{2\pi i\big[\tfrac{1}{2}z_0x^2-x(\overline{v}-yz_0-\epsilon y^{\ast}) \big]} \kappa_{z_0}(v-y\overline{z_0}-\epsilon y^{\ast},v-y\overline{z_0}-\epsilon y^{\ast})e^{-2\pi i[\frac{1}{2}(y\overline{z_0}+\epsilon y^{*})y+(v-y\overline{z_0}-\epsilon y^{\ast})y]}\phi([\epsilon,v])d_{z_0}(v)\\
&=  2^{\tfrac{1}{4}} \psi(\epsilon t) \int_{\C}e^{2\pi i\big[\tfrac{1}{2}z_0x^2-x(\overline{v}-yz_0-\epsilon y^{\ast}) \big]} \kappa_{z_0}(v-y\overline{z_0}-\epsilon y^{\ast},v-y\overline{z_0}-\epsilon y^{\ast})e^{-2\pi i[-\frac{1}{2}(y\overline{z_0}+\epsilon y^{*})y+vy]}\phi([\epsilon,v])d_{z_0}(v)\\
&=  2^{\tfrac{1}{4}} \psi(\epsilon t) \int_{\C}e^{2\pi i\big[\tfrac{1}{2}z_0x^2-x(\overline{v}-yz_0-\epsilon y^{\ast}) \big]} \kappa_{z_0}(v,v)e^{2\pi i[\frac{1}{2}(yz_0+\epsilon y^{*})y-\overline{v}y]}\phi([\epsilon,v])d_{z_0}(v);
\end{align*}
\begin{align*}
&\Pi_{X^{\ast},\psi}(h) [\theta_{  X^{\ast},F_{z_0}}(\phi)]([\epsilon,x])\\
&=\psi(\epsilon t+\epsilon \langle x,y^{\ast}\rangle+\epsilon \tfrac{1}{2}\langle y, y^{\ast}\rangle) [\theta_{  X^{\ast},F_{z_0}}(\phi)]([\epsilon, x+y])\\
&=\psi(\epsilon t+\epsilon \langle x,y^{\ast}\rangle+\epsilon \tfrac{1}{2}\langle y, y^{\ast}\rangle) \int_{\C} 2^{\tfrac{1}{4}} e^{2\pi i\big[\tfrac{1}{2}\langle x+y,(x+y) z_0\rangle-\langle x+y,\overline{v}\rangle \big]} \kappa_{z_0}(v,v)\phi([\epsilon ,v])d_{z_0}(v)\\
&=2^{\tfrac{1}{4}} \psi(\epsilon t)  \int_{\C} \psi\big( \tfrac{1}{2}z_0x^2+ x(yz_0+\epsilon  y^{\ast}-\overline{v})\big) e^{2\pi i\big(\tfrac{1}{2} y (\epsilon  y^{\ast}) +\tfrac{1}{2}z_0y^2-y\overline{v}\big)} \kappa_{z_0}(v,v)\phi([\epsilon, v])d_{z_0}(v)\\
&=\theta_{  X^{\ast},F_{z_0}}([\Pi_{F_{z_0}}(h)\phi])([\epsilon,x]).
\end{align*}
3) 
\begin{align*}
 &\theta_{ F_{z_0}, X^{\ast}} [\Pi_{X^{\ast},\psi }(h_{-1}) f]([\epsilon,v])\\
 &=2^{\tfrac{1}{4}} \int_{\R} e^{2\pi i \big( \tfrac{1}{2} z_0 x^2 +xv\big)}[\Pi_{X^{\ast},\psi }(h_{-1}) f]([\epsilon,x])dx\\
 &=2^{\tfrac{1}{4}} \int_{\R} e^{2\pi i \big( \tfrac{1}{2} z_0 x^2 +xv\big)}f([-\epsilon,x])dx;
  \end{align*}
   \begin{align*}
& \Pi_{F_{z_0}}(h_{-1}) [\theta_{ F_{z_0}, X^{\ast}} (f)]([\epsilon,v])\\
 &=[\theta_{ F_{z_0}, X^{\ast}} (f)]([-\epsilon,v])\\
 &=2^{\tfrac{1}{4}} \int_{\R} e^{2\pi i\big( \tfrac{1}{2} z_0 x^2 +xv\big)}f([-\epsilon,x])dx\\
 &=\theta_{ F_{z_0}, X^{\ast}} [\Pi_{X^{\ast},\psi }(h_{-1}) f]([\epsilon,v]).
 \end{align*}
 4) \begin{align*}
&\theta_{  X^{\ast},F_{z_0}}([\Pi_{F_{z_0}}(h_{-1})\phi])([\epsilon,x])\\
&= 2^{\tfrac{1}{4}}\int_{\C}  e^{2\pi i\big[\tfrac{1}{2}z_0x^2-x\overline{v} \big]} \kappa_{z_0}(v,v)[\Pi_{F_{z_0}}(h_{-1})\phi]([\epsilon, v])d_{z_0}(v)\\
&= 2^{\tfrac{1}{4}}\int_{\C}  e^{2\pi i\big[\tfrac{1}{2}z_0x^2-x\overline{v} \big]} \kappa_{z_0}(v,v)\phi([-\epsilon, v])d_{z_0}(v);
\end{align*}
\begin{align*}
&\Pi_{X^{\ast},\psi}(h_{-1}) [\theta_{  X^{\ast},F_{z_0}}(\phi)]([\epsilon,x])\\
&=\theta_{  X^{\ast},F_{z_0}}(\phi)([-\epsilon, x])\\
&=  2^{\tfrac{1}{4}} \int_{\C}  e^{2\pi i \big[\tfrac{1}{2}z_0x^2-x\overline{v} \big]} e^{-2 \pi (v_y)^2}\phi([-\epsilon, v])dv_x dv_y\\
&=\theta_{  X^{\ast},F_{z_0}}([\Pi_{F_{z_0}}(h_{-1})\phi])([\epsilon,x]).
\end{align*}
Let us verify that $\theta_{  X^{\ast},F_{z_0}}$ and  $\theta_{  F_{z_0},X^{\ast}}$ are indeed  inverses of  one another. It suffices to show the following:
\begin{align*}
&\theta_{  X^{\ast},F_{z_0}}\Big[\theta_{F_{z_0},X^{\ast}}(f)\Big]([\epsilon,x])\\
&=  2^{\tfrac{1}{4}} \int_{\C}  e^{2\pi i \big[\tfrac{1}{2}z_0x^2-x\overline{v} \big]} e^{-2 \pi (v_y)^2}\Big[\theta_{F_{z_0},X^{\ast}}(f)\Big]([\epsilon, v])dv_x dv_y\\
&=  2^{\tfrac{1}{2}} \int_{\C}  \int_{\R} e^{2\pi i \big[\tfrac{1}{2}z_0x^2-x\overline{v} \big]} e^{-2 \pi (v_y)^2}e^{2\pi i\big( \tfrac{1}{2} z_0 y^2 +yv\big)}f([\epsilon,y])dyd_{z_0}(v) \\  
&= 2^{\tfrac{1}{2}} \int_{\R} e^{2\pi i\big( \tfrac{1}{2} z_0 x^2\big)} e^{2\pi i\big( \tfrac{1}{2} z_0 y^2\big)}f([\epsilon,y]) \bigg[ \int_{\C}  e^{-2\pi i x\overline{v}}  e^{-2 \pi (v_y)^2} e^{2\pi i\big( yv\big)}d_{z_0}(v) \bigg] dy\\
&=  2^{\tfrac{1}{2}} \int_{\R} e^{2\pi i\big( \tfrac{1}{2} z_0 x^2\big)} e^{2\pi i\big( \tfrac{1}{2} z_0 y^2\big)}f([\epsilon,y])\bigg[ \int_{\C}  e^{-2\pi (v_y)^2} e^{-2\pi \big( (x+y) v_y\big)}e^{2\pi i \big( (-x+y) v_x\big)}d_{z_0}(v) \bigg] dy\\
&=2^{\tfrac{1}{2}}\int_{\R} e^{\pi (x  y- \tfrac{1}{2}x^2- \tfrac{1}{2}y^2)}f([\epsilon,y]) dy\cdot \bigg[ \int_{\C}  e^{-2\pi(\tfrac{1}{2}x+\tfrac{1}{2}y+ v_y)(\tfrac{1}{2}x+\tfrac{1}{2}y+v_y)} e^{2\pi i \big( (-x+y) v_x\big)}d_{z_0}(v) \bigg] \\
&= \int_{\R}  e^{\pi (x  y- \tfrac{1}{2}x^2- \tfrac{1}{2}y^2)}f([\epsilon,y]) dy\int_{\R}\bigg[ e^{2\pi i \big( (-x+y) v_x\big)}dv_x\bigg] \\
&=f(x).
\end{align*}
\begin{proposition}\label{Propequ}
For every \(g\in\GL_2(\R)\),
$
\theta_{F_{z_0},X^*}\!\bigl(\Pi_{X^*,\psi}(g)f\bigr)
=\widetilde s(g)^{-1}\Pi_{F_{z_0},\psi}(g)\!\bigl(\theta_{F_{z_0},X^*}(f)\bigr).
$
\end{proposition}
\subsection{ Proof of Proposition \ref{Propequ}}
The identity is already known for $g=h_{-1}$ and $g=\operatorname{diag}(r,r)$ ($r>0$), so we may restrict to $g\in\SL_2(\R)$.
Recall that $\SL_2(\R)$ is generated by $u(b)$, $h(a)$ and $\omega$, and that Lemma~\ref{constants1} gives $\widetilde{s}( g)^{-1}=1$ for these generators.
Below (Sections~\ref{ubxasttoFz01}--\ref{ubxasttoFz03})  we verify the identity for these generators in turn; consequently
\[
\theta_{F_{z_0},X^*}\!\bigl(\pi_{X^*,\psi}(g)f\bigr)
=c_g\,\pi_{F_{z_0},\psi}(g)\!\bigl(\theta_{F_{z_0},X^*}(f)\bigr)
\quad\text{for all }g\in\SL_2(\R)
\]
with some constants $c_g$.
Comparing both sides on $g_1g_2$ yields
\begin{align*}
\theta_{F_{z_0},X^*}\!\bigl(\pi_{X^*,\psi}(g_1g_2)f\bigr)
&=\widetilde c_{X^*}(g_1,g_2)^{-1}\theta_{F_{z_0},X^*}\!\bigl(\pi_{X^*,\psi}(g_1)\pi_{X^*,\psi}(g_2)f\bigr)\\
&=\widetilde c_{X^*}(g_1,g_2)^{-1}c_{g_1}c_{g_2}\alpha^{-1}_{z_0}(g_1,g_2)\pi_{F_{z_0},\psi}(g_1g_2)\!\bigl(\theta_{F_{z_0},X^*}(f)\bigr),
\end{align*}
so
\[
\alpha_{z_0}(g_1,g_2)^{-1}
=\widetilde c_{X^*}(g_1,g_2)\,c_{g_1g_2}\,c_{g_2}^{-1}c_{g_1}^{-1}.
\]
Since these constants $c_{g}$ are uniquely determined by this equality,  Lemma \ref{eightcoz0} and Remark \ref{equalityonab} force $c_g=\widetilde s(g)^{-1}$. Let $f=\big(\theta_{  X^{\ast},F_{z_0}}(\phi)\big)$, for $\phi \in S(\R)\subseteq L^2(\R)$.\\
\subsubsection{}\label{ubxasttoFz01} Let  $g=u(b)$.
\begin{align*}
&\theta_{ F_{z_0}, X^{\ast}}\big(\Pi_{X^{\ast}, \psi}(g)(f)\big)([\epsilon,v])\\
&=2^{\tfrac{1}{4}} \int_{\R} e^{2\pi i\big( \tfrac{1}{2}  z_0 x^2 +xv\big)}\big(\Pi_{X^{\ast}, \psi}(g)f\big)([\epsilon,x])dx\\
&=2^{\tfrac{1}{4}} \int_{\R} e^{2\pi i\big( \tfrac{1}{2} z_0 x^2 +xv\big)}\psi(\tfrac{1}{2} \epsilon bx^2) f([\epsilon, x])dx\\
&=2^{\tfrac{1}{4}} \int_{\R} e^{2\pi i\big( \tfrac{1}{2} z_0 x^2 +xv+\tfrac{1}{2}\epsilon bx^2\big)}f([\epsilon,x])dx\\
&=2^{\tfrac{1}{2}} \int_{\R} \int_{\C}  e^{2\pi i\big( \tfrac{1}{2} z_0 x^2 +xv+\tfrac{1}{2}\epsilon bx^2\big)}e^{2\pi i\big[\tfrac{1}{2}z_0x^2-x\overline{v'} \big]} k_{\widehat{z_0}}(v',v')\phi([\epsilon,v'])d_{\widehat{z_0}}(v')dx\\
&=2^{\tfrac{1}{2}} \int_{\C}  \Big[\int_{\R} e^{2\pi i\big( \tfrac{1}{2} z_0 x^2 +xv+\tfrac{1}{2}\epsilon bx^2\big)}e^{2\pi i\big[\tfrac{1}{2}z_0x^2-x\overline{v'}  \big]} dx\Big]k_{z_0}(v',v')\phi([\epsilon,v'])d_{z_0}(v')\\
&=2^{\tfrac{1}{2}}  \int_{\C}  \Big[\int_{\R} e^{-\pi \big(   (2-\epsilon z_0 b) x^2 \big)}e^{2\pi ix\big(v-\overline{v'}  \big)} dx\Big]k_{z_0}(v',v')\phi([\epsilon,v'])d_{z_0}(v')\\
&\xlongequal{\textrm{ \cite[App.A, Thm.1]{Fo} }} 2^{\tfrac{1}{2}} \int_{\C}  \Big[( 2-\epsilon z_0 b)^{-1/2} e^{-\pi(( 2-\epsilon z_0b)^{-1} (v-\overline{v'})^2)  }\Big]k_{z_0}(v',v')\phi([\epsilon,v'])d_{z_0}(v')\\
&=(\tfrac{ 2i+\epsilon  b}{2i})^{-1/2} \int_{\C}  \Big[ e^{-\pi i\big[( 2i+\epsilon b)^{-1} (v-\overline{v'})^2 \big] }\Big]k_{z_0}(v',v')\phi([\epsilon,v'])d_{z_0}(v')\\
&\xlongequal{\textrm{Eq. (\ref{z_04}) }}\Pi_{ F_{z_0}}(g)(\phi)([\epsilon,v])\\
&=\theta_{ F_{z_0}, X^{\ast}}\big(\Pi_{X^{\ast}, \psi}(g)(f)\big)([\epsilon,v]).
\end{align*}
\subsubsection{}\label{ubxasttoFz02}  Let  $g=h(a)$.  
\begin{align*}
&\theta_{ F_{z_0}, X^{\ast}}\big(\Pi_{X^{\ast}, \psi}(g)(f)\big)([\epsilon,v])\\
&=2^{\tfrac{1}{4}} \int_{\R} e^{2\pi i\big( \tfrac{1}{2} z_0 x^2 +xv\big)}\big(\Pi_{X^{\ast}, \psi}(g)f\big)([\epsilon,x])dx\\
&=2^{\tfrac{1}{4}} |a|^{1/2}  \int_{\R} e^{2\pi i\big( \tfrac{1}{2} z_0 x^2 +xv\big)}(a,\epsilon)_\R f([\epsilon,xa])dx\\
&=2^{\tfrac{1}{2}} |a|^{1/2}(a,\epsilon)_\R \int_{\C} \int_{\R} e^{2\pi i\big( \tfrac{1}{2} z_0 x^2 +xv\big)} e^{2\pi i\big[\tfrac{1}{2}z_0a^2x^2-xa\overline{v'} \big]} k_{z_0}(v',v')\phi(v')d_{z_0}(v')dx\\
&=2^{\tfrac{1}{2}} |a|^{1/2}(a,\epsilon)_\R\int_{\C} \Big[\int_{\R} e^{-\pi    (1+a^2) x^2}   e^{2\pi i\big(x(v-a\overline{v'}) \big)} dx\Big] k_{z_0}(v',v')\phi([\epsilon,v'])d_{z_0}(v') \\
&\xlongequal{\textrm{ \cite[App.A, Thm.1]{Fo} }}2^{\tfrac{1}{2}} |a|^{1/2}(a,\epsilon)_\R  \int_{\C}  \Big[(1+a^2)^{-1/2} e^{-\pi\big(( 1+a^2)^{-1} (v-\overline{v}'a)^2\big)}\Big]k_{z_0}(v',v')\phi([\epsilon,v'])d_{z_0}(v')\\
& =(\tfrac{2|a|}{1+a^2})^{1/2}(a,\epsilon)_\R \int_{\C} \Big[e^{-\pi\big(( 1+a^2)^{-1} (v-\overline{v}'a)^{2} \big) }\Big]k_{z_0}(v',v')\phi([\epsilon,v'])d_{z_0}(v')\\
&\xlongequal{\textrm{Eq. (\ref{z_05}) }}\Pi_{ F_{z_0}}(g)(\phi)([\epsilon,v])=\Pi_{ F_{z_0}}(g)\big(\theta_{ F_{z_0}, X^{\ast}}(f)\big)([\epsilon,v]).
\end{align*}
\subsubsection{}\label{ubxasttoFz03}   Let  $g=\omega$. 

\begin{align*}
&\theta_{ F_{z_0}, X^{\ast}}\big(\Pi_{X^{\ast}, \psi}(g)(f)\big)([\epsilon,v])\\
&=2^{\tfrac{1}{4}} \int_{\R} e^{2\pi i\big( \tfrac{1}{2} z_0 x^2 +xv\big)}\big(\Pi_{X^{\ast}, \psi}(g)f\big)([\epsilon,x])dx\\
&=2^{\tfrac{1}{4}} \int_{\R} e^{2\pi i\big( \tfrac{1}{2} z_0 x^2 +xv\big)}dx\int_{\R} \nu(\epsilon,\omega)e^{2\pi i \epsilon x y}  f([\epsilon,-y]) dy\\
&=2^{\tfrac{1}{4}}\nu(\epsilon,\omega)\int_{\R}  f([\epsilon,-y]) dy \Big[\int_{\R} e^{2\pi i\big( \tfrac{1}{2} z_0 x^2 +x(v+\epsilon y)\big)}dx\Big] \\
&\xlongequal{\textrm{ \cite[App.A, Thm.1]{Fo} }}2^{\tfrac{1}{4}}\nu(\epsilon,\omega) \int_{\R}  f([\epsilon,-y]) \Big[e^{-\pi(v+\epsilon y)^2}\Big] dy ; 
\end{align*}
By Eq. \ref{z_02} and \ref{z_03}, we have:
\begin{align*}
&\Pi_{ F_{z_0}}(g)\big(\theta_{ F_{z_0}, X^{\ast}}(f)\big)([1,v])\\
&=e^{-\pi i (-i) v^2}\theta_{ F_{z_0}, X^{\ast}}(f)([1,v(-i)])\\
&=2^{\tfrac{1}{4}} \int_{\R} e^{-\pi  v^2} e^{2\pi i\big( \tfrac{1}{2} z_0 x^2 -xiv\big)}f([1,x])dx\\
&=2^{\tfrac{1}{4}}\int_{\R} e^{-\pi   v^2} e^{\pi \big( - x^2 +2xv\big)}f([1,x])dx\\
&=2^{\tfrac{1}{4}}\int_{\R} e^{-\pi  v^2} e^{\pi \big( - y^2 -2yv\big)}f([1,-y])dy\\
&=\theta_{ F_{z_0}, X^{\ast}}\big(\Pi_{X^{\ast}, \psi}(g)(f)\big)([1,v]).
\end{align*}
\begin{align*}
&\Pi_{ F_{z_0}}(g)\big(\theta_{ F_{z_0}, X^{\ast}}(f)\big)([-1,v])\\
&=i e^{-\pi v^2} \theta_{ F_{z_0}, X^{\ast}}(f)([-1,vi])\\
&=2^{\tfrac{1}{4}} i \int_{\R} e^{-\pi  v^2} e^{2\pi i\big( \tfrac{1}{2} z_0 x^2 +xiv\big)}f([-1,x])dx\\
&=2^{\tfrac{1}{4}}i \int_{\R} e^{-\pi   v^2} e^{\pi \big( - x^2 -2xv\big)}f([-1,x])dx\\
&=2^{\tfrac{1}{4}}i \int_{\R} e^{-\pi  v^2} e^{\pi \big( - y^2 +2yv\big)}f([-1,-y])dy\\
&=\theta_{ F_{z_0}, X^{\ast}}\big(\Pi_{X^{\ast}, \psi}(g)(f)\big)([-1,v]).
\end{align*}
\subsection{} Retain the notations from  Example \ref{exm2}.

1) \begin{align*}
\theta_{  X^{\ast},F_{z_0}}(f^{1/2}_-)([-1,x])&=\theta_{  X^{\ast},F_{z_0}}(f^{1/2}_+)([1,x])\\
&= 2^{\tfrac{1}{4}} \int_{\C}  e^{2\pi i \big[\tfrac{1}{2}z_0x^2-x\overline{v} \big]} e^{-2 \pi (v_y)^2} e^{-\tfrac{\pi}{2}v^2}dv_x dv_y\\
&= 2^{\tfrac{1}{4}} e^{-\pi x^2}\int_{\C}  e^{2\pi i \big[-x(v_x-iv_y) \big]} e^{-2 \pi (v_y)^2} e^{-\tfrac{\pi}{2}v^2}dv_x dv_y\\
&=2^{1/4}e^{-\pi x^2}.
\end{align*}
2) \begin{align*}
\theta_{  X^{\ast},F_{z_0}}(f^{3/2}_-)([-1,x]) &= \theta_{  X^{\ast},F_{z_0}}(f^{3/2}_+)([1,x])\\
&=\tfrac{1}{2} \cdot 2^{\tfrac{1}{4}} \int_{\C}  e^{2\pi i \big[\tfrac{1}{2}z_0x^2-x\overline{v} \big]} e^{-2 \pi (v_y)^2} ive^{-\tfrac{\pi}{2}v^2}dv_x dv_y\\
 &=   i \tfrac{1}{2} \cdot 2^{\tfrac{1}{4}} e^{-\pi x^2}\int_{\C}  e^{2\pi i \big[-x(v_x-iv_y) \big]} e^{-2 \pi (v_y)^2} v e^{-\tfrac{\pi}{2}v^2}dv_x dv_y\\
 &=i\tfrac{1}{2} \cdot 2^{1/4}e^{-\pi x^2} (-2 ix)= 2^{1/4} x e^{-\pi x^2}.
\end{align*}
\begin{definition}
Let  
\(A^+([\epsilon,x])=\begin{cases}\,e^{-\pi x^2}&\epsilon=1,\\0&\epsilon=-1,\end{cases}\) \(A^-([\epsilon,x])=\begin{cases}0&\epsilon=1,\\\,e^{-\pi x^2}&\epsilon=-1,\end{cases}\)   

\(B^+([\epsilon,x])=\begin{cases}\,x e^{-\pi x^2}&\epsilon=1,\\ 0&\epsilon=-1,\end{cases}\) \(B^-([\epsilon,x])=\begin{cases}0&\epsilon=1,\\\,x e^{-\pi x^2}&\epsilon=-1.\end{cases}\)
\end{definition}
\begin{lemma}\label{ABM}
Let  $g=\begin{pmatrix}a & -b\\ b& a\end{pmatrix}\in \SO_2(\R)$, and $u=a+bi\in \U(\C)$.
\begin{itemize}
\item[(1)] $\Pi_{X^{\ast},\psi }(g) A^{+}=\widetilde{s}(g)^{-1}A^{+}$, $\Pi_{X^{\ast},\psi }(g) A^{-}=\widetilde{s}(g)^{-1}uA^{-}$.
\item[(2)]  $\Pi_{X^{\ast},\psi }(g) B^{+}=\widetilde{s}(g)^{-1}\overline{u}B^{+}$, $\Pi_{X^{\ast},\psi }(g) B^{-}=\widetilde{s}(g)^{-1}u^2B^{-}$.
\end{itemize}
\end{lemma}
\begin{proof}
 By the above 1) and 2), $\theta_{  X^{\ast},F_{z_0}}(2^{-1/4} f^{1/2}_{\pm})=A^{\pm}$,  and $\theta_{  X^{\ast},F_{z_0}}( 2^{-1/4} f^{3/2}_{\pm})=B^{\pm}$. 
\begin{align*}
\Pi_{X^{\ast},\psi }(g) A^+&=\Pi_{X^{\ast},\psi }(g)\theta_{  X^{\ast},F_{z_0}}(2^{-1/4} f^{1/2}_+)\\
&=\widetilde{s}(g)^{-1}\theta_{  X^{\ast},F_{z_0}}\Big(2^{-1/4}\Pi_{F_{z_0},\psi}(g)f^{1/2}_+\Big)\\
&\stackrel{\textrm{Ex. \ref{exm2}}}{=}\widetilde{s}(g)^{-1}\theta_{  X^{\ast},F_{z_0}}(2^{-1/4} f^{1/2}_+)\\
&=\widetilde{s}(g)^{-1}A^+;\\
& \\
\Pi_{X^{\ast},\psi }(g) B^+&=\Pi_{X^{\ast},\psi }(g)\theta_{  X^{\ast},F_{z_0}}(2^{-1/4} f^{3/2}_+)\\
&=\widetilde{s}(g)^{-1}\theta_{  X^{\ast},F_{z_0}}\Big( 2^{-1/4}\Pi_{F_{z_0},\psi}(g)f^{3/2}_+\Big)\\
&\stackrel{\textrm{Ex. \ref{exm2}}}{=}\widetilde{s}(g)^{-1}\theta_{  X^{\ast},F_{z_0}}\Big(2^{-1/4}(-bi+a)f^{3/2}_+\Big)\\
&=\widetilde{s}(g)^{-1}(-bi+a)B^+.
\end{align*}
The remaining equalities follow in the same fashion.
\end{proof}
Following Lemma \ref{deltapm},  for $[g,\epsilon] \in \widetilde{\Oa}_2(\R)$, we define: 
\[\widetilde{\delta}^+: [g,\epsilon] \to \epsilon\widetilde{s}(g)^{-1} \delta^+(g),\] 
\[\widetilde{\delta}^-: [g,\epsilon] \to \epsilon\widetilde{s}(g)^{-1} \delta^-(g).\] 
\begin{lemma}
 $\widetilde{\delta}^+$ and $\widetilde{\delta}^-$ both are    irreducible representations of $\widetilde{\Oa}_2(\R)$.
 \end{lemma}
 \begin{proof}
 Recall the isomorphism from (\ref{overwidehatSL}):
 $$\iota': \widetilde{\Oa}_2(\R) \to \Im(\iota') (\subseteq\widehat{\Oa}_2(\R)) ; [g,\epsilon] \longmapsto [g,\epsilon\widetilde{s}(g)^{-1}].$$
 Then $\widetilde{\delta}^+=\delta^+ \circ \iota'$ and $\widetilde{\delta}^-=\delta^- \circ \iota'$. So both representations are irreducible.
 \end{proof}
 We also denote:
\[ \mathcal{D}^{+}=\Span_{\C}\{ A^+, A^-\}, \quad \mathcal{D}^-=\Span_{\C}\{ B^+, B^-\}.\]
For any functions $A\in \mathcal{D}^{+}$, $B\in \mathcal{D}^-$ and $g\in \Oa_2(\R)$,  we have:
\[\Pi_{X^{\ast},\psi }([g, \epsilon]) A=\widetilde{\delta}^+([g, \epsilon])A, \quad \Pi_{X^{\ast},\psi }([g, \epsilon]) B=\widetilde{\delta}^-([g, \epsilon])B.\]
Similarly, for $[g,\epsilon] \in \overline{\Oa}_2(\R)$, we define: 
\[\overline{\delta}^+: [g,\epsilon] \to \epsilon\overline{s}(g)^{-1} \delta^+(g),\] 
\[\overline{\delta}^-: [g,\epsilon] \to \epsilon\overline{s}(g)^{-1} \delta^-(g).\] 
Then $\overline{\delta}^+$ and $\overline{\delta}^-$ both are    irreducible representations. Moreover,  $\overline{s}^{-1}: [g,\epsilon] \to \epsilon\overline{s}(g)^{-1} $ defines a character. 
\begin{equation}\label{deltapmrep}
\overline{\delta}^+\simeq \Ind_{\overline{\SO}_2(\R)}^{\overline{\Oa}_2(\R)} (\overline{s}^{-1}\cdot 1_{\mu_2}), \quad \overline{\delta}^-\simeq \Ind_{\overline{\SO}_2(\R)}^{\overline{\Oa}_2(\R)} (\overline{s}^{-1} \overline{u}\cdot 1_{\mu_2}).
\end{equation}

\section{Siegel modular forms associated to Weil representations: $[\overline{\Gamma}_{\theta} \cap \overline{\Gamma}_0(N^2)]$ \& $\overline{\Gamma}_0(N^2)$}\label{GammathetaN2}
Unless otherwise stated, the following conventions are used throughout this section:
\begin{itemize}
\item $z\in\mathbb H$ and $g=\begin{pmatrix}a&b\\ c&d\end{pmatrix}\in\SL_2(\R)$;
\item $g_i=\begin{pmatrix}a_i&b_i\\ c_i&d_i\end{pmatrix}\in\SL_2(\R)$ for $i=1,2$;
\item the character $\psi$ is fixed to be $\psi_0$.
\item $N$ is a positive integer number, and  $\chi$ is a  primitive Dirichlet character modulo $N$. We say $\chi$  even if $\chi(-1)=1$ and odd if $\chi(-1)=-1$. In particular, $N=1$, $\chi\equiv1$.
\end{itemize}
We recall the functions $A^{\pm}$ and $B^{\pm}$ from Lemma~\ref{ABM}; as we work exclusively with the Weil representation of $\widetilde{\SL}_2(\R)$, we abbreviate $A(x)=A^{+}([1,x])$ and $B(x)=B^+([1,x])$ for $x\in X\simeq \R$.
Let us define the Gauss sum:
 $$G(n, \overline{\chi}) =\sum_{k \bmod N} \overline{\chi}(k) e^{\tfrac{2\pi i k n}{N}}=\sum_{k \bmod N, (k,N)=1} \overline{\chi}(k) e^{\tfrac{2\pi i k n}{N}}.$$
 The following result is known:
\begin{lemma}
\begin{itemize}
\item[(1)] $|G(1, \overline{\chi})|=\sqrt{N}$.
\item[(2)] $G(n,\overline{\chi})=\chi(n)G(1, \overline{\chi})$.
\end{itemize}
\end{lemma}

\subsection{ $ P_{>0}(\R)$}
Let $p=\begin{pmatrix} a & b\\ 0& a^{-1} \end{pmatrix}\in P_{>0}(\R)$ and write $z_p=p(i)=ba+ia^{2}\in\mathbb{H}$.    By \eqref{chap2representationsp2} and \eqref{chap2representationsp3} we have $\widetilde{s}(p)=1$ and
\begin{align*}
\pi_{X^{\ast}, \psi}(p)A(x)
&= \pi_{X^{\ast}, \psi}(h(a)u(a^{-1}b))A(x) \\
&= |a|^{1/2}\pi_{X^{\ast}, \psi}(u(a^{-1}b))A(xa) \\
&= |a|^{1/2}\psi(\tfrac{1}{2}a^{-1}b\cdot a^2x^2)A(xa) \\
&= |a|^{1/2}e^{\pi i bax^2}\cdot e^{-\pi a^2x^2} \\
&= |a|^{1/2}e^{i\pi z_p x^2},
\end{align*}
\begin{align*}
\pi_{X^{\ast}, \psi}(p)B(x)
&= \pi_{X^{\ast}, \psi}(h(a)u(a^{-1}b))B(x) \\
&= |a|^{1/2}\pi_{X^{\ast}, \psi}(u(a^{-1}b))B(xa) \\
&= |a|^{1/2}\psi(\tfrac{1}{2}a^{-1}b\cdot a^2x^2)B(xa) \\
&= |a|^{1/2}xa\,e^{i\pi z_p x^2}.
\end{align*}
Moreover,
\[
J_{1/2}(p,z_0)=\epsilon(p;z_0,z_0)\cdot|J(p,z_0)|^{1/2}=|a|^{-1/2},\qquad
J_{3/2}(p,z_0)=J_{1/2}(p,z_0)J(p,z_0)=|a|^{-1/2}a^{-1}.
\]
Consequently,
\[
J_{1/2}(p,z_0)\pi_{X^{\ast},\psi}(p)A(x)=e^{i\pi z_p x^2},\qquad
J_{3/2}(p,z_0)\pi_{X^{\ast},\psi}(p)B(x)=e^{i\pi z_p x^2}x.
\]
 \subsection{ $ \SL_2(\R)$}
 Write $z_g=g(i)=x_g+iy_g\in \mathbb{H}$, $g=p_gk_g$(cf. Eq. (\ref{depkg})), $z_g=z_{p_g}$. Then:
\[ p_g=\begin{pmatrix} \sqrt{y_g} & x_g(\sqrt{y_g})^{-1}\\ 0& (\sqrt{y_g})^{-1} \end{pmatrix}, k_g = \begin{pmatrix}\sqrt{y_g}^{-1}a -(\sqrt{y_g})^{-1}x_g c & \sqrt{y_g}^{-1}b-(\sqrt{y_g})^{-1}x_g d\\ \sqrt{y_g} c &  \sqrt{y_g}d\end{pmatrix}. \]
\[ J_{1/2}(p_gk_g, z_0)\stackrel{\textrm{Lem. }  \ref{J121J1}}{=}J_{1/2}(p_g, k_g(z_0))J_{1/2}(k_g, z_0)\widetilde{c}_{X^{\ast}}(p_g,k_g)=J_{1/2}(p_g, z_0)\widetilde{s}(k_g).\]
\[ J_{3/2}(p_gk_g, z_0)= J_{1/2}(p_gk_g, z_0)(cz_0+d)=(cz_0+d)J_{1/2}(p_g, z_0)\widetilde{s}(k_g).\]
\begin{align*}
&J_{1/2}(g, z_0)\pi_{X^{\ast}, \psi} (g)A( x)\\
&= J_{1/2}(p_g, z_0)\widetilde{s}(k_g)\widetilde{c}_{X^{\ast}}(p_g, k_g)^{-1}\pi_{X^{\ast}, \psi} (p_g)\pi_{X^{\ast}, \psi} (k_g)A( x)\\
&= J_{1/2}(p_g, z_0)\pi_{X^{\ast}, \psi} (p_g)\widetilde{s}(k_g)\pi_{X^{\ast}, \psi} (k_g)A( x)\\
&\stackrel{\textrm{Lem. } \ref{ABM}}{=}J_{1/2}(p_g, z_0)\pi_{X^{\ast}, \psi} (p_g)A( x)\\
&=e^{i   \pi xz_g x }.
\end{align*}
\begin{align*}
&J_{3/2}(g, z_0)\pi_{X^{\ast}, \psi} (g)B( x)\\
&=(cz_0+d) J_{1/2}(p_g, z_0)\widetilde{s}(k_g)\widetilde{c}_{X^{\ast}}(p_g, k_g)^{-1}\pi_{X^{\ast}, \psi} (p_g)\pi_{X^{\ast}, \psi} (k_g)B( x)\\
&=(cz_0+d) J_{1/2}(p_g, z_0)\pi_{X^{\ast}, \psi} (p_g)\widetilde{s}(k_g)\pi_{X^{\ast}, \psi} (k_g)B( x)\\
&\stackrel{\textrm{Lem. } \ref{ABM}}{=}(cz_0+d) J_{1/2}(p_g, z_0)\pi_{X^{\ast}, \psi} (p_g)[(-\sqrt{y_g} c z_0+ \sqrt{y_g}d)B]( x)\\
&=(cz_0+d) e^{i   \pi xz_g x }(-\sqrt{y_g} c z_0+ \sqrt{y_g}d)  \sqrt{y_g}x\\
&=e^{i   \pi xz_g x }(cz_0+d) (-cz_0+d)y_gx\\
&=e^{i   \pi xz_g x }x\\
&=J_{3/2}(p_g, z_0)\pi_{X^{\ast}, \psi} (p_g)B( x).
\end{align*}
\subsection{ $\Gamma_{\theta}$}\label{gammatheta}
Let $f\in \mathcal H(X^*) $.   Let $r=\begin{pmatrix}a& b\\ c& d\end{pmatrix}\in \Gamma_{\theta}$.
\begin{align*}
 \theta_{L, X^{\ast}}(f)(0)=\sum_{l\in L\cap X}f(l)=\sum_{n\in \Z}f( n).
 \end{align*}
 \begin{align*}
\overline{ \chi}(k)\theta_{L, X^{\ast}}(f)(\tfrac{k}{N}e^{\ast})=\sum_{l\in L\cap X} \overline{ \chi}(k)f(l)\psi(\tfrac{kl}{N})=\sum_{n\in \Z}f( n)\overline{ \chi}(k)\psi(\tfrac{kn}{N}).
 \end{align*}
 \begin{align*}
\overline{ \chi}(k)\theta_{L, X^{\ast}}(f)(cNke+\tfrac{kd}{N}e^{\ast})&=\overline{ \chi}(k)\theta_{L, X^{\ast}}(f)(\tfrac{kd}{N}e^{\ast})\\
&=\chi(d)\sum_{l\in L\cap X} \overline{ \chi}(kd)f(l)\psi(\tfrac{kdl}{N})\\
&=\chi(d)\sum_{n\in \Z}f( n)\overline{ \chi}(kd)\psi(\tfrac{kdn}{N}).
 \end{align*}
 \begin{align*}
 \theta^{1/2}(g)&\stackrel{Def.}{=} \theta_{L, X^{\ast}}\Big(J_{1/2}(g, z_0)\pi_{X^{\ast}, \psi} (g)A\Big)(0)\\
&=\sum_{n\in \Z}J_{1/2}(g, z_0)\pi_{X^{\ast}, \psi} (g)A(  n)\\
&=\sum_{n\in \Z} e^{i   \pi z_g n^2 }.
 \end{align*}
  \begin{align*}
\theta^{1/2}_{\chi}(g)&\stackrel{Def.}{=} \tfrac{1}{G(1, \overline{\chi})}\sum_{k \bmod N}\overline{ \chi}(k) \theta_{L, X^{\ast}}\Big(J_{1/2}(g, z_0)\pi_{X^{\ast}, \psi} (g)A\Big)(\tfrac{k}{N}e^{\ast})\\
&=\tfrac{1}{G(1, \overline{\chi})} \sum_{k \bmod N} \sum_{n\in \Z}J_{1/2}(g, z_0)\pi_{X^{\ast}, \psi} (g)A(  n)\overline{ \chi}(k) \psi(\tfrac{kn}{N})\\
&= \tfrac{1}{G(1, \overline{\chi})}\sum_{k \bmod N} \sum_{n\in \Z} e^{i   \pi z_g n^2 }\overline{ \chi}(k) \psi(\tfrac{kn}{N})\\
&=\tfrac{1}{G(1, \overline{\chi})}\sum_{n\in \Z} e^{i   \pi z_g n^2 }\Big(\sum_{k \bmod N}\overline{ \chi}(k) \psi(\tfrac{kn}{N})\Big)\\
&=\sum_{n\in \Z} e^{i   \pi z_g n^2}  \chi(n) .
 \end{align*}
 \begin{align*}
 \theta^{3/2}_{\chi}(g)&\stackrel{Def.}{=} \tfrac{1}{G(1, \overline{\chi})}\sum_{k \bmod N}\overline{ \chi}(k)\theta_{L, X^{\ast}}\Big(J_{3/2}(g, z_0)\pi_{X^{\ast}, \psi} (g)B\Big)(\tfrac{k}{N}e^{\ast})\\
 &= \tfrac{1}{G(1, \overline{\chi})}  \sum_{k \bmod N} \sum_{n\in \Z}J_{1/2}(g, z_0)\pi_{X^{\ast}, \psi} (g)B(  n)\overline{ \chi}(k) \psi(\tfrac{kn}{N})\\
 &=\tfrac{1}{G(1, \overline{\chi})}\sum_{n\in \Z} ne^{i   \pi z_g n^2 }\Big(\sum_{k \bmod N}\overline{ \chi}(k) \psi(\tfrac{kn}{N})\Big)\\
&=\sum_{n\in \Z} ne^{i   \pi z_g n^2}  \chi(n) 
  \end{align*}
 Note that $\theta^{1/2}_{\chi}(g)$ and $ \theta^{3/2}_{\chi}(g)$ only depend on $z_g=g(z_0)$. 
 \begin{definition}
If $\chi$ is even,  define
$
\theta^{1/2}_{\chi}(z) = \sum_{n \in \mathbb{Z}} \chi(n)\, e^{i\pi z n^2}$,

and if $\chi$ is odd, define
$\theta^{3/2}_{\chi}(z) = \sum_{n \in \mathbb{Z}} \chi(n)\, n\, e^{i\pi n^2 z}$.
\end{definition}
Let us consider the explicit action of $\Gamma_{\theta}$ on $\theta^{1/2}_{\chi}(z)$ and  $\theta^{3/2}_{\chi}(z)$.

1)  For $r=\begin{pmatrix}   a & b \\  c& d\end{pmatrix} \in \Gamma_{\theta}$, we have:
 \begin{align*}
\theta^{1/2}(r g)&= \theta^{1/2}(r p_g)\\
&=\theta_{L, X^{\ast}}\Big(J_{1/2}(r p_g, z_0)\pi_{X^{\ast}, \psi} (r p_g)A\Big)(0)\\
&\stackrel{\widetilde{c}_{X^{\ast}}(r, p_g)=1}{=} J_{1/2}(r ,  p_g(z_0))\pi_{L, \psi}(r)\theta_{L, X^{\ast}}\Big(J_{1/2}(p_g,z_0)\pi_{X^{\ast}, \psi} ( p_g)A\Big)(0) \\
&\stackrel{(\ref{chapeq8})}{=}J_{1/2}(r , z_g)\epsilon_{r} \theta^{1/2}( g)\\
&=\sqrt{cz_g+d} m_{X^{\ast}}(r)\epsilon_{r} \theta^{1/2}( g)\\
&= m_{X^{\ast}}(r)\widetilde{\beta}^{-1}(r)\sqrt{cz_g+d} \cdot \theta^{1/2}( g)\\
&\stackrel{(\ref{lambda})}{=}\lambda(r)\sqrt{cz_g+d}\cdot \theta^{1/2}( g).
\end{align*} 

 Note that $[h(N^{-1}) \SL_2(\Z) h(N)]\cap \SL_2(\Z)=\Gamma_0(N^2)$.  \\
2) Let $r=\begin{pmatrix} a& b\\ cN^2 & d\end{pmatrix}\in   \Gamma_{\theta} \cap  \Gamma_0(N^2 )$.
\begin{align*}
\theta^{1/2}_{\chi}(r g)&= \theta^{1/2}_{\chi}(r p_g)\\
&= \tfrac{1}{G(1, \overline{\chi})}\sum_{k \bmod N}\overline{ \chi}(k) \theta_{L, X^{\ast}}\Big(J_{1/2}(rp_g, z_0)\pi_{X^{\ast}, \psi} (rp_g)A\Big)(\tfrac{k}{N}e^{\ast})\\
&\stackrel{\widetilde{c}_{X^{\ast}}(r, p_g)=1}{=}\tfrac{1}{G(1, \overline{\chi})}\sum_{k \bmod N}\overline{ \chi}(k)  J_{1/2}(r ,  p_g(z_0))\pi_{L, \psi}(r)\theta_{L, X^{\ast}}\Big(J_{1/2}(p_g,z_0)\pi_{X^{\ast}, \psi} ( p_g)A\Big)(\tfrac{k}{N}e^{\ast}) \\
&\stackrel{(\ref{chapeq8})}{=}J_{1/2}(r ,  p_g(z_0)) \epsilon_{r} \tfrac{1}{G(1, \overline{\chi})}\sum_{k \bmod N}\overline{ \chi}(k) \theta_{L, X^{\ast}}\Big(J_{1/2}(p_g,z_0)\pi_{X^{\ast}, \psi} ( p_g)A\Big)(cNke+\tfrac{kd}{N}e^{\ast})\\
&=J_{1/2}(r ,  p_g(z_0)) \epsilon_{r} \chi(d) \tfrac{1}{G(1, \overline{\chi})}\sum_{k \bmod N}\sum_{n\in \Z}\Big(J_{1/2}(p_g,z_0)\pi_{X^{\ast}, \psi} ( p_g)A\Big)( n)\overline{ \chi}(kd)\psi(\tfrac{kdn}{N}) \\
&=J_{1/2}(r ,  p_g(z_0)) \epsilon_{r} \chi(d)  \theta^{1/2}_{\chi}( g)\\
&=\chi(d) \lambda(r)\sqrt{cN^2z_g+d}\cdot\theta^{1/2}_{\chi}( g).
\end{align*} 
\begin{align*}
\theta^{3/2}_{\chi}(r g)&= \theta^{3/2}_{\chi}(r p_g)\\
&= \tfrac{1}{G(1, \overline{\chi})}\sum_{k \bmod N}\overline{ \chi}(k) \theta_{L, X^{\ast}}\Big(J_{3/2}(rp_g, z_0)\pi_{X^{\ast}, \psi} (rp_g)B\Big)(\tfrac{k}{N}e^{\ast})\\
&\stackrel{\widetilde{c}_{X^{\ast}}(r, p_g)=1}{=}\tfrac{1}{G(1, \overline{\chi})}\sum_{k \bmod N}\overline{ \chi}(k)  J_{3/2}(r ,  p_g(z_0))\pi_{L, \psi}(r)\theta_{L, X^{\ast}}\Big(J_{3/2}(p_g,z_0)\pi_{X^{\ast}, \psi} ( p_g)B\Big)(\tfrac{k}{N}e^{\ast}) \\
&\stackrel{(\ref{chapeq8})}{=}J_{3/2}(r ,  p_g(z_0)) \epsilon_{r} \tfrac{1}{G(1, \overline{\chi})}\sum_{k \bmod N}\overline{ \chi}(k) \theta_{L, X^{\ast}}\Big(J_{3/2}(p_g,z_0)\pi_{X^{\ast}, \psi} ( p_g)B\Big)(cNke+\tfrac{kd}{N}e^{\ast})\\
\end{align*} 
\begin{align*}
&=J_{3/2}(r ,  p_g(z_0)) \epsilon_{r} \chi(d) \tfrac{1}{G(1, \overline{\chi})}\sum_{k \bmod N}\sum_{n\in \Z}\Big(J_{3/2}(p_g,z_0)\pi_{X^{\ast}, \psi} ( p_g)B\Big)( n)\overline{ \chi}(kd)\psi(\tfrac{kdn}{N}) \\
&=J_{3/2}(r ,  p_g(z_0)) \epsilon_{r} \chi(d)  \theta^{3/2}_{\chi}( g)\\
&=\chi(d) \lambda(r)\sqrt{cN^2z_g+d}(cN^2z_g+d)\cdot\theta^{3/2}_{\chi}( g).
\end{align*} 

 We conclude: 
 \begin{theorem}\label{mainthm1}
 For   $r=\begin{pmatrix} a& b\\ c & d\end{pmatrix}\in    \Gamma_{\theta} \cap  \Gamma_0(N^2 )$,  we have:
 \begin{itemize}
 \item[(1)] $\theta^{1/2}_{\chi}(r z)=\chi(d) \lambda(r) \sqrt{cz+d} \, \theta^{1/2}_{\chi}( z)$;
 \item[(2)] $\theta^{3/2}_{\chi}(r z)=\chi(d)\lambda(r) \sqrt{cz+d}(cz+d)\,\theta^{3/2}_{\chi}( z)$. 
 \end{itemize}
\end{theorem}
Note that $\chi$ defines a character of $\Gamma_0(N^2)/\Gamma_1(N^2)$.
\subsection{$\overline{\Gamma}_{\theta}$}\label{Gammatheta}
We let $\overline{\SL}_{2}(\mathbb{R})$ act on $\mathbb{H}$ through its projection to $\SL_{2}(\mathbb{R})$.

\begin{lemma}\label{thetachichar}
The map $\overline{\lambda}_{\chi} \colon \overline{\Gamma}_{\theta} \cap \overline{\Gamma}_0(N^2) \to \mathbb{T}$ defined by $[r,\epsilon] \longmapsto \lambda(r)\chi(r)\epsilon$ is a character.
\end{lemma}
 \begin{proof}
 Recall the character $\overline{\lambda}$ from Lemma \ref{lambdaoverlin}. The character $\chi$ can also  be viewed as a character of $\overline{\Gamma}_0(N^2 )$. Then $\overline{\lambda}_{\chi}=\overline{\lambda}\otimes \chi$.
 \end{proof}
 According to Theorem \ref{mainthm1}, we obtain the following:
 \begin{theorem}\label{mainthm1'}
 For $\overline{r}=[r,  \epsilon]\in \overline{\Gamma}_{\theta} \cap  \overline{\Gamma}_0(N^2 )$,  we have:
 \begin{itemize}
 \item[(1)] $\theta^{1/2}_{\chi}(\overline{r} z)=\overline{\lambda}_{\chi}(\overline{r})J_{1/2}(\overline{r},z)\theta^{1/2}_{\chi}( z)$;
 \item[(2)] $\theta^{3/2}_{\chi}(\overline{r} z)=\overline{\lambda}_{\chi}(\overline{r})J_{3/2}(\overline{r},z)\theta^{3/2}_{\chi}( z)$. 
 \end{itemize}
\end{theorem}
\subsection{ Twisted by $\overline{\mathfrak{w}} $ }\label{twistmathw}

Let $\mathfrak{w} \in \GL_2(\Q)$, and let $\overline{\mathfrak{w}} = (\mathfrak{w}, t_{\mathfrak{w}}) \in \overline{\GL}_2(\Q)$. For $\overline{s} \in \overline{\Gamma}_{\theta} \cap \overline{\Gamma}_0(N^2)$, let $\overline{r} = \overline{\mathfrak{w}}^{-1} \overline{s} \overline{\mathfrak{w}} \in \overline{\mathfrak{w}}^{-1}[\overline{\Gamma}_{\theta} \cap \overline{\Gamma}_0(N^2)]\overline{\mathfrak{w}}$. We define a character on this  subgroup by
\[
 \overline{\lambda}_{\chi}^{\overline{\mathfrak{w}}}(\overline{r}) := \overline{\lambda}_{\chi}(\overline{s}).
\]

\begin{definition}
Define the slash operators for the theta functions as follows:
\begin{itemize}
\item[(1)] $[\theta^{1/2}_{\chi}]^{\overline{\mathfrak{w}}}(z) := \det(\mathfrak{w})^{1/4} J_{1/2}(\overline{\mathfrak{w}},z)^{-1} \theta^{1/2}_{\chi}(\overline{\mathfrak{w}} z)$;
\item[(2)] $[\theta^{3/2}_{\chi}]^{\overline{\mathfrak{w}}}(z) := \det(\mathfrak{w})^{3/4} J_{3/2}(\overline{\mathfrak{w}},z)^{-1} \theta^{3/2}_{\chi}(\overline{\mathfrak{w}} z)$.
\end{itemize}
\end{definition}
\begin{proposition}\label{slashpro}
For $\overline{r} \in \overline{\mathfrak{w}}^{-1}[\overline{\Gamma}_{\theta} \cap \overline{\Gamma}_0(N^2)]\overline{\mathfrak{w}}$, we have:
\begin{itemize}
\item[(1)] $[\theta^{1/2}_{\chi}]^{\overline{\mathfrak{w}}}(\overline{r}z) = J_{1/2}(\overline{r},z)\overline{\lambda}_{\chi}^{\overline{\mathfrak{w}}}(\overline{r})[\theta^{1/2}_{\chi}]^{\overline{\mathfrak{w}}}(z)$; 
\item[(2)] $[\theta^{3/2}_{\chi}]^{\overline{\mathfrak{w}}}(\overline{r}z) = J_{3/2}(\overline{r},z)\overline{\lambda}_{\chi}^{\overline{\mathfrak{w}}}(\overline{r})[\theta^{3/2}_{\chi}]^{\overline{\mathfrak{w}}}(z)$.
\end{itemize}
\end{proposition}

\begin{proof}
(1) \begin{align*}
& J_{1/2}(\overline{\mathfrak{w}},\overline{r}z)^{-1} J_{1/2}(\overline{s},\overline{\mathfrak{w}}z) J_{1/2}(\overline{\mathfrak{w}},z) \\
&\stackrel{\textrm{ Coro.} \ref{Automorphicfactor2} }{=} J_{1/2}(\overline{\mathfrak{w}},\overline{r}z)^{-1} J_{1/2}(\overline{s}\overline{\mathfrak{w}},z) \\
&= J_{1/2}(\overline{\mathfrak{w}},\overline{r}z)^{-1} J_{1/2}(\overline{\mathfrak{w}}\overline{r},z) \\
&= J_{1/2}(\overline{r},z).
\end{align*}
\begin{align*}
[\theta^{1/2}_{\chi}]^{\overline{\mathfrak{w}}}(\overline{r}z)
&= \det(\mathfrak{w})^{1/4} J_{1/2}(\overline{\mathfrak{w}},\overline{r}z)^{-1}\theta^{1/2}_{\chi}(\overline{\mathfrak{w}}\overline{r}z) \\
&=\det(\mathfrak{w})^{1/4} J_{1/2}(\overline{\mathfrak{w}},\overline{r}z)^{-1}\theta^{1/2}_{\chi}(\overline{s}\overline{\mathfrak{w}}z) \\
&= \det(\mathfrak{w})^{1/4}J_{1/2}(\overline{\mathfrak{w}},\overline{r}z)^{-1} J_{1/2}(\overline{s},\overline{\mathfrak{w}}z)\overline{\lambda}_{\chi}(\overline{s})\theta^{1/2}_{\chi}(\overline{\mathfrak{w}}z) \\
&= J_{1/2}(\overline{\mathfrak{w}},\overline{r}z)^{-1} J_{1/2}(\overline{s},\overline{\mathfrak{w}}z) J_{1/2}(\overline{\mathfrak{w}},z)\overline{\lambda}_{\chi}^{\overline{\mathfrak{w}}}(\overline{r})[\theta^{1/2}_{\chi}]^{\overline{\mathfrak{w}}}(z) \\
&= J_{1/2}(\overline{r},z)\overline{\lambda}_{\chi}^{\overline{\mathfrak{w}}}(\overline{r})[\theta^{1/2}_{\chi}]^{\overline{\mathfrak{w}}}(z).
\end{align*}

(2) \begin{align*}
& J_{3/2}(\overline{\mathfrak{w}},\overline{r}z)^{-1} J_{3/2}(\overline{s},\overline{\mathfrak{w}}z) J_{3/2}(\overline{\mathfrak{w}},z) \\
&= J_{3/2}(\overline{\mathfrak{w}},\overline{r}z)^{-1} J_{3/2}(\overline{s}\overline{\mathfrak{w}},z) \\
&= J_{3/2}(\overline{\mathfrak{w}},\overline{r}z)^{-1} J_{3/2}(\overline{\mathfrak{w}}\overline{r},z) \\
&= J_{3/2}(\overline{r},z).
\end{align*}
\begin{align*}
[\theta^{3/2}_{\chi}]^{\overline{\mathfrak{w}}}(\overline{r}z)
&= \det(\mathfrak{w})^{3/4}J_{3/2}(\overline{\mathfrak{w}},\overline{r}z)^{-1}\theta^{3/2}_{\chi}(\overline{\mathfrak{w}}\overline{r}z) \\
&= \det(\mathfrak{w})^{3/4}J_{3/2}(\overline{\mathfrak{w}},\overline{r}z)^{-1}\theta^{3/2}_{\chi}(\overline{s}\overline{\mathfrak{w}}z) \\
&=\det(\mathfrak{w})^{3/4} J_{3/2}(\overline{\mathfrak{w}},\overline{r}z)^{-1} J_{3/2}(\overline{s},\overline{\mathfrak{w}}z)\overline{\lambda}_{\chi}(\overline{s})\theta^{3/2}_{\chi}(\overline{\mathfrak{w}}z) \\
&= J_{3/2}(\overline{\mathfrak{w}},\overline{r}z)^{-1} J_{3/2}(\overline{s},\overline{\mathfrak{w}}z) J_{3/2}(\overline{\mathfrak{w}},z)\overline{\lambda}_{\chi}^{\overline{\mathfrak{w}}}(\overline{r})[\theta^{3/2}_{\chi}]^{\overline{\mathfrak{w}}}(z) \\
&= J_{3/2}(\overline{r},z)\overline{\lambda}_{\chi}^{\overline{\mathfrak{w}}}(\overline{r})[\theta^{3/2}_{\chi}]^{\overline{\mathfrak{w}}}(z).
\end{align*}
\end{proof}
Let $t$ now be  a positive integer, and let  $\mathfrak{w}=\begin{pmatrix} t& 0\\ 0 & 1\end{pmatrix}$. We consider the associated element  $\overline{\mathfrak{w}}=[\mathfrak{w}, 1]\in \overline{\GL}_2(\Q)$. 
Then: 
\[[\theta^{1/2}_{\chi}]^{\overline{\mathfrak{w}}}( z)=t^{1/4}\theta^{1/2}_{\chi}(tz)=t^{1/4}\sum_{n\in \Z} \chi(n)e^{\pi i  z tn^2}\]
\[[\theta^{3/2}_{\chi}]^{\overline{\mathfrak{w}}}( z)=t^{3/4}\theta^{3/2}_{\chi}(tz)=t^{3/4}\sum_{n\in \Z} \chi(n)n e^{ \pi i  z tn^2 }.\]
\begin{definition}\label{defmainprop1}
\begin{itemize}
\item $\overline{\Gamma}_{\theta,t}\stackrel{ Def.}{=}\overline{\mathfrak{w}}^{-1} \overline{\Gamma}_{\theta}\overline{\mathfrak{w}}$ and $\overline{\Gamma}_0(tN^2, t^{-1})\stackrel{ Def.}{=}\overline{\mathfrak{w}}^{-1} \overline{\Gamma}_0(N^2)\overline{\mathfrak{w}}$.
\item $\theta^{1/2}_{\chi, t}( z)\stackrel{ Def.}{=}\sum_{n\in \Z} \chi(n)e^{\pi i  z tn^2}$.
\item $\theta^{3/2}_{\chi,t}( z)\stackrel{ Def.}{=} \sum_{n\in \Z} \chi(n)n e^{ \pi i  z tn^2 }$.
\end{itemize}
\end{definition}
\begin{proposition}\label{mainprop1}
 For $\overline{r}= \overline{\mathfrak{w}}^{-1} \overline{s}\overline{\mathfrak{w}}\in \overline{\Gamma}_{\theta,t}\cap \overline{\Gamma}_0(tN^2, t^{-1})$, we have:
 \begin{itemize}
\item[(1)] $[\theta^{1/2}_{\chi,t}](\overline{r}z) = J_{1/2}(\overline{r},z)\overline{\lambda}_{\chi}(\overline{s})[\theta^{1/2}_{\chi,t}](z)$; 
\item[(2)] $[\theta^{3/2}_{\chi,t}](\overline{r}z) = J_{3/2}(\overline{r},z)\overline{\lambda}_{\chi}(\overline{s})[\theta^{3/2}_{\chi,t}](z)$.
\end{itemize}
\end{proposition}

\subsection{Induction  to $\overline{\Gamma}_0(N^2)$: $N$ odd}\label{Gammaaonodd} 
 Recall notations from Def. \ref{Mq123}.  By Lemma \ref{coniso}, we have:
\[ \overline{\Gamma}_0(N^2)=\sqcup_{i=1}^3[\overline{\Gamma}_{\theta} \cap  \overline{\Gamma}_0(N^2 )]  \overline{M}_{q_i}.\]
\begin{definition}
$\overline{\gamma}_{\chi}=\Ind_{\overline{\Gamma}_{\theta} \cap  \overline{\Gamma}_0(N^2 )}^{\overline{\Gamma}_0(N^2)} [\overline{\lambda}_{\chi}]^{-1}, \overline{M}_{\chi}=\Ind_{\overline{\Gamma}_{\theta} \cap  \overline{\Gamma}_0(N^2 )}^{\overline{\Gamma}_0(N^2)} \C.$
\end{definition}
\begin{lemma}\label{gammachiodd}
$\overline{\gamma}_{\chi}$ is an irreducible representation.
\end{lemma}
\begin{proof}
Note that $\chi$ is a character of $\overline{\Gamma}_0(N^2 )$. So
$$\overline{\gamma}_{\chi}\simeq \chi^{-1}\Big(\Ind_{\overline{\Gamma}_{\theta} \cap  \overline{\Gamma}_0(N^2 )}^{\overline{\Gamma}_0(N^2)} [\overline{\lambda}]^{-1}\Big) \simeq \chi^{-1}\otimes \Res_{\overline{\Gamma}_0(N^2)}^{\overline{\SL}_2(\Z)}\overline{\gamma}.$$
Then the result follows from Corollary \ref{irrdgammaodd}.
\end{proof}
For our purposes, we give an explicit analysis of the representation $\overline{\gamma}_{\chi}$. The vector space  $\overline{M}_{\chi}$ consists of the functions $f: \overline{\Gamma}_0(N^2) \longrightarrow \C$ such that $f(\overline{r}\overline{g})=\overline{\lambda}_{\chi}(\overline{r})^{-1}f(\overline{g})$, for $\overline{r}\in\overline{\Gamma}_{\theta} \cap  \overline{\Gamma}_0(N^2 )$ and $\overline{g}\in \overline{\Gamma}_0(N^2)$.  Let $\overline{e}_i$ denote the function of $\overline{M}_{\chi}$, supported on $[\overline{\Gamma}_{\theta} \cap  \overline{\Gamma}_0(N^2 )] \overline{M}_{q_i}$ and $\overline{e}_i(\overline{M}_{q_i})=1$. Then $\overline{M}=\oplus_{i=1}^3 \C \overline{e}_i$.   Write 
    $$\overline{\gamma}_{\chi}(\overline{g})(\overline{e}_1, \overline{e}_2, \overline{e}_3)=(\overline{e}_1, \overline{e}_2, \overline{e}_3)M(\overline{g}),$$ for some $3\times 3$-matrix $M(\overline{g})$. Then 
    \[\begin{array}{rcll}
    \overline{\gamma}_{\chi}: & \overline{\Gamma}_0(N^2) & \longrightarrow &  M_3(\C);\\
                        &\overline{g}                    & \longmapsto     & M(\overline{g}),
    \end{array}\] gives a matrix representation.  More precisely, we have:
\begin{itemize}
 \item \begin{itemize}
\item[(a)]  $\overline{M}_{q_i} \overline{r}=\overline{s}\overline{M}_{q_j}$, for some $\overline{s}\in \overline{\Gamma}_{\theta} \cap  \overline{\Gamma}_0(N^2 )$ iff $[\overline{\Gamma}_{\theta} \cap  \overline{\Gamma}_0(N^2 )] \overline{M}_{q_i} \overline{r}=[\overline{\Gamma}_{\theta} \cap  \overline{\Gamma}_0(N^2 )] \overline{M}_{q_j}$.
\item[(b)] If the above condition holds, then $ [\overline{\gamma}_{\chi}(\overline{r})(\overline{e}_j)](\overline{M}_{q_i})=\overline{\lambda}_{\chi}(\overline{s})^{-1}$, and $ \supp [\overline{\gamma}_{\chi}(\overline{r})](\overline{e}_j)=[\overline{\Gamma}_{\theta} \cap  \overline{\Gamma}_0(N^2 )] \overline{M}_{q_i}$. So $[\overline{\gamma}_{\chi}(\overline{r})](\overline{e}_j)=\overline{\lambda}_{\chi}(\overline{s})^{-1}\overline{e}_i$, and $[\overline{\gamma}_{\chi}(\overline{r}^{-1})](\overline{e}_i)=\overline{\lambda}_{\chi}(\overline{s})\overline{e}_j$.
\end{itemize}
 \item   
 \begin{itemize}
\item[(c)] $\overline{M}_{q_i}\overline{M}_{q_j}=\overline{s}\overline{M}_{q_r}$, for some $\overline{s}\in \overline{\Gamma}_{\theta} \cap  \overline{\Gamma}_0(N^2 )$ iff $[\overline{\Gamma}_{\theta} \cap  \overline{\Gamma}_0(N^2 )]\overline{M}_{q_i}\overline{M}_{q_j}=[\overline{\Gamma}_{\theta} \cap  \overline{\Gamma}_0(N^2 )]\overline{M}_{q_r}$.
    \item [(d)] If the above condition holds, then $ [\overline{\gamma}_{\chi}(\overline{M}_{q_j})(\overline{e}_r)](\overline{M}_{q_i})=\overline{\lambda}_{\chi}(\overline{s})^{-1}$,  and $ \supp \overline{\gamma}_{\chi}(\overline{M}_{q_j})(\overline{e}_r)=[\overline{\Gamma}_{\theta} \cap  \overline{\Gamma}_0(N^2 )] \overline{M}_{q_i}$.  So $[\overline{\gamma}_{\chi}(\overline{M}_{q_j})](\overline{e}_r)=\overline{\lambda}_{\chi}(\overline{s})^{-1}\overline{e}_i$, and  $[\overline{\gamma}_{\chi}(\overline{M}_{q_j}^{-1})](\overline{e}_i)=\overline{\lambda}_{\chi}(\overline{s})\overline{e}_r$.
    \end{itemize} 
\end{itemize}
We now define:
$$\Theta^{1/2}_{\chi}:  \mathbb{H} \longrightarrow M_{1 \times 3}(\C); z\longmapsto \Big([\theta^{1/2}_{\chi}]^{\overline{M}_{q_1}}(z),[\theta^{1/2}_{\chi}]^{\overline{M}_{q_2}}(z), [\theta^{1/2}_{\chi}]^{\overline{M}_{q_3}}(z)\Big),$$
$$\Theta^{3/2}_{\chi}:  \mathbb{H} \longrightarrow M_{1 \times 3}(\C); z\longmapsto \Big([\theta^{3/2}_{\chi}]^{\overline{M}_{q_1}}(z),[\theta^{3/2}_{\chi}]^{\overline{M}_{q_2}}(z), [\theta^{3/2}_{\chi}]^{\overline{M}_{q_3}}(z)\Big).$$
\subsubsection{}\label{Mqpr1} Let $\overline{r}=(r,t_r)\in \overline{\Gamma}_{\theta} \cap  \overline{\Gamma}_0(N^2 )$. Let $z=z_g\in \mathbb{H}$.  Write $\overline{M}_{q_i}\overline{r}=\overline{s}\overline{M}_{q_j}$, for some $\overline{M}_{q_j}$ , $\overline{s}\in \overline{\Gamma}_{\theta} \cap  \overline{\Gamma}_0(N^2 )$. 
\begin{align*}
[\theta^{1/2}_{\chi}]^{\overline{M}_{q_i}}(\overline{r}z)
&=J_{1/2}(\overline{M}_{q_i},\overline{r}z)^{-1}\theta^{1/2}_{\chi}(\overline{M}_{q_i}\overline{r}z)\\
&=J_{1/2}(\overline{M}_{q_i},\overline{r}z)^{-1}\theta^{1/2}_{\chi}(\overline{s}\overline{M}_{q_j}z)\\
&=J_{1/2}(\overline{M}_{q_i},\overline{r}z)^{-1}J_{1/2}(\overline{s} ,\overline{M}_{q_j}z)\overline{\lambda}_{\chi}(\overline{s}) \theta^{1/2}_{\chi}( \overline{M}_{q_j}z)\\
&=J_{1/2}(\overline{M}_{q_i},\overline{r}z)^{-1}J_{1/2}(\overline{s} ,\overline{M}_{q_j}z)J_{1/2}(\overline{M}_{q_j}, z)\overline{\lambda}_{\chi}(\overline{s})[\theta^{1/2}_{\chi}]^{\overline{M}_{q_j}}(z)\\
&=J_{1/2}(\overline{r},z)\overline{\lambda}_{\chi}(\overline{s})[\theta^{1/2}_{\chi}]^{\overline{M}_{q_j}}(z).
\end{align*}
\subsubsection{}\label{Mqpr2} Let $\overline{M}_{q_i}, \overline{M}_{q_j}\in \overline{\mathcal{M}}_{2}$. Write $\overline{M}_{q_i}\overline{M}_{q_j}=\overline{s}\overline{M}_{q_r}$, for some $\overline{M}_{q_r}\in \overline{\mathcal{M}}_{2}$ , $\overline{s}\in \overline{\Gamma}_{\theta} \cap  \overline{\Gamma}_0(N^2 )$.  
\begin{align*}
[\theta^{1/2}_{\chi}]^{\overline{M}_{q_i}}(\overline{M}_{q_j}z)
&=J_{1/2}(\overline{M}_{q_i},\overline{M}_{q_j}z)^{-1}\theta^{1/2}_{\chi}(\overline{M}_{q_i}\overline{M}_{q_j}z)\\
&=J_{1/2}(\overline{M}_{q_i},\overline{M}_{q_j}z)^{-1}\theta^{1/2}_{\chi}(\overline{s}\overline{M}_{q_r}z)\\
&=J_{1/2}(\overline{M}_{q_i},\overline{M}_{q_j}z)^{-1}J_{1/2}(\overline{s} ,\overline{M}_{q_r}z)\overline{\lambda}_{\chi}(\overline{s}) \theta^{1/2}_{\chi}( \overline{M}_{q_r}z)\\
&=J_{1/2}(\overline{M}_{q_i},\overline{M}_{q_j}z)^{-1}J_{1/2}(\overline{s} ,\overline{M}_{q_r}z)J_{1/2}(\overline{M}_{q_r}, z) \overline{\lambda}_{\chi}(\overline{s}) [\theta^{1/2}_{\chi}]^{\overline{M}_{q_r}}(z)\\
&=J_{1/2}(\overline{M}_{q_j},z)\overline{\lambda}_{\chi}(\overline{s}) [\theta^{1/2}_{\chi}]^{\overline{M}_{q_r}}(z).
\end{align*}
\begin{theorem}\label{maintheorem1}
Let $\overline{r}\in \overline{\Gamma}_0(N^2)$. Then:
 \begin{itemize}
 \item[(1)]$\Theta^{1/2}_{\chi}(\overline{r}z)=J_{1/2}(\overline{r},z)\Theta^{1/2}_{\chi}(z)\overline{\gamma}_{\chi}(\overline{r}^{-1})$.
 \item[(2)] $\Theta^{3/2}_{\chi}(\overline{r}z)=J_{1/2}(\overline{r},z)\Theta^{3/2}_{\chi}(z)\overline{\gamma}_{\chi}(\overline{r}^{-1})$.
 \end{itemize}
\end{theorem}
\begin{proof}
1a) If $\overline{r}\in  \overline{\Gamma}_{\theta} \cap  \overline{\Gamma}_0(N^2 )$,  then by the results of Section \ref{Mqpr1} and the above (a)(b), the equality holds.\\
1b)  If  $\overline{r}=\overline{M}_{q_i}$, by the results of Section \ref{Mqpr2} and the above (c)(d), the equality holds.\\
1c) If $\overline{r}=\overline{s}\overline{M}_{q_i}$, for some $\overline{s}\in \overline{\Gamma}_{\theta} \cap  \overline{\Gamma}_0(N^2 )$, then 
\begin{align*}
\Theta^{1/2}_{\chi}(\overline{r}z)&=\Theta^{1/2}_{\chi}(\overline{s}\overline{M}_{q_i}z)\\
&=J_{1/2}(\overline{s},\overline{M}_{q_i}z)\Theta^{1/2}_{\chi}(\overline{M}_{q_i}z)\overline{\gamma}_{\chi}(\overline{s}^{-1})\\
&=J_{1/2}(\overline{s},\overline{M}_{q_i}z)J_{1/2}(\overline{M}_{q_i},z)\Theta^{1/2}_{\chi}(z)\overline{\gamma}_{\chi}(\overline{M}_{q_i}^{-1})\overline{\gamma}_{\chi}(\overline{s}^{-1})\\
&=J_{1/2}(\overline{r},z)\Theta^{1/2}_{\chi}(z)\overline{\gamma}_{\chi}(\overline{r}^{-1}).
\end{align*}
The  proof of Part (2) is similar.
\end{proof}
\subsection{Induction   to $\overline{\Gamma}_0(N^2)$: $N$ even}\label{Gammaaoneven}
 In this case, let $\overline{M}_{q_1}=[u(1), 1]$ and $\overline{M}_{q_2}=[1, 1]$.  By Lemma \ref{conisoeve}, we have:
\[ \overline{\Gamma}_0(N^2)=\sqcup_{i=1}^2[\overline{\Gamma}_{\theta} \cap  \overline{\Gamma}_0(N^2 )]  \overline{M}_{q_i}, \qquad [\overline{\Gamma}_{\theta} \cap  \overline{\Gamma}_0(N^2 )] \unlhd \overline{\Gamma}_0(N^2).\]

\begin{definition}
$\overline{\gamma}_{\chi}=\Ind_{\overline{\Gamma}_{\theta} \cap  \overline{\Gamma}_0(N^2 )}^{\overline{\Gamma}_0(N^2)} [\overline{\lambda}_{\chi}]^{-1}, \overline{M}_{\chi}=\Ind_{\overline{\Gamma}_{\theta} \cap  \overline{\Gamma}_0(N^2 )}^{\overline{\Gamma}_0(N^2)} \C.$
\end{definition}
\begin{lemma}\label{exenchigamma}
\begin{itemize}
\item[(1)]  If $N\equiv 2(\bmod 4)$, then $\overline{\gamma}_{\chi}$ is an irreducible representation of dimension $2$.
\item[(2)]  If $N\equiv 0(\bmod 4)$,  then $\overline{\gamma}_{\chi}$ is a reducible representation of dimension $2$ with two different irreducible components.
\end{itemize}
\end{lemma}
\begin{proof}
Let $\overline{r}=[\begin{pmatrix}
a&b \\
c& d
\end{pmatrix}, \epsilon]\in \overline{\Gamma}_{\theta} \cap  \overline{\Gamma}_0(N^2 ) \subseteq \overline{\Gamma}(2)$. Then
\[\overline{M}_{q_1}^{-1}\overline{r}\overline{M}_{q_1}=[u(-1)ru(1), \epsilon]=[\begin{pmatrix}
a - c & a + b - c - d \\
c & c + d
\end{pmatrix}, \epsilon],\]
\[
\chi(\overline{M}_{q_{1}}^{-1}\overline{r}\,\overline{M}_{q_{1}})=\chi(\overline{r}),
\]
\begin{gather*}
\overline{\lambda}(\overline{r})
=\epsilon\Bigl(\frac{2c}{d}\Bigr)\epsilon_{d}^{-1},
\qquad
\overline{\lambda}(\overline{M}_{q_{1}}^{-1}\overline{r}\,\overline{M}_{q_{1}})
=\epsilon\Bigl(\frac{2c}{c+d}\Bigr)\epsilon_{c+d}^{-1}.
\end{gather*}
Since $N^2\mid c$, $4\mid c$ and  $\epsilon_{c+d}^{-1}=\epsilon_{d}^{-1} $.  To prove the result, it suffices to check  whether $ \left(\tfrac{2c}{d}\right)$ and $\Bigl(\frac{2c}{c+d}\Bigr)$ are equal or not.
\begin{itemize}
\item[(1)] If $N\equiv 2(\bmod 4)$, then $N^2\equiv 4(\bmod 8)$.  We choose:
\[\begin{pmatrix}
a&b \\
c& d
\end{pmatrix}=\begin{pmatrix}
-1-2N^2&2 \\
N^2& -1
\end{pmatrix}.\]
Then:
\[(\tfrac{2c}{d})=(\tfrac{2N^2}{-1})=1, \quad \Bigl(\frac{2c}{c+d}\Bigr)=\Bigl(\frac{2N^2}{N^2-1}\Bigr)=\Bigl(\frac{2}{N^2-1}\Bigr)=-1.\]
\item[(2)] If $N\equiv 0(\bmod 4)$, then $N^2\equiv 0(\bmod 16)$. Write $c=N^2 r=[\pm 2^{k}] r_1$ with an odd positive number $r_1$.  If $c=0$, $\left(\tfrac{2c}{d}\right)=\Bigl(\frac{2c}{c+d}\Bigr)$. In the following, we assume that $c\neq 0$. Then 
\[\Bigl(\tfrac{2c}{d}\Bigr)\Bigl(\frac{2c}{c+d}\Bigr)=\Bigl(\frac{2c}{(c+d)d}\Bigr)=\Bigl( \frac{\pm 2^k}{(c+d)d}\Bigr)\Bigl(\frac{r_1}{(c+d)d}\Bigr).\]
\begin{itemize}
\item[(i)] If $(c+d)d>0$, then $(c+d)d\equiv cd+d^2\equiv d^2\equiv 1(\bmod 8)$. So $\Bigl(\frac{2}{(c+d)d}\Bigr)=1$.
\item[(ii)] If $(c+d)d<0$, then $-(c+d)d\equiv -cd-d^2\equiv -d^2\equiv -1(\bmod 8)$. So $\Bigl(\frac{2}{(c+d)d}\Bigr)=\Bigl(\frac{2}{-1}\Bigr)\Bigl(\frac{2}{-(c+d)d}\Bigr)=1$.
\item[(iii)] If $(c+d)d>0$, then $(c+d)d\equiv cd+d^2\equiv d^2\equiv 1(\bmod 4)$. So $\Bigl(\frac{-1}{(c+d)d}\Bigr)=1$.
\item[(iv)] If $(c+d)d<0$, then $-(c+d)d\equiv -cd-d^2\equiv -d^2\equiv 3(\bmod 4)$. So $\Bigl(\frac{-1}{(c+d)d}\Bigr)=\Bigl(\frac{-1}{-1}\Bigr)\Bigl(\frac{-1}{-(c+d)d}\Bigr)=(-1)(-1)=1$.
\item[(v)]  If $(c+d)d>0$,  then $\Bigl(\frac{r_1}{(c+d)d}\Bigr)=\Bigl(\frac{(c+d)d}{r_1}\Bigr) (-1)^{\tfrac{(r_1-1)}{2}\tfrac{((c+d)d-1)}{2}}=\Bigl(\frac{d^2}{r_1}\Bigr) (-1)^{\tfrac{(r_1-1)}{2}\tfrac{((c+d)d-1)}{2}}= (-1)^{\tfrac{(r_1-1)}{2}\tfrac{((c+d)d-1)}{2}}=1$.
\item[(vi)]  If $(c+d)d<0$,  then $\Bigl(\frac{r_1}{(c+d)d}\Bigr)=\Bigl(\frac{r_1}{-(c+d)d}\Bigr)=\Bigl(\frac{-(c+d)d}{r_1}\Bigr) (-1)^{\tfrac{(r_1-1)}{2}\tfrac{(-(c+d)d-1)}{2}}=\Bigl(\frac{-d^2}{r_1}\Bigr) (-1)^{\tfrac{(r_1-1)}{2}\tfrac{(-(c+d)d-1)}{2}}=\Bigl(\frac{-1}{r_1}\Bigr) (-1)^{\tfrac{(r_1-1)}{2}\tfrac{(-d^2-1)}{2}}= (-1)^{\tfrac{(r_1-1)}{2}\tfrac{(-d^2+1)}{2}}=1$.
\end{itemize}
Hence, $\Bigl(\tfrac{2c}{d}\Bigr)\Bigl(\frac{2c}{c+d}\Bigr)=1$ and $\Bigl(\tfrac{2c}{d}\Bigr)=\Bigl(\frac{2c}{c+d}\Bigr)$.
\end{itemize}
\end{proof}
We now define:
$$\Theta^{1/2}_{\chi}:  \mathbb{H} \longrightarrow M_{1 \times 2}(\C); z\longmapsto \Big([\theta^{1/2}_{\chi}]^{\overline{M}_{q_1}}(z),[\theta^{1/2}_{\chi}]^{\overline{M}_{q_2}}(z)\Big),$$
$$\Theta^{3/2}_{\chi}:  \mathbb{H} \longrightarrow M_{1 \times 2}(\C); z\longmapsto \Big([\theta^{3/2}_{\chi}]^{\overline{M}_{q_1}}(z),[\theta^{3/2}_{\chi}]^{\overline{M}_{q_2}}(z)\Big).$$
Analogous to the odd case, we have:
\begin{theorem}\label{maintheorem1even}
Let $\overline{r}\in \overline{\Gamma}_0(N^2)$. Then:
 \begin{itemize}
 \item[(1)]$\Theta^{1/2}_{\chi}(\overline{r}z)=J_{1/2}(\overline{r},z)\Theta^{1/2}_{\chi}(z)\overline{\gamma}_{\chi}(\overline{r}^{-1})$;
 \item[(2)] $\Theta^{3/2}_{\chi}(\overline{r}z)=J_{3/2}(\overline{r},z)\Theta^{3/2}_{\chi}(z)\overline{\gamma}_{\chi}(\overline{r}^{-1})$.
 \end{itemize}
\end{theorem}
\subsection{Tensor induction   to $\overline{\Gamma}_0(N^2)$}\label{tensorGammaaonodd}
Let us give the following definitions:
\begin{itemize}
\item If $N$ is odd, 
\[[\theta^{1/2}_{\chi}]^{\otimes}(z)\stackrel{\mathrm{Def.}}{=}[\theta^{1/2}_{\chi}]^{\overline{M}_{q_1}}(z)[\theta^{1/2}_{\chi}]^{\overline{M}_{q_2}}(z)[\theta^{1/2}_{\chi}]^{\overline{M}_{q_3}}(z),\]
\[[\theta^{3/2}_{\chi}]^{\otimes}(z)\stackrel{\mathrm{Def.}}{=}[\theta^{3/2}_{\chi}]^{\overline{M}_{q_1}}(z)[\theta^{3/2}_{\chi}]^{\overline{M}_{q_2}}(z)[\theta^{3/2}_{\chi}]^{\overline{M}_{q_3}}(z).\]
\item If $N$ is even,
\[[\theta^{1/2}_{\chi}]^{\otimes}(z)\stackrel{\mathrm{Def.}}{=}[\theta^{1/2}_{\chi}]^{\overline{M}_{q_1}}(z)[\theta^{1/2}_{\chi}]^{\overline{M}_{q_2}}(z),\]
\[[\theta^{3/2}_{\chi}]^{\otimes}(z)\stackrel{\mathrm{Def.}}{=}[\theta^{3/2}_{\chi}]^{\overline{M}_{q_1}}(z)[\theta^{3/2}_{\chi}]^{\overline{M}_{q_2}}(z).\]
\end{itemize}
\begin{theorem}\label{Twistedbypro2}
Let $g\in G=\overline{\Gamma}_0(N^2)$, $H=\overline{\Gamma}_{\theta} \cap  \overline{\Gamma}_0(N^2 )$.
\begin{itemize}
\item[(1)] If $N$ is odd, then  $[\theta^{\kappa/2}_{\chi}]^{\otimes} (gz)=J_{\kappa/2}(g,z)^3 [\theta^{\kappa/2}_{\chi}]^{\otimes}(z)  \overline{\lambda}_{\chi}(V_{G\to H}(g))$, for $\kappa=1$ or $3$.
\item[(2)] If $N$ is even, then $[\theta^{\kappa/2}_{\chi}]^{\otimes} (gz)=J_{\kappa/2}(g,z)^2 [\theta^{\kappa/2}_{\chi}]^{\otimes}(z)  \overline{\lambda}_{\chi}(V_{G\to H}(g))$, for $\kappa=1$ or $3$.
  \end{itemize}
   \end{theorem}
 \begin{proof}  
 The proof is similar to Section \ref{exampI}. Retain the notations  from Sect. \ref{tensorind}. In this case, $ x_i^{-1}=\overline{M}_{q_i}$, $\sigma=\overline{\lambda}_{\chi}$.  Write $$g^{-1}\overline{M}^{-1}_{q_i}=\overline{M}^{-1}_{q_{\tau_g(i)}}h^{-1}_{g,x_{ \tau_g(i)}}, \quad \overline{M}_{q_i} g=h_{g,x_{ \tau_g(i)}}\overline{M}_{q_{\tau_g(i)}}, \quad \quad  h_{g,x_{ \tau_g(i)}}\in H.$$ \\
 1) By Sections \ref{Mqpr1}--\ref{Mqpr2} 
   \[  [\theta^{\kappa/2}_{\chi}]^{\overline{M}_{q_i}}(gz)=J_{\kappa/2}(g,z)[\theta^{\kappa/2}_{\chi}]^{\overline{M}_{q_{\tau_g(i)}}}(z) \sigma(h_{g,x_{ \tau_g(i)}}).\]
   Hence:
   \begin{align*}
   &[\theta^{\kappa/2}_{\chi}]^{\overline{M}_{q_1}}(gz)[\theta^{\kappa/2}_{\chi}]^{\overline{M}_{q_2}}(gz) [\theta^{\kappa/2}_{\chi}]^{\overline{M}_{q_3}}(gz)\\
   &=J_{\kappa/2}(g,z)^3[\theta^{\kappa/2}_{\chi}]^{\overline{M}_{q_1}}(z)[\theta^{\kappa/2}_{\chi}]^{\overline{M}_{q_2}}(z) [\theta^{\kappa/2}_{\chi}]^{\overline{M}_{q_3}}(z) \sigma^{-1}(V_{G\to H}(g^{-1}))\\
   &=J_{\kappa/2}(g,z)^3[\theta^{\kappa/2}_{\chi}]^{\overline{M}_{q_1}}(z)[\theta^{\kappa/2}_{\chi}]^{\overline{M}_{q_2}}(z) [\theta^{\kappa/2}_{\chi}]^{\overline{M}_{q_3}}(z) \sigma(V_{G\to H}(g)).
   \end{align*}
 2) Similarly.
   \end{proof}
   \section{Siegel modular forms associated to Weil representations: $[\overline{\Gamma}^{\pm}_{\theta} \cap \overline{\Gamma}^{\pm}_0(N^2)]$ \& $\overline{\Gamma}^{\pm}_0(N^2)$}\label{GammathetaN2PM}
   Unless otherwise stated, the following conventions are used throughout this section:
\begin{itemize}
\item $z\in\mathbb H^{\pm}$ and $g=\begin{pmatrix}a&b\\ c&d\end{pmatrix}\in\SL_2^{\pm}(\R)$;
\item $g_i=\begin{pmatrix}a_i&b_i\\ c_i&d_i\end{pmatrix}\in\SL^{\pm}_2(\R)$ for $i=1,2$;
\item the character $\psi$ is fixed to be $\psi_0$;
\item $N$ is a positive integer number, and $\chi$ is a  primitive Dirichlet character modulo $N$;
\item Define a Galois action of $\Gal(\mathbb{C}/\mathbb{R})=\langle \sigma \rangle$ on $\mathbb{C}$ by
$
z\mapsto z^{\sigma}
$.
\end{itemize}
\subsection{Weil representations of $\overline{\GL}_2(\R)$}
Note that 
\[  \overline{\GL}_2(\R) \subseteq \overline{\GL}^8_2(\R), \quad \quad \overline{\GL}^8_2(\R) \stackrel{\iota}{\simeq } \widetilde{\GL}_2(\R); [h,\epsilon] \longmapsto [h,m_{X^{\ast}}(h)\epsilon].\]
We modify the Weil representation from  the $8$-fold covering group $\widetilde{\GL}_2(\R)$  to the  two-fold covering groups $\overline{\GL}_2(\R)$ 
via the above embedding map. We denote  Weil representations of   $\overline{\GL}_2(\R) $ by $\overline{\Pi}_{X^{\ast},\psi}$ and $\overline{\Pi}_{L,\psi}$.  They  differ from  $\Pi_{X^{\ast},\psi}$ and $\Pi_{L,\psi}$ by the factors $m_{X^{\ast}}(h)$. For simplicity,  we will take these  factors in account without  repeating the corresponding results.  
\subsection{Automorphic factors}\label{dewiaf}
We recall   $\overline{\delta}^+$ and $\overline{\delta}^-$ from (\ref{deltapmrep}).  By Lemma \ref{weilDelignegroup}, there exists a group embedding of
$\overline{\Oa}_2(\R)$ into $\mathbb{W}_{\R}$. We extend $\overline{\delta}^+$ and $\overline{\delta}^-$ to $\mathbb{W}_{\R}$. Let $z^{k}\overline{z}^l$ denote the character of  $\C^{\times}$  by sending $z$ to $ z^{k}\overline{z}^l$.  Now define:
\[ (\overline{\delta}^+=\Ind^{\mathbb{W}_{\R}}_{\C^{\times}} z^{-1}, \mathcal{D}^+=\Ind^{\mathbb{W}_{\R}}_{\C^{\times}}\C), \qquad ( \overline{\delta}^-=\Ind^{\mathbb{W}_{\R}}_{\C^{\times}} z^{-3} , \mathcal{D}^-=\Ind^{\mathbb{W}_{\R}}_{\C^{\times}}\C).\]
\begin{lemma}
The restrictions of $\overline{\delta}^+$ and $\overline{\delta}^-$ to $\overline{\Oa}_2(\R)$ are compatible with  the representations from (\ref{deltapmrep}).
\end{lemma}
\begin{proof}
Recall:
\[\overline{\SO}_2(\R) \hookrightarrow \C^{\times}; [g=\begin{pmatrix}a&-b\\b&a\end{pmatrix}, \epsilon] \longmapsto  \epsilon^{-1} \overline{s}(g)=\epsilon^{-1}\sqrt{u};\]
\[[h_{-1},1] \longmapsto \begin{pmatrix}i&0\\0&-i\end{pmatrix} \omega_{\sigma}.\]
Form (\ref{deltapmrep}), \[\overline{\delta}^+\simeq \Ind_{\overline{\SO}_2(\R)}^{\overline{\Oa}_2(\R)} (\overline{s}^{-1}\cdot 1_{\mu_2}), \quad \overline{\delta}^-\simeq \Ind_{\overline{\SO}_2(\R)}^{\overline{\Oa}_2(\R)} (\overline{s}^{-1} \overline{u}\cdot 1_{\mu_2}).\]
By observation, the result holds.
\end{proof}
Recall  $\mathbb{W}_{\R}=\C^{\times} \cup \C^{\times} \omega_{\sigma}$ from Sect. \ref{theWdg1}.  Let $e_1$ and $e_2$ be a basis of functions in  $\mathcal{D}^+$ such that (1) $\supp e_1=\C^{\times}$ and $\supp e_2=\C^{\times} \omega_{\sigma}$, (2) $e_1(1)=1=e_2(i\omega_{\sigma})$.   For any $t\in \C^{\times}$, 
\[ \overline{\delta}^+(t) e_1(1)= e_1(t)=t^{-1} e_1(1), \quad \overline{\delta}^+(t) e_2(i\omega_{\sigma})= e_2(i\omega_{\sigma}t)= e_2(\overline{t}i\omega_{\sigma})=\overline{t}^{-1} e_2(i\omega_{\sigma}).\]
So the restriction of $\overline{\delta}^+$ to $\C^{\times}$ is isomorphic to the representation:
 $$z\to \begin{pmatrix} z^{-1}&0 \\ 0& \overline{z}^{-1}\end{pmatrix}.$$ Similarly,  the restriction of $\overline{\delta}^-$ to $\C^{\times}$ is isomorphic to the representation:
 $$z\to \begin{pmatrix}z^{-3} &0 \\ 0& \overline{z}^{-3}\end{pmatrix}.$$
For $\overline{h} \in \overline{\GL}_2(\R)$ and $z\in \mathbb{H}^{\pm}$, let us define:
\begin{itemize}
\item[(1)] $J_{\delta^+}(\overline{h}, z)=  (\overline{\delta}^+)^{-1}(J_{1/2}(\overline{h}, z))= \begin{pmatrix} J_{1/2}(\overline{h}, z)&0 \\ 0& \overline{J_{1/2}(\overline{h}, z)}\end{pmatrix}\\
=\begin{pmatrix} \epsilon^{-1}\sgn(h,z)\sqrt{\sgn(\det h)(cz+d)}&0 \\ 0& \overline{\epsilon^{-1}\sgn(h,z)\sqrt{\sgn(\det h)(cz+d)}}\end{pmatrix}$ ;
\item[(2)] $J_{\delta^-}(\overline{h}, z)= \sgn(h) [(\overline{\delta}^-)^{-1}( J_{1/2}(\overline{h}, z))]=\begin{pmatrix}\sgn(h)J_{1/2}(\overline{h}, z)^3&0 \\ 0& \sgn(h)\overline{J_{1/2}(\overline{h}, z)}^3\end{pmatrix}$\\
    $=\begin{pmatrix}J_{3/2}(\overline{h}, z) &0 \\ 0&\overline{J_{3/2}(\overline{h}, z)}\end{pmatrix}$\\
$=\begin{pmatrix}\epsilon^{-1}\sgn(h,z)(\sqrt{\sgn(\det h)(cz+d)})(cz+d) &0 \\ 0&\epsilon^{-1}\sgn(h,z)\overline{\big(\sqrt{\sgn(\det h)(cz+d)}\big)(cz+d)}\end{pmatrix}$. 
\end{itemize}
\begin{lemma}
$J_{\delta^{\pm}}(\overline{h}_1\overline{h}_2, z)=J_{\delta^{\pm}}(\overline{h}_1,\overline{h}_2 z)J_{\delta^{\pm}}(\overline{h}_2, z)$, for $\overline{h}_i \in \overline{\GL}_2(\R)$.
\end{lemma}
\begin{proof}
It follows from Lemma \ref{Automorphicfactor}.
\end{proof}
\subsection{ $\overline{\Oa}_2(\R)$}We recall the functions $A^{\pm}$, $B^{\pm}$ from Lemma~\ref{ABM}. For $\overline{g}=[g,\epsilon]\in \overline{\SO}_2(\R)$, we have:
\[\overline{\Pi}_{X^{\ast},\psi }(\overline{g}) A^{+}=[\Pi_{X^{\ast},\psi }(g)  A^{+}] \epsilon m_{X^{\ast}}(g) =\widetilde{s}(g)^{-1} m_{X^{\ast}}(g) \epsilon A^+=\overline{s}(g)^{-1}\epsilon A^+=\overline{\delta}^+(\overline{g})  A^+,\]
\[\overline{\Pi}_{X^{\ast},\psi }(\overline{g}) A^{-}=[\Pi_{X^{\ast},\psi }(g)  A^{-}] \epsilon m_{X^{\ast}}(g) =u\widetilde{s}(g)^{-1} m_{X^{\ast}}(g) \epsilon A^{-}=\overline{s}(g)^{-1}\epsilon uA^{-}=\overline{\delta}^+(\overline{g})  A^-.\]
\[\overline{\Pi}_{X^{\ast},\psi }([h_{-1},1]) A^{+}=[\Pi_{X^{\ast},\psi }(h_{-1})  A^{+}]  m_{X^{\ast}}(h_{-1})=\Pi_{X^{\ast},\psi }(h_{-1})  A^{+}=A^-=\overline{\delta}^+([h_{-1},1])  A^+.\]
\[\overline{\Pi}_{X^{\ast},\psi }([h_{-1},1]) A^{-}=A^{+}=\overline{\delta}^+([h_{-1},1])  A^-.\]
Therefore, we have:
\[\overline{\Pi}_{X^{\ast},\psi }(\overline{g}) \begin{bmatrix} A^{+}\\ A^{-}\end{bmatrix}= \overline{\delta}^+(\overline{g}) \begin{bmatrix} A^{+}\\ A^{-}\end{bmatrix}, \quad \textrm{ and similarly } \overline{\Pi}_{X^{\ast},\psi }(\overline{g}) \begin{bmatrix} B^{+}\\ B^{-}\end{bmatrix}= \overline{\delta}^-(\overline{g}) \begin{bmatrix} B^{+}\\ B^{-}\end{bmatrix},\]
where
\begin{itemize}
\item $\overline{\delta}^+: [g=\begin{pmatrix} a& -b\\ b& a\end{pmatrix}, \epsilon] \to \begin{pmatrix} \epsilon\sqrt{ai+b}^{-1}&0 \\ 0& \epsilon\overline{\sqrt{ai+b}}^{-1}\end{pmatrix}$ and $ [g=\begin{pmatrix} 1& 0\\ 0&-1\end{pmatrix}, 1] \to \begin{pmatrix}0&1 \\ 1& 0\end{pmatrix}$, 
     \item $\overline{\delta}^-: [g=\begin{pmatrix} a& -b\\ b& a\end{pmatrix}, \epsilon] \to \begin{pmatrix} \epsilon \sqrt{ai+b}^{-3}&0 \\ 0& \epsilon \overline{\sqrt{ai+b}}^{-3}\end{pmatrix} $ and  $[g=\begin{pmatrix} 1& 0\\ 0&-1\end{pmatrix}, 1] \to \begin{pmatrix}0&1 \\ 1& 0\end{pmatrix}$.
\end{itemize}
For $\overline{g}=[g,\epsilon]\in \overline{\SO}_2(\R)$, we have:
\[J_{\delta^+}(\overline{g}, z_0)\overline{\Pi}_{X^{\ast},\psi }(\overline{g}) \begin{bmatrix} A^{+}\\ A^{-}\end{bmatrix}=\begin{bmatrix} A^{+}\\ A^{-}\end{bmatrix}, \quad J_{\delta^-}(\overline{g}, z_0)\overline{\Pi}_{X^{\ast},\psi }(\overline{g}) \begin{bmatrix} B^{+}\\ B^{-}\end{bmatrix}=\begin{bmatrix}B^{+}\\ B^{-}\end{bmatrix}.\]
For $\overline{g}=[h_{-1},1]\in \overline{\Oa}_2(\R)$, we have:
\[J_{\delta^+}(\overline{g}, z_0)\overline{\Pi}_{X^{\ast},\psi }(\overline{g}) \begin{bmatrix} A^{+}\\ A^{-}\end{bmatrix}=\begin{bmatrix} A^{-}\\ A^{+}\end{bmatrix}, \quad J_{\delta^-}(\overline{g}, z_0)\overline{\Pi}_{X^{\ast},\psi }(\overline{g}) \begin{bmatrix} B^{+}\\ B^{-}\end{bmatrix}=\begin{bmatrix}-B^{-}\\ -B^{+}\end{bmatrix}.\]
\subsection{ $ P^{\pm}_{>0}(\R)$}
Let $p=\begin{pmatrix} a & b\\ 0&  \det(p) a^{-1} \end{pmatrix}\in P^{\pm}_{>0}(\R)$($a>0$) and write $z_p=p(i)=\det(p)ba+i\det(p)a^{2}\in\mathbb{H}^{\pm}$.    
By \eqref{chap4representationsp22}--\eqref{chap4representationsp24}, we have: 
\[m_{X^{\ast}}(p)=1,\]
\begin{align*}
\overline{\Pi}_{X^{\ast}, \psi} (p)f( [\epsilon,x])&=\Pi_{X^{\ast}, \psi} (p)f( [\epsilon,x])\\
&= \Pi_{X^{\ast}, \psi} (h_{\det(p)}h(a) u(a^{-1}b))f( [\epsilon,x])\\
&= \Pi_{X^{\ast}, \psi} (h(a) u(a^{-1}b))f( [\det(p)\epsilon,x])\\
&=|a|^{1/2} (a,\det(p)\epsilon)_\R \Pi_{X^{\ast}, \psi} ( u(a^{-1}b))f([\det(p)\epsilon,xa])\\
&=|a|^{1/2} (a,\det(p)\epsilon)_\R \psi(\tfrac{1}{2}\det(p)\epsilon a^{-1}b (xa)^2)f([\det(p)\epsilon,xa])\\
&=|a|^{1/2} (a,\det(p)\epsilon)_\R \psi(\tfrac{1}{2}\det(p)\epsilon ab x^2)f([\det(p)\epsilon,xa])\\
&= |a|^{1/2} (a,\det(p)\epsilon)_\R e^{\pi i \det(p)\epsilon ab x^2}f([\det(p)\epsilon,xa])\\
&=a^{1/2} e^{\pi i \det(p)\epsilon ab x^2}f([\det(p)\epsilon,xa]).
\end{align*}
\begin{align*}
 J_{1/2}(p,  \pm z_0)\stackrel{\textrm{ Lem. } \ref{jpz}}{=}|a|^{-1/2}=a^{-1/2};
 \end{align*}
\begin{align*}
 [J_{1/2}(p,  \pm z_0)]^3=a^{-3/2};
 \end{align*}
If $\det(p)>0$, then:
$$J_{1/2}(p, z_0)\overline{\Pi}_{X^{\ast}, \psi} (p)A^{+}( [\epsilon,x])=\left\{\begin{array}{lr} e^{  \pi i x^2 z_p}& \epsilon=1,\\ 0 &  \epsilon=-1,\end{array}\right.$$
$$J_{1/2}(p, z_0)\overline{\Pi}_{X^{\ast}, \psi} (p)A^{-}( [\epsilon,x])=\left\{\begin{array}{lr}0& \epsilon=1,\\   e^{- \pi i x^2 z_p^{\sigma}} &  \epsilon=-1,\end{array}\right.$$
$$[J_{1/2}(p,   z_0)]^3\overline{\Pi}_{X^{\ast}, \psi} (p)B^{+}( [\epsilon,x])=\left\{\begin{array}{lr}xe^{ \pi i x^2 z_p}& \epsilon=1,\\ 0 &  \epsilon=-1,\end{array}\right.$$
$$[J_{1/2}(p,   z_0)]^3\overline{\Pi}_{X^{\ast}, \psi} (p)B^{-}( [\epsilon,x])=\left\{\begin{array}{lr}0& \epsilon=1,\\  xe^{-  \pi i x^2 z_p^{\sigma}} &  \epsilon=-1.\end{array}\right.$$
If $\det(p)<0$, then:
$$J_{1/2}(p, z_0)\overline{\Pi}_{X^{\ast}, \psi} (p)A^{+}( [\epsilon,x])=e^{-\pi i \epsilon ab x^2}A^{+}([-\epsilon,xa])=\left\{\begin{array}{lr}0 & \epsilon=1,\\ e^{ - \pi i x^2 z_p} &  \epsilon=-1,\end{array}\right.$$
$$J_{1/2}(p, z_0)\overline{\Pi}_{X^{\ast}, \psi} (p)A^{-}( [\epsilon,x])=\left\{\begin{array}{lr}e^{\pi i x^2 z_p^{\sigma}}& \epsilon=1,\\  0  &  \epsilon=-1,\end{array}\right.$$
$$[J_{1/2}(p,   z_0)]^3\overline{\Pi}_{X^{\ast}, \psi} (p)B^{+}( [\epsilon,x])=\left\{\begin{array}{lr}0& \epsilon=1,\\xe^{ -\pi i x^2 z_p} &  \epsilon=-1,\end{array}\right.$$
$$[J_{1/2}(p,   z_0)]^3\overline{\Pi}_{X^{\ast}, \psi} (p)B^{-}( [\epsilon,x])=\left\{\begin{array}{lr}xe^{\pi i x^2 z_p^{\sigma}} & \epsilon=1,\\ 0 &  \epsilon=-1.\end{array}\right.$$
We conclude:
\begin{itemize}
\item For $\overline{p}=[p,1]\in P_{>0}(\R)$, we have:
 \[J_{\delta^+}(\overline{p}, z_0)\overline{\Pi}_{X^{\ast},\psi }(\overline{p}) \begin{bmatrix} A^{+}([1,x])&  A^{+}([-1,x])\\ A^{-}([1,x])&  A^{-}([-1,x])\end{bmatrix}=\begin{bmatrix}e^{  \pi i x^2 z_p}& 0\\ 0& e^{- \pi i x^2 z_p^{\sigma}}\end{bmatrix},\]
\[J_{\delta^-}(\overline{p}, z_0)\overline{\Pi}_{X^{\ast},\psi }(\overline{p}) \begin{bmatrix} B^{+}([1,x])&  B^{+}([1,x])\\ B^{-}([1,x])& B^{-}([-1,x])\end{bmatrix}=\begin{bmatrix}xe^{  \pi i x^2 z_p}& 0\\ 0&x e^{- \pi i x^2 z_p^{\sigma}}\end{bmatrix}.\]
\item  For $\overline{h}_{-1}=[h_{-1},1]$, we have:
\[ \overline{h}_{-1} \overline{p}=[h_{-1}p, \overline{C}_{X^{\ast}}(h_{-1},p)]=[h_{-1}p,1],\]
\[z_{h_{-1}p}=-z_p,\]
\begin{align*}
&J_{\delta^+}(\overline{h}_{-1}\overline{p}, z_0)\overline{\Pi}_{X^{\ast},\psi }(\overline{h}_{-1}\overline{p}) \begin{bmatrix} A^{+}([1,x])&  A^{+}([-1,x])\\ A^{-}([1,x])&  A^{-}([-1,x])\end{bmatrix}\\
&=\begin{bmatrix}0& e^{  \pi i x^2 z_p}\\  e^{- \pi i x^2 z_p^{\sigma}}& 0\end{bmatrix}\\
&=\begin{bmatrix}e^{  \pi i x^2 z_p}& 0\\ 0& e^{- \pi i x^2 z_p^{\sigma}}\end{bmatrix}\overline{\delta}^+(\overline{h}_{-1}),
\end{align*}
\begin{align*}
&J_{\delta^-}(\overline{h}_{-1}\overline{p}, z_0)\overline{\Pi}_{X^{\ast},\psi }(\overline{h}_{-1}\overline{p}) \begin{bmatrix} B^{+}([1,x])&  B^{+}([-1,x])\\ B^{-}([1,x])& B^{-}([-1,x])\end{bmatrix}\\
&=\begin{bmatrix}0& -xe^{  \pi i x^2 z_p}\\  -xe^{- \pi i x^2 z_p^{\sigma}}& 0\end{bmatrix}\\
&=\begin{bmatrix}-xe^{  \pi i x^2 z_p}& 0\\ 0& -xe^{- \pi i x^2 z_p^{\sigma}}\end{bmatrix}\overline{\delta}^-(\overline{h}_{-1}).
\end{align*}
\end{itemize}
 \subsection{ $ \SL^{\pm}_2(\R)$}
 Write $z_g=g(i)=x_g+iy_g\in \mathbb{H}^{\pm}$, $g=p_gk_g$(cf. Lem. \ref{decomgl2}), $z_g=z_{p_g}$. Then:
\[ p_g=\begin{pmatrix} \sqrt{y_g (\det g)} &  (\det g) x_g(\sqrt{y_g (\det g)})^{-1}\\ 0&  (\det g) (\sqrt{y_g (\det g)})^{-1} \end{pmatrix};\]
\[ k_g = \begin{pmatrix}(\sqrt{y_g (\det g)})^{-1}a - (\sqrt{y_g (\det g)})^{-1}x_g c & (\sqrt{y_g (\det g)})^{-1}b- (\sqrt{y_g (\det g)})^{-1}x_g d\\  (\det g) (\sqrt{y_g (\det g)}) c &   (\det g) (\sqrt{y_g (\det g)})d\end{pmatrix}. \]
For $\overline{g}=[g,\epsilon]\in \overline{\SL}^{\pm}_2(\R)$, we write $\overline{g}=p_g \overline{k}_g$ with $p_g\in P^{\pm}_{>0}(\R)$, $\overline{k}_g \in \overline{\SO}_2(\R)$.\\
(1)
 \begin{align*}
&J_{\delta^+}(\overline{g}, z_0)\overline{\Pi}_{X^{\ast}, \psi} (\overline{g}) \begin{bmatrix} A^{+} \\ A^{-} \end{bmatrix}([1,x])\\
&=J_{\delta^+}(p_g, z_0)\overline{\Pi}_{X^{\ast}, \psi} (p_g)J_{\delta^+}(\overline{k}_g, z_0)\overline{\Pi}_{X^{\ast}, \psi} (\overline{k}_g) \begin{bmatrix} A^{+} \\ A^{-} \end{bmatrix}([1,x])\\
&=J_{\delta^+}(p_g, z_0)\overline{\Pi}_{X^{\ast}, \psi} (p_g) \begin{bmatrix} A^{+} \\ A^{-} \end{bmatrix}([1,x])\\
&=\left\{ \begin{array}{lr} \begin{bmatrix}e^{  \pi i x^2 z_g}\\ 0\end{bmatrix}, & \det(g)>0,\\ \begin{bmatrix}0\\e^{\pi i x^2 z_g^{\sigma}}\end{bmatrix}, & \det(g)<0.\end{array}\right.
\end{align*}
 \begin{align*}
&J_{\delta^+}(\overline{g}, z_0)\overline{\Pi}_{X^{\ast}, \psi} (\overline{g}) \begin{bmatrix} A^{+} \\ A^{-} \end{bmatrix}([-1,x])\\
&=J_{\delta^+}(p_g, z_0)\overline{\Pi}_{X^{\ast}, \psi} (p_g)J_{\delta^+}(\overline{k}_g, z_0)\overline{\Pi}_{X^{\ast}, \psi} (\overline{k}_g) \begin{bmatrix} A^{+} \\ A^{-} \end{bmatrix}([-1,x])\\
&=J_{\delta^+}(p_g, z_0)\overline{\Pi}_{X^{\ast}, \psi} (p_g) \begin{bmatrix} A^{+} \\ A^{-} \end{bmatrix}([-1,x])\\
&=\left\{ \begin{array}{lr} \begin{bmatrix}0\\  e^{- \pi i x^2 z_p^{\sigma}}\end{bmatrix}, & \det(g)>0,\\ \begin{bmatrix}e^{-\pi i x^2 z_g}\\0 \end{bmatrix}, & \det(g)<0.\end{array}\right.
\end{align*}
(2) \begin{align*}
&J_{\delta^-}(\overline{g}, z_0)\overline{\Pi}_{X^{\ast}, \psi} (\overline{g}) \begin{bmatrix} B^{+} \\ B^{-} \end{bmatrix}([1,x])\\
&=J_{\delta^-}(p_g, z_0)\overline{\Pi}_{X^{\ast}, \psi} (p_g)J_{\delta^+}(\overline{k}_g, z_0)\overline{\Pi}_{X^{\ast}, \psi} (\overline{k}_g) \begin{bmatrix}B^{+} \\ B^{-} \end{bmatrix}([1,x])\\
&=J_{\delta^-}(p_g, z_0)\overline{\Pi}_{X^{\ast}, \psi} (p_g) \begin{bmatrix} B^{+} \\ B^{-} \end{bmatrix}([1,x])\\
&=\left\{ \begin{array}{lr} \begin{bmatrix}xe^{  \pi i x^2 z_g}\\ 0\end{bmatrix}, & \det(g)>0,\\ \begin{bmatrix}0\\-xe^{\pi i x^2 z_g^{\sigma}}\end{bmatrix}, & \det(g)<0.\end{array}\right.
\end{align*}
 \begin{align*}
&J_{\delta^-}(\overline{g}, z_0)\overline{\Pi}_{X^{\ast}, \psi} (\overline{g}) \begin{bmatrix} B^{+} \\ B^{-} \end{bmatrix}([-1,x])\\
&=J_{\delta^-}(p_g, z_0)\overline{\Pi}_{X^{\ast}, \psi} (p_g)J_{\delta^+}(\overline{k}_g, z_0)\overline{\Pi}_{X^{\ast}, \psi} (\overline{k}_g) \begin{bmatrix} B^{+} \\ B^{-} \end{bmatrix}([-1,x])\\
&=J_{\delta^-}(p_g, z_0)\overline{\Pi}_{X^{\ast}, \psi} (p_g) \begin{bmatrix} B^{+} \\ B^{-} \end{bmatrix}([-1,x])\\
&=\left\{ \begin{array}{lr} \begin{bmatrix}0\\ x e^{- \pi i x^2 z_p^{\sigma}}\end{bmatrix}, & \det(g)>0,\\ \begin{bmatrix}-xe^{-\pi i x^2 z_g}\\0 \end{bmatrix}, & \det(g)<0.\end{array}\right.
\end{align*}
\subsection{ $\Gamma_{\theta}^{\pm}$} 
Let $f\in \mathcal{H}^{\pm}(X^{\ast})$, $g\in \SL^{\pm}_2(\R)$ and  $ \overline{g}=[g,\epsilon] \in \overline{\SL}^{\pm}_2(\R)$. 
 \begin{equation}
\theta_{L, X^{\ast}}(f)([\epsilon,0])=\sum_{l\in L\cap X}f([\epsilon, l])=\sum_{n\in \Z}f( [\epsilon,n]). 
 \end{equation}
\begin{align*}
\overline{ \chi}(k)\theta_{L, X^{\ast}}(f)([\epsilon,\tfrac{k}{N}e^{\ast}])=\sum_{n\in \Z}f([\epsilon, n])\overline{ \chi}(k)\psi(\tfrac{\epsilon kn}{N}).
 \end{align*}
 \begin{align*}
\overline{ \chi}(k)\theta_{L, X^{\ast}}(f)([\epsilon,cNke+\tfrac{kd}{N}e^{\ast}])
&=\chi(d)\sum_{n\in \Z}f([\epsilon, n])\overline{ \chi}(kd)\psi(\tfrac{\epsilon kdn}{N}).
 \end{align*}

  \begin{align*}
   \begin{bmatrix}\theta^{+, 1/2}_{1, \chi}(\overline{g}) \\\theta^{-, 1/2}_{1,\chi}(\overline{g}) \end{bmatrix}&\stackrel{Def.}{=}  \tfrac{1}{G(1, \overline{\chi})}\sum_{k \bmod N}\overline{ \chi}(k) \theta_{L, X^{\ast}}\Big(J_{\delta^+}(\overline{g}, z_0)\overline{\Pi}_{X^{\ast}, \psi} (\overline{g})\begin{bmatrix} A^{+} \\ A^{-} \end{bmatrix}\Big)([1, \tfrac{k}{N}e^{\ast}])\\
&=\tfrac{1}{G(1, \overline{\chi})} \sum_{k \bmod N} \sum_{n\in \Z}J_{\delta^+}(\overline{g}, z_0)\overline{\Pi}_{X^{\ast}, \psi} (\overline{g})\begin{bmatrix} A^{+} \\ A^{-} \end{bmatrix}( [1, n])\overline{ \chi}(k) \psi(\tfrac{kn}{N})\\
&= \left\{ \begin{array}{c}  \left\{ \begin{array}{cr} \tfrac{1}{G(1, \overline{\chi})}\sum_{k \bmod N} \sum_{n\in \Z} e^{i   \pi z_g n^2 }\overline{ \chi}(k) \psi(\tfrac{kn}{N})& \det(g)>0, \\0 & \det(g)<0,\end{array}\right.\\  \left\{\begin{array}{cr}0  & \det(g)>0, \\  \tfrac{1}{G(1, \overline{\chi})}\sum_{k \bmod N} \sum_{n\in \Z} e^{i   \pi z^{\sigma}_g n^2 }\overline{ \chi}(k) \psi(\tfrac{kn}{N}) & \det(g)<0,\end{array}\right.\end{array}\right.\\
&=\left\{ \begin{array}{c} \left\{ \begin{array}{cr} \sum_{n\in \Z} e^{i   \pi z_g n^2}  \chi(n) & \det(g)>0, \\0 & \det(g)<0.\end{array}\right. \\  \left\{ \begin{array}{cr} 0 & \det(g)>0, \\\sum_{n\in \Z} e^{i   \pi z_g^{\sigma} n^2}  \chi(n) & \det(g)<0.\end{array}\right.\end{array}\right.
 \end{align*}
  \begin{align*}
   \begin{bmatrix}\theta^{+, 1/2}_{-1, \chi}(\overline{g}) \\\theta^{-, 1/2}_{-1,\chi}(\overline{g}) \end{bmatrix}&\stackrel{Def.}{=}  \tfrac{1}{G(1, \overline{\chi})}\sum_{k \bmod N}\overline{ \chi}(k) \theta_{L, X^{\ast}}\Big(J_{\delta^+}(\overline{g}, z_0)\overline{\Pi}_{X^{\ast}, \psi} (\overline{g})\begin{bmatrix} A^{+} \\ A^{-} \end{bmatrix}\Big)([-1, \tfrac{k}{N}e^{\ast}])\\
&=\tfrac{1}{G(1, \overline{\chi})} \sum_{k \bmod N} \sum_{n\in \Z}J_{\delta^+}(\overline{g}, z_0)\overline{\Pi}_{X^{\ast}, \psi} (\overline{g})\begin{bmatrix} A^{+} \\ A^{-} \end{bmatrix}( [-1, n])\overline{ \chi}(k) \psi(-\tfrac{kn}{N})\\
&= \left\{ \begin{array}{c}  \left\{ \begin{array}{cr}0 & \det(g)>0, \\ \tfrac{1}{G(1, \overline{\chi})}\sum_{k \bmod N} \sum_{n\in \Z} e^{-i   \pi z_g n^2 }\overline{ \chi}(k) \psi(-\tfrac{kn}{N}) & \det(g)<0,\end{array}\right.\\  \left\{\begin{array}{cr}\tfrac{1}{G(1, \overline{\chi})}\sum_{k \bmod N} \sum_{n\in \Z} e^{-i   \pi z^{\sigma}_g n^2 }\overline{ \chi}(k) \psi(-\tfrac{kn}{N})  & \det(g)>0, \\   0 & \det(g)<0,\end{array}\right.\end{array}\right.\\
&\stackrel{\chi(n)=\chi(-n)}{=}\left\{ \begin{array}{c} \left\{ \begin{array}{cr}0 & \det(g)>0, \\ \sum_{n\in \Z} e^{-i   \pi z_g n^2}  \chi(n)  & \det(g)<0.\end{array}\right. \\  \left\{ \begin{array}{cr} \sum_{n\in \Z} e^{-i   \pi z_g^{\sigma} n^2}  \chi(n)  & \det(g)>0, \\0& \det(g)<0.\end{array}\right.\end{array}\right.
 \end{align*}
   \begin{align*}
   \begin{bmatrix}\theta^{+, 3/2}_{1, \chi}(\overline{g}) \\\theta^{-, 3/2}_{1, \chi}(\overline{g}) \end{bmatrix}&\stackrel{Def.}{=}  \tfrac{1}{G(1, \overline{\chi})}\sum_{k \bmod N}\overline{ \chi}(k) \theta_{L, X^{\ast}}\Big(J_{\delta^-}(\overline{g}, z_0)\overline{\Pi}_{X^{\ast}, \psi} (\overline{g})\begin{bmatrix} B^{+} \\ B^{-} \end{bmatrix}\Big)([1, \tfrac{k}{N}e^{\ast}])\\
&=\tfrac{1}{G(1, \overline{\chi})} \sum_{k \bmod N} \sum_{n\in \Z}J_{\delta^-}(\overline{g}, z_0)\overline{\Pi}_{X^{\ast}, \psi} (\overline{g})\begin{bmatrix} B^{+} \\ B^{-} \end{bmatrix}( [1, n])\overline{ \chi}(k) \psi(\tfrac{kn}{N})\\
&= \left\{ \begin{array}{c}  \left\{ \begin{array}{cr} \tfrac{1}{G(1, \overline{\chi})}\sum_{k \bmod N} \sum_{n\in \Z} ne^{i   \pi z_g n^2 }\overline{ \chi}(k) \psi(\tfrac{kn}{N})& \det(g)>0, \\0 & \det(g)<0,\end{array}\right.\\  \left\{\begin{array}{cr}0  & \det(g)>0, \\  \tfrac{1}{G(1, \overline{\chi})}\sum_{k \bmod N} \sum_{n\in \Z}(-n) e^{i   \pi z^{\sigma}_g n^2 }\overline{ \chi}(k) \psi(\tfrac{kn}{N}) & \det(g)<0,\end{array}\right.\end{array}\right.\\
&=\left\{ \begin{array}{c} \left\{ \begin{array}{cr} \sum_{n\in \Z}n e^{i   \pi z_g n^2}  \chi(n) & \det(g)>0, \\0 & \det(g)<0.\end{array}\right. \\  \left\{ \begin{array}{cr} 0 & \det(g)>0, \\\sum_{n\in \Z} ne^{i   \pi z_g^{\sigma} n^2}  \chi(-n) & \det(g)<0.\end{array}\right.\end{array}\right.
 \end{align*}
 \begin{align*}
   \begin{bmatrix}\theta^{+, 3/2}_{-1, \chi}(\overline{g}) \\\theta^{-, 3/2}_{-1,\chi}(\overline{g}) \end{bmatrix}&\stackrel{Def.}{=}  \tfrac{1}{G(1, \overline{\chi})}\sum_{k \bmod N}\overline{ \chi}(k) \theta_{L, X^{\ast}}\Big(J_{\delta^+}(\overline{g}, z_0)\overline{\Pi}_{X^{\ast}, \psi} (\overline{g})\begin{bmatrix} B^{+} \\ B^{-} \end{bmatrix}\Big)([-1, \tfrac{k}{N}e^{\ast}])\\
&=\tfrac{1}{G(1, \overline{\chi})} \sum_{k \bmod N} \sum_{n\in \Z}J_{\delta^-}(\overline{g}, z_0)\overline{\Pi}_{X^{\ast}, \psi} (\overline{g})\begin{bmatrix} B^{+} \\ B^{-} \end{bmatrix}( [-1, n])\overline{ \chi}(k) \psi(-\tfrac{kn}{N})\\
&= \left\{ \begin{array}{c}  \left\{ \begin{array}{cr}0 & \det(g)>0, \\ \tfrac{1}{G(1, \overline{\chi})}\sum_{k \bmod N} \sum_{n\in \Z}(-n) e^{-i   \pi z_g n^2 }\overline{ \chi}(k) \psi(-\tfrac{kn}{N}) & \det(g)<0,\end{array}\right.\\  \left\{\begin{array}{cr}\tfrac{1}{G(1, \overline{\chi})}\sum_{k \bmod N} \sum_{n\in \Z} ne^{-i   \pi z^{\sigma}_g n^2 }\overline{ \chi}(k) \psi(-\tfrac{kn}{N})  & \det(g)>0, \\   0 & \det(g)<0,\end{array}\right.\end{array}\right.\\
&\stackrel{\chi(-n)=-\chi(n)}{=}\left\{ \begin{array}{c} \left\{ \begin{array}{cr}0 & \det(g)>0, \\ \sum_{n\in \Z} ne^{-i   \pi z_g n^2}  \chi(n)  & \det(g)<0.\end{array}\right. \\  \left\{ \begin{array}{cr} \sum_{n\in \Z} ne^{-i   \pi z_g^{\sigma} n^2}  \chi(-n)  & \det(g)>0, \\0& \det(g)<0.\end{array}\right.\end{array}\right.
 \end{align*}
Note that $\theta^{\pm,1/2}_{\epsilon, \chi}(\overline{g})$, $\theta^{\pm, 3/2}_{\epsilon, \chi}(\overline{g})$  only depend on $z_g=g(z_0)$, $z_g^{\sigma}=g(-z_0)$. 
 \begin{definition}
\begin{itemize}
\item[(1)] If $\chi$ is even, define:
\begin{itemize}
\item[(a)]
$
\theta^{+, 1/2}_{1, \chi}(z) = \left\{ \begin{array}{cr}\sum_{n \in \mathbb{Z}} \chi(n)\, e^{i\pi z n^2}& \Im(z)>0, \\0 & \Im(z)<0,\end{array}\right.$ $
\theta^{+, 1/2}_{-1, \chi}(z) = \left\{ \begin{array}{cr}0& \Im(z)>0, \\\sum_{n \in \mathbb{Z}} \chi(n)\, e^{-i\pi z n^2} & \Im(z)<0.\end{array}\right.$
\item[(b)] $
\theta^{-, 1/2}_{1, \chi}(z) = \left\{ \begin{array}{cr} 0 &\Im(z)>0, \\\sum_{n\in \Z} \chi(n) e^{i   \pi z^{\sigma} n^2}  & \Im(z)<0,\end{array}\right.$ $\theta^{-, 1/2}_{-1, \chi}(z) = \left\{ \begin{array}{cr} \sum_{n\in \Z} \chi(n) e^{-i   \pi z^{\sigma} n^2}  &\Im(z)>0, \\0 & \Im(z)<0.\end{array}\right.$
\end{itemize}
\item[(2)] If $\chi$ is odd,  define:
\begin{itemize}
\item[(a)]
$
\theta^{+, 3/2}_{1, \chi}(z) = \left\{ \begin{array}{cr}\sum_{n \in \mathbb{Z}} \chi(n)\, ne^{i\pi z n^2}& \Im(z)>0, \\0 & \Im(z)<0,\end{array}\right.$ $
\theta^{+, 3/2}_{-1, \chi}(z) = \left\{ \begin{array}{cr}0& \Im(z)>0, \\\sum_{n \in \mathbb{Z}} \chi(n)\, ne^{-i\pi z n^2} & \Im(z)<0.\end{array}\right.$
\item[(b)] $
\theta^{-, 3/2}_{1, \chi}(z) = \left\{ \begin{array}{cr} 0 &\Im(z)>0, \\\sum_{n\in \Z} \chi(-n)n e^{i   \pi z^{\sigma} n^2}  & \Im(z)<0,\end{array}\right.$ $\theta^{-, 3/2}_{-1, \chi}(z) = \left\{ \begin{array}{cr} \sum_{n\in \Z} \chi(-n)n e^{-i   \pi z^{\sigma} n^2}  &\Im(z)>0, \\0 & \Im(z)<0.\end{array}\right.$
\end{itemize}
\item[(3)]\begin{itemize}
\item[(a)]
$\theta^{\pm,\delta^+}_{\chi}(z) = \begin{bmatrix}\theta^{+, 1/2}_{1, \chi}(z) &\theta^{+, 1/2}_{-1, \chi}(z)  \\\theta^{-, 1/2}_{1, \chi}(z) &\theta^{-, 1/2}_{-1, \chi}(z)  \end{bmatrix}$.
\item[(b)] 
$\theta^{\pm, \delta^-}_{\chi}(z) = \begin{bmatrix}\theta^{+, 3/2}_{1, \chi}(z) &\theta^{+, 3/2}_{-1, \chi}(z)  \\\theta^{-, 3/2}_{1, \chi}(z) &\theta^{-, 3/2}_{-1, \chi}(z)  \end{bmatrix}$.
\end{itemize}
\end{itemize}
\end{definition}
Similar to Theorems  \ref{mainthm1} \ref{mainthm1'}, the following result holds:
\begin{lemma}\label{overrgamma0}
For   $\overline{r}=[r,  \epsilon]\in \overline{\Gamma}_{\theta} \cap  \overline{\Gamma}_0(N^2 )$, we have:
 \begin{itemize}
 \item[(1)] $\theta^{\pm, \delta^+}_{\chi}(\overline{r} z)=J_{\delta^+}(\overline{r},z)\theta^{ \pm,\delta^+}_{\chi}( z)\begin{bmatrix} \overline{\lambda}_{\chi}(\overline{r}) &  0\\ 0& \overline{\lambda}_{\chi}(\overline{r}^{h_{-1}})\end{bmatrix} $.
 \item[(2)] $\theta^{\pm, \delta^-}_{\chi}(\overline{r} z)=J_{\delta^-}(\overline{r},z)\theta^{ \pm, \delta^-}_{\chi}( z)\begin{bmatrix} \overline{\lambda}_{\chi}(\overline{r}) &  0\\ 0& \overline{\lambda}_{\chi}(\overline{r}^{h_{-1}})\end{bmatrix} $. 
 \end{itemize}
 \end{lemma}
 \begin{proof}
 The proof is similar to that of Theorem \ref{mainthm1}. To be easily checked, we give a detailed proof for the first statement. 
  Let $r=\begin{pmatrix} a& b\\ cN^2 & d\end{pmatrix}\in   \Gamma_{\theta} \cap  \Gamma_0(N^2 )$.
  \begin{align*}
  \begin{bmatrix}\theta^{+, 1/2}_{1, \chi}(r\overline{g}) \\\theta^{-, 1/2}_{1,\chi}(r\overline{g}) \end{bmatrix}&=  \tfrac{1}{G(1, \overline{\chi})}\sum_{k \bmod N}\overline{ \chi}(k) \theta_{L, X^{\ast}}\Big(J_{\delta^+}(r\overline{g}, z_0)\overline{\Pi}_{X^{\ast}, \psi} (r\overline{g})\begin{bmatrix} A^{+} \\ A^{-} \end{bmatrix}\Big)([1, \tfrac{k}{N}e^{\ast}])\\
  &= \tfrac{1}{G(1, \overline{\chi})}\sum_{k \bmod N}\overline{ \chi}(k) \theta_{L, X^{\ast}}\Big(J_{\delta^+}(rp_g, z_0)\overline{\Pi}_{X^{\ast}, \psi} (rp_g) \begin{bmatrix} A^{+} \\ A^{-} \end{bmatrix}\Big)([1, \tfrac{k}{N}e^{\ast}])\\
  &= \tfrac{1}{G(1, \overline{\chi})}\sum_{k \bmod N}\overline{ \chi}(k)J_{\delta^+}(r, p_gz_0)\overline{\Pi}_{L, \psi} (r) \theta_{L, X^{\ast}} \Big(J_{\delta^+}(p_g, z_0)\overline{\Pi}_{X^{\ast}, \psi} (p_g) \begin{bmatrix} A^{+} \\ A^{-} \end{bmatrix}\Big)([1, \tfrac{k}{N}e^{\ast}])\\
 & \stackrel{(\ref{chapeq8})}{=} \tfrac{1}{G(1, \overline{\chi})}\sum_{k \bmod N}\overline{ \chi}(k) J_{\delta^+}(r, p_gz_0) \epsilon_{r}m_{X^{\ast}}(r) \theta_{L, X^{\ast}} \Big(J_{\delta^+}(p_g, z_0)\overline{\Pi}_{X^{\ast}, \psi} (p_g) \begin{bmatrix} A^{+} \\ A^{-} \end{bmatrix}\Big)([1, cNke+\tfrac{kd}{N}e^{\ast}])\\
 &=\tfrac{1}{G(1, \overline{\chi})} J_{\delta^+}(r, p_gz_0) \epsilon_{r}m_{X^{\ast}}(r) \sum_{k \bmod N}\overline{ \chi}(k)\theta_{L, X^{\ast}} \Big(J_{\delta^+}(p_g, z_0)\overline{\Pi}_{X^{\ast}, \psi} (p_g) \begin{bmatrix} A^{+} \\ A^{-} \end{bmatrix}\Big)([1, cNke+\tfrac{kd}{N}e^{\ast}])\\
 &=J_{\delta^+}(r, p_gz_0) \epsilon_{r}m_{X^{\ast}}(r)\chi(d)\tfrac{1}{G(1, \overline{\chi})}  \sum_{k \bmod N} \sum_{n\in \Z} \Big(J_{\delta^+}(p_g, z_0)\overline{\Pi}_{X^{\ast}, \psi} (p_g) \begin{bmatrix} A^{+} \\ A^{-} \end{bmatrix}\Big)([1, n])\overline{ \chi}(kd)\psi(\tfrac{kdn}{N})\\
  &=J_{\delta^+}(r, p_gz_0) \epsilon_{r}m_{X^{\ast}}(r)\chi(d)\begin{bmatrix}\theta^{+, 1/2}_{1, \chi}(\overline{g}) \\\theta^{-, 1/2}_{1,\chi}(\overline{g}) \end{bmatrix}\\
  &=J_{\delta^+}(r,z_g) \chi(d)\begin{bmatrix}\theta^{+, 1/2}_{1, \chi}(\overline{g}) \\\theta^{-, 1/2}_{1,\chi}(\overline{g}) \end{bmatrix} \lambda(r).
  \end{align*} 
\begin{align*}
\begin{bmatrix}\theta^{+, 1/2}_{-1, \chi}(r\overline{g}) \\\theta^{-, 1/2}_{-1,\chi}(r\overline{g}) \end{bmatrix}
  &=  \tfrac{1}{G(1, \overline{\chi})}\sum_{k \bmod N}\overline{ \chi}(k) \theta_{L, X^{\ast}}\Big(J_{\delta^+}(r\overline{g}, z_0)\overline{\Pi}_{X^{\ast}, \psi} (r\overline{g})\begin{bmatrix} A^{+} \\ A^{-} \end{bmatrix}\Big)([-1, \tfrac{k}{N}e^{\ast}])\\
  &= \tfrac{1}{G(1, \overline{\chi})}\sum_{k \bmod N}\overline{ \chi}(k) \theta_{L, X^{\ast}}\Big(J_{\delta^+}(rp_g, z_0)\overline{\Pi}_{X^{\ast}, \psi} (rp_g) \begin{bmatrix} A^{+} \\ A^{-} \end{bmatrix}\Big)([-1, \tfrac{k}{N}e^{\ast}])\\
  &= \tfrac{1}{G(1, \overline{\chi})}\sum_{k \bmod N}\overline{ \chi}(k)J_{\delta^+}(r, p_gz_0)\overline{\Pi}_{L, \psi} (r) \theta_{L, X^{\ast}} \Big(J_{\delta^+}(p_g, z_0)\overline{\Pi}_{X^{\ast}, \psi} (p_g) \begin{bmatrix} A^{+} \\ A^{-} \end{bmatrix}\Big)([-1, \tfrac{k}{N}e^{\ast}])\\
 & \stackrel{(\ref{chapeq8})}{=} \tfrac{\epsilon_{r}m_{X^{\ast}}(r)}{G(1, \overline{\chi})}\sum_{k \bmod N}\overline{ \chi}(k) J_{\delta^+}(r, p_gz_0) \chi_{\omega}(r)  \theta_{L, X^{\ast}} \Big(J_{\delta^+}(p_g, z_0)\overline{\Pi}_{X^{\ast}, \psi} (p_g) \begin{bmatrix} A^{+} \\ A^{-} \end{bmatrix}\Big)([-1, cNke+\tfrac{kd}{N}e^{\ast}])\\
 &=\tfrac{\epsilon_{r}m_{X^{\ast}}(r)}{G(1, \overline{\chi})} J_{\delta^+}(r, p_gz_0)  \chi_{\omega}(r)\sum_{k \bmod N}\overline{ \chi}(k)\theta_{L, X^{\ast}} \Big(J_{\delta^+}(p_g, z_0)\overline{\Pi}_{X^{\ast}, \psi} (p_g) \begin{bmatrix} A^{+} \\ A^{-} \end{bmatrix}\Big)([-1, cNke+\tfrac{kd}{N}e^{\ast}])\\
 &=J_{\delta^+}(r, p_gz_0) \chi(d)\chi_{\omega}(r)\tfrac{\epsilon_{r}m_{X^{\ast}}(r)}{G(1, \overline{\chi})}  \sum_{k \bmod N} \sum_{n\in \Z} \Big(J_{\delta^+}(p_g, z_0)\overline{\Pi}_{X^{\ast}, \psi} (p_g) \begin{bmatrix} A^{+} \\ A^{-} \end{bmatrix}\Big)([-1, n])\overline{ \chi}(kd)\psi(-\tfrac{kdn}{N})\\
  &=J_{\delta^+}(r, p_gz_0) \epsilon_{r}m_{X^{\ast}}(r)\chi_{\omega}(r)\chi(d)\begin{bmatrix}\theta^{+, 1/2}_{-1, \chi}(\overline{g}) \\\theta^{-, 1/2}_{-1,\chi}(\overline{g}) \end{bmatrix}\\
  &\stackrel{\textrm{ Lem. }\ref{hminustoh}}{=}J_{\delta^+}(r,z_g) \chi(d)\begin{bmatrix}\theta^{+, 1/2}_{-1, \chi}(\overline{g}) \\\theta^{-, 1/2}_{-1,\chi}(\overline{g}) \end{bmatrix} \lambda(r^{h_{-1}}).
  \end{align*} 
 \end{proof}
For $r=h_{-1}$, we have:
\begin{align*}
\begin{bmatrix}\theta^{+, 1/2}_{1, \chi}\\\theta^{-, 1/2}_{1,\chi} \end{bmatrix}(r z_g)
&=\tfrac{1}{G(1, \overline{\chi})}\sum_{k \bmod N}\overline{ \chi}(k) \theta_{L, X^{\ast}}\Big(J_{\delta^+}(r\overline{g}, z_0)\overline{\Pi}_{X^{\ast}, \psi} (r\overline{g})\begin{bmatrix} A^{+} \\ A^{-} \end{bmatrix}\Big)([1, \tfrac{k}{N}e^{\ast}])\\
&=\tfrac{1}{G(1, \overline{\chi})}\sum_{k \bmod N}\overline{ \chi}(k) J_{\delta^+}(r,gz_0)\overline{\Pi}_{L, \psi} (r) \theta_{L, X^{\ast}}\Big(J_{\delta^+}(\overline{g}, z_0)\overline{\Pi}_{X^{\ast}, \psi} (\overline{g})\begin{bmatrix} A^{+} \\ A^{-} \end{bmatrix}\Big)([1, \tfrac{k}{N}e^{\ast}])\\
&\stackrel{m_{X^{\ast}}(h_{-1})=1}{=}\tfrac{1}{G(1, \overline{\chi})}\sum_{k \bmod N}\overline{ \chi}(k)\Pi_{L, \psi} (r)  \theta_{L, X^{\ast}}\Big(J_{\delta^+}(\overline{g}, z_0)\overline{\Pi}_{X^{\ast}, \psi} (\overline{g})\begin{bmatrix} A^{+} \\ A^{-} \end{bmatrix}\Big)([1, \tfrac{k}{N}e^{\ast}])\\
&=\tfrac{1}{G(1, \overline{\chi})}\sum_{k \bmod N} \overline{ \chi}(k)\theta_{L, X^{\ast}}\Big(J_{\delta^+}(\overline{g}, z_0)\overline{\Pi}_{X^{\ast}, \psi} (\overline{g})\begin{bmatrix} A^{+} \\ A^{-} \end{bmatrix}\Big)([-1, -\tfrac{k}{N}e^{\ast}])\\
&=\left\{ \begin{array}{c} \left\{ \begin{array}{cr} 0& \det(g)>0, \\\sum_{n\in \Z} e^{-i   \pi z_g n^2}  \chi(n)  & \det(g)<0.\end{array}\right. \\  \left\{ \begin{array}{cr}\sum_{n\in \Z}  e^{-i   \pi z_g^{\sigma} n^2}  \chi(n)& \det(g)>0, \\ 0& \det(g)<0.\end{array}\right.\end{array}\right.
\end{align*}
\begin{align*}
\begin{bmatrix}\theta^{+, 1/2}_{-1, \chi}\\\theta^{-, 1/2}_{-1,\chi}\end{bmatrix}(r z_g)
&=\tfrac{1}{G(1, \overline{\chi})}\sum_{k \bmod N}\overline{ \chi}(k) \theta_{L, X^{\ast}}\Big(J_{\delta^+}(r\overline{g}, z_0)\overline{\Pi}_{X^{\ast}, \psi} (r\overline{g})\begin{bmatrix} A^{+} \\ A^{-} \end{bmatrix}\Big)([-1, \tfrac{k}{N}e^{\ast}])\\
&=\tfrac{1}{G(1, \overline{\chi})}\sum_{k \bmod N}\overline{ \chi}(k) J_{\delta^+}(r,gz_0)\overline{\Pi}_{L, \psi} (r) \theta_{L, X^{\ast}}\Big(J_{\delta^+}(\overline{g}, z_0)\overline{\Pi}_{X^{\ast}, \psi} (\overline{g})\begin{bmatrix} A^{+} \\ A^{-} \end{bmatrix}\Big)([-1, \tfrac{k}{N}e^{\ast}])\\
&=\tfrac{1}{G(1, \overline{\chi})}\sum_{k \bmod N}\overline{ \chi}(k)  \theta_{L, X^{\ast}}\Big(J_{\delta^+}(\overline{g}, z_0)\overline{\Pi}_{X^{\ast}, \psi} (\overline{g})\begin{bmatrix} A^{+} \\ A^{-} \end{bmatrix}\Big)([1, -\tfrac{k}{N}e^{\ast}])\\
&\stackrel{\chi(-n)=\chi(n)}{=}\left\{ \begin{array}{c} \left\{ \begin{array}{cr}\sum_{n\in \Z} e^{i   \pi z_g n^2}  \chi(n) & \det(g)>0, \\ 0  & \det(g)<0.\end{array}\right. \\  \left\{ \begin{array}{cr} 0  & \det(g)>0, \\\sum_{n\in \Z} e^{i   \pi z_g^{\sigma} n^2}  \chi(n)& \det(g)<0.\end{array}\right.\end{array}\right.
\end{align*}
If $ \Im(z)>0$, 
\begin{align}\label{hminus1zn}
\theta^{\pm, \delta^+}_{\chi}(rz) = \begin{bmatrix}\theta^{+, 1/2}_{1, \chi}(rz) &\theta^{+, 1/2}_{-1, \chi}(rz)  \\\theta^{-, 1/2}_{1, \chi}(rz) &\theta^{-, 1/2}_{-1, \chi}(rz)  \end{bmatrix}=\begin{bmatrix}0 &\sum_{n\in \Z} e^{i   \pi z n^2}  \chi(n)  \\ \sum_{n\in \Z}  e^{-i   \pi z^{\sigma} n^2}  \chi(n)&0  \end{bmatrix}=\theta^{\pm, \delta^+}_{\chi}(z) \begin{bmatrix} 0& 1\\ 1& 0\end{bmatrix}.
\end{align}
If $ \Im(z)<0$,
\begin{align}\label{hminus1zp}
\theta^{\pm, \delta^+}_{\chi}(rz) = \begin{bmatrix}\theta^{+, 1/2}_{1, \chi}(rz) &\theta^{+, 1/2}_{-1, \chi}(rz)  \\\theta^{-, 1/2}_{1, \chi}(rz) &\theta^{-, 1/2}_{-1, \chi}(rz)  \end{bmatrix}=\begin{bmatrix}\sum_{n\in \Z} e^{-i   \pi z n^2}  \chi(n)  &0  \\ 0&\sum_{n\in \Z} e^{i   \pi z^{\sigma} n^2}  \chi(n) \end{bmatrix}=\theta^{\pm, \delta^+}_{\chi}(z) \begin{bmatrix} 0& 1\\ 1& 0\end{bmatrix}.
\end{align}
Similarly, we have:
\begin{align*}
\begin{bmatrix}\theta^{+, 3/2}_{1, \chi}\\\theta^{-, 3/2}_{1,\chi} \end{bmatrix}(r z_g)&=\tfrac{1}{G(1, \overline{\chi})}\sum_{k \bmod N}  \overline{ \chi}(k)\theta_{L, X^{\ast}}\Big(J_{\delta^-}(\overline{g}, z_0)\overline{\Pi}_{X^{\ast}, \psi} (\overline{g})\begin{bmatrix} B^{+} \\ B^{-} \end{bmatrix}\Big)([-1, -\tfrac{k}{N}e^{\ast}])\\
&=\left\{ \begin{array}{c} \left\{ \begin{array}{cr}0 & \det(g)>0, \\ \sum_{n\in \Z} ne^{-i   \pi z_g n^2}  \chi(n)  & \det(g)<0.\end{array}\right. \\  \left\{ \begin{array}{cr} \sum_{n\in \Z} ne^{-i   \pi z_g^{\sigma} n^2}  \chi(-n)  & \det(g)>0, \\0& \det(g)<0.\end{array}\right.\end{array}\right.
 \end{align*}
 \begin{align*}
\begin{bmatrix}\theta^{+, 3/2}_{-1, \chi}\\\theta^{-, 3/2}_{-1,\chi}\end{bmatrix}(r z_g)
&=\tfrac{1}{G(1, \overline{\chi})}\sum_{k \bmod N}\overline{ \chi}(k)  \theta_{L, X^{\ast}}\Big(J_{\delta^-}(\overline{g}, z_0)\overline{\Pi}_{X^{\ast}, \psi} (\overline{g})\begin{bmatrix} B^{+} \\ B^{-} \end{bmatrix}\Big)([1, -\tfrac{k}{N}e^{\ast}])\\
&=\left\{ \begin{array}{c} \left\{ \begin{array}{cr} \sum_{n\in \Z}n e^{i   \pi z_g n^2}  \chi(n) & \det(g)>0, \\0 & \det(g)<0.\end{array}\right. \\  \left\{ \begin{array}{cr} 0 & \det(g)>0, \\\sum_{n\in \Z} ne^{i   \pi z_g^{\sigma} n^2}  \chi(-n) & \det(g)<0.\end{array}\right.\end{array}\right.
 \end{align*}
 If $ \Im(z)>0$, 
\begin{align}\label{hminus1zn32}
\theta^{\pm, \delta^-}_{\chi}(rz) = \begin{bmatrix}\theta^{+, 3/2}_{1, \chi}(rz) &\theta^{+, 3/2}_{-1, \chi}(rz)  \\\theta^{-, 3/2}_{1, \chi}(rz) &\theta^{-, 3/2}_{-1, \chi}(rz)  \end{bmatrix}=\begin{bmatrix}0 &\sum_{n\in \Z}n e^{i   \pi z n^2}  \chi(n)  \\ \sum_{n\in \Z} n e^{-i   \pi z^{\sigma} n^2}  \chi(-n)&0  \end{bmatrix}=\theta^{\pm, \delta^-}_{\chi}(z) \begin{bmatrix} 0& 1\\ 1& 0\end{bmatrix}.
\end{align}
If $ \Im(z)<0$,
\begin{align}\label{hminus1zp32}
\theta^{\pm, \delta^-}_{\chi}(rz) = \begin{bmatrix}\theta^{+, 3/2}_{1, \chi}(rz) &\theta^{+, 3/2}_{-1, \chi}(rz)  \\\theta^{-, 3/2}_{1, \chi}(rz) &\theta^{-, 3/2}_{-1, \chi}(rz)  \end{bmatrix}=\begin{bmatrix}\sum_{n\in \Z}n e^{-i   \pi z n^2}  \chi(n)  &0  \\ 0&\sum_{n\in \Z}n e^{i   \pi z^{\sigma} n^2}  \chi(-n) \end{bmatrix}=\theta^{\pm, \delta^-}_{\chi}(z) \begin{bmatrix} 0& 1\\ 1& 0\end{bmatrix}.
\end{align}
$$\theta^{ \pm, \delta^-}_{\chi}(r z)= \theta^{\pm,\delta^-}_{ \chi}(z) \begin{bmatrix} 0& 1\\ 1& 0\end{bmatrix}.$$
  \begin{definition}
 $\overline{\lambda}^{\pm}_{\chi}\stackrel{Def.}{=}\Ind_{\overline{\Gamma}_{\theta} \cap \overline{\Gamma}_0(N^2)}^{ \overline{\Gamma}^{\pm}_{\theta} \cap  \overline{\Gamma}^{\pm}_0(N^2 )} \overline{\lambda}_{\chi}^{-1}.$
\end{definition}
\begin{lemma}\label{gammatwodimplus}
$\overline{\lambda}^{\pm}_{\chi}$ is an irreducible representation of two dimension.
\end{lemma}
\begin{proof}
 Recall the character $\overline{\lambda}^{\pm}$ from Lemma \ref{gammatwodim}. The character $\chi$ can also  be viewed as a character of $\overline{\Gamma}^{\pm}_0(N^2 )$. Then $\overline{\lambda}^{\pm}_{\chi}=[\overline{\lambda}^{\pm}]^{-1}\otimes \chi^{-1}$.
 \end{proof}
 Note that   $[ \overline{\Gamma}^{\pm}_{\theta} \cap  \overline{\Gamma}^{\pm}_0(N^2 )]=[\overline{\Gamma}_{\theta} \cap \overline{\Gamma}_0(N^2)] \cup [\overline{\Gamma}_{\theta} \cap \overline{\Gamma}_0(N^2)] h_{-1}$. Let $e_1$ and $e_2$ denote  the complex  functions on $ \overline{\Gamma}^{\pm}_{\theta} \cap  \overline{\Gamma}^{\pm}_0(N^2 )$ such that (1) $\supp e_1=\overline{\Gamma}_{\theta} \cap \overline{\Gamma}_0(N^2)$ and $\supp e_2=[\overline{\Gamma}_{\theta} \cap \overline{\Gamma}_0(N^2)] h_{-1}$, (2) $e_1(1)=1=e_2(h_{-1})$. Let $V$ be the vector space spanned by $e_1,e_2$. The representation $\overline{\lambda}^{\pm}_{\chi}$ can be realized  on $V$. For any $f\in V$, we have:
 \[ \overline{\lambda}^{\pm}_{\chi}(h_{-1}) f(1)=f(h_{-1}), \quad \overline{\lambda}^{\pm}_{\chi}(h_{-1}) f(h_{-1})=f(1).\]
  For any $\overline{r}\in \overline{\Gamma}_{\theta} \cap \overline{\Gamma}_0(N^2)$, 
\[ \overline{\lambda}^{\pm}_{\chi}(\overline{r}) f(1)= [\overline{\lambda}_{\chi}]^{-1}(\overline{r})f(1), \quad \overline{\lambda}^{\pm}_{\chi}(\overline{r}) f(h_{-1})= [\overline{\lambda}_{\chi}]^{-1}(\overline{r}^{h_{-1}})f(h_{-1}).\]
Therefore,
\begin{align}
\overline{\lambda}^{\pm}_{\chi}(h_{-1}) (e_1,e_2)& =(e_1,e_2) \begin{bmatrix} 0 &1\\ 1& 0\end{bmatrix},\label{lamdh1}\\
\overline{\lambda}^{\pm}_{\chi}(\overline{r}) (e_1,e_2)&=(e_1,e_2) \begin{bmatrix} \overline{\lambda}_{\chi}(\overline{r})^{-1} &  0\\ 0& \overline{\lambda}_{\chi}(\overline{r}^{h_{-1}})^{-1}\end{bmatrix}.\label{lamdr}
\end{align}
Similar  to Theorem \ref{maintheorem1}, we have:
 \begin{theorem}\label{mainthm1234}
 For $\overline{r}=[r,  \epsilon]\in \overline{\Gamma}^{\pm}_{\theta} \cap  \overline{\Gamma}^{\pm}_0(N^2 )$,  we have:
 \begin{itemize}
 \item[(1)] $\theta^{\pm, \delta^+}_{\chi}(\overline{r} z)=J_{\delta^+}(\overline{r},z)\theta^{\pm, \delta^+}_{\chi}( z)\overline{\lambda}^{\pm}_{\chi}(\overline{r}^{-1})$.
 \item[(2)] $\theta^{\pm, \delta^-}_{\chi}(\overline{r} z)=J_{\delta^-}(\overline{r},z)\theta^{\pm, \delta^-}_{\chi}( z)\overline{\lambda}^{\pm}_{\chi}(\overline{r}^{-1})$. 
 \end{itemize}
\end{theorem}
\begin{proof}
1) a) If $\overline{r}= \overline{r}_0\in \overline{\Gamma}_{\theta} \cap  \overline{\Gamma}_0(N^2 )$, the result follows from Lemma \ref{overrgamma0} and (\ref{lamdr}).\\
b) If $\overline{r}=h_{-1}$, then $J_{\delta^+}(\overline{r},z)=I_2$ and the result follows from (\ref{hminus1zn}), (\ref{hminus1zp})and (\ref{lamdh1}).\\
c) For the general $\overline{r}=h_{-1}\overline{r}_0$, 
\begin{align*}
\theta^{\pm, \delta^+}_{\chi}(h_{-1}\overline{r}_0z)&=J_{\delta^+}(h_{-1},\overline{r}_0z)\theta^{\pm, \delta^+}_{\chi}( \overline{r}_0z)\overline{\lambda}^{\pm}_{\chi}(h_{-1}^{-1})\\
&=J_{\delta^+}(h_{-1},\overline{r}_0z)J_{\delta^+}(\overline{r}_0, z)\theta^{\pm, \delta^+}_{\chi}( z)\overline{\lambda}^{\pm}_{\chi}(\overline{r}_0^{-1})\overline{\lambda}^{\pm}_{\chi}(h_{-1}^{-1})\\
&=J_{\delta^+}(\overline{r},z)\theta^{\pm, \delta^+}_{\chi}( z)\overline{\lambda}^{\pm}_{\chi}(\overline{r}^{-1}).
\end{align*}
2) The proof of Part (2) is similar.
\end{proof}
\subsection{ Twisted by $\overline{\mathfrak{w}} $ }\label{twistmathwII}
Let $\mathfrak{w} \in \GL_2(\Q)$, and let $\overline{\mathfrak{w}} = (\mathfrak{w}, t_{\mathfrak{w}}) \in \overline{\GL}_2(\Q)$. For $\overline{s} \in \overline{\Gamma}^{\pm}_{\theta} \cap  \overline{\Gamma}^{\pm}_0(N^2 )$, let $\overline{r} = \overline{\mathfrak{w}}^{-1} \overline{s} \overline{\mathfrak{w}} \in \overline{\mathfrak{w}}^{-1}[\overline{\Gamma}^{\pm}_{\theta} \cap  \overline{\Gamma}^{\pm}_0(N^2 )]\overline{\mathfrak{w}}$. We define a character on this conjugated subgroup by
\[
 \overline{\lambda}_{\chi}^{\overline{\mathfrak{w}}}(\overline{r}) := \overline{\lambda}_{\chi}(\overline{s}).
\]

\begin{definition}
Define the slash operators for the theta functions as follows:
\begin{itemize}
\item[(1)] $[\theta^{\pm, \delta^+}_{\chi}]^{\overline{\mathfrak{w}}}(z) := \det(\mathfrak{w})^{1/4} J_{\delta^+}(\overline{\mathfrak{w}},z)^{-1} \theta^{\pm, \delta^+}_{\chi}(\overline{\mathfrak{w}} z)$;
\item[(2)] $[\theta^{\pm, \delta^-}_{\chi}]^{\overline{\mathfrak{w}}}(z) := \det(\mathfrak{w})^{3/4} J_{\delta^-}(\overline{\mathfrak{w}},z)^{-1}\theta^{\pm, \delta^-}_{\chi}(\overline{\mathfrak{w}} z)$.
\end{itemize}
\end{definition}
\begin{proposition}\label{slashpro2}
For $\overline{r} \in \overline{\mathfrak{w}}^{-1}[\overline{\Gamma}^{\pm}_{\theta} \cap  \overline{\Gamma}^{\pm}_0(N^2 )]\overline{\mathfrak{w}}$, we have:
\begin{itemize}
\item[(1)] $[\theta^{\pm, \delta^+}_{\chi}]^{\overline{\mathfrak{w}}}(\overline{r}z) = J_{\delta^+}(\overline{r},z)[\theta^{\pm, \delta^+}_{\chi}]^{\overline{\mathfrak{w}}}(z)\overline{\lambda}_{\chi}^{\overline{\mathfrak{w}}}(\overline{r}^{-1})$. 
\item[(2)] $[\theta^{\pm, \delta^-}_{\chi}]^{\overline{\mathfrak{w}}}(\overline{r}z) = J_{\delta^-}(\overline{r},z)[\theta^{\pm, \delta^-}_{\chi}]^{\overline{\mathfrak{w}}}(z)\overline{\lambda}_{\chi}^{\overline{\mathfrak{w}}}(\overline{r}^{-1})$.
\end{itemize}
\end{proposition}
\begin{proof}
The proof parallels that of Proposition~\ref{slashpro}. For completeness, we provide the proof of statement~(1); the case of~(2) is analogous.
Write $\overline{r}=\overline{\mathfrak{w}}^{-1}\overline{s}\overline{\mathfrak{w}}$.\\
 \begin{align*}
& J_{\delta^+}(\overline{\mathfrak{w}},\overline{r}z)^{-1} J_{\delta^+}(\overline{s},\overline{\mathfrak{w}}z) J_{\delta^+}(\overline{\mathfrak{w}},z) \\
&= J_{\delta^+}(\overline{\mathfrak{w}},\overline{r}z)^{-1} J_{\delta^+}(\overline{s}\overline{\mathfrak{w}},z) \\
&= J_{\delta^+}(\overline{\mathfrak{w}},\overline{r}z)^{-1} J_{\delta^+}(\overline{\mathfrak{w}}\overline{r},z) \\
&= J_{\delta^+}(\overline{r},z).
\end{align*}
\begin{align*}
[\theta^{\pm, \delta^+}_{\chi}]^{\overline{\mathfrak{w}}}(\overline{r}z)
&= \det(\mathfrak{w})^{1/4} J_{\delta^+}(\overline{\mathfrak{w}},\overline{r}z)^{-1}\theta^{\pm, \delta^+}_{\chi}(\overline{\mathfrak{w}}\overline{r}z) \\
&=\det(\mathfrak{w})^{1/4} J_{\delta^+}(\overline{\mathfrak{w}},\overline{r}z)^{-1}\theta^{\pm, \delta^+}_{\chi}(\overline{s}\overline{\mathfrak{w}}z) \\
&= \det(\mathfrak{w})^{1/4}J_{\delta^+}(\overline{\mathfrak{w}},\overline{r}z)^{-1} J_{\delta^+}(\overline{s},\overline{\mathfrak{w}}z)\theta^{\pm, \delta^+}_{\chi}(\overline{\mathfrak{w}}z) \overline{\lambda}_{\chi}(\overline{s}^{-1})\\
&= J_{\delta^+}(\overline{\mathfrak{w}},\overline{r}z)^{-1} J_{\delta^+}(\overline{s},\overline{\mathfrak{w}}z) J_{\delta^+}(\overline{\mathfrak{w}},z)[\theta^{\pm, \delta^+}_{\chi}]^{\overline{\mathfrak{w}}}(z)\overline{\lambda}_{\chi}^{\overline{\mathfrak{w}}}(\overline{r}^{-1}) \\
&= J_{\delta^+}(\overline{r},z)[\theta^{\pm, \delta^+}_{\chi}]^{\overline{\mathfrak{w}}}(z)\overline{\lambda}_{\chi}^{\overline{\mathfrak{w}}}(\overline{r}^{-1}).
\end{align*}
\end{proof}
We extend the results of Sections \ref{Gammaaonodd} and \ref{Gammaaoneven} to the group $\overline{\Gamma}^{\pm}_0(N^2 )$. 

\begin{itemize}
\item Follow the notions form Section \ref{Gammaaonodd}.  If $N$ is odd, we define:
$$\Theta^{\pm, 1/2}_{\chi}:  \mathbb{H} \longrightarrow M_{2 \times 3}(\C); z\longmapsto \Big([\theta^{\pm, \delta^+}_{\chi}]^{\overline{M}_{q_1}}(z), \cdots , [\theta^{\pm, \delta^+}_{\chi}]^{\overline{M}_{q_3}}(z)\Big),$$
$$\Theta^{\pm, 3/2}_{\chi}:  \mathbb{H} \longrightarrow M_{2 \times 3}(\C); z\longmapsto \Big([\theta^{\pm, \delta^-}_{\chi}]^{\overline{M}_{q_1}}(z),\cdots, [\theta^{\pm, \delta^-}_{\chi}]^{\overline{M}_{q_3}}(z)\Big).$$
\item Follow the notions form Section \ref{Gammaaoneven}.  If $N$ is even, we define:
$$\Theta^{\pm, 1/2}_{\chi}:  \mathbb{H} \longrightarrow M_{2 \times 3}(\C); z\longmapsto \Big([\theta^{\pm, \delta^+}_{\chi}]^{\overline{M}_{q_1}}(z),  [\theta^{\pm, \delta^+}_{\chi}]^{\overline{M}_{q_2}}(z)\Big),$$
$$\Theta^{\pm, 3/2}_{\chi}:  \mathbb{H} \longrightarrow M_{2 \times 3}(\C); z\longmapsto \Big([\theta^{\pm, \delta^-}_{\chi}]^{\overline{M}_{q_1}}(z), [\theta^{\pm, \delta^-}_{\chi}]^{\overline{M}_{q_2}}(z)\Big).$$
\end{itemize}
\begin{definition}
$\overline{\gamma}^{\pm}_{\chi}=\Ind_{ \overline{\Gamma}^{\pm}_{\theta} \cap  \overline{\Gamma}^{\pm}_0(N^2 )}^{\overline{\Gamma}^{\pm}_0(N^2)} [\overline{\lambda}^{\pm}_{\chi}].$
\end{definition}
\begin{lemma}\label{irrdgamachi}
\begin{itemize}
\item[(1)] If $N$ is odd, then $\overline{\gamma}^{\pm}_{\chi}$ is an irreducible representation.
\item[(2)]
\begin{itemize}
\item[(a)]  If $N\equiv 2(\bmod 4)$, then $\overline{\gamma}^{\pm}_{\chi}$ is an irreducible representation.
\item[(b)]  If $N\equiv 0(\bmod 4)$,  then $\overline{\gamma}^{\pm}_{\chi}$ is a reducible representation with  two different irreducible components.
\end{itemize}
\end{itemize}
\end{lemma}
\begin{proof}
1) Note that $\chi$ is a character of $\overline{\Gamma}^{\pm}_0(N^2 )$. So
$$\overline{\gamma}^{\pm}_{\chi}\simeq \Ind_{\overline{\Gamma}_{\theta} \cap \overline{\Gamma}_0(N^2)}^{ \overline{\Gamma}^{\pm}_0(N^2)} \overline{\lambda}_{\chi}^{-1}\simeq \chi^{-1}\Big(\Ind_{\overline{\Gamma}_{\theta} \cap  \overline{\Gamma}_0(N^2 )}^{\overline{\Gamma}^{\pm}_0(N^2)} [\overline{\lambda}]^{-1}\Big) \simeq \chi^{-1}\otimes \Res_{\overline{\Gamma}^{\pm}_0(N^2)}^{\overline{\SL}^{\pm}_2(\Z)}\overline{\gamma}^{\pm}.$$
By Coro. \ref{irrdgammapm2nn}, the result holds.\\
2) $N$ is even. Retains  the notations from the proof of  Lem. \ref{exenchigamma}.\\
a)   If take $\overline{r}=[r,1]$ with $r=\begin{pmatrix}
a&b \\
c& d
\end{pmatrix}=\begin{pmatrix}
-1-2N^2&2 \\
N^2& -1
\end{pmatrix}$, then:
\[\overline{M}_{q_{1}}^{-1}\overline{r}\,\overline{M}_{q_{1}}=[\begin{pmatrix}
-1-3N^2 & -3N^2 + 2 \\
N^2 & N^2  -1
\end{pmatrix},1],\]
\[ h_{-1}\overline{r}h_{-1}^{-1}=[\begin{pmatrix}
-1-2N^2&-2 \\
-N^2& -1
\end{pmatrix}, \nu_2(-1, r)]=[\begin{pmatrix}
-1-2N^2&-2 \\
-N^2& -1
\end{pmatrix}, 1],\]
\[ h_{-1}\overline{M}_{q_{1}}^{-1}\overline{r}\,\overline{M}_{q_{1}}h_{-1}^{-1}=[\begin{pmatrix}
-1-3N^2 &3N^2- 2 \\
-N^2 & N^2  -1
\end{pmatrix}, \nu_2(-1,M_{q_{1}}^{-1}rM_{q_{1}})]=[\begin{pmatrix}
-1-3N^2 &3N^2- 2 \\
-N^2 & N^2  -1
\end{pmatrix}, 1].\]
So $$\overline{\lambda}(\overline{r})
=\Bigl(\frac{2N^2}{-1}\Bigr)\epsilon_{-1}^{-1} \neq \overline{\lambda}^{h_{-1}}(\overline{r})
=\Bigl(\frac{-2N^2}{-1}\Bigr)\epsilon_{-1}^{-1}, $$
$$\overline{\lambda}(\overline{M}_{q_{1}}^{-1}\overline{r}\,\overline{M}_{q_{1}})
=\Bigl(\frac{2N^2}{N^2-1}\Bigr)\epsilon_{N^2-1}^{-1} \neq \overline{\lambda}^{h_{-1}}(\overline{M}_{q_{1}}^{-1}\overline{r}\,\overline{M}_{q_{1}})
=\Bigl(\frac{-2N^2}{N^2-1}\Bigr)\epsilon_{N^2-1}^{-1}.$$
If take $\overline{r}_1=[r_1,1]$ with  $r_1=\begin{pmatrix}
a&b \\
c& d
\end{pmatrix}=\begin{pmatrix}
-1-4N^2&2 \\
2N^2& -1
\end{pmatrix}$, then:
\[\overline{M}_{q_{1}}^{-1}\overline{r}_1\,\overline{M}_{q_{1}}=[\begin{pmatrix}
-1-6N^2 & -6N^2 + 2  \\
2N^2 & 2N^2 -1
\end{pmatrix},1],\]
\[ h_{-1}\overline{M}_{q_{1}}^{-1}\overline{r}_1\,\overline{M}_{q_{1}}h_{-1}^{-1}=[\begin{pmatrix}
-1-6N^2 & -6N^2 + 2  \\
2N^2 & 2N^2 -1
\end{pmatrix}, \nu_2(-1,M_{q_{1}}^{-1}r_1M_{q_{1}})]=[\begin{pmatrix}
-1-6N^2 & -6N^2 + 2  \\
2N^2 & 2N^2 -1
\end{pmatrix}, 1],\]
\[\overline{\lambda}^{h_{-1}}(\overline{r}_1)
=\Bigl(\frac{-4N^2}{-1}\Bigr)\epsilon_{-1}^{-1}=-\epsilon_{-1}^{-1}\neq\overline{\lambda}(\overline{M}_{q_{1}}^{-1}\overline{r}_1\,\overline{M}_{q_{1}})=\Bigl(\frac{4N^2}{2N^2-1}\Bigr)\epsilon_{2N^2-1}^{-1} =\epsilon_{-1}^{-1}.\]
So the result of Part (a) holds.\\
b) We take the same $\overline{r}$ as above. Then 
\[\overline{\lambda}(\overline{r})\neq \overline{\lambda}^{h_{-1}}(\overline{r}), \quad \overline{\lambda}(\overline{M}_{q_{1}}^{-1}\overline{r}\,\overline{M}_{q_{1}})\neq \overline{\lambda}^{h_{-1}}(\overline{M}_{q_{1}}^{-1}\overline{r}\,\overline{M}_{q_{1}}).\]
So the result of Part (b) holds.
 \end{proof}
Analogous to  Theorem \ref{maintheorem1}, we have:
\begin{theorem}\label{maintheoremplus}
Let $\overline{r}\in \overline{\Gamma}^{\pm}_0(N^2 )$. 
 \begin{itemize}
 \item[(1)]$\Theta^{\pm, \delta^+}_{\chi}(\overline{r}z)=J_{\delta^+}(\overline{r},z)\Theta^{\pm, \delta^+}_{\chi}(z)\overline{\gamma}^{\pm}_{\chi}(\overline{r}^{-1})$.
 \item[(2)] $\Theta^{\pm, \delta^-}_{\chi}(\overline{r}z)=J_{\delta^-}(\overline{r},z)\Theta^{\pm, \delta^-}_{\chi}(z)\overline{\gamma}^{\pm}_{\chi}(\overline{r}^{-1})$.
 \end{itemize}
\end{theorem}
\section{Examples}\label{examples}
\subsection{Example 1 to Sect. \ref{twistmathw}}\label{Examples11} $\mathfrak{w}=\begin{pmatrix} 2& 0\\ 0 & 1\end{pmatrix}$ and $\overline{\mathfrak{w}}=[\mathfrak{w}, 1]$. Recall $\omega$ and $h(a)$ from Introduction.
Note:
\[ \Gamma_{\theta}=\Gamma(2) \sqcup \Gamma(2) \omega,  \mathfrak{w}^{-1} \Gamma( 2) \mathfrak{w}=\Gamma_0(4),  \mathfrak{w}^{-1}\omega \mathfrak{w}=\begin{pmatrix}0 & -\tfrac{1}{2}\\ 2 & 0\end{pmatrix}=\omega h(2), \mathfrak{w}^{-1}\Gamma_{\theta} \mathfrak{w}=\Gamma_0(4) \sqcup \Gamma_0(4)\omega h(2).\]
\begin{definition}
\begin{itemize}
\item[(1)] $\theta(z)\stackrel{ \textrm{ Def.}}{=} 2^{-1/4}[\theta^{1/2}]^{\overline{\mathfrak{w}}}(z) =\sum_{n\in \Z} e^{2 \pi i   z n^2 }$.
\item[(2)] $\Gamma_{0}(4)^{\kappa}\stackrel{ Def.}{=}\mathfrak{w}^{-1} \Gamma_{\theta}\mathfrak{w}=\Gamma_0(4) \sqcup \Gamma_0(4)\omega h(2)$.
\item[(3)] $\nu_{\theta, 2}(r)=  \lambda(\mathfrak{w}r\mathfrak{w}^{-1})$, for $r\in \Gamma_{0}(4)^{\kappa}$.
\item[(4)]  $\Gamma'_{0}(4)^{\kappa}\stackrel{ Def.}{=}\{[\mathfrak{w}^{-1} r \mathfrak{w}, \lambda(r)^{-1}] \mid r\in \Gamma_{\theta}\}=\{ [r, \nu_{\theta, 2}(r)^{-1}] \mid r\in \Gamma_{0}(4)^{\kappa}\} $.
\end{itemize}
\end{definition}
\begin{lemma}
\begin{itemize}
\item[(1)] There exists a group isomorphism $\iota: \Gamma_{0}(4)^{\kappa} \to \Gamma'_{0}(4)^{\kappa}; r \to [r, \nu_{\theta, 2}(r)^{-1}]$.
\item[(2)] $\nu_{\theta, 2}(r)=\left\{ \begin{array}{lcl}  \left(\frac{c}{d}\right)\epsilon^{-1}_d, &   & r=\begin{pmatrix} a& b\\ c& d\end{pmatrix}\in \Gamma_0(4),\\
 \left(\frac{c}{d}\right)\epsilon^{-1}_d e^{-\tfrac{i \pi}{4}} (-c,d)_{\R},& & r=\begin{pmatrix} a& b\\ c& d\end{pmatrix} \omega h(2)\in  \Gamma_0(4)\omega h(2).\end{array}\right.$
\end{itemize}
\end{lemma}
\begin{proof}
1) $\iota(r_1) \iota(r_2)= [r_1, \nu_{\theta, 2}(r_1)^{-1}] \cdot [r_2, \nu_{\theta, 2}(r_2)^{-1}]=[r_1r_2, \lambda(\mathfrak{w}r_1\mathfrak{w}^{-1})^{-1} \lambda(\mathfrak{w}r_2\mathfrak{w}^{-1})^{-1} \overline{c}_{X^{\ast}}(r_1,r_2)]=[r_1r_2, \lambda(\mathfrak{w}r_1\mathfrak{w}^{-1})^{-1} \lambda(\mathfrak{w}r_2\mathfrak{w}^{-1})^{-1} \overline{c}_{X^{\ast}}(\mathfrak{w}r_1\mathfrak{w}^{-1},\mathfrak{w}r_2\mathfrak{w}^{-1})]\stackrel{\textrm{ Lem. } \ref{mutirr12}}{=}[r_1r_2, \nu_{\theta, 2}(r_1r_2)^{-1}]=\iota(r_1r_2) $.\\
2)
\[r=\begin{pmatrix} a& b\\ c& d\end{pmatrix}\in \Gamma_{\theta}, \quad \mathfrak{w} r \mathfrak{w}^{-1}=\begin{pmatrix} a& 2b\\ \tfrac{1}{2}c& d\end{pmatrix}, \quad \mathfrak{w}\omega h(2) \mathfrak{w}^{-1}=\omega ;\]
\begin{align*}
  \overline{c}_{X^{\ast}}(r, \omega h(2))&= \overline{c}_{X^{\ast}}(r, \omega h(2)) \overline{c}_{X^{\ast}}(\omega,  h(2))= \overline{c}_{X^{\ast}}(r, \omega) \overline{c}_{X^{\ast}}(r\omega,  h(2))=\overline{c}_{X^{\ast}}(r, \omega)=(-c,d)_{\R}.
   \end{align*}
\begin{align*}
\nu_{\theta, 2}(r)&=\lambda(\mathfrak{w}r\mathfrak{w}^{-1})\stackrel{\textrm{ Lem. }\ref{gammar}}{=}\left(\frac{c}{d}\right)\epsilon^{-1}_d;\\
\nu_{\theta, 2}(\omega h(2))&=\lambda(\mathfrak{w}r\mathfrak{w}^{-1})=\lambda(\omega )=\overline{\lambda}([\omega,1])=e^{-\tfrac{i \pi}{4}};\\
\nu_{\theta, 2}((r\omega h(2))&=\overline{c}_{X^{\ast}}(r, \omega h(2)) \nu_{\theta, 2}((r)\nu_{\theta, 2}((\omega h(2))=(-c,d)_{\R}\left(\frac{c}{d}\right)\epsilon^{-1}_de^{-\tfrac{i \pi}{4}}.\end{align*}
\end{proof}
For $r=\begin{pmatrix} a& b\\ c& d\end{pmatrix}\in \Gamma_{0}(4)^{\kappa}$, $\overline{r}=[r,1]$, we have:
\[  J_{1/2}(\overline{r},z)=\sqrt{cz+d};\]
\[  \overline{\lambda}^{\overline{\mathfrak{w}}}(\overline{r})= \overline{\lambda}(\overline{\mathfrak{w}}\overline{r}\overline{\mathfrak{w}}^{-1}) =\overline{\lambda}([\mathfrak{w}r\mathfrak{w}^{-1},1])=\lambda(\mathfrak{w}r\mathfrak{w}^{-1})=\nu_{\theta, 2}(r).\]
In this case, we conclude:
\begin{proposition}\label{thetashi}
 $\theta(r z)= \nu_{\theta, 2}(r) \sqrt{cz+d}\theta(z)$, for $r=\begin{pmatrix} a& b\\ c& d\end{pmatrix}\in \Gamma_{0}(4)^{\kappa} $. 
\end{proposition}\label{proatheta}
\begin{remark}
For $r\in \Gamma_{0}(4)$, it is comparable with the classical result in Shimura paper \cite{Sh}. We  extend it to a little bigger group $\Gamma_{0}(4)^{\kappa}$ by following Lion-Vergne's paper \cite{LiVe}.
\end{remark}

\subsection{Example  2 to Sect. \ref{Gammaaonodd}}\label{Examples12} 
We consider where  $N=1$ and $\chi\equiv 1$. In this case, $ \overline{M}_{q_1}=[u(1),1], \overline{M}_{q_2}=[1,1], \overline{M}_{q_3}=[u_-(-1),1]$. The theta series are related to concepts introduced by Don Zagier in \cite[p. 27]{Za}, specifically the ``minus sign theta series'' and the ``fermionic theta series.'' 
 \[ \theta^{M,1/2}(z)\stackrel{Def.}{=}\sum_{n\in \Z}  (-1)^n e^{i   \pi n^2 z} \quad \quad \textrm{(minus sign)}\]
\[ \theta^{F,1/2}(z)\stackrel{Def.}{=}\sum_{n\in \Z} e^{i   \pi (n+\tfrac{1}{2})^2 z} \quad \quad \textrm{(fermionic)}.\]
Then:
\begin{itemize}
\item[(1)] $[\theta^{1/2}]^{\overline{M}_{q_1}}(z) =\sum_{n\in \Z} (-1)^{n}e^{i   \pi z n^2 }= \theta^{M,1/2}(z)$;
\item[(2)] $[\theta^{1/2}]^{\overline{M}_{q_2}}(z) =\sum_{n\in \Z}e^{i   \pi z n^2 }$;
\item[(3)] $[\theta^{1/2}]^{\overline{M}_{q_3}}(z)=e^{-\tfrac{\pi i}{4}}e^{i   \pi z ( n+\tfrac{1}{2})^2 }=e^{-\tfrac{\pi i}{4}}\theta^{F,1/2}(z)$.
\item[(4)] $\Theta^{1/2}:  \mathbb{H} \longrightarrow M_{1 \times 3}(\C); z\longmapsto \Big([\theta^{1/2}]^{\overline{M}_{q_1}}(z),[\theta^{1/2}]^{\overline{M}_{q_2}}(z), [\theta^{1/2}]^{\overline{M}_{q_3}}(z)\Big)$.
\end{itemize}
\begin{itemize}
\item[1)] $[\theta^{1/2}]^{\overline{M}_{q_1}}(z) =  J_{1/2}( \overline{M}_{q_1},z)^{-1} \theta^{1/2}( z+1)=\sum_{n\in \Z} (-1)^{n}e^{i   \pi z n^2 }$.
\item[3)] \begin{align*}
 \theta_{L, X^{\ast}}(A)(\tfrac{1}{2} e)=\sum_{l\in L\cap X}A(l+\tfrac{1}{2} e)=\sum_{n\in \Z}A( n+\tfrac{1}{2}).
 \end{align*}
 \begin{align*}
 &\theta_{L, X^{\ast}}\Big(J_{1/2}(p_g,z_0)\pi_{X^{\ast}, \psi} ( p_g)A\Big)(\tfrac{1}{2} e)\\
 &=\sum_{n\in \Z}\Big(J_{1/2}(p_g,z_0)\pi_{X^{\ast}, \psi} ( p_g)A\Big)( n+\tfrac{1}{2})\\
 &=\sum_{n\in \Z} e^{i   \pi z_g ( n+\tfrac{1}{2})^2 }.
 \end{align*}
 \begin{align*}
&\theta^{1/2}(\overline{M}_{q_3}z_g)\\
&=\theta^{1/2}(u_-(-1)  p_g)\\
&=\theta_{L, X^{\ast}}\Big(J_{1/2}(u_-(-1) p_g, z_0)\pi_{X^{\ast}, \psi} (u_-(-1) p_g)A\Big)(0)\\
&= J_{1/2}(u_-(-1) ,  p_gz_0)\pi_{L, \psi}(u_-(-1))\theta_{L, X^{\ast}}\Big(J_{1/2}(p_g,z_0)\pi_{X^{\ast}, \psi} ( p_g)A\Big)( 0) \\
&\stackrel{\textrm{Sect. \ref{Gamma2I2}} \  \textrm{Case  8}}{=} J_{1/2}(u_-(-1) , z_g)m_{X^{\ast}}(u_-(-1))^{-1} \theta_{L, X^{\ast}}\Big(J_{1/2}(p_g,z_0)\pi_{X^{\ast}, \psi} ( p_g)A\Big)(\tfrac{1}{2} e)\\
&=\sum_{n\in \Z}J_{1/2}( \overline{M}_{q_3}, z_g) m_{X^{\ast}}(u_-(-1))^{-1} e^{i   \pi z_g ( n+\tfrac{1}{2})^2 }.
\end{align*}
So $[\theta^{1/2}]^{\overline{M}_{q_3}}(z)= m_{X^{\ast}}(u_-(-1))^{-1} e^{i   \pi z_g ( n+\tfrac{1}{2})^2 }=e^{-\tfrac{\pi i}{4}}e^{i   \pi z ( n+\tfrac{1}{2})^2 }$.
\end{itemize}
In this case, the result is the following:
\begin{proposition}\label{mainthm112}
Let $\overline{r}=(r,\epsilon)\in \overline{\SL}_2(\Z)$ with  $r=\begin{pmatrix}a& b\\ c&d\end{pmatrix}\in \SL_2(\Z)$ and  $\epsilon\in \mu_2$. Then
$
\Theta^{1/2}(\overline{r}z)
=\epsilon^{-1}\sqrt{cz+d}\;\Theta^{1/2}(z)\;\overline{\gamma}(\overline{r}^{-1}).
$
\end{proposition}
\subsection{Example 3  to Sect. \ref{twistmathw}}\label{Examples13}
Let $\overline{\mathfrak{w}}= [u(1),1]$, for $\mathfrak{w}=u(1)=\begin{pmatrix} 1& 1\\ 0& 1\end{pmatrix}$. 
Let $\overline{s}=[s,\lambda({s})^{-1}]\in  \overline{\Gamma}^{\mu_8}_{\theta} \subseteq  \overline{\SL}^{\mu_8}_2(\mathbb R)$ and $r=u(-1)s u(1)$. Note:
\[ u(-1)\Gamma_{\theta} u(1)=\Gamma^0(2), \quad  u(-1)\Gamma(2) u(1)=\Gamma(2), \quad \overline{\mathfrak{w}}^{-1}= [u(-1),1].\]
\begin{lemma}
\begin{itemize}
\item[(1)] $\nu_2(\mathfrak{w}^{-1}, g)=1$, for $g\in \SL_2(\R)$.
\item[(2)] $\overline{\mathfrak{w}}^{-1}\overline{s}\overline{\mathfrak{w}}=[r, \lambda^{\mathfrak{w}}(r)^{-1}]$.
\end{itemize}
\end{lemma}
\begin{proof}
1) $\nu_2(u(-1), g)\stackrel{\textrm{ Ex.} \ref{twocoverex}}{=} \overline{c}_{X^{\ast}}(u(-1)^{-1},g)\overline{c}_{X^{\ast}}(u(-1)^{-1}g,u(-1))\stackrel{\textrm{ Lem.} \ref{prog0}}{=} 1$.\\
2) 
$
 \overline{\mathfrak{w}}^{-1}\overline{s}\overline{\mathfrak{w}}
=[u(-1) s u(1), \overline{c}_{X^{\ast}}(u(-1), s)\overline{c}_{X^{\ast}}(u(-1) s, u(1))\lambda({s})^{-1}]
=[r, \lambda({s})^{-1}].$
\end{proof}
 \[[\theta^{1/2}]^{\overline{\mathfrak{w}}}(z)=\sum_{n\in \Z} (-1)^{n}e^{i   \pi z n^2 }.\]

\begin{definition}
\begin{itemize}
\item[(1)] $ \theta^{M,1/2}(z)\stackrel{ \textrm{ Def.}}{=} \sum_{n\in \Z} (-1)^{n}e^{i   \pi z n^2 }$  .
\item[(2)] $\Gamma^{0'}(2)\stackrel{ \textrm{ Def.}}{=}\{ [r, \lambda^{\mathfrak{w}}(r)^{-1}] \mid r\in \Gamma^0(2)\}$.
\item[(3)] $\nu_{\theta^M}(r)\stackrel{ \textrm{ Def.}}{=} \lambda^{\mathfrak{w}}(r)$.
\end{itemize}
\end{definition}
\begin{lemma}
 There exists a group isomorphism $\iota: \Gamma^{0}(2) \simeq \Gamma^{0'}(2); r\to [r, \nu_{\theta^M}(r)^{-1}]$.
\end{lemma}
\begin{proof}
 \begin{align*}
 \iota(r_1) \iota(r_2)&= [r_1, \nu_{\theta^M}(r_1)^{-1}] \cdot [r_2, \nu_{\theta^M}(r_2)^{-1}]\\
 &=[r_1r_2, \lambda(\mathfrak{w}r_1\mathfrak{w}^{-1})^{-1} \lambda(\mathfrak{w}r_2\mathfrak{w}^{-1})^{-1} \nu_2(\mathfrak{w}^{-1},r_1)^{-1}\nu_2(\mathfrak{w}^{-1},r_2)^{-1}\overline{c}_{X^{\ast}}(r_1,r_2)]\\
 &=[r_1r_2, \lambda(\mathfrak{w}r_1\mathfrak{w}^{-1})^{-1} \lambda(\mathfrak{w}r_2\mathfrak{w}^{-1})^{-1} \overline{c}_{X^{\ast}}(\mathfrak{w}r_1\mathfrak{w}^{-1},\mathfrak{w}r_2\mathfrak{w}^{-1})\nu_2(\mathfrak{w}^{-1},r_1r_2)^{-1}]\\
 &\stackrel{\textrm{ Lem. }\ref{mutirr12}}{=}[r_1r_2, \lambda(\mathfrak{w}r_1r_2\mathfrak{w}^{-1})^{-1}\nu_2(\mathfrak{w}^{-1},r_1r_2)^{-1}]\\
 &=[r_1r_2, \nu_{\theta^M}(r_1r_2)^{-1}]=\iota(r_1r_2) .
 \end{align*}
\end{proof}
Note that 
\[\Gamma(2)/\{\pm I\} \simeq \langle \begin{pmatrix} 1& 0\\ 2& 1\end{pmatrix}, \begin{pmatrix} 1& 2\\ 0& 1\end{pmatrix} \rangle=\{ g=\begin{pmatrix}a& b\\ c& d\end{pmatrix}\in \Gamma(2)\mid  a\equiv d\equiv 1 (\bmod 4)\}.\]
\begin{lemma}
Let $\delta_M: \{ g=\begin{pmatrix}a& b\\ c& d\end{pmatrix}\in \Gamma(2)\mid  a\equiv d\equiv 1 (\bmod 4)\} \to \mu_4; g\to (- i)^{c/2}$. Then $\delta_M$ is a character. 
\end{lemma}
\begin{proof}
$$g_1g_2=\begin{pmatrix} a_1 & b_1 \\ c_1 & d_1 \end{pmatrix}
\begin{pmatrix} a_2 & b_2 \\ c_2 & d_2 \end{pmatrix}
= \begin{pmatrix}
a_1 a_2 + b_1 c_2 & a_1 b_2 + b_1 d_2 \\
c_1 a_2 + d_1 c_2 & c_1 b_2 + d_1 d_2
\end{pmatrix}, a_2=1+4k_1, d_1=1+4l_2, 2\mid c_1, 2\mid c_2;$$
$$\delta_M(g_1g_2)=(-i)^{\tfrac{c_1 a_2 + d_1 c_2}{2}}=(-i)^{\tfrac{ c_2 + c_1 +4k_1c_1+4l_2c_2}{2}}=(-i)^{c_1/2+c_2/2}=\delta_M(g_1)\delta_M(g_2).$$
\end{proof}
So $\delta_M$ is also a character of $\Gamma(2)$ such that $\delta_M(-I)=1=\delta_M(u(2))$ and $\delta_M(u_-(2))=(-i)$. More precisely, 
\[\delta_M(\begin{pmatrix}a& b\\ c& d\end{pmatrix})=(-i)^{c/2} (\tfrac{-1}{d})^{c/2}.\]
\begin{lemma}
Let $r=\begin{pmatrix} a& b\\ c& d\end{pmatrix} \in \Gamma(2)$. Then $\lambda^{\mathfrak{w}}(r )=\lambda(r)\delta_M(r)$.
\end{lemma}
\begin{proof}
\[v_2(\mathfrak{w}^{-1}, r)\stackrel{ \textrm{ Ex.} \ref{twocoverex}}{=}\overline{c}_{X^{\ast}}(\mathfrak{w}, r\mathfrak{w}^{-1})\overline{c}_{X^{\ast}}(r,\mathfrak{w}^{-1})=1.\]
\begin{align*}
\overline{c}_{X^{\ast}}(r_1^{\mathfrak{w}^{-1}}, r_2^{\mathfrak{w}^{-1}}) &\stackrel{ (\ref{chap3crao2})}{=}
 \overline{c}_{X^{\ast}}(r_1, r_2) v_2(\mathfrak{w}^{-1}, r_1)^{-1}v_2(\mathfrak{w}^{-1}, r_2)^{-1}v_2(\mathfrak{w}^{-1}, r_1r_2)\\
 &\stackrel{\textrm{Lem. }\ref{mutirr12}}{=} \lambda(r_1)\lambda(r_2)\lambda(r_1r_2)^{-1}v_2(\mathfrak{w}^{-1}, r_1)^{-1}v_2(\mathfrak{w}^{-1}, r_2)^{-1}v_2(\mathfrak{w}^{-1}, r_1r_2)\\
 &\stackrel{\textrm{Lem. }\ref{mutirr12}}{=}\lambda(r_1^{\mathfrak{w}^{-1}})\lambda(r_2^{\mathfrak{w}^{-1}})\lambda(r_1^{\mathfrak{w}^{-1}}r_2^{\mathfrak{w}^{-1}})^{-1}.
 \end{align*}
 So $\lambda(r^{\mathfrak{w}^{-1}})=\lambda(r)v_2(\mathfrak{w}^{-1}, r)^{-1} \delta'(r)$, for some character $\delta'$ of $\Gamma(2)$. Note that $\Gamma(2)$ is generated by $u(2)$, $u_-(2)$ and $-I$. 
\[v_2(\mathfrak{w}^{-1}, r)^{-1}\stackrel{ \textrm{Ex .} \ref{twocoverex}}{=}\overline{c}_{X^{\ast}}(\mathfrak{w}, r\mathfrak{w}^{-1})^{-1}\overline{c}_{X^{\ast}}(r,\mathfrak{w}^{-1})^{-1}\stackrel{ \textrm{ Lem.} \ref{prog0}}{=}1.\] 
 \begin{itemize}
\item[(I)] If $r=-I$, $\lambda(r^{\mathfrak{w}^{-1}})=\lambda(r) $.
So $\delta'(r)=1=\delta_M(r)$.
 \item[(II)]  If $r=u(2)$, $r^{\mathfrak{w}^{-1}}=u(1)u(2) u(-1)=u(2)$, then $\lambda(r^{\mathfrak{w}^{-1}})=\lambda(r)$. So $\delta'(r)=1=\delta_M(r)$.
 \item[(III)]  If $r=u_-(2)$, $r^{\mathfrak{w}^{-1}}=u(1)u_-(2) u(-1)=\begin{pmatrix} 3& -2\\2& -1\end{pmatrix}$, then  
 $$\lambda(r^{\mathfrak{w}^{-1}})=\lambda(\begin{pmatrix} 3& -2\\2& -1\end{pmatrix}) \stackrel{ \textrm{ Lem.} \ref{gammar}}{=}(\frac{4}{-1}) i^{-1}=-i, \quad 
 \lambda(r)\stackrel{ \textrm{ Lem.} \ref{gammar}}{=}1$$ 
  So $\delta'(r)=-i=\delta_M(r)$.
  \end{itemize}
\end{proof}
\begin{lemma}
$\lambda^{\mathfrak{w}}(r)^{-1}=\left\{ \begin{array}{rl} \lambda(r)^{-1} \delta_M(r)^{-1},  & \textrm{ if } r\in \Gamma(2),\\e^{-\tfrac{\pi i}{4}} \lambda(r_1)^{-1} \delta_M(r_1)^{-1}, &  \textrm{ if } r=r_1 u_-(-1) \textrm{ with } r_1\in \Gamma(2).\end{array} \right.$ 
\end{lemma}
\begin{proof}
The first statement comes from the above lemma. Let   $r=r_1 u_-(-1)$ with $r_1=\begin{pmatrix} a& b\\ c& d\end{pmatrix}  \in \Gamma(2)$. 
\[\mathfrak{w}r_1\mathfrak{w}^{-1}=\begin{pmatrix}1& 1\\0& 1\end{pmatrix}\begin{pmatrix}a& b\\ c& d\end{pmatrix}\begin{pmatrix}1& -1\\0& 1\end{pmatrix}=\begin{pmatrix}a+c& b+d-a-c\\c&d -c \end{pmatrix};\]
\[\mathfrak{w}r_1\mathfrak{w}^{-1}\begin{pmatrix} 0&1\\-1 & 0\end{pmatrix}=\begin{pmatrix}-b-d+a+c& a+c\\c-d&c \end{pmatrix};\]
\[\lambda(\mathfrak{w}u_-(-1)\mathfrak{w}^{-1} )=\lambda(\begin{pmatrix} 0&1\\-1 &2\end{pmatrix})\stackrel{ ( \ref{lambda})}{=}m_{X^{\ast}}(\begin{pmatrix} 0&1\\-1 & 2\end{pmatrix})\widetilde{\beta}^{-1}(\begin{pmatrix} 0&1\\-1 & 2\end{pmatrix})=e^{\tfrac{\pi i}{4}};\]
\begin{align*}
\overline{c}_{X^{\ast}}(r_1^{\mathfrak{w}^{-1}},u_-(-1)^{\mathfrak{w}^{-1}})
&=\overline{c}_{X^{\ast}}(r_1^{\mathfrak{w}^{-1}},\begin{pmatrix} 0&1\\-1 & 0\end{pmatrix}\begin{pmatrix} 1&-2\\0 & 1\end{pmatrix})\\
&=\overline{c}_{X^{\ast}}(r_1^{\mathfrak{w}^{-1}},\begin{pmatrix} 0&1\\-1 & 0\end{pmatrix}\begin{pmatrix} 1&-2\\0 & 1\end{pmatrix})\overline{c}_{X^{\ast}}(\begin{pmatrix} 0&1\\-1 & 0\end{pmatrix},\begin{pmatrix} 1&-2\\0 & 1\end{pmatrix})\\
&=\overline{c}_{X^{\ast}}(r_1^{\mathfrak{w}^{-1}},\begin{pmatrix} 0&1\\-1 & 0\end{pmatrix}) \overline{c}_{X^{\ast}}(r_1^{\mathfrak{w}^{-1}}\begin{pmatrix} 0&1\\-1 & 0\end{pmatrix},\begin{pmatrix} 1&-2\\0 & 1\end{pmatrix})\\
&=\overline{c}_{X^{\ast}}(r_1^{\mathfrak{w}^{-1}},\begin{pmatrix} 0&1\\-1 & 0\end{pmatrix})\\
&=\left\{\begin{array}{lr} (c,d-c)_{\R} & c\neq 0,\\
(d,-1)_{\R} & c=0,\end{array}\right.\\
&=(-1,d-c)_{\R}(-c,d-c)_{\R}.
\end{align*}

\begin{align*}
\lambda^{\mathfrak{w}}(r_1u_-(-1))^{-1}
&\stackrel{\textrm{Lem. }\ref{mutirr12}}{=}\lambda^{\mathfrak{w}}(r_1)^{-1}\lambda^{\mathfrak{w}}(u_-(-1))^{-1}\overline{c}_{X^{\ast}}(r_1^{\mathfrak{w}^{-1}}, u_-(-1)^{\mathfrak{w}^{-1}})\\
&=\lambda(r_1)^{-1} \delta_M(r_1)^{-1}e^{-\tfrac{\pi i}{4}}(-1,d-c)_{\R}(-c,d-c)_{\R}.
\end{align*}
\end{proof}
As a consequence, we have:
\begin{lemma}

$$\nu_{\theta^M}(r)=\left\{ \begin{array}{ll} \left(\frac{2c}{d}\right)\epsilon^{-1}_d (\tfrac{-1}{d})^{c/2}(-i)^{c/2},  & r=\begin{pmatrix} a& b\\ c& d\end{pmatrix}\in \Gamma(2),\\ \left(\frac{2c}{d}\right)\epsilon^{-1}_d (- i)^{c/2}(\tfrac{-1}{d})^{c/2} e^{\tfrac{i \pi}{4}} (-1,d-c)_{\R}(-c,d-c)_{\R},  & r=\begin{pmatrix} a& b\\ c& d\end{pmatrix} \begin{pmatrix} 1& 0\\ -1& 1\end{pmatrix}\in  \Gamma(2)u_-(-1).\end{array}\right.$$
\end{lemma}
For $r'=\overline{\mathfrak{w}}^{-1}s'\overline{\mathfrak{w}}=[r, \lambda^{\mathfrak{w}}(r)^{-1}]\in \Gamma^{0'}(2) \subseteq \overline{\SL}^{\mu_8}_2(\R)$, we have:
$$ J_{1/2}(r',z)=J_{1/2}(r,z) m_{X^{\ast}}(r)^{-1}\lambda^{\mathfrak{w}}(r)=J_{1/2}(r,z)m_{X^{\ast}}(r)^{-1}\nu_{\theta^M}(r).$$
In this case, the result is the following:
\begin{proposition}\label{proathetb}
 $\theta^{M,1/2}(r z)=  \nu_{\theta^M}(r)\sqrt{cz+d}\theta^{M,1/2}(z)$, for $r=\begin{pmatrix} a& b\\ c& d\end{pmatrix}\in \Gamma^0(2)$.
\end{proposition}
\begin{remark}
The result is similar as Prop.\ref{thetashi}.  This means that Zagier's minus theta function is a $1/2$-modular form on  $\Gamma^0(2)$. Certainly, it is just a consequence of  Lion-Vergne \cite{LiVe}.
\end{remark}
\subsection{Example 4 to Sect. \ref{twistmathw}}\label{Examples14}
Let $\overline{\mathfrak{w}}= [u_-(-1),1]$, for $\mathfrak{w}=u_-(-1)=\begin{pmatrix} 1& 0\\ -1& 1\end{pmatrix}$. 
Let $\overline{s}=[s,\lambda({s})^{-1}]\in  \overline{\Gamma}^{\mu_8}_{\theta} \subseteq  \overline{\SL}^{\mu_8}_2(\mathbb R)$ and $r=u_-(1)s u_-(-1)$. Note:
\[ u_-(1)\Gamma_{\theta} u_-(-1)=\Gamma_0(2), \quad  u_-(1)\Gamma(2) u_-(-1)=\Gamma(2), \quad \overline{\mathfrak{w}}^{-1}= [u_-(1),1].\]
\begin{lemma}
\begin{itemize}
\item[(1)] $\nu_2(\mathfrak{w}^{-1}, g)=\nu_2(\omega , g)\nu_2(\omega,  g^{\omega^{-1}u(-1)})$, for $g\in \SL_2(\R)$.
\item[(2)] $\overline{\mathfrak{w}}^{-1}\overline{s}\overline{\mathfrak{w}}=[r, \nu_2(\mathfrak{w}^{-1},r)^{-1}\lambda^{\mathfrak{w}}(r)^{-1}]$.
\end{itemize}
\end{lemma}
\begin{proof}
1) $\nu_2(u(-1), g)\stackrel{\textrm{ Ex.} \ref{twocoverex}}{=} \overline{c}_{X^{\ast}}(u(-1)^{-1},g)\overline{c}_{X^{\ast}}(u(-1)^{-1}g,u(-1))\stackrel{\textrm{ Lem.} \ref{prog0}}{=} 1$, 
$ \nu_2(-I, g)= 1$. Note that $u_-(1)=\omega^{-1} u(-1)\omega$. Then: 
\begin{align*}
\nu_2(u_-(1), g)&=\nu_2(\omega^{-1} u(-1)\omega, g)\\
&\stackrel{\textrm{ Lem.} \ref{chap3twocover}}{=}\nu_2(\omega^{-1} , g) \nu_2(u(-1)\omega, g^{\omega^{-1}})\\
&=\nu_2(\omega^{-1} , g)\nu_2(u(-1),  g^{\omega^{-1}}) \nu_2(\omega,  g^{\omega^{-1}u(-1)})\\
&=\nu_2(\omega^{-1} , g)\nu_2(\omega,  g^{\omega^{-1}u(-1)})\\
&=\nu_2(\omega , g)\nu_2(\omega,  g^{\omega^{-1}u(-1)}).
\end{align*}
2) $\nu_2(u_-(-1), s)\stackrel{\textrm{ Ex.} \ref{twocoverex}}{=} \overline{c}_{X^{\ast}}(u_-(1),s)\overline{c}_{X^{\ast}}(u_-(1)s,u_-(-1))$.
\[1=\nu_2(u_-(-1)u_-(1), s)=\nu_2(u_-(-1), s)\nu_2(u_-(1), s^{u_-(-1)})=\nu_2(u_-(-1), s)\nu_2(u_-(1), r).\]
\begin{align*}
 \overline{\mathfrak{w}}^{-1}\overline{s}\overline{\mathfrak{w}}
&=[u_-(1) s u_-(-1), \overline{c}_{X^{\ast}}(u_-(1), s)\overline{c}_{X^{\ast}}(u_-(1) s, u_-(-1))\lambda({s})^{-1}]\\
&=[u_-(1)s u_-(-1), \nu_2(u_-(-1),s)\lambda({s})^{-1}]\\
&=[r, \nu_2(u_-(1),r)^{-1}\lambda({s})^{-1}].
\end{align*}
\end{proof}
 \[[\theta^{1/2}]^{\overline{\mathfrak{w}}}(z)= m_{X^{\ast}}(u_-(-1))^{-1} e^{i   \pi z_g ( n+\tfrac{1}{2})^2 }=e^{-\tfrac{\pi i}{4}}e^{i   \pi z ( n+\tfrac{1}{2})^2 }.\]
\begin{definition}
\begin{itemize}
\item[(1)] $ \theta^{F,1/2}(z)\stackrel{ \textrm{ Def.}}{=}\sum_{n\in \Z}e^{i   \pi z ( n+\tfrac{1}{2})^2 }$  .
\item[(2)] $\Gamma'_{0}(2)\stackrel{ \textrm{ Def.}}{=}\{ [r, \lambda^{\mathfrak{w}}(r)^{-1} \nu_2(\mathfrak{w}^{-1},r)^{-1}] \mid r\in \Gamma_0(2)\}$.
\item[(3)] $\nu_{\theta^{F}}(r)\stackrel{ \textrm{ Def.}}{=} \lambda^{\mathfrak{w}}(r)\nu_2(\mathfrak{w}^{-1},r)$.
\end{itemize}
\end{definition}
\begin{lemma}
 There exists a group isomorphism $\iota: \Gamma_{0}(2) \simeq \Gamma'_{0}(2); r\to [r, \nu_{\theta^{F}}(r)^{-1}]$.
\end{lemma}
\begin{proof}
 \begin{align*}
 \iota(r_1) \iota(r_2)&= [r_1, \nu_{\theta^{F}}(r_1)^{-1}] \cdot [r_2, \nu_{\theta^{F}}(r_2)^{-1}]\\
 &=[r_1r_2, \lambda(\mathfrak{w}r_1\mathfrak{w}^{-1})^{-1} \lambda(\mathfrak{w}r_2\mathfrak{w}^{-1})^{-1} \nu_2(\mathfrak{w}^{-1},r_1)^{-1}\nu_2(\mathfrak{w}^{-1},r_2)^{-1}\overline{c}_{X^{\ast}}(r_1,r_2)]\\
 &=[r_1r_2, \lambda(\mathfrak{w}r_1\mathfrak{w}^{-1})^{-1} \lambda(\mathfrak{w}r_2\mathfrak{w}^{-1})^{-1} \overline{c}_{X^{\ast}}(\mathfrak{w}r_1\mathfrak{w}^{-1},\mathfrak{w}r_2\mathfrak{w}^{-1})\nu_2(\mathfrak{w}^{-1},r_1r_2)^{-1}]\\
 &\stackrel{\textrm{ Lem. }\ref{mutirr12}}{=}[r_1r_2, \lambda(\mathfrak{w}r_1r_2\mathfrak{w}^{-1})^{-1}\nu_2(\mathfrak{w}^{-1},r_1r_2)^{-1}]\\
 &=[r_1r_2, \nu_{\theta^{F}}(r_1r_2)^{-1}]=\iota(r_1r_2) .
 \end{align*}
\end{proof}
Note that 
\[\Gamma(2)/\{\pm I\} \simeq \langle \begin{pmatrix} 1& 0\\ 2& 1\end{pmatrix}, \begin{pmatrix} 1& 2\\ 0& 1\end{pmatrix} \rangle=\{ g=\begin{pmatrix}a& b\\ c& d\end{pmatrix}\in \Gamma(2)\mid  a\equiv d\equiv 1 (\bmod 4)\}.\]
\begin{lemma}
Let $\delta_F: \{ g=\begin{pmatrix}a& b\\ c& d\end{pmatrix}\in \Gamma(2)\mid  a\equiv d\equiv 1 (\bmod 4)\} \to \mu_4; g\to i^{b/2}$. Then $\delta_F$ is a character. 
\end{lemma}
\begin{proof}
$$g_1g_2=\begin{pmatrix} a_1 & b_1 \\ c_1 & d_1 \end{pmatrix}
\begin{pmatrix} a_2 & b_2 \\ c_2 & d_2 \end{pmatrix}
= \begin{pmatrix}
a_1 a_2 + b_1 c_2 & a_1 b_2 + b_1 d_2 \\
c_1 a_2 + d_1 c_2 & c_1 b_2 + d_1 d_2
\end{pmatrix}, a_1=1+4k_1, d_2=1+4l_2, 2\mid b_1, 2\mid b_2;$$
$$\delta_F(g_1g_2)=i^{\tfrac{a_1 b_2 + b_1 d_2}{2}}=i^{\tfrac{ b_2 + b_1 +4k_1b_2+4l_2b_1}{2}}=i^{b_1/2+b_2/2}=\delta_F(g_1)\delta_F(g_2).$$
\end{proof}
So $\delta_F$ is also a character of $\Gamma(2)$ such that $\delta_F(-I)=1=\delta_F(u_-(2))$ and $\delta_F(u(2))=i$. More precisely, 
\[\delta_F(\begin{pmatrix}a& b\\ c& d\end{pmatrix})=i^{b/2} (\tfrac{-1}{d})^{b/2}.\]
\begin{lemma}\label{deltaf}
Let $r=\begin{pmatrix} a& b\\ c& d\end{pmatrix} \in \Gamma(2)$. Then $\lambda^{\mathfrak{w}}(r )=\lambda(r)v_2(\mathfrak{w}^{-1}, r)^{-1} \delta_F(r)$.
\end{lemma}
\begin{proof}
\[v_2(\mathfrak{w}^{-1}, r)\stackrel{ \textrm{ Lem.} \ref{twocoverex}}{=}\overline{c}_{X^{\ast}}(\mathfrak{w}, r\mathfrak{w}^{-1})\overline{c}_{X^{\ast}}(r,\mathfrak{w}^{-1}).\]
\begin{align*}
\overline{c}_{X^{\ast}}(r_1^{\mathfrak{w}^{-1}}, r_2^{\mathfrak{w}^{-1}}) &\stackrel{ (\ref{chap3crao2})}{=}
 \overline{c}_{X^{\ast}}(r_1, r_2) v_2(\mathfrak{w}^{-1}, r_1)^{-1}v_2(\mathfrak{w}^{-1}, r_2)^{-1}v_2(\mathfrak{w}^{-1}, r_1r_2)\\
 &\stackrel{\textrm{Lem. }\ref{mutirr12}}{=} \lambda(r_1)\lambda(r_2)\lambda(r_1r_2)^{-1}v_2(\mathfrak{w}^{-1}, r_1)^{-1}v_2(\mathfrak{w}^{-1}, r_2)^{-1}v_2(\mathfrak{w}^{-1}, r_1r_2)\\
 &\stackrel{\textrm{Lem. }\ref{mutirr12}}{=}\lambda(r_1^{\mathfrak{w}^{-1}})\lambda(r_2^{\mathfrak{w}^{-1}})\lambda(r_1^{\mathfrak{w}^{-1}}r_2^{\mathfrak{w}^{-1}})^{-1}.
 \end{align*}
 So $\lambda(r^{\mathfrak{w}^{-1}})=\lambda(r)v_2(\mathfrak{w}^{-1}, r)^{-1} \delta'(r)$, for some character $\delta'$ of $\Gamma(2)$. Note that $\Gamma(2)$ is generated by $u(2)$, $u_-(2)$ and $-I$. 
 \begin{itemize}
\item[(I)] If $r=-I$, $\lambda(r^{\mathfrak{w}^{-1}})=\lambda(r) $ and 
\[v_2(\mathfrak{w}^{-1}, r)^{-1}\stackrel{ \textrm{ Ex.} \ref{twocoverex}}{=}\overline{c}_{X^{\ast}}(\mathfrak{w}, r\mathfrak{w}^{-1})^{-1}\overline{c}_{X^{\ast}}(r,\mathfrak{w}^{-1})^{-1}\stackrel{ \textrm{ Lem.} \ref{relacto}(1)}{=}1.\] So $\delta'(r)=1=\delta_F(r)$.
 \item[(II)]  If $r=u(2)$, $r^{\mathfrak{w}^{-1}}=u_-(-1)u(2) u_-(1)=\begin{pmatrix} 3& 2\\-2& -1\end{pmatrix}$, then: 
 $$\lambda(r^{\mathfrak{w}^{-1}})=\lambda(\begin{pmatrix} 3& 2\\-2& -1\end{pmatrix})\stackrel{ \textrm{ Lem.} \ref{gammar}}{=}(\frac{-4}{-1}) i^{-1}=i, \quad 
 \lambda(r)\stackrel{ \textrm{ Lem.} \ref{gammar}}{=}1,$$ 
 $$ r^{\omega^{-1}u(-1)}=\begin{pmatrix} 1& 1\\0& 1\end{pmatrix}\begin{pmatrix} 0& -1\\1& 0\end{pmatrix}\begin{pmatrix} 1& 2\\0& 1\end{pmatrix}\begin{pmatrix} 0& 1\\-1& 0\end{pmatrix}\begin{pmatrix} 1& -1\\0& 1\end{pmatrix}=\begin{pmatrix} -1& 2\\-2& 3\end{pmatrix},$$ 
  $$v_2(\mathfrak{w}^{-1}, r)^{-1}=\nu_2(\omega , r)^{-1}\nu_2(\omega,  r^{\omega^{-1}u(-1)})^{-1}\stackrel{ \textrm{ Lem.} \ref{nu2abcd}}{=}
  \nu_2(\omega,  \begin{pmatrix} -1& 2\\-2& 3\end{pmatrix})^{-1} \stackrel{ \textrm{ Lem.} \ref{nu2abcd}}{=} 1.$$ 
  So $\delta'(r)=i=\delta_F(r)$.
 \item[(III)]  If $r=u_-(2)$, $r^{\mathfrak{w}^{-1}}=u_-(-1)u_-(2) u_-(1)=r$, then  $\lambda(r^{\mathfrak{w}^{-1}})=\lambda(r) $, 
 $$r^{\omega^{-1}u(-1)}=\begin{pmatrix} 1& 1\\0& 1\end{pmatrix}\begin{pmatrix} 0& -1\\1& 0\end{pmatrix}\begin{pmatrix} 1&0\\2& 1\end{pmatrix}\begin{pmatrix} 0& 1\\-1& 0\end{pmatrix}\begin{pmatrix} 1& -1\\0& 1\end{pmatrix}=\begin{pmatrix} 1& -2\\0& 1\end{pmatrix},$$
 $$v_2(\mathfrak{w}^{-1}, r)^{-1}=\nu_2(\omega , r)^{-1}\nu_2(\omega,  r^{\omega^{-1}u(-1)})^{-1}\stackrel{ \textrm{ Lem.} \ref{nu2abcd}}{=}
  \nu_2(\omega,  \begin{pmatrix} 1& -2\\0& 1\end{pmatrix})^{-1} \stackrel{ \textrm{ Lem.} \ref{nu2abcd}}{=} 1.$$ 
  So $\delta'(r)=1=\delta_F(r)$.
  \end{itemize}
\end{proof}
\begin{lemma}
$\lambda^{\mathfrak{w}}(r)^{-1}\nu_2(\mathfrak{w}^{-1},r)^{-1}=\left\{ \begin{array}{rl} \lambda(r)^{-1} \delta_F(r)^{-1},  & \textrm{ if } r\in \Gamma(2),\\e^{-\tfrac{\pi i}{4}} \lambda(r_1)^{-1} \delta_F(r_1)^{-1}, &  \textrm{ if } r=r_1 u(1) \textrm{ with } r_1\in \Gamma(2).\end{array} \right.$ 
\end{lemma}
\begin{proof}
The first statement comes from the above lemma. If  $r=r_1 u(1)$, then:
\[v_2(\mathfrak{w}^{-1}, u(1))\stackrel{ \textrm{ Ex.} \ref{twocoverex}}{=}\overline{c}_{X^{\ast}}(\mathfrak{w}, u(1)\mathfrak{w}^{-1})\overline{c}_{X^{\ast}}(u(1),\mathfrak{w}^{-1})=\overline{c}_{X^{\ast}}(\begin{pmatrix} 1& 0\\-1& 1\end{pmatrix},\begin{pmatrix} 2& 1\\1& 1\end{pmatrix})=1;\]
\[\lambda(\mathfrak{w}u(1)\mathfrak{w}^{-1} )=\lambda(\begin{pmatrix} 2&1\\-1 & 0\end{pmatrix})\stackrel{ \textrm{(\ref{lambda})}}{=}m_{X^{\ast}}(\begin{pmatrix} 2&1\\-1 & 0\end{pmatrix})\widetilde{\beta}^{-1}(\begin{pmatrix} 2&1\\-1 & 0\end{pmatrix})=e^{\tfrac{\pi i}{4}};\]
\begin{align*}
&\lambda^{\mathfrak{w}}(r_1u(1))^{-1}\nu_2(\mathfrak{w}^{-1},r_1u(1))^{-1}\\
&=\overline{c}_{X^{\ast}}(r_1, u(1))\lambda^{\mathfrak{w}}(r_1)^{-1}\nu_2(\mathfrak{w}^{-1},r_1)^{-1}\lambda^{\mathfrak{w}}(u(1))^{-1}\nu_2(\mathfrak{w}^{-1},u(1))^{-1}\\
&=\lambda^{\mathfrak{w}}(r_1)^{-1}\nu_2(\mathfrak{w}^{-1},r_1)^{-1}\lambda^{\mathfrak{w}}(u(1))^{-1}\nu_2(\mathfrak{w}^{-1},u(1))^{-1}\\
&=\lambda(r_1)^{-1} \delta_F(r_1)^{-1}e^{-\tfrac{\pi i}{4}}.
\end{align*}
\end{proof}
As a consequence, we have:
\begin{lemma}
$\nu_{\theta^{F}}(r)=\left\{ \begin{array}{ccl} \left(\frac{2c}{d}\right)\epsilon^{-1}_d (\tfrac{-1}{d})^{b/2}i^{b/2}, & & r=\begin{pmatrix} a& b\\ c& d\end{pmatrix}\in \Gamma(2),\\ \left(\frac{2c}{d}\right)\epsilon^{-1}_d i^{b/2}(\tfrac{-1}{d})^{b/2} e^{\tfrac{i \pi}{4}}, & & r=\begin{pmatrix} a& b\\ c& d\end{pmatrix} \begin{pmatrix} 1& 1\\ 0& 1\end{pmatrix}\in  \Gamma(2)u(1).\end{array}\right.$
\end{lemma}
For $r'=\overline{\mathfrak{w}}^{-1}s'\overline{\mathfrak{w}}=[r, \nu_2(\mathfrak{w}^{-1},r)^{-1}\lambda^{\mathfrak{w}}(r)^{-1}]\in \Gamma'_{0}(2) \subseteq \overline{\SL}^{\mu_8}_2(\R)$, we have:
$$ J_{1/2}(r',z)=J_{1/2}(r,z) m_{X^{\ast}}(r)^{-1} \nu_2(\mathfrak{w}^{-1},r)\lambda^{\mathfrak{w}}(r)=J_{1/2}(r,z)m_{X^{\ast}}(r)^{-1}\nu_{\theta^{F}}(r).$$
In this case, the result is the following:
\begin{proposition}\label{proathetc}
 $\theta^{F,1/2}(r z)=  \nu_{\theta^{F}}(r)\sqrt{cz+d}\theta^{F,1/2}(z)$, for $r\in \Gamma_0(2)$.
\end{proposition}
In particular, if  $r=\begin{pmatrix} 1& 8\\ 0&1\end{pmatrix}$, then  $\theta^{F,1/2}(r z)=\theta^{F,1/2}(z)$.  
\begin{remark}
The result is similar as Prop. \ref{thetashi}.  This means that Zagier's fermionic theta function is a $1/2$-modular form on  $\Gamma_0(2)$, not just $\Gamma(2)$ or $\Gamma_0(4)$. Certainly, it is just a consequence of  Lion-Vergne \cite{LiVe}.
\end{remark}
\subsection{Example 5 to Sect. \ref{twistmathw}}\label{Examples15}
Let now $\mathfrak{w}=\begin{pmatrix} t& 0\\ 0 & 1\end{pmatrix}$ with positive integer  $t$ and $\overline{\mathfrak{w}}=[\mathfrak{w}, 1]$.  Then:
\[\mathfrak{w}^{-1}\Gamma_{\theta}\mathfrak{w}=\mathfrak{w}^{-1}\Gamma(2)\mathfrak{w} \sqcup \mathfrak{w}^{-1}\Gamma(2)\mathfrak{w} \begin{pmatrix}0& -t^{-1}\\ t& 0\end{pmatrix}.\]
If $r\in \mathfrak{w}^{-1}\Gamma_{\theta}\mathfrak{w}$, then $r'=\mathfrak{w}r\mathfrak{w}^{-1}\in \Gamma_{\theta}$. We define:
\[\nu_{\theta,t}(r)=\nu_{\theta}(r').\]
If $r=\begin{pmatrix} a& b\\ c & d\end{pmatrix} \in \mathfrak{w}^{-1}\Gamma(2)\mathfrak{w} $, then $r'=\begin{pmatrix} a& tb\\ t^{-1}c & d\end{pmatrix} \in \Gamma(2)$. Hence,
\[\nu_{\theta,t}(r)=\nu_{\theta}(\mathfrak{w}r\mathfrak{w}^{-1})=  \left(\frac{2c/t}{d}\right)\epsilon^{-1}_d= \left(\frac{2ct}{d}\right)\epsilon^{-1}_d.\]
If $r\in \mathfrak{w}^{-1}\Gamma(2)\mathfrak{w} \begin{pmatrix}0& -t^{-1}\\ t& 0\end{pmatrix}$, we write:
\[r=\begin{pmatrix} a& b\\ c & d\end{pmatrix}  \begin{pmatrix}0& -t^{-1}\\ t& 0\end{pmatrix}, \quad r'=\begin{pmatrix} a& tb\\ t^{-1}c & d\end{pmatrix} \begin{pmatrix} 0& -1\\ 1&0\end{pmatrix}.\]
Hence,
\[\nu_{\theta,t}(r)=\nu_{\theta}(\mathfrak{w}r\mathfrak{w}^{-1})= \left(\frac{2ct^{-1}}{d}\right)\epsilon^{-1}_de^{-\tfrac{i \pi}{4}}(-ct^{-1},d)_{\R}=\left(\frac{2ct}{d}\right)\epsilon^{-1}_de^{-\tfrac{i \pi}{4}}(-ct,d)_{\R}=\left(\frac{2ct}{d}\right)\epsilon^{-1}_de^{-\tfrac{i \pi}{4}}(-c,d)_{\R}.\]
Let:
 \[r=\begin{pmatrix} a& b\\ c & d\end{pmatrix}\in  \mathfrak{w}^{-1}[\Gamma_{\theta}\cap \Gamma_0(N^2)]\mathfrak{w}, \quad \quad r'=\mathfrak{w}r\mathfrak{w}^{-1}=\begin{pmatrix} a& tb\\ ct^{-1} & d\end{pmatrix}\in \Gamma_{\theta}\cap \Gamma_0(N^2).\] Then:
\[\overline{\mathfrak{w}}^{-1}=[\mathfrak{w}^{-1}, 1], \quad \overline{\mathfrak{w}}^{-1}[r', \nu_{\theta}(r')^{-1}]\overline{\mathfrak{w}}=[\mathfrak{w}^{-1}r'\mathfrak{w}, \nu_{\theta}(r')^{-1}]=[r, \nu_{\theta}(\mathfrak{w}r\mathfrak{w}^{-1})^{-1}],\]
\[[\theta^{ 1/2}_{\chi}]^{\overline{\mathfrak{w}}}( z)=t^{1/4}J_{1/2}(\overline{\mathfrak{w}},z)^{-1}  \theta^{ 1/2}_{\chi}(tz)= t^{1/4}\theta^{ 1/2}_{\chi}(tz)= t^{1/4}\sum_{n \in \mathbb{Z}} \chi(n)e^{\pi i  tz n^2},\]
\[[\theta^{ 3/2}_{\chi}]^{\overline{\mathfrak{w}}}( z)=t^{3/4}J_{3/2}(\overline{\mathfrak{w}},z)^{-1}  \theta^{ 3/2}_{\chi}(tz)=t^{3/4}\theta^{ 3/2}_{\chi}(tz)= t^{3/4}\sum_{n \in \mathbb{Z}} \chi(n)ne^{\pi i  tz n^2}.\]
Recall  $\theta^{ 1/2}_{\chi,t}$ and $\theta^{3/2}_{\chi,t} $ from  Def. \ref{defmainprop1}.
\begin{definition}
 $\nu_{\theta,t}(r)= \left\{ \begin{array}{lcl}  \left(\frac{2ct}{d}\right)\epsilon^{-1}_d, & & r=\begin{pmatrix} a& tb\\ ct^{-1} & d\end{pmatrix}\in \Gamma(2),\\ \left(\frac{2ct}{d}\right)\epsilon^{-1}_de^{-\tfrac{i \pi}{4}}(-c,d)_{\R}, & & r=\begin{pmatrix} a& tb\\ ct^{-1} & d\end{pmatrix}\begin{pmatrix} 0& -1\\ 1& 0\end{pmatrix}\in  \Gamma(2)\omega.\end{array}\right.$
\end{definition}
Then the result is the following:
\begin{proposition}\label{examthetaN}
  $\theta^{ \kappa/2}_{\chi,t}( rz)=\nu_{\theta,t}(r)(cz+d)^{\kappa/2}\chi(d)\theta^{ \kappa/2}_{\chi,t}(z)$,  for  $r=\begin{pmatrix} a& b\\ c & d\end{pmatrix}\in   \mathfrak{w}^{-1}[\Gamma_{\theta}\cap \Gamma_0(N^2)]\mathfrak{w}$, $\kappa=1$ or $3$.
\end{proposition}
\subsection{Example 6 to Sect. \ref{twistmathw} }\label{Examples16} Let $N$ be a positive natural number.
 Let $\mathfrak{w}=u(1)\in \Gamma_0(N^2)$ and  $\overline{\mathfrak{w}}=[\mathfrak{w}, 1]\in \overline{\SL}_2(\Q)$. Then:
  \[[\theta^{1/2}_{\chi}]^{\overline{\mathfrak{w}}}( z)=J_{1/2}(\overline{\mathfrak{w}},z)^{-1} \theta^{1/2}_{\chi}(z+1)= \sum_{n \in \mathbb{Z}} \chi(n)\, e^{i\pi (z+1) n^2}=\sum_{n \in \mathbb{Z}} (-1)^{n}\chi(n)\, e^{i\pi z n^2} \]
\[[\theta^{3/2}_{\chi}]^{\overline{\mathfrak{w}}}( z)=J_{3/2}(\overline{\mathfrak{w}},z)^{-1} \theta^{3/2}_{\chi}(z+1)=\sum_{n\in \Z}(-1)^n\chi(n) ne^{\pi i zn^2} .\]
Let $s=\begin{pmatrix} a& b\\ c& d\end{pmatrix}\in \Gamma_{\theta}$,  $s'=[s,\lambda(s)^{-1}] \in  \overline{\Gamma}^{\mu_8}_{\theta} \subseteq  \overline{\SL}^{\mu_8}_2(\mathbb R)$, $r=\mathfrak{w}^{-1}s \mathfrak{w}=\begin{pmatrix} a-c & (a-c)+(b-d) \\ c & c+d \end{pmatrix}$ and $r'=\overline{\mathfrak{w}}^{-1} s'\overline{\mathfrak{w}}$.  According to Sect. \ref{Examples13},  
\[\overline{\mathfrak{w}}^{-1}=[\mathfrak{w}^{-1}, 1], \quad \overline{\mathfrak{w}}^{-1}s'\overline{\mathfrak{w}}=[r, \lambda^{\mathfrak{w}}(r)^{-1}]=[r, \nu_{\theta^M}(r)^{-1}] \in  \Gamma^{0'}(2).\]
If $s\in \Gamma_{\theta}\cap \Gamma_0(N^2)$, we have: 
\[\chi^{\overline{\mathfrak{w}}}(r') = \chi(s')=\chi(s)=\chi(d)=\chi( c+d ).\]
\begin{definition}
\begin{itemize}
\item $\theta^{M, 1/2}_{\chi}( z)\stackrel{ Def.}{=} \sum_{n \in \mathbb{Z}} (-1)^{n}\chi(n)\, e^{i\pi z n^2} $.
\item $\theta^{M,3/2}_{\chi}( z)\stackrel{ Def.}{=} \sum_{n\in \Z}(-1)^n\chi(n) ne^{\pi i zn^2} $.
\end{itemize}
\end{definition}
In this case, the result is the following:
\begin{proposition}\label{examthetaMN}
 $\theta^{M, \kappa/2}_{\chi}( rz)=\nu_{\theta^M}(r)(cz+d)^{\kappa/2}\chi(d)\theta^{M, \kappa/2}_{\chi}(z)$, for  $r=\begin{pmatrix} a& b\\ c & d\end{pmatrix}\in  \Gamma^0(2)\cap \Gamma_0(N^2)$, $\kappa=1$ or $3$.
\end{proposition}
\subsubsection{}\label{exf5} Let now $\mathfrak{w}=\begin{pmatrix} t& 0\\ 0 & 1\end{pmatrix}$ with positive integer  and $\overline{\mathfrak{w}}=[\mathfrak{w}, 1]$.  Then:
\[\mathfrak{w}^{-1}\Gamma^0(2)\mathfrak{w}=\mathfrak{w}^{-1}\Gamma(2)\mathfrak{w} \sqcup \mathfrak{w}^{-1}\Gamma(2)\mathfrak{w} \begin{pmatrix} 1& 0\\ -t& 1\end{pmatrix}.\]
If $r\in \mathfrak{w}^{-1}\Gamma^0(2)\mathfrak{w}$, then $r'=\mathfrak{w}r\mathfrak{w}^{-1}\in \Gamma^0(2)$. We define:
\[\nu_{\theta^M,t}(r)=\nu_{\theta^M}(r').\]
If $r=\begin{pmatrix} a& b\\ c & d\end{pmatrix} \in \mathfrak{w}^{-1}\Gamma(2)\mathfrak{w} $, then $r'=\begin{pmatrix} a& tb\\ t^{-1}c & d\end{pmatrix} \in \Gamma(2)$. Hence,
\begin{align*}
\nu_{\theta^M,t}(r)&=\nu_{\theta^M}(\mathfrak{w}r\mathfrak{w}^{-1})= \left(\frac{2c/t}{d}\right)\epsilon^{-1}_d (\tfrac{-1}{d})^{c/2t}(-i)^{c/2t}=\left(\frac{2ct}{d}\right)\epsilon^{-1}_d (\tfrac{-1}{d})^{c/2t}(-i)^{c/2t} .
\end{align*}
If $r\in \mathfrak{w}^{-1}\Gamma(2)\mathfrak{w} \begin{pmatrix} 1& 0\\ -t& 1\end{pmatrix}$, we write:
\[r=\begin{pmatrix} a& b\\ c & d\end{pmatrix} \begin{pmatrix} 1& 0\\ -t& 1\end{pmatrix}, \quad r'=\begin{pmatrix} a& tb\\ t^{-1}c & d\end{pmatrix} \begin{pmatrix} 1& 0\\ -1& 1\end{pmatrix}.\]
Hence,
\begin{align*}\nu_{\theta^M,t}(r)=\nu_{\theta^M}(\mathfrak{w}r\mathfrak{w}^{-1})&= \left(\frac{2ct}{d}\right)\epsilon^{-1}_d (\tfrac{-1}{d})^{c/2t}(-i)^{c/2t}e^{\tfrac{\pi i}{4}}(-1,d-t^{-1}c)_{\R}(-t^{-1}c,d-t^{-1}c)_{\R}\\
&= \left(\frac{2ct}{d}\right)\epsilon^{-1}_d (\tfrac{-1}{d})^{c/2t}(-i)^{c/2t}e^{\tfrac{\pi i}{4}}(-1,dt-c)_{\R}(-c,dt-c)_{\R}.
\end{align*}
Let:
 \[r=\begin{pmatrix} a& b\\ c & d\end{pmatrix}\in  \mathfrak{w}^{-1}[\Gamma^0(2)\cap \Gamma_0(N^2)]\mathfrak{w}, \quad \quad r'=\mathfrak{w}r\mathfrak{w}^{-1}=\begin{pmatrix} a& tb\\ ct^{-1} & d\end{pmatrix}\in \Gamma^0(2)\cap \Gamma_0(N^2).\] Then:
\[\overline{\mathfrak{w}}^{-1}=[\mathfrak{w}^{-1}, 1], \quad \overline{\mathfrak{w}}^{-1}[r', \nu_{\theta^M}(r')^{-1}]\overline{\mathfrak{w}}=[\mathfrak{w}^{-1}r'\mathfrak{w}, \nu_{\theta^M}(r')^{-1}]=[r, \nu_{\theta^M}(\mathfrak{w}r\mathfrak{w}^{-1})^{-1}],\]
\[[\theta^{M, 1/2}_{\chi}]^{\overline{\mathfrak{w}}}( z)=t^{1/4}J_{1/2}(\overline{\mathfrak{w}},z)^{-1}  \theta^{M, 1/2}_{\chi}(tz)=t^{1/4}\theta^{M, 1/2}_{\chi}(tz)= t^{1/4}\sum_{n \in \mathbb{Z}} (-1)^{n}\chi(n)\, e^{i\pi tz n^2},\]
\[[\theta^{M, 3/2}_{\chi}]^{\overline{\mathfrak{w}}}( z)=t^{3/4}J_{3/2}(\overline{\mathfrak{w}},z)^{-1}  \theta^{M, 3/2}_{\chi}(tz)= t^{3/4}\theta^{M, 3/2}_{\chi}(tz)= t^{3/4}\sum_{n \in \mathbb{Z}} (-1)^{n}\chi(n)\, ne^{i\pi tz n^2}.\]
\begin{definition}
\begin{itemize}
\item $\theta^{M, 1/2}_{\chi,t}( z)\stackrel{ Def.}{=} \sum_{n \in \mathbb{Z}} (-1)^{n}\chi(n)\, e^{i\pi tz n^2} $.
\item $\nu_{\theta^M,t}(r)= \left\{ \begin{array}{ll} \left(\frac{2ct}{d}\right)\epsilon^{-1}_d (\tfrac{-1}{d})^{\tfrac{c}{2t}}(-i)^{\tfrac{c}{2t}}, & r=\begin{pmatrix} a& tb\\ ct^{-1} & d\end{pmatrix}\in \Gamma(2),\\ \left(\frac{2ct}{d}\right)\epsilon^{-1}_d (\tfrac{-1}{d})^{\tfrac{c}{2t}}(-i)^{\tfrac{c}{2t}} e^{\tfrac{i \pi}{4}}(-1,dt-c)_{\R}(-c,dt-c)_{\R},  & r=\begin{pmatrix} a& tb\\ ct^{-1} & d\end{pmatrix}\begin{pmatrix} 1& 0\\ -1& 1\end{pmatrix}\in  \Gamma(2)u_-(-1).\end{array}\right.$
\item $\theta^{M,3/2}_{\chi,t}( z)\stackrel{ Def.}{=} \sum_{n\in \Z}(-1)^n\chi(n) ne^{\pi i tzn^2} $.
\end{itemize}
\end{definition}
Then the result is the following:
\begin{proposition}\label{examthetaMNt} 
$\theta^{M, \kappa/2}_{\chi,t}( rz)=\nu_{\theta^M,t}(r)(cz+d)^{\kappa/2}\chi(d)\theta^{M, \kappa/2}_{\chi,t}(z)$, for $r=\begin{pmatrix} a& b\\ c & d\end{pmatrix}\in   \mathfrak{w}^{-1}[\Gamma^0(2)\cap \Gamma_0(N^2)]\mathfrak{w}$, $\kappa=1$ or $3$.
\end{proposition}

\subsection{Example 7  to Sect. \ref{twistmathw} }\label{Examples17} Let $N$ be an odd positive integer.
 Let $\mathfrak{w}=u_-(-N^2)$ and  $\overline{\mathfrak{w}}=[\mathfrak{w}, 1]\in \overline{\SL}_2(\Q)$. Then:
  \[[\theta^{1/2}_{\chi}]^{\overline{\mathfrak{w}}}( z)=J_{1/2}(\overline{\mathfrak{w}},z)^{-1} \theta^{1/2}_{\chi}(\mathfrak{w} z) \]
\[[\theta^{3/2}_{\chi}]^{\overline{\mathfrak{w}}}( z)=J_{3/2}(\overline{\mathfrak{w}},z)^{-1} \theta^{3/2}_{\chi}(\mathfrak{w}z).\]
1)
  \begin{align*}
 \theta^{1/2}_{\chi}(\mathfrak{w} g)&= \tfrac{1}{G(1, \overline{\chi})}\sum_{k \bmod N}\overline{ \chi}(k) \theta_{L, X^{\ast}}\Big(J_{1/2}(\mathfrak{w} g, z_0)\pi_{X^{\ast}, \psi} (\mathfrak{w} g)A\Big)(\tfrac{k}{N}e^{\ast})\\
&= \tfrac{1}{G(1, \overline{\chi})}\sum_{k \bmod N}\overline{ \chi}(k)  J_{1/2}(\mathfrak{w} ,  p_g(z_0))\pi_{L, \psi}(\mathfrak{w} )\theta_{L, X^{\ast}}\Big(J_{1/2}(p_g,z_0)\pi_{X^{\ast}, \psi} ( p_g)A\Big)(\tfrac{k}{N}e^{\ast}) \\
&\stackrel{\textrm{Sect. \ref{Gamma2I2}} \  \textrm{Case  $8'$}}{=}e^{-\tfrac{\pi i}{4}} \psi(\tfrac{1}{4}\tfrac{k}{N}) \tfrac{1}{G(1, \overline{\chi})}\sum_{k \bmod N}\overline{ \chi}(k)  J_{1/2}(\mathfrak{w} ,  z_g)\theta_{L, X^{\ast}}\Big(J_{1/2}(p_g,z_0)\pi_{X^{\ast}, \psi} ( p_g)A\Big)([\tfrac{k}{N} e^{\ast} -kN e+\tfrac{1}{2} e])\\
&=e^{-\tfrac{\pi i}{4}} \psi(\tfrac{1}{4}\tfrac{k}{N}) \tfrac{1}{G(1, \overline{\chi})}\sum_{k \bmod N}\overline{ \chi}(k)  J_{1/2}(\mathfrak{w} ,  z_g)\theta_{L, X^{\ast}}\Big(J_{1/2}(p_g,z_0)\pi_{X^{\ast}, \psi} ( p_g)A\Big)([\tfrac{k}{N} e^{\ast} +\tfrac{1}{2} e]) \psi(\tfrac{k^2}{2})\\
&=e^{-\tfrac{\pi i}{4}} \tfrac{1}{G(1, \overline{\chi})}\sum_{k \bmod N}\sum_{n\in \Z}\overline{ \chi}(k)  J_{1/2}(\mathfrak{w} ,  z_g) J_{1/2}(p_g,z_0)\pi_{X^{\ast}, \psi} ( p_g)A( \tfrac{1}{2}+n)\psi(\tfrac{kn}{N} ) \psi(\tfrac{k^2}{2}+\tfrac{k}{2N}) \\
&=e^{-\tfrac{\pi i}{4}} \tfrac{1}{G(1, \overline{\chi})}\sum_{k \bmod N}\sum_{n\in \Z}\overline{ \chi}(k)  J_{1/2}(\mathfrak{w} ,  z_g) J_{1/2}(p_g,z_0)\pi_{X^{\ast}, \psi} ( p_g)A( \tfrac{1}{2}+n)\psi(\tfrac{kn}{N} ) \psi(\tfrac{k}{2}+\tfrac{k}{2N}) \\
&= J_{1/2}(\mathfrak{w} ,  z_g) e^{-\tfrac{\pi i}{4}} \tfrac{1}{G(1, \overline{\chi})} \sum_{k \bmod N} \sum_{n\in \Z}J_{1/2}(p_g, z_0)\pi_{X^{\ast}, \psi} (p_g)A( \tfrac{1}{2}+n)\overline{ \chi}(k) \psi(\tfrac{k}{N}[n+\tfrac{N+1}{2}])\\
&= J_{1/2}(\mathfrak{w} ,  z_g) e^{-\tfrac{\pi i}{4}}   \tfrac{1}{G(1, \overline{\chi})}\sum_{k \bmod N} \sum_{n\in \Z} e^{i   \pi z_g (n+\tfrac{1}{2})^2 }\overline{ \chi}(k) \psi(\tfrac{k}{N}[n+\tfrac{N+1}{2}])\\
&=J_{1/2}(\mathfrak{w} ,  z_g) e^{-\tfrac{\pi i}{4}} \tfrac{1}{G(1, \overline{\chi})}\sum_{n\in \Z} e^{i   \pi z_g (n+\tfrac{1}{2})^2 }\Big(\sum_{k \bmod N}\overline{ \chi}(k)\psi(\tfrac{k}{N}[n+\tfrac{N+1}{2}])\Big)\\
&=J_{1/2}(\mathfrak{w} ,  z_g) e^{-\tfrac{\pi i}{4}}\sum_{n\in \Z} e^{i   \pi z_g (n+\tfrac{1}{2})^2 }   \chi(n+\tfrac{N+1}{2}).
 \end{align*}
2) Similarly, we have:\\
  \begin{align*}
 \theta^{3/2}_{\chi}(\mathfrak{w} g)&=J_{3/2}(\mathfrak{w} ,  z_g) e^{-\tfrac{\pi i}{4}}\sum_{n\in \Z} (n+\tfrac{1}{2})e^{i   \pi z_g (n+\tfrac{1}{2})^2 } \chi(n+\tfrac{N+1}{2}).
 \end{align*}
Let $s=\begin{pmatrix} a& b\\ c& d\end{pmatrix}\in \Gamma_{\theta}$,  $s'=[s,\lambda(s)^{-1}] \in  \overline{\Gamma}^{\mu_8}_{\theta} \subseteq  \overline{\SL}^{\mu_8}_2(\mathbb R)$, $r=\mathfrak{w}^{-1}s \mathfrak{w}=\begin{pmatrix} a-bN^2 & b \\ c+(a-d)N^2-bN^4 & d+bN^2 \end{pmatrix}$ and $r'=\overline{\mathfrak{w}}^{-1} s'\overline{\mathfrak{w}}$. 
\[ \mathfrak{w}^{-1} \omega\mathfrak{w}= \begin{pmatrix} N^2 & -1\\ 1+N^4 & -N^2 \end{pmatrix}\stackrel{(\bmod 2)}{\equiv } \begin{pmatrix} 1 & 1\\ 0&1 \end{pmatrix}, \quad \mathfrak{w}^{-1} \Gamma(2)\mathfrak{w}=\Gamma(2), \quad \mathfrak{w}^{-1}\Gamma_{\theta}\mathfrak{w}=\Gamma_0(2).\]

Following Sect. \ref{Examples13},  we have:
\begin{lemma}
\begin{itemize}
\item[(1)] $\nu_2(\mathfrak{w}^{-1}, g)=\nu_2(\omega , g)\nu_2(\omega,  g^{\omega^{-1}u(-N^2)})$, for $g\in \SL_2(\R)$.
\item[(2)] $\overline{\mathfrak{w}}^{-1}s'\overline{\mathfrak{w}}=[r, \nu_2(\mathfrak{w}^{-1},r)^{-1}\lambda^{\mathfrak{w}}(r)^{-1}]$.
\end{itemize}
\end{lemma}
\begin{proof}
1) $\nu_2(u(-N^2), g)\stackrel{\textrm{ Ex.} \ref{twocoverex}}{=} \overline{c}_{X^{\ast}}(u(-N^2)^{-1},g)\overline{c}_{X^{\ast}}(u(-N^2)^{-1}g,u(-N^2))\stackrel{\textrm{ Lem.} \ref{prog0}}{=} 1$, 
$ \nu_2(-I, g)= 1$. Note that $u_-(N^2)=\omega^{-1} u(-N^2)\omega$. Then: 
\begin{align*}
\nu_2(u_-(N^2), g)&=\nu_2(\omega^{-1} u(-N^2)\omega, g)\\
&\stackrel{\textrm{ Lem.} \ref{chap3twocover}}{=}\nu_2(\omega^{-1} , g) \nu_2(u(-N^2)\omega, g^{\omega^{-1}})\\
&=\nu_2(\omega^{-1} , g)\nu_2(u(-N^2),  g^{\omega^{-1}}) \nu_2(\omega,  g^{\omega^{-1}u(-N^2)})\\
&=\nu_2(\omega^{-1} , g)\nu_2(\omega,  g^{\omega^{-1}u(-N^2)})\\
&=\nu_2(\omega , g)\nu_2(\omega,  g^{\omega^{-1}u(-N^2)}).
\end{align*}
2) $\nu_2(u_-(-N^2), s)\stackrel{\textrm{ Ex.} \ref{twocoverex}}{=} \overline{c}_{X^{\ast}}(u_-(N^2),s)\overline{c}_{X^{\ast}}(u_-(N^2)s,u_-(-N^2))$.
\[1=\nu_2(u_-(-N^2)u_-(N^2), s)=\nu_2(u_-(-N^2), s)\nu_2(u_-(N^2), s^{u_-(-N^2)})=\nu_2(u_-(-N^2), s)\nu_2(u_-(N^2), r).\]
\begin{align*}
 \overline{\mathfrak{w}}^{-1}s'\overline{\mathfrak{w}}
&=[u_-(N^2) s u_-(-N^2), \overline{c}_{X^{\ast}}(u_-(N^2), s)\overline{c}_{X^{\ast}}(u_-(N^2) s, u_-(-N^2))\lambda({s})^{-1}]\\
&=[u_-(N^2)s u_-(-N^2), \nu_2(u_-(-N^2),s)\lambda({s})^{-1}]\\
&=[r, \nu_2(u_-(N^2),r)^{-1}\lambda({s})^{-1}].
\end{align*}
\end{proof}
\begin{lemma}\label{deltafN^2}
Let $r=\begin{pmatrix} a& b\\ c& d\end{pmatrix} \in \Gamma(2)$. Then $\lambda^{\mathfrak{w}}(r )=\lambda(r)v_2(\mathfrak{w}^{-1}, r)^{-1} \delta_F(r)$.
\end{lemma}
\begin{proof}
The proof is similar to that of Lem. \ref{deltaf}. Then  $\lambda(r^{\mathfrak{w}^{-1}})=\lambda(r)v_2(\mathfrak{w}^{-1}, r)^{-1} \delta'(r)$, for some character $\delta'$ of $\Gamma(2)$. 
 \begin{itemize}
\item[(I)] If $r=-I$, $\lambda(r^{\mathfrak{w}^{-1}})=\lambda(r) $ and 
\[v_2(\mathfrak{w}^{-1}, r)^{-1}\stackrel{ \textrm{ Ex.} \ref{twocoverex}}{=}\overline{c}_{X^{\ast}}(\mathfrak{w}, r\mathfrak{w}^{-1})^{-1}\overline{c}_{X^{\ast}}(r,\mathfrak{w}^{-1})^{-1}\stackrel{ \textrm{ Lem.} \ref{relacto}(1)}{=}1.\] So $\delta'(r)=1=\delta_F(r)$.
 \item[(II)]  If $r=u(2)$, $r^{\mathfrak{w}^{-1}}=u_-(-N^2)u(2) u_-(N^2)=\begin{pmatrix} 2N^2+1 & 2 \\ -2N^4 & 1-2N^2 \end{pmatrix}$, then: 
 $$\lambda(r^{\mathfrak{w}^{-1}})=\lambda(\begin{pmatrix} 2N^2+1 & 2 \\ -2N^4 & 1-2N^2 \end{pmatrix})\stackrel{ \textrm{ Lem.} \ref{gammar}}{=}(\frac{-4N^4}{1-2N^2}) i^{-1}=i, \quad 
 \lambda(r)\stackrel{ \textrm{ Lem.} \ref{gammar}}{=}1,$$ 
 $$ r^{\omega^{-1}u(-N^2)}=\begin{pmatrix} 1& N^2\\0& 1\end{pmatrix}\begin{pmatrix} 0& -1\\1& 0\end{pmatrix}\begin{pmatrix} 1& 2\\0& 1\end{pmatrix}\begin{pmatrix} 0& 1\\-1& 0\end{pmatrix}\begin{pmatrix} 1& -N^2\\0& 1\end{pmatrix}= \begin{pmatrix} 1-2N^2 & 2N^4 \\ -2 & 2N^2+1 \end{pmatrix},$$ 
  $$v_2(\mathfrak{w}^{-1}, r)^{-1}=\nu_2(\omega , r)^{-1}\nu_2(\omega,  r^{\omega^{-1}u(-N^2)})^{-1}\stackrel{ \textrm{ Lem.} \ref{nu2abcd}}{=}
  \nu_2(\omega, \begin{pmatrix} 1-2N^2 & 2N^4 \\ -2 & 2N^2+1 \end{pmatrix})^{-1} \stackrel{ \textrm{ Lem.} \ref{nu2abcd}}{=} 1.$$ 
  So $\delta'(r)=i=\delta_F(r)$.
 \item[(III)]  If $r=u_-(2)$, $r^{\mathfrak{w}^{-1}}=u_-(-N^2)u_-(2) u_-(N^2)=r$, then  $\lambda(r^{\mathfrak{w}^{-1}})=\lambda(r) $, 
 $$r^{\omega^{-1}u(-N^2)}=\begin{pmatrix} 1& N^2\\0& 1\end{pmatrix}\begin{pmatrix} 0& -1\\1& 0\end{pmatrix}\begin{pmatrix} 1&0\\2& 1\end{pmatrix}\begin{pmatrix} 0& 1\\-1& 0\end{pmatrix}\begin{pmatrix} 1& -N^2\\0& 1\end{pmatrix}=\begin{pmatrix} 1& -2\\0& 1\end{pmatrix},$$
 $$v_2(\mathfrak{w}^{-1}, r)^{-1}=\nu_2(\omega , r)^{-1}\nu_2(\omega,  r^{\omega^{-1}u(-N^2)})^{-1}\stackrel{ \textrm{ Lem.} \ref{nu2abcd}}{=}
  \nu_2(\omega,  \begin{pmatrix} 1& -2\\0& 1\end{pmatrix})^{-1} \stackrel{ \textrm{ Lem.} \ref{nu2abcd}}{=} 1.$$ 
  So $\delta'(r)=1=\delta_F(r)$.
  \end{itemize}
\end{proof}
\begin{lemma}
$\lambda^{\mathfrak{w}}(r)^{-1}\nu_2(\mathfrak{w}^{-1},r)^{-1}=\left\{ \begin{array}{rl} \lambda(r)^{-1} \delta_F(r)^{-1},  & \textrm{ if } r\in \Gamma(2),\\e^{-\tfrac{\pi i}{4}} \lambda(r_1)^{-1} \delta_F(r_1)^{-1}, &  \textrm{ if } r=r_1 u(1) \textrm{ with } r_1\in \Gamma(2).\end{array} \right.$ 
\end{lemma}
\begin{proof}
The first statement comes from the above lemma. If  $r=r_1 u(1)$, then:
\[v_2(\mathfrak{w}^{-1}, u(1))\stackrel{ \textrm{ Ex.} \ref{twocoverex}}{=}\overline{c}_{X^{\ast}}(\mathfrak{w}, u(1)\mathfrak{w}^{-1})\overline{c}_{X^{\ast}}(u(1),\mathfrak{w}^{-1})=\overline{c}_{X^{\ast}}(\begin{pmatrix} 1& 0\\-N^2& 1\end{pmatrix},\begin{pmatrix} N^2+1 & 1 \\ N^2 & 1 \end{pmatrix})=1;\]
\[\lambda(\mathfrak{w}u(1)\mathfrak{w}^{-1} )=\lambda(\begin{pmatrix} N^2+1 & 1 \\ -N^4 & 1-N^2 \end{pmatrix})\stackrel{ \textrm{(\ref{lambda})}}{=}m_{X^{\ast}}(\begin{pmatrix} N^2+1 & 1 \\ -N^4 & 1-N^2 \end{pmatrix})\widetilde{\beta}^{-1}(\begin{pmatrix} N^2+1 & 1 \\ -N^4 & 1-N^2 \end{pmatrix})=e^{\tfrac{\pi i}{4}};\]
\begin{align*}
&\lambda^{\mathfrak{w}}(r_1u(1))^{-1}\nu_2(\mathfrak{w}^{-1},r_1u(1))^{-1}\\
&=\overline{c}_{X^{\ast}}(r_1, u(1))\lambda^{\mathfrak{w}}(r_1)^{-1}\nu_2(\mathfrak{w}^{-1},r_1)^{-1}\lambda^{\mathfrak{w}}(u(1))^{-1}\nu_2(\mathfrak{w}^{-1},u(1))^{-1}\\
&=\lambda^{\mathfrak{w}}(r_1)^{-1}\nu_2(\mathfrak{w}^{-1},r_1)^{-1}\lambda^{\mathfrak{w}}(u(1))^{-1}\nu_2(\mathfrak{w}^{-1},u(1))^{-1}\\
&=\lambda(r_1)^{-1} \delta_F(r_1)^{-1}e^{-\tfrac{\pi i}{4}}.
\end{align*}
\end{proof}
As a consequence, we have:
$$\lambda^{\mathfrak{w}}(r)\nu_2(\mathfrak{w}^{-1},r)=\nu_F(r).$$
If $s\in \Gamma_{\theta}\cap \Gamma_0(N^2)$, we have: 
\[\chi^{\overline{\mathfrak{w}}}(r') = \chi(s')=\chi(s)=\chi(d)=\chi( bN^2+d ).\]
\begin{definition}
\begin{itemize}
\item $\theta^{F, 1/2}_{\chi}( z)\stackrel{ Def.}{=} \sum_{n\in \Z} e^{i   \pi z (n+\tfrac{1}{2})^2 } \chi(n+\tfrac{N+1}{2}).$
\item $\theta^{F, 3/2}_{\chi}( z)\stackrel{ Def.}{=}\sum_{n\in \Z} (n+\tfrac{1}{2})e^{i   \pi z(n+\tfrac{1}{2})^2 }\chi(n+\tfrac{N+1}{2}).$
\end{itemize}
\end{definition}
In this case, the result is the following:
\begin{proposition}\label{examthetaFN}
 $\theta^{F,\kappa/2}_{\chi}( rz)=\nu_{\theta^F}(r)(cz+d)^{\kappa/2}\chi(d)\theta^{F, \kappa/2}_{\chi}(z)$, for  $r=\begin{pmatrix} a& b\\ c & d\end{pmatrix}\in  \Gamma_0(2)\cap \Gamma_0(N^2)$, $\kappa=1$ or $3$.
\end{proposition}
\subsubsection{}\label{exf17} Let now $\mathfrak{w}=\begin{pmatrix} t& 0\\ 0 & 1\end{pmatrix}$ with positive integer and $\overline{\mathfrak{w}}=[\mathfrak{w}, 1]$.  Then:
\[\mathfrak{w}^{-1}\Gamma_0(2)\mathfrak{w}=\mathfrak{w}^{-1}\Gamma(2)\mathfrak{w} \sqcup \mathfrak{w}^{-1}\Gamma(2)\mathfrak{w} \begin{pmatrix} 1& t^{-1}\\ 0& 1\end{pmatrix}.\]
If $r\in \mathfrak{w}^{-1}\Gamma_0(2)\mathfrak{w}$, then $r'=\mathfrak{w}r\mathfrak{w}^{-1}\in \Gamma_0(2)$. We define:
\[\nu_{\theta^F,t}(r)=\nu_{\theta^F}(r').\]
If $r=\begin{pmatrix} a& b\\ c & d\end{pmatrix} \in \mathfrak{w}^{-1}\Gamma(2)\mathfrak{w} $, then $r'=\begin{pmatrix} a& tb\\ t^{-1}c & d\end{pmatrix} \in \Gamma(2)$. Hence,
\[\nu_{\theta^F,t}(r)=\nu_{\theta^F}(\mathfrak{w}r\mathfrak{w}^{-1})=  \left(\frac{2c/t}{d}\right)\epsilon^{-1}_d (\tfrac{-1}{d})^{bt/2}i^{bt/2}= \left(\frac{2ct}{d}\right)\epsilon^{-1}_d (\tfrac{-1}{d})^{bt/2}i^{bt/2}.\]
If $r\in \mathfrak{w}^{-1}\Gamma(2)\mathfrak{w} \begin{pmatrix} 1& t^{-1}\\ 0& 1\end{pmatrix}$, we write:
\[r=\begin{pmatrix} a& b\\ c & d\end{pmatrix} \begin{pmatrix} 1& t^{-1}\\ 0& 1\end{pmatrix}, \quad r'=\begin{pmatrix} a& tb\\ t^{-1}c & d\end{pmatrix} \begin{pmatrix} 1& 1\\0& 1\end{pmatrix}.\]
Hence,
\[\nu_{\theta^F,t}(r)=\nu_{\theta^F}(\mathfrak{w}r\mathfrak{w}^{-1})= \left(\frac{2ct}{d}\right)\epsilon^{-1}_d (\tfrac{-1}{d})^{bt/2}i^{bt/2}e^{\tfrac{\pi i}{4}}.\]
Let:
 \[r=\begin{pmatrix} a& b\\ c & d\end{pmatrix}\in  \mathfrak{w}^{-1}[\Gamma_0(2)\cap \Gamma_0(N^2)]\mathfrak{w}, \quad \quad r'=\mathfrak{w}r\mathfrak{w}^{-1}=\begin{pmatrix} a& tb\\ ct^{-1} & d\end{pmatrix}\in \Gamma_0(2)\cap \Gamma_0(N^2).\] Then:
\[\overline{\mathfrak{w}}^{-1}=[\mathfrak{w}^{-1}, 1], \quad \overline{\mathfrak{w}}^{-1}[r', \nu_{\theta^F}(r')^{-1}]\overline{\mathfrak{w}}=[\mathfrak{w}^{-1}r'\mathfrak{w}, \nu_{\theta^F}(r')^{-1}]=[r, \nu_{\theta^F}(\mathfrak{w}r\mathfrak{w}^{-1})^{-1}],\]
\[[\theta^{F, 1/2}_{\chi}]^{\overline{\mathfrak{w}}}( z)=t^{1/4}J_{1/2}(\overline{\mathfrak{w}},z)^{-1}  \theta^{F, 1/2}_{\chi}(tz)=\theta^{F, 1/2}_{\chi}(tz)= t^{1/4}\sum_{n\in \Z} e^{i   \pi z (n+\tfrac{1}{2})^2 } \chi(n+\tfrac{N+1}{2}),\]
\[[\theta^{F, 3/2}_{\chi}]^{\overline{\mathfrak{w}}}( z)=t^{3/4}J_{3/2}(\overline{\mathfrak{w}},z)^{-1}  \theta^{F, 3/2}_{\chi}(tz)=\theta^{F, 3/2}_{\chi}(tz)= t^{3/4}\sum_{n\in \Z} (n+\tfrac{1}{2})e^{i   \pi z(n+\tfrac{1}{2})^2 }\chi(n+\tfrac{N+1}{2}).\]
\begin{definition}
\begin{itemize}
\item $\nu_{\theta^F,t}(r)= \left\{ \begin{array}{lcl}\left(\frac{2ct}{d}\right)\epsilon^{-1}_d (\tfrac{-1}{d})^{bt/2}i^{bt/2}, & & r=\begin{pmatrix} a& tb\\ ct^{-1} & d\end{pmatrix}\in \Gamma(2),\\ \left(\frac{2ct}{d}\right)\epsilon^{-1}_d (\tfrac{-1}{d})^{bt/2}i^{bt/2}e^{\tfrac{i \pi}{4}}, & & r=\begin{pmatrix} a& tb\\ ct^{-1} & d\end{pmatrix}\begin{pmatrix} 1& 1\\0& 1\end{pmatrix}\in  \Gamma(2)u(1).\end{array}\right.$
\item $\theta^{F, 1/2}_{\chi,t}( z)\stackrel{ Def.}{=}\sum_{n\in \Z} e^{i   \pi z (n+\tfrac{1}{2})^2 } \chi(n+\tfrac{N+1}{2})$.
\item $\theta^{F,3/2}_{\chi,t}( z)\stackrel{ Def.}{=}\sum_{n\in \Z} (n+\tfrac{1}{2})e^{i   \pi z(n+\tfrac{1}{2})^2 }\chi(n+\tfrac{N+1}{2})$.
\end{itemize}
\end{definition}
Then the result is the following:
\begin{proposition}\label{examthetaFNt}
  $\theta^{F, \kappa/2}_{\chi,t}( rz)=\nu_{\theta^F,t}(r)(cz+d)^{\kappa/2}\chi(d)\theta^{F, 1/2}_{\chi,t}(z)$, for  $r=\begin{pmatrix} a& b\\ c & d\end{pmatrix}\in   \mathfrak{w}^{-1}[\Gamma_0(2)\cap \Gamma_0(N^2)]\mathfrak{w}$, $\kappa=1$ or $3$.
\end{proposition}

\subsection{Example 8: The discriminant function}\label{Examples18}
Let $\eta(z)$ denote the classical Dedekind eta function and $\Delta(z)$ the discriminant cusp form (see Zagier’s treatise \cite{Za}). We keep the notations $[\theta^{1/2}]^{\overline{M}_{q_i}}$ from Section  \ref{Examples12}.
\begin{theorem}\label{eta3}
\begin{itemize}
\item[(1)] $\eta(z)^3= \tfrac{1}{2} e^{\tfrac{\pi i}{4}}[\theta^{1/2}]^{\overline{M}_{q_1}}(z) \cdot [\theta^{1/2}]^{\overline{M}_{q_2}}(z) \cdot  [\theta^{1/2}]^{\overline{M}_{q_3}}(z)=\tfrac{1}{2} e^{\tfrac{\pi i}{4}}[\theta^{1/2}]^{\otimes}(z)$.
\item[(2)]  $\Delta(z) = \tfrac{1}{256}\Big[ [\theta^{1/2}]^{\overline{M}_{q_1}}(z) \cdot [\theta^{1/2}]^{\overline{M}_{q_2}}(z) \cdot  [\theta^{1/2}]^{\overline{M}_{q_3}}(z)\Big]^8$.
\end{itemize}
\end{theorem}
\begin{proof}
See \cite[p.29]{Za}. 
\end{proof}

Let $g\in G=\overline{\SL}_2(\Z)$, $H=\overline{\Gamma}_{\theta}$.  According to Thm. \ref{Twistedbypro2}, we have:
 \begin{proposition}
   For $g\in G$, $\eta^3(gz)=J_{1/2}(g,z)^3 \eta^3(z)  \overline{\lambda}(V_{G\to H}(g))$.
   \end{proposition}
   As a consequence, we have:
\begin{proposition}\label{Twistedbypro24}
For $g\in G=\overline{\SL}_2(\Z) $,  $\eta^{12}(gz)=J(g,z)^{6} \eta^{12}(z)  \overline{\lambda}^4(V_{G\to H}(g))$.
   \end{proposition}
   \begin{lemma}
   $\overline{\lambda}^4(V_{G\to H}(-))$ is a character of $\SL_2(\Z)$ determined by $\overline{\lambda}^4|_{\Gamma(2)}=1$ and $\overline{\lambda}^4=$ the sign character  on $\SL_2(\Z)/\Gamma(2)\simeq S_3$.
   \end{lemma}
   \begin{proof}
   By Lemma \ref{gammar}, $\overline{\lambda}^4(g)=1$ on $\overline{\Gamma}(2)$ and $\overline{\lambda}^4([\omega, 1])=-1$. So the result follows.    
   \end{proof}
   As a consequence, we have:
   \begin{proposition}\label{Twistedbypro242}
  $\eta^{12}(gz)=(cz+d)^{6} \eta^{12}(z) \sgn(g)$, for  $g=\begin{pmatrix} a& b\\ c& d\end{pmatrix}\in\SL_2(\Z) $, $\sgn: \SL_2(\Z)/\Gamma(2)\to \{\pm1\}$.
   \end{proposition}
It is comparable with the following  results from Ono \cite{Ono}.
   \begin{lemma}\label{cuspformsgamma26}
 \begin{itemize}
  \item[(1)] $\eta^{12}(z)$ is a cusp form of $\Gamma(2)$ of weight $6$.
  \item[(2)] $\eta^{12}(2z)$ is a cusp form of $\Gamma_0(4)$ of weight $6$.
  \end{itemize}
   \end{lemma}
   \begin{proof}
   See \cite[p.3]{Ono}.
   \end{proof}

\labelwidth=4em
\setlength\labelsep{0pt}
\addtolength\parskip{\smallskipamount}

\end{document}